\pdfoutput=1

\documentclass[12pt,a4paper,reqno]{amsart}
\usepackage[utf8]{inputenc}
\usepackage[T1]{fontenc}
\usepackage{indentfirst}
\usepackage{textcomp,lscape,rotating}
\usepackage[english]{babel}
\usepackage{enumerate}
\usepackage[colorlinks=true,citecolor=DarkOrchid,linkcolor=NavyBlue]{hyperref}
\usepackage{amsmath,amsfonts,amssymb,amsopn,amscd,amsthm}
\usepackage{mathrsfs,bbm}
\usepackage{stmaryrd}
\usepackage{graphicx}
\usepackage[dvipsnames]{xcolor}
\usepackage{tikz}
\usepackage{mathpazo}
\usepackage{todonotes}

\newtheorem{lemma}{Lemma}[section]
\newtheorem{theorem}[lemma]{Theorem}
\newtheorem{proposition}[lemma]{Proposition}
\newtheorem{corollary}[lemma]{Corollary}
\newtheorem{definition}[lemma]{Definition}

\theoremstyle{remark}
\newtheorem{example}[lemma]{Example}
\newtheorem{remark}[lemma]{Remark}

\def\lle{[\![}
\def\rre{]\!]}
\newcommand{\N}{\mathbb{N}}
\newcommand{\Z}{\mathbb{Z}}
\newcommand{\R}{\mathbb{R}}
\newcommand{\C}{\mathbb{C}}
\renewcommand{\AA}{\mathbb{A}}
\newcommand{\ST}{\mathrm{ST}}
\newcommand{\SF}{\mathrm{F}}
\newcommand{\E}{\mathrm{e}}
\newcommand{\I}{\mathrm{i}}
\newcommand{\esper}{\mathbb{E}}
\newcommand{\proba}{\mathbb{P}}
\newcommand{\comment}[1]{}
\newcommand{\eps}{\varepsilon}

\newcommand{\alphan}{\alpha_n}
\newcommand{\betan}{\beta_n}
\newcommand{\SubGraph}{X^{(n)}}
\newcommand{\RCV}{X^{(n)}_{\rho}}
\newcommand{\RCVk}{X^{(n)}_{(k)}}
\newcommand{\fall}[1]{^{\downarrow #1}}
\newcommand{\arr}{\mathfrak{A}}
\newcommand{\qym}{\mathfrak{Q}}
\newcommand{\sym}{\mathfrak{S}}
\newcommand{\pym}{\mathfrak{P}}
\newcommand{\pa}{\alpha}

\newcommand{\supp}{\mathrm{supp}}
\DeclareMathOperator\Var{Var}
\DeclareMathOperator\Set{\textsc{Set}}

\renewcommand{\Re}{\mathrm{Re}}
\newcommand{\GraphTwoLoops}{
\begin{tikzpicture}[scale=.7]
    \tikzstyle{vertex}=[circle,fill=black,inner sep=0pt,minimum size=2mm]
    \node (v1) at (0,0) [vertex] {};
    \draw (v1) .. controls +(150:8mm) and +(100:8mm) .. (v1);
    \draw (v1) .. controls +(80:8mm) and +(30:8mm) .. (v1);
\end{tikzpicture}}
\newcommand{\GraphOneLoop}{
\begin{tikzpicture}[scale=.7]
    \tikzstyle{vertex}=[circle,fill=black,inner sep=0pt,minimum size=2mm]
    \node (v1) at (0,0) [vertex] {};
    \draw (v1) .. controls +(125:8mm) and +(55:8mm) .. (v1);
\end{tikzpicture}}
\def\B{\mathbb{B}}
\def\ka{\kappa}
\renewcommand{\footnote}[1]{}

\title[Mod-$\phi$ convergence, I]{Mod-$\phi$ convergence, I: \\ 
Normality zones and precise deviations}
\author{Valentin F\'eray}
\address{Institut f\"ur Mathematik, Universit\"at Z\"urich --- Winterthurerstrasse 190, 8057 Z\"urich, Switzerland}
\email[V.~F\'eray]{valentin.feray@math.uzh.ch}
\author{Pierre-Lo\"ic M\'eliot}
\address{Laboratoire de Math\'ematiques, B\^atiment 425 --- Facult\'e des Sciences d'Orsay --- Universit\'e Paris-Sud, 91400 Orsay, France}
\email[P.-L.~M\'eliot]{pierre-loic.meliot@math.u-psud.fr}
\author{Ashkan Nikeghbali}
\address{Institut f\"ur Mathematik, Universit\"at Z\"urich --- Winterthurerstrasse 190, 8057 Z\"urich, Switzerland}
\email[A.~Nikeghbali]{ashkan.nikeghbali@math.uzh.ch}

\date{\today}

\begin{document}

\begin{abstract}
In this paper, we use the framework of mod-$\phi$ convergence to prove precise large or moderate deviations for quite general sequences of real valued random variables $(X_{n})_{n \in \N}$, which can be lattice or non-lattice distributed. We establish precise estimates of the fluctuations $\proba[X_{n} \in t_{n}B]$, instead of the usual estimates for the rate of exponential decay  $\log( \proba[X_{n}\in t_{n}B])$. Our approach provides us with a systematic way to characterise the normality zone, that is the zone in which the Gaussian approximation for the tails is still valid. Besides, the residue function measures the extent to which this approximation fails to hold at the edge of the normality zone.\medskip

The first sections of the article are devoted to a proof of these abstract results and comparisons with existing results. We then propose new examples covered by this theory and coming from various areas of mathematics: classical probability  theory, number theory (statistics of additive arithmetic functions), combinatorics (statistics of random permutations), random matrix theory (characteristic polynomials of random matrices in compact Lie groups), graph theory (number of subgraphs in a random Erd\H{o}s-R\'enyi graph), and non-commutative probability theory (asymptotics of random character values of symmetric groups). In particular, we complete our theory of precise deviations by a concrete method of cumulants and dependency graphs, which applies to many examples of sums of ``weakly dependent'' random variables. The large number as well as the variety of examples hint at a universality class for second order fluctuations. 
\end{abstract}

\maketitle

\clearpage

\tableofcontents

\clearpage

\section{Introduction}
\subsection{Mod-\texorpdfstring{$\phi$}{phi} convergence}
The notion of mod-$\phi$ convergence has been studied in the articles \cite{JKN11,DKN11,KN10,KN12,BKN09}, in connection with problems from number theory, random matrix theory and  probability theory. The main idea was to look for a natural renormalization of the characteristic functions of random variables which do not converge in law (instead of a renormalization of the random variables themselves). After this renormalization, the sequence of characteristic functions converges to some non-trivial limiting function. Here is the definition of mod-$\phi$ convergence that we will use throughout this article
(see Section~\ref{subsec:discussion_hypotheses} for a discussion on the different parts of this definition).

\begin{definition}\label{def:modphi}
    Let $(X_{n})_{n \in \N}$ be a sequence of real-valued random variables,
    and let us denote by \hbox{$\varphi_{n}(z)=\esper[\E^{zX_{n}}]$} their moment generating functions, which we assume to all exist in a strip 
    $$\mathcal{S}_{(c,d)}=\{z,\,\,c < \Re\, z <d\},$$
     with $c<d$ extended real numbers (we allow $c=-\infty$ and $d=+\infty$). We assume that there exists a non-constant infinitely divisible distribution $\phi$ with moment generating function $\int_{\R} \E^{zx}\, \phi(dx)=\exp(\eta(z))$ that is well defined on $\mathcal{S}_{(c,d)}$, and an analytic function $\psi(z)$ that does not vanish on the real part of $\mathcal{S}_{(c,d)}$, such that locally uniformly in $z \in \mathcal{S}_{(c,d)}$, 
\begin{equation}
\exp(-t_{n}\,\eta(z))\,\varphi_{n}(z) \to \psi(z),\label{eq:modphi}
\end{equation}
where $(t_{n})_{n \in \N}$  is some sequence going to $+\infty$. We then say that $(X_{n})_{n \in \N}$ \emph{converges mod-$\phi$} on $\mathcal{S}_{(c,d)}$, with parameters $(t_{n})_{n \in \N}$ and limiting function $\psi$. In the following we denote $\psi_{n}(z)$ the left-hand side of \eqref{eq:modphi}.
\end{definition}\medskip

\noindent When $\phi$ is the standard Gaussian (resp.~Poisson) distribution, we will speak of mod-Gaussian (resp.~mod-Poisson) convergence. 
Besides, unless explicitely stat\-ed, we shall always assume that $0$ belongs to the band of convergence $\mathcal{S}_{(c,d)}$, \emph{i.e.}, $c<0<d$. Under this assumption, Definition \ref{def:modphi} implies mod-$\phi$ convergence in the sense of \cite[Definition 1.1]{JKN11} or \cite[Section 2]{DKN11}.\bigskip    

It is immediate to see that mod-$\phi$ convergence implies a \emph{central limit theorem} if the sequence of parameters $t_{n}$ goes to infinity (see the remark after Theorem \ref{thm:cltlattice}). But in fact there is much more information encoded in mod-$\phi$ convergence than merely the central limit theorem. Indeed, mod-$\phi$ convergence appears as a natural extension of the framework of sums of independent random variables (see Example \ref{ex:sumiid}): many interesting asymptotic results  that hold for sums of independent random variables can also be established for sequences of random variables converging in the mod-$\phi$ sense (\cite{JKN11,DKN11,KN10, KN12, BKN09}). For instance, under some general extra assumptions on the convergence in Equation \eqref{eq:modphi}, it is proved in \cite{DKN11,KN12,FMN14} that one can establish local limit theorems for the random variables $X_n$. Then the local limit theorem of Stone appears as a special case of the local limit theorem for mod-$\phi$ convergent sequences.
But the latter also applies to a variety of situations where the random variables under consideration exhibit some dependence structure (\emph{e.g.} the Riemann zeta function on the critical line,  some probabilistic models of primes, the winding number for the planar Brownian motion, the characteristic polynomial of random  matrices, finite fields $L$-functions, \emph{etc.}). It is also shown in \cite{BKN09} that mod-Poisson convergence (in fact mod-$\phi$ convergence for $\phi$ a lattice distribution) implies very sharp distributional approximation in the total variation distance (among other distances) for a large class of random variables. 
In particular, the total number of distinct prime divisors $\omega(n)$ of an integer $n$ chosen at random can be approximated in the total variation distance with an arbitrary precision by explicitly computable measures. \bigskip

Besides these quantitative aspects, mod-$\phi$ convergence also sheds some new light on the nature of some conjectures in analytic number theory. Indeed it is shown in \cite{KN10} that the structure of the limiting function appearing in the moments conjecture for the Riemann zeta function by Keating and Snaith \cite{KS1} is shared by other arithmetic functions and that the limiting function $\psi$ accounts for the fact that prime numbers do not behave independently of each other. More precisely, the limiting function $\psi$ can be used to measure the deviation of the true result from what the  probabilistic models based on a naive independence assumption would predict. One should note that these naive probabilistic models are usually enough to predict central limit theorems for arithmetic functions (\emph{e.g.} the naive probabilistic model made with a sum of independent Bernoulli random variables to predict the Erd\"os-Kac central limit theorem  for $\omega(n)$ or the stochastic zeta function to predict Selberg's central limit theorem for the Riemann zeta function) but fail to predict accurately mod-$\phi$ convergence by a factor which is contained in $\psi$.
There is another example, where dependence appears in the limiting function $\psi$,
while we have independence at the scale of central limit theorem: the log of the characteristic polynomial of a random unitary matrix, as a vector in $\mathbb R^2$, converges in the mod-Gaussian sense to a limiting function which is not the product of the limiting functions of each component considered individually although when properly normalized it converges to a Gaussian vector with independent components \cite{KN12}. 
\bigskip

\subsection{Theoretical results}
The goal of this paper is to prove that the framework of mod-$\phi$ convergence as described in Definition~\ref{def:modphi} is suitable to obtain {\em precise}  large and moderate deviation results for the sequence $(X_n)_{n \in \N}$
(throughout the paper, we call {\em precise deviation result} an equivalent of the deviation probability itself and not on its logarithm). Namely, our results are the following.\vspace{2mm}

\begin{itemize}
\item We give equivalents for the quantity $\proba[X_n \ge t_n x]$, where $x$ is a fixed positive real number (see  Theorem~\ref{thm:mainlattice} for a lattice distribution $\phi$ and Theorem~\ref{thm:mainnonlattice} for a non-lattice distribution). This can be viewed as an analog (or an extension, see Section \ref{subsect:OnInfDiv}) of Bahadur-Rao theorem \cite{BR60}.\vspace{2mm}

\item We also consider probabilities of the kind $\proba[X_n \in t_n B]$ where $B$ is a Borelian set, and we give upper and lower bounds on this probability which coincide at first order for a {\em nice} Borelian set $B$, see Theorem~\ref{thm:hyperellis}. This result is an analog of Ellis-G\"artner theorem \cite[Theorem 2.3.6]{DZ98} (see also the original papers \cite{Gar77,Ellis84}): we have stronger hypotheses than in Ellis-Gärtner theorem, but also a more precise conclusion (the bounds involve the probability itself, not its logarithm).\vspace{2mm}

\item Besides, we give an equivalent for the probability $\proba[X_n -\esper[X_n] \ge s_n t_n]$, where $s_n=o(1)$, covering all intermediate scales between central limit theorem and deviations of order $t_n$ (Theorem~\ref{thm:cltlattice} in the lattice case, and Theorem~\ref{thm:cltnonlattice} in the non-lattice case).
\end{itemize}    
We also address the question of {\em normality zone}, {\it i.e.} the scale up to which the Gaussian approximation 
(coming from the central limit theorem) for the tail of the distribution of $X_n$ is valid.
In particular, our methods provide us with a systematic way to detect it and also explains how this approximation breaks at the edge of this zone; see Section \ref{sec:cumulantechnic}. 
The problem of detecting the normality zone for sums of i.i.d. random variables has received some attention in the literature on limit theorems (origninally in \cite{Cramer38}, see also \cite{IL71}).
Our framework enables an extension of such results, going beyond the setting of independent random variables:\vspace{1mm}
\begin{itemize}
\item we cover more situations, {\em e.g.}~sums of dependent random variables with a sparse dependency graph, or integer valued random variables, such as random additive functions, satisfying mod-Poisson convergence;\vspace{1mm}
\item we describe the correction to the normal approximation needed at the edge of the normality zone.
\end{itemize}

\noindent An interesting fact in our deviation results is the appearance of the limiting function $\psi$ in deviations at scale $t_n$. This means that, at smaller scales, a sequence $X_n$ converging mod-$\phi$ behaves exactly as a sum of $t_n$ i.i.d.~variables with distribution $\phi$. However, at scale $t_n$, this is not true any more and the limiting function $\psi$ gives us exactly the correcting factor.\medskip

In particular, in the case of mod-Gaussian convergence, the scale $t_n$ is the first scale where the equivalent given by the central limit theorem is not valid anymore. In this case, one often observes a symmetry breaking phenomenon
which is explained by the appearance of function $\psi$; see Section \ref{subsec:normalityzone}.
\bigskip

A special case of mod-Gaussian convergence is the case where \eqref{eq:modphi} is proved using bounds on the cumulants of $X_n$ --- see Section \ref{subsec:modgaussfromcumulants}. This case is particularly interesting as:\vspace{2mm}
\begin{itemize}
    \item it contains a large class of examples, see below in Section \ref{subsec:applications};\vspace{2mm}
    \item in this setting, one can obtain deviation results at a scale larger that $t_n$ (typically, $o((t_n)^{5/4})$, see Proposition \ref{prop:largedeviationscumulants}).\vspace{2mm}
\end{itemize}

The arguments involved in the proofs of our deviation results are standard, but they nonetheless need to be carefully adapted to the framework of Definition \ref{def:modphi}: elementary complex analysis, the method of change of probability measure or tilting due to Cram\'er, or adaptations of Berry-Esseen type inequalities with smoothing techniques. 
\bigskip

\begin{remark}
    \label{rem:hwang}
    We should here mention the work of Hwang \cite{Hwa96}, with some similarities with ours. Hwang works with hypotheses similar to Definition~\ref{def:modphi}, except that the convergence takes place uniformly on all compact sets contained in a given disk centered at the origin (while we assume convergence in a strip; thus this is weaker than our hypothesis, see Remark~\ref{rem:diskorstrip} for a discussion on this point).
    Under this hypothesis (and an hypothesis of the convergence speed), Hwang obtains an equivalent of the probability $\proba[X_n -\esper[X_n] \ge s_n t_n]$ with $s_n=o(1)$, and even gives some asymptotic expansion of this probability. However, Hwang does not give any deviation result at the scale $t_n$ and hence, none of his results show the role of $\psi$ in deviation probabilities. Besides, he has no results in the multi-dimensional setting.
\end{remark}
\bigskip

\subsection{Applications}\label{subsec:applications}
After proving our abstract results, we provide a large set of (new) examples where these results can be applied. 
We have thus devoted the second half of the paper to examples, from a variety of different areas.\bigskip

Section \ref{sec:firstexamples} contains examples where the moment generating function is explicit,
or given as a path integral of an explicit function.
%
%
First, in Section \ref{subsec:arithmetic},
we recover results of Radziwill \cite{Rad09} on precise large deviations for additive arithmetic functions,
by carefully recalling the principle of the Selberg-Delange method.
The next examples --- Sections \ref{subsec:cycles} and \ref{subsec:rises} --- involve the total number of cycles (resp. rises) for random permutations.
The precise large deviation result in the case of cycles was announced in a recent paper of Nikeghbali and Zeindler \cite{NZ13},
where the mod-convergence was proved by the singularity analysis method.
Finally, in Section \ref{subsec:rmtpol}, we compute deviation probabilities of the characteristic polynomial of random matrices in compact Lie groups.
This completes previous results by Hughes, Keating and O'Connell \cite{HKO01} on large deviations for the characteristic polynomial.
\medskip 

Surprisingly, mod-Gaussian convergence can also be established in some cases,
even if neither the moment generating function nor an appropriate bivariate generating series is known explicitly.
A first example of this situation is given in Section \ref{sec:zeros}.
We give a criterion based on the location of the zeroes of the probability generating function,
which ensures mod-Gaussian convergence. We then apply this result to the number of blocks in a uniform random set-partition of size $n$. As a consequence, we obtain the normality zone for this statistics,
refining the central limit theorem of Harper \cite{HarperCLTBlocks}.
\medskip

Our next examples lie in the framework in which mod-Gaussian convergence
is obtained {\em via} bounds on cumulants (Section \ref{subsec:modgaussfromcumulants}). 
In Section \ref{sec:depgraph}, we show that such bounds on cumulants typically arise 
for $X_n=\sum_{i=1}^{N_n} Y_i$,
where the $Y_i's$ have a {\em sparse dependency graph}
 (references and details are provided in Section \ref{sec:depgraph}).
With weak hypothesis on the second and third cumulants, this implies mod-Gaussian convergence of a renormalized version of $X_n$
(Theorem \ref{thm:modgaussiangeneralsparsegraph}).
This allows us to provide new examples of variables converging in the mod-Gaussian sense.\vspace{2mm}
%
%
\begin{itemize}
    \item First, we consider subgraph count statistics $X_\gamma$ in Erd\"os-R\'enyi random graph $G(n,p)$ (Theorem \ref{thm:moddevgraphs}) for a fixed $p$ between $0$ and $1$. 
        Moderate deviation probabilities in this case are given and compared with the literature on the subject in Section~\ref{sec:erdosrenyi}. We are also able to determine the size of the normality zone of $X_\gamma$.\vspace{2mm}

    \item In our last application in Section \ref{sec:central}, we use the machinery of depen\-den\-cy gra\-phs in non-commutative probability spaces, namely, the algebras $\C\sym(n)$ of the symmetric groups, all endowed with the restriction of a trace of the infinite symmetric group $\sym(\infty)$. The technique of cumulants still works and it gives the fluctuations of random integer partitions under so-called {\em central measures} in the terminology of Kerov and Vershik. Thus, one obtains a central limit theorem and moderate deviations for the values of the random irreducible characters of symmetric groups under these measures.\vspace{2mm}
\end{itemize}    
The variety of the many examples that fall in the seemingly more restrictive setting of mod-Gaussian convergence makes it tempting to assert that it can be considered as a universality class for second order fluctuations. 
\bigskip

\begin{remark}
The idea of using bounds on  cumulants to show moderate deviations for a family of random variable with some given dependency graph is not new --- see in particular \cite{DE12}. Nevertheless, the bounds we obtain in Theorem \ref{thm:dependencygraphs} (and also in Theorem \ref{thm:dependencygraphsrefined}) are stronger than those which were previously known and, as a consequence, we obtain deviation results at a larger scale. Another advantage of our method is that it gives estimates of the deviation probability itself,
and not only of its logarithm.
\end{remark}
\bigskip

\subsection{Forthcoming works}
As an intermediate step for our deviation estimates,
we give Berry-Esseen estimates for random variables 
that converge mod-$\phi$ (Proposition \ref{prop:berryesseen}).
These estimates are optimal for this setting,
though they can be improved in some special cases,
such as sums of independent or weakly dependent variables.
In a companion paper \cite{FMN14}, we establish optimal 
Berry-Esseen bounds in these cases,
providing a mod-$\phi$ alternative to Stein's method.
\medskip

In another direction, in \cite{Multidim},
we extend some results of this paper to a multi-dimensional framework.
This situation requires more care since the geometry of the Borel set $B$,
when considering $\mathbb P[X_n\in t_n B]$, plays a crucial role.
\bigskip

\subsection{Discussion on our hypotheses}
\label{subsec:discussion_hypotheses}
The following remarks explain the role of each hypothesis of Definition~\ref{def:modphi}. As we shall see later, some assumptions can be removed in some of our results (\emph{e.g.}, the infinite-divisibility of the reference law), but Definition \ref{def:modphi} provides a coherent setting where all the techniques presented in the paper do apply without further verification. 

\begin{remark}[Analyticity]
The existence of the relevant moment generating function on a strip is crucial in our proof, as we consider in Section \ref{sec:nonlattice} the Fourier transform of $\widetilde{X_n}$, obtained from $X_n$ by an exponential change of measure. We also use respectively the existence of continuous derivatives up to order $3$ for $\eta$ and $\psi$ on the strip $\mathcal{S}_{(c,d)}$, and the local uniform convergence of $\psi_n$ and its first derivatives (say, up to order $3$) toward those of $\psi$. By Cauchy's formula, the local uniform convergence of analytic functions imply those of their derivatives, so it provides a natural framework where convergence of derivatives are automatically verified.
\medskip

Let us mention however that these assumptions of analyticity are a bit restrictive, as they imply that the $X_{n}$'s and $\phi$ have moments of all order; in particular, $\phi$ cannot be any infinitely divisible distribution (for instance the Cauchy distribution is excluded). That explains that the theory of mod-$\phi$ convergence was initially developed with characteristic functions on the real line  rather than moment generating functions in a complex domain. With this somehow weaker hypothesis, one can find many examples for instance of mod-Cauchy convergence (see \emph{e.g.} \cite{DKN11,KNN13}), while the concept of mod-Cauchy convergence does not even make sense in the sense of Definition~\ref{def:modphi}. We are unfortunately not able to give precise deviation results  in this framework.
\end{remark}\medskip

\begin{remark}[Infinite divisibility and non-vanishing of the terms of mod-$\phi$ convergence]
The non-vanishing of $\psi$ is a natural hypothesis since evaluations of $\psi$ appear in many estimates of non-zero probabilities, and also in denominators in fractions, see for instance Lemma \ref{lem:exponchange}. 
The assumption that $\phi$ is an infinitely divisible distribution 
will be discussed in Section \ref{subsect:OnInfDiv}.
\end{remark}
\bigskip

\section*{Acknowledgements}
The authors would like to thank Martin Wahl for his input at the beginning of this project and for sharing with us his ideas. We would also like to address special thanks to Andrew Barbour, Reda Chhaibi and Kenny Maples for many fruitful discussions which helped us improve some of our arguments.
\bigskip
\bigskip

\section{Preliminaries}
\subsection{Basic examples of mod-convergence}
\label{subsec:basicexamples}
Let us give a few examples of mod-$\phi$ convergence, which will guide our intuition throughout the paper. In these examples, it will be useful sometimes to precise the speed of convergence in Definition \ref{def:modphi}.

\begin{definition}
We say that the sequence of random variables $(X_{n})_{n \in \N}$ converges mod-$\phi$ at speed $O((t_{n})^{-v})$ if the difference of the two sides of Equation \eqref{eq:modphi} can be bounded by $C_{K}\,(t_{n})^{-v}$ for any $z$ in a given compact subset $K$ of $\mathcal{S}_{(c,d)}$. We use the analogue definition with the $o(\cdot)$ notation.
\end{definition}\medskip

\begin{example}\label{ex:sumiid}
Let $(Y_{n})_{n \in \N}$ be a sequence of centered, independent and identically distributed real-valued random variables, with $\esper[\E^{zY}]=\esper[\E^{zY_{1}}]$ analytic and non-vani\-shing on a strip $\mathcal{S}_{(c,d)}$, possibly with $c=-\infty$ and/or $d=+\infty$. Set $S_{n}=Y_{1}+\cdots+Y_{n}$. If the distribution of $Y$ is infinitely divisible,
then $S_n$ converges mod-$Y$ towards the limiting function $\psi \equiv 1$ with parameter $t_n=n$.\medskip

But there is another mod-convergence hidden in this framework (we now drop the assumption of infinite divisibility of the law of $Y$). The cumulant generating series of $S_{n}$ is
$$\log \esper[\E^{zS_{n}}]=n \,\log \esper[\E^{zY}] = n \sum_{r = 2}^{\infty}\frac{\kappa^{(r)}(Y)}{r!}\,z^{r},$$
which is also analytic on $\mathcal{S}_{(c,d)}$ --- the coefficients $\kappa^{(r)}(Y)$ are the \emph{cumulants} of the variable $Y$. Let $v\geq 3$ be an integer such that $\kappa^{(r)}(Y)=0$ for each integer $r$ strictly between $3$ and $v-1$, and set $X_{n}=\frac{S_{n}}{n^{1/v}}$. It is always possible to take $v=3$, but sometimes we can also consider higher value of $v$, for instance $v=4$ as soon as $Y$ is a symmetric random random variable, and has therefore its odd moments and cumulants that vanish. One has
$$\log \varphi_{n}(z) = n^{\frac{v-2}{v}} \,\frac{\kappa^{(2)}(Y)}{2}\,z^{2} + \frac{\kappa^{(v)}(Y)}{v!}\,z^{v} + \sum_{r = v+1}^{\infty} \frac{\kappa^{(r)}(Y)}{r!\,n^{r/v-1}}\,z^{r},$$
and locally uniformly on $\C$ the right-most term is bounded by $\frac{C}{n^{1/v}}$. Consequently, 
$$\psi_{n}(z)=\exp\left(-n^{\frac{v-2}{v}}\,\frac{\sigma^{2}z^{2}}{2}\right)\,\varphi_{n}(z) \to \exp\left(\frac{\kappa^{(v)}(Y)}{v!}\,z^{v}\right) + O(n^{-1/v}),$$
that is, $(X_{n})_{n \in \N}$ converges in the mod-Gaussian sense with parameters $t_{n}=\sigma^{2}\,n^{\frac{v-2}{v}}$, speed $O(n^{-1/v})$ and limiting function $\psi(z)=\exp(\kappa^{(v)}(Y)\,z^{v}/v!)$.  Note that this first example was used in \cite{KNN13} to characterize the set of limiting functions in the setting of mod-$\phi$ convergence.
\bigskip

Through this article, we shall commonly rescale random variables in order to get estimates of fluctuations at different regimes. In order to avoid any confusion, we provide the reader with the following scheme, which details each possible scaling, and for each scaling, the regimes of fluctuations that can be deduced from the mod-$\phi$ convergence, as well as their scope. We also underline or frame the scalings and regimes that will be studied in this paper, and give references for the other kinds of fluctuations.

\begin{center}
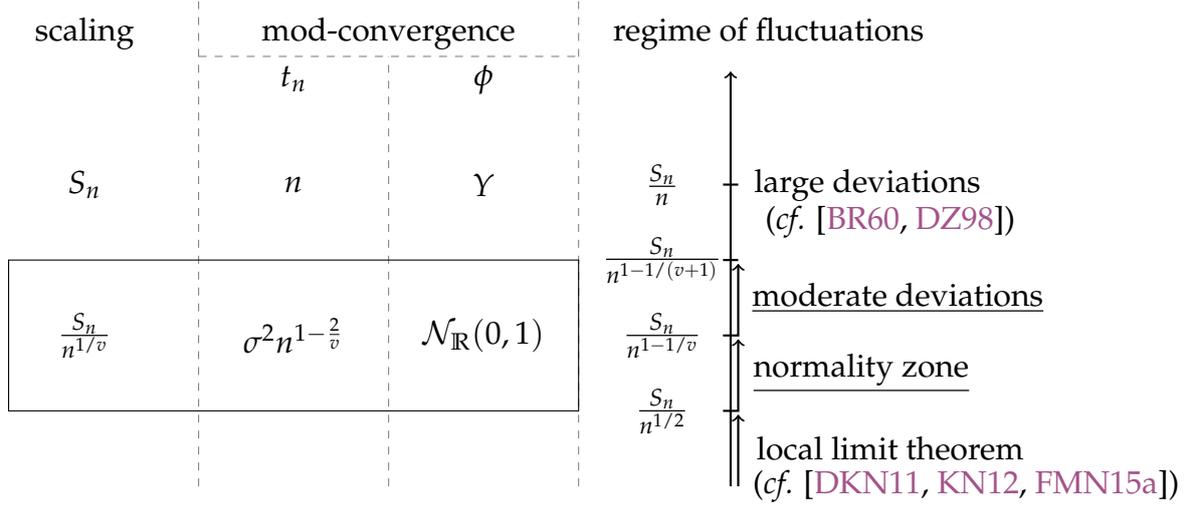
\begin{figure}[ht]
\begin{tikzpicture}
\draw (1,6) node {scaling};
\draw (1,4) node {$S_n$};
\draw (1,2) node {$\frac{S_n}{n^{1/v}}$};
\draw (5,6) node {mod-convergence};
\draw (3.75,5.4) node {$t_n$};
\draw (6.25,5.4) node {$\phi$};
\draw (3.75,4) node {$n$};
\draw (6.25,4) node {$Y$};
\draw (3.75,2) node {$\sigma^2 n^{1-\frac{2}{v}}$};
\draw (6.25,2) node {$\mathcal{N}_{\R}(0,1)$};
\draw[->,thick] (9.5,0) -- (9.5,5.5);
\draw[thick] (9.4,4) -- (9.6,4);
\draw (8.6,4) node {$\frac{S_n}{n}$};
\draw (8.6,3) node {$\frac{S_n}{n^{1-1/(v+1)}}$};
\draw (8.6,2) node {$\frac{S_n}{n^{1-1/v}}$};
\draw (8.6,1) node {$\frac{S_n}{n^{1/2}}$};
\draw (11.3,4) node {large deviations};
\draw (11.6,3.5) node {(\emph{cf.} \cite{BR60,DZ98})};
\draw (11.7,2.5) node {\underline{moderate deviations}};
\draw[thick,->] (9.4,2) -- (9.6,2) -- (9.6,2.95);
\draw[thick] (9.4,3) -- (9.6,3) ;
\draw (10,6) node {regime of fluctuations};
\draw[dashed,gray] (2.5,0) -- (2.5,6.5);
\draw[dashed,gray] (7.5,0) -- (7.5,6.5);
\draw[dashed,gray] (5,0) -- (5,5.7);
\draw[dashed,gray] (2.5,5.7) -- (7.5,5.7);
\draw[thick,->] (9.4,1) -- (9.6,1) -- (9.6,1.95);
\draw[thick,->] (9.6,0) -- (9.6,0.95) ;
\draw (11.22,1.5) node {\underline{normality zone}};
\draw (11.6,0.5) node {local limit theorem};
\draw (12.6,0.) node {(\emph{cf.} \cite{DKN11,KN12,FMN14})};
\draw (0,1) rectangle (7.5,3);
\end{tikzpicture}
\caption{Panorama of the fluctuations of a sum of $n$ i.i.d.~random variables.\label{fig:panoramaiid}}
\end{figure}
\end{center}
The content of this scheme will be fully explained in Section \ref{sec:nonlattice} (see in particular Section \ref{subsec:normalityzone}).
\end{example}\medskip

\begin{example}\label{ex:cycle}
Denote $X_{n}$ the number of disjoint cycles (including fixed points) of a random permutation chosen uniformly in the symmetric group $\mathfrak{S}(n)$. Feller's coupling (\emph{cf.} \cite[Chapter 1]{ABT03}) shows that $X_{n}=_{(\text{law})}\sum_{i=1}^{n}\mathcal{B}_{(1/i)},$
where $\mathcal{B}_{p}$ denotes a Bernoulli variable equal to $1$ with probability $p$ and to $0$ with probability $1-p$, and the Bernoulli variables are independent in the previous expansion. So,
$$\esper[\E^{zX_{n}}]=\prod_{i=1}^{n} \left(1+\frac{\E^{z}-1}{i}\right) = \E^{H_{n}(\E^{z}-1)}\,\prod_{i=1}^{n}\frac{1+\frac{\E^{z}-1}{i}}{\E^{\frac{\E^{z}-1}{i}}}$$
where $H_{n}=\sum_{i=1}^{n}\frac{1}{i}=\log n+\gamma+O(n^{-1})$. The Weierstrass infinite product in the right-hand side converges locally uniformly to an entire function, therefore (see \cite{WW27}),
$$\esper[\E^{zX_{n}}]\,\E^{-(\E^{z}-1)\log n} \to \E^{\gamma\,(\E^{z}-1)}\prod_{i=1}^{\infty}\frac{1+\frac{\E^{z}-1}{i}}{\E^{\frac{\E^{z}-1}{i}}}=\frac{1}{\Gamma(\E^{z})}$$
locally uniformly, \emph{i.e.}, one has mod-Poisson convergence with parameters $t_{n}=\log n$ and limiting function $1/\Gamma(\E^{z})$. Moreover, the speed of convergence is a $O(n^{-1})$, hence, a $o((t_{n})^{-v})$ for any integer $v$. We shall study generalizations of this example in Section \ref{subsec:cycles}.

\begin{center}
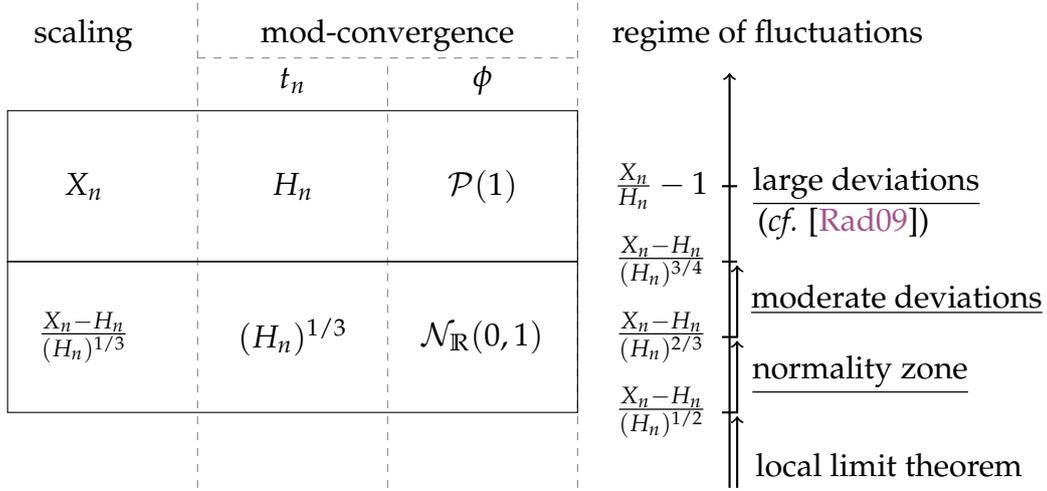
\begin{figure}[ht]
\begin{tikzpicture}
\draw (1,6) node {scaling};
\draw (1,4) node {$X_n$};
\draw (1,2) node {$\frac{X_n-H_n}{(H_n)^{1/3}}$};
\draw (5,6) node {mod-convergence};
\draw (3.75,5.4) node {$t_n$};
\draw (6.25,5.4) node {$\phi$};
\draw (3.75,4) node {$H_n$};
\draw (6.25,4) node {$\mathcal{P}(1)$};
\draw (3.75,2) node {$(H_n)^{1/3}$};
\draw (6.25,2) node {$\mathcal{N}_{\R}(0,1)$};
\draw[->,thick] (9.5,0) -- (9.5,5.5);
\draw[thick] (9.4,4) -- (9.6,4);
\draw (8.6,4) node {$\frac{X_n}{H_n}-1$};
\draw (8.6,3) node {$\frac{X_n-H_n}{(H_n)^{3/4}}$};
\draw (8.6,2) node {$\frac{X_n-H_n}{(H_n)^{2/3}}$};
\draw (8.6,1) node {$\frac{X_n-H_n}{(H_n)^{1/2}}$};
\draw (11.3,4) node {\underline{large deviations}};
\draw (11,3.5) node {(\emph{cf.} \cite{Rad09})};
\draw (11.7,2.5) node {\underline{moderate deviations}};
\draw[thick,->] (9.4,2) -- (9.6,2) -- (9.6,2.95);
\draw[thick] (9.4,3) -- (9.6,3) ;
\draw (10,6) node {regime of fluctuations};
\draw[dashed,gray] (2.5,0) -- (2.5,6.5);
\draw[dashed,gray] (7.5,0) -- (7.5,6.5);
\draw[dashed,gray] (5,0) -- (5,5.7);
\draw[dashed,gray] (2.5,5.7) -- (7.5,5.7);
\draw[thick,->] (9.4,1) -- (9.6,1) -- (9.6,1.95);
\draw[thick,->] (9.6,0) -- (9.6,0.95) ;
\draw (11.22,1.5) node {\underline{normality zone}};
\draw (11.6,0.3) node {local limit theorem};
\draw (0,1) rectangle (7.5,3);
\draw (0,3) rectangle (7.5,5);
\end{tikzpicture}
\caption{Panorama of the fluctuations of the number of cycles $X_n$ of a random permutation of size $n$.}
\end{figure}
\end{center}
Once again, there is another mod-convergence hidden in this example. Indeed, consider $Y_n=\frac{X_n-H_n}{(H_n)^{1/3}}$. Its generating function has asymptotics
$$\esper[\E^{zY_n}] = \E^{H_n\left(\E^{\frac{z}{(H_n)^{1/3}}}-1\right) - z(H_n)^{2/3}}\,(1+o(1)) = \E^{(H_n)^{1/3}\,\frac{z^2}{2}}\,\exp\left(\frac{z^3}{6}\right)\,(1+o(1)).$$
Therefore, one has mod-Gaussian convergence of $Y_n$ with parameters $t_n=(H_n)^{1/3}$ and limiting function $\exp(z^3/6)$.
\medskip

This is in fact a particular case of a more general phenomenon:
every sequence that converges mod-$\phi$
converges with a different rescaling in the mod-Gaussian sense.
\begin{proposition}
    Assume $X_n$ converges mod-$\phi$ with parameters $t_n$ and limiting function $\psi$,
    where $\phi$ is not the Gaussian distribution.
    Let 
    \[m = \min_{i \ge 3} \{i \mid \eta^{(i)}(0) \ne 0 \}.\]
    Then, the sequence of random variables $Y_n=(X_n-t_n \eta'(0))/(t_n)^{1/m}$ converges in the mod-Gaussian sense
    with parameters $(t_n)^{1-2/m} \eta''(0)$ towards the limiting function $\Psi(z)=\exp\big(\eta^{(m)}(0) z^m/m!)$.
\end{proposition}
\begin{proof}
    This follows from a simple computation
    \begin{align*}
    &\esper\! \left[  \exp \left( \frac{z (X_n-t_n \eta'(0))}{(t_n)^{1/m}}  \right)\right] \\
    &=\exp\!\left( \frac{-t_n \eta'(0))}{(t_n)^{1/m}} \right) \, \exp \!\left( t_n\, \eta\!\left(\frac{z}{(t_n)^{1/m}}\right) \right)\, \psi\!\left(\frac{z}{(t_n)^{1/m}}\right) (1+o(1)).
    \end{align*}
    The factor $\psi(\frac{z}{(t_n)^{1/m}})$ tends to $1$ and we do a Taylor expansion of $\eta(\frac{z}{(t_n)^{1/m}})$. We get
    \[\esper \!\left[  \exp \left( \frac{z (X_n-t_n \eta'(0))}{(t_n)^{1/m}}  \right)\right]             
        = \exp\!\left( (t_n)^{1-2/m} \eta''(0) \frac{z^2}{2} + \eta^{(m)}(0) \frac{z^m}{m!} + o(1) \right)( 1+o(1)). \qedhere\]
\end{proof}
Naturally, the mod-$\phi$ convergence gives more information than the implied mod-Gaussian convergence:
our deviation results --- Theorems \ref{thm:mainlattice} and \ref{thm:mainnonlattice} ---
for the former involve deviation probabilities of $X_n$ at scale $O(t_n)$,
while with the mod-Gaussian convergence, we get deviation probabilities of $Y_n$ at scale $O((t_n)^{1-2/m})$,
that is deviations of $X_n$ at scale $O((t_n)^{1-1/m})$.
\end{example}
\bigskip

\subsection{Legendre-Fenchel transforms}
We now present the definition and some properties of the Legendre-Fenchel transform, a classical tool in large deviation theory (see \emph{e.g.} \cite[Section 2.2]{DZ98}) that we shall use a lot in this paper. The Legendre-Fenchel transform is the following operation on (convex) functions:
\begin{definition}
The Legendre-Fenchel transform of a function $\eta$ is defined by:
$$F(x)=\sup_{h \in \R}\,( hx-\eta(h)).$$
This is an involution on convex lower semi-continuous functions. 
\end{definition}\medskip

Assume that $\eta$ is the logarithm of the moment generating series of a random variable. In this case, $\eta$ is a convex function (by H\"older's inequality). Then $F$ is always non-negative, and the unique $h$ maximizing $hx-\eta(h)$, if it exists, is then defined by the implicit equation $\eta'(h)=x$ (note that $h$ depends on $x$, but we have chosen not to write $h(x)$ to make notation lighter). This implies the following useful identities:
$$ F(x)=xh-\eta(h)\qquad;\qquad F'(x)=h\qquad;\qquad F''(x)=h'(x)=\frac{1}{\eta''(h)}.$$
\begin{example} If $\eta(z)=mz+\frac{\sigma^{2}z^{2}}{2}$ (Gaussian variable with mean $m$ and variance $\sigma^{2}$), then 
$$h=\frac{x-m}{\sigma^{2}} \qquad;\qquad F_{\mathcal{N}_{\R}(m,\sigma^{2})}(x)=\frac{(x-m)^{2}}{2\sigma^{2}}$$
whereas if $\eta(z)=\lambda(\E^{z}-1)$ (Poisson law with parameter $\lambda$), then
$$h=\log\frac{x}{\lambda}\qquad;\qquad F_{\mathcal{P}(\lambda)}(x)=\begin{cases} x\log\frac{x}{\lambda}-(x-\lambda) & \text{if }x>0,\\
+ \infty & \text{otherwise}\end{cases}.$$
\begin{center}
\vspace{-3mm}
\begin{figure}[ht]
\includegraphics{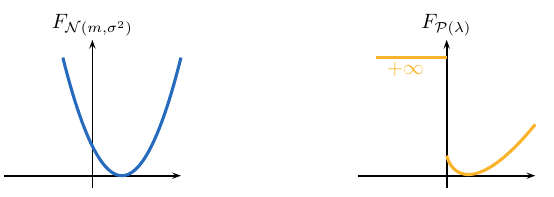}
\caption{The Legendre-Fenchel transforms of a Gaussian law and of a Poisson law.}
\end{figure}
\end{center}
\vspace{-5mm}
\end{example}
\bigskip

\subsection{Gaussian integrals}
Some computations involving the Gaussian density are used several times throughout the paper, so we decided to present them together here.

\begin{lemma}[Gaussian integrals]\label{lem:gaussintegral}~
\begin{enumerate}
\item \label{item:gaussmoment} moments: $$\frac{1}{\sqrt{2\pi}}\int_\R \E^{-\frac{x^2}{2}}\,x^{2m}\,dx=(2m-1)!!=(2m-1)(2m-3)\cdots 3\,1,$$ and the odd moments vanish.\vspace{2mm}
\item \label{item:gaussfourier} Fourier transform: with $g(x)=\frac{\E^{-\frac{x^2}{2}}}{\sqrt{2\pi}}$, one has
$$g^*(\zeta)=\int_\R g(x)\,\E^{\I x\zeta}\,dx=\E^{-\frac{\zeta^2}{2}}.$$
More generally, with the Hermite polynomials $H_{r}(x)=(-1)^{r}\,\E^{\frac{x^{2}}{2}}\,\,\frac{\partial^{r}}{\partial x^{r}}(\E^{-\frac{x^{2}}{2}})$, one has
$$(g\,H_{r})^{*}(\zeta)=(\I \zeta)^{r}\,\E^{-\frac{\zeta^{2}}{2}}.$$
\item \label{item:gausstail} tails: if $a \to +\infty$, then
\begin{align*}
&\int_{0}^{\infty} \E^{-\frac{(y+a)^{2}}{2}}\,dy=\frac{\E^{-\frac{a^{2}}{2}}}{a}\left(1-\frac{1}{a^{2}}+O\!\left(\frac{1}{a^{4}}\right)\right)\\
&\int_{0}^{\infty} y\,\E^{-\frac{(y+a)^{2}}{2}}\,dy=\frac{\E^{-\frac{a^{2}}{2}}}{a^{2}}\left(1+O\!\left(\frac{1}{a^{2}}\right)\right)\end{align*}
\begin{align*}
&\int_{0}^{\infty} y^{2}\,\E^{-\frac{(y+a)^{2}}{2}}\,dy=O\!\left(\frac{\E^{-\frac{a^{2}}{2}}}{a^{3}}\right)\\
&\int_{0}^{\infty} y^{3}\,\E^{-\frac{(y+a)^{2}}{2}}\,dy=O\!\left(\frac{\E^{-\frac{a^{2}}{2}}}{a^{2}}\right)
\end{align*}
In particular, the tail of the Gaussian distribution is $\frac{1}{\sqrt{2\pi}}\int_{a}^\infty\E^{-\frac{x^2}{2}}\,dx\simeq\frac{1}{a\sqrt{2\pi}}\,\E^{-\frac{a^2}{2}}$.
\vspace{1mm}
\item \label{item:gausscomplex} complex transform: for $\beta>0$,
$$
\int_{\R}\frac{\E^{-\frac{\beta^{2}}{2}}}{2\pi}\,\frac{\E^{-\frac{w^{2}}{2}}}{\beta+\I w}\,dw=\int_{\beta}^{\infty} \frac{\E^{-\frac{\alpha^{2}}{2}}}{\sqrt{2\pi}}\,d\alpha=\proba[\mathcal{N}_{\R}(0,1)\geq \beta].
$$
\vspace{1mm}
\end{enumerate}
\end{lemma}

\begin{proof}    
Recall that the generating series of Hermite polynomials (\cite[Chapter 5]{Sze75}) is 
$$\sum_{r =0}^{\infty} H_{r}(x)\,\frac{t^{r}}{r!}=\E^{\frac{x^2}{2}}\,\sum_{r=0}^\infty\frac{(-t)^r}{r!}\,\frac{\partial^r}{\partial x^r}\left(\E^{-\frac{x^2}{2}}\right)= \E^{\frac{x^2}{2}}\,\E^{-\frac{(x-t)^2}{2}}=\E^{-\frac{t^2}{2}+tx}.$$ 
Integrating against $g(x)\,\E^{\I x \zeta}\,dx$ yields
\begin{align*}\sum_{r=0}^\infty (g(x)\,H_r(x))^*(\zeta)\,\frac{t^r}{r!}&=\frac{1}{\sqrt{2\pi}}\int_\R \E^{-\frac{(x-t)^2}{2}+\I\zeta x}\,dx\\
&=\frac{\E^{\I\zeta t}}{\sqrt{2\pi}}\int_\R \E^{-\frac{y^2}{2}+\I\zeta y}\,dy = \E^{\I\zeta t-\frac{\zeta^2}{2}} = \sum_{r=0}^\infty (\I \zeta)^r\E^{-\frac{\zeta^2}{2}}\frac{t^r}{r!}
\end{align*}
whence the identity \eqref{item:gaussfourier} for Fourier transforms.\bigskip

With $r=0$, one gets the Fourier transform of the Gaussian $g^*(\zeta)=\E^{-\frac{\zeta^2}{2}}$, hence the moments \eqref{item:gaussmoment} by derivation at $\zeta=0$. The estimate of tails \eqref{item:gausstail} is obtained by an integration by parts; notice that similar techniques yield the tails of distributions $x^m\,\E^{-x^2/2}\,dx$ with $m\geq 1$. Finally, as for the complex transform \eqref{item:gausscomplex}, remark that 
$$F(\beta)=\int_{\R}\frac{\E^{-\frac{\beta^{2}}{2}}}{2\pi}\,\frac{\E^{-\frac{w^{2}}{2}}}{\beta+\I w}\,dw=\frac{1}{2\I\pi} \int_{\Gamma=\beta+ \I \R}\frac{\E^{\frac{(z-\beta)^{2}-\beta^{2}}{2}}}{z}\,dz,$$ the second integral being along the complex curve $\Gamma=\beta+\I \R$. 
By standard complex analysis arguments, this integral is the same along any line $\Gamma'=\beta'+i\R$ (for $\beta' >0$).
Namely
\[F(\beta)= \frac{1}{2\I\pi} \int_{\Gamma'=\beta'+ \I \R}\frac{\E^{\frac{(z-\beta)^{2}-\beta^{2}}{2}}}{z}\,dz.\]

Since $\lim_{\beta \to +\infty} F(\beta)=0$, 
$$
F(\beta)=-\int_\beta^\infty F'(\alpha)\,d\alpha
=\int_{\beta}^{\infty} \left(\frac{1}{2\I \pi}\int_{\Gamma'=\beta'+\I \R} \E^{\frac{(z-\alpha)^{2}-\alpha^{2}}{2}}\,dz\right)d\alpha.
$$
Again, the integration line $\Gamma'$ in the second integral can be replaced by 
$\Gamma=\alpha+\I \R$ and we get
$$F(\beta)
=\int_{\beta}^{\infty} \left(\frac{1}{2\I \pi}\int_{\Gamma=\alpha+\I \R} \E^{\frac{(z-\alpha)^{2}-\alpha^{2}}{2}}\,dz\right)d\alpha.
=\int_{\beta}^{\infty} \frac{\E^{-\frac{\alpha^{2}}{2}}}{\sqrt{2\pi}}\,d\alpha,$$
which is the tail $\proba[\mathcal{N}_{\R}(0,1)\geq \beta]$ of a standard Gaussian law.
\end{proof}
\noindent Also, there will be several instances of the Laplace method for asymptotics of integrals, but each time in a different setting; so we found it more convenient to reprove it each time.
\bigskip
\bigskip

\section{Fluctuations in the case of lattice distributions}\label{sec:lattice}

\subsection{Lattice and non-lattice distributions}
If $\phi$ is an infinitely divisible distribution, recall that its characteristic function writes uniquely as
\begin{equation}
\int_{\R} \E^{\I u x }\,\phi(dx)=\exp\left(\I m u - \frac{\sigma^2 u^2}{2}+\int_{\R \setminus \{0\}} \left(\E^{\I u x}-1-\frac{\I u x}{1+x^2}\right)\Pi(dx)\right),\label{eq:levykhintchine}
\end{equation}
where $\Pi$ is the L\'evy measure of $\phi$, and is required to integrate $1 \wedge x^2$ (see \cite[Chapter 13]{Kal97}). If $\sigma^2>0$, then $\phi$ has a normal component and its support 
$$\supp(\phi)=\big(\text{smallest closed subset $S$ of $\R$ with }\phi(S)=1\big)$$
is the whole real line, since $\phi$ can be seen as the convolution of some probability measure with a non-degenerate Gaussian law. Suppose now $\sigma^2=0$, and, set
$$\gamma=m-\int_{\R \setminus \{0\}} \frac{x}{1+x^2}\,\Pi(dx),$$
which is the shift parameter of $\phi$; note the integral above is not always convergent, so that $\gamma$ is not always defined.
\begin{lemma}\label{lem:support}
    \cite[Chapter 4, Theorem 8.4]{SVH04}
\begin{enumerate}
	\item If $\gamma$ is well-defined and finite, and if $\Pi([-\eps,\eps]\setminus \{0\})=0$ for some $\eps>0$, then 
	$$\supp(\phi)=\gamma + \overline{\N[\supp(\Pi)]},$$
	where $\N[S]$ is the semigroup generated by a part $S$ of $\R$ (the set of all sums of elements of $S$, including the empty sum $0$), and $\overline{\N[S]}$ is its closure.\vspace{2mm}
	\item Otherwise, the support of $\phi$ is either $\R$, or the half-line $[\gamma,+\infty)$, or the half-line $(-\infty,\gamma]$.
\end{enumerate}
\end{lemma}\bigskip

Recall that an additive subgroup of $\R$ is either discrete of type $\lambda \Z$ with $\lambda\geq 0$, or dense in $\R$. We call an infinitely divisible distribution {\em discrete}, or of type {\em lattice}, if $\sigma^2=0$, if $\gamma$ is well-defined and finite, and if the subgroup $\Z[\supp(\Pi)]$ is discrete. Otherwise, we say that $\phi$ is a non-lattice infinitely divisible distribution.

\begin{proposition}\label{prop:latticeandnonlattice}
An infinitely divisible distribution $\phi$ is of type lattice if and only if one of the following equivalent assertions is satisfied:\vspace{2mm}
\begin{enumerate}
    \item \label{item:supp} Its support is included in a set $\gamma+\lambda \Z$ for some parameters $\gamma$ and $\lambda> 0$. \vspace{2mm}
    \item \label{item:cf} For some parameter $\lambda> 0$, the characteristic function $\phi(\E^{\I u x})$ has modulus $|\phi(\E^{\I u x})|=1$ if and only if $u \in \frac{2\pi}{\lambda}\,\Z$.\vspace{2mm}
\end{enumerate}
If both hold and if $\lambda$ is chosen maximal in \eqref{item:supp}, then the parameters $\lambda$ in \eqref{item:supp} and \eqref{item:cf} coincide.
\medskip

\noindent Moreover, an infinitely divisible distribution $\phi$ is of type non-lattice if and only if $|\phi(\E^{\I u x})|<1$ for all $u \neq 0$.
\end{proposition}

\begin{proof}
In the following we exclude the case of a degenerate Dirac distribution $\phi=\delta_\gamma$, which is trivial. We can also assume that $\sigma^2=0$: otherwise, $\phi$ is of type non-lattice and with support $\R$, and the inequality $|\phi(\E^{\I u x})|<1$ for $u \neq 0$ is true for any non-degenerate Gaussian law, and therefore by convolution for every infinitely divisible law with parameter $\sigma^2\neq 0$.\bigskip

Suppose $\phi$ of type lattice. Then, since $\Z[\supp(\Pi)]=\lambda \Z$ for some $\lambda > 0$, the semigroup $\N[\supp(\Pi)] \subset \lambda \Z$ is discrete, and hence closed.  It thus follows from Lemma \ref{lem:support} that
$$\supp(\phi)=\gamma + \N[\supp(\Pi)]\subset \gamma + \lambda \Z.$$
Conversely, if $\supp(\phi)$ is included in a shifted lattice $\gamma+\lambda \Z$, then the second case of Lemma \ref{lem:support} is excluded, so $\gamma$ is well-defined and finite, and then 
$$\supp(\phi)=\gamma + \overline{\N[\supp(\Pi)]}.$$
But $\supp(\phi) \subset \gamma + \lambda \Z$, so this forces $\N[\supp(\Pi)]\subset \lambda \Z$, and therefore $\Z[\supp(\Pi)]\subset \lambda \Z$. Hence, $\phi$ is of type lattice. We have proved that the first assertion is indeed equivalent to the definition of a lattice infinitely divisible distribution.\bigskip

The equivalence of the two assertions \eqref{item:supp} and \eqref{item:cf} is then a general fact on probability measures $\phi$ on the real line. If $\phi$ is such a measure, let $G_\phi = \{u \in \R\,|\,|\phi(\E^{\I u x})| =1\}$. We claim that $G_\phi$ is an additive subgroup of $\R$. Indeed, if $u \neq 0$, then 
\begin{align*}
u \in G_{\phi} &\iff \left|\int_{\R}\E^{\I u x}\phi(dx)\right| = 1 \\
&\iff \text{the phase of $\E^{\I u x}$ is constant $\phi$-almost surely} \\
&\iff \phi \text{ is supported on a set }\gamma_u + \frac{2\pi}{u}\,\Z.
\end{align*}
Suppose that $u\neq 0$ and $v\neq 0$ belong to $G_\phi$. Then, 
$$\mathrm{supp}(\phi) \subset \left(\gamma_u + \frac{2\pi \Z}{u} \right)\cap \left(\gamma_v + \frac{2\pi \Z}{v} \right),$$
and the right-hand side of this formula is again a (shifted) discrete subgroup $\gamma_w + \frac{2\pi \Z}{w}$, with $\frac{k}{u}=\frac{l}{v}=\frac{1}{w}$ for some non-zero integers $k$ and $l$. In particular,
$$(k+l)w = kw+lw = u+v \quad;\quad \frac{1}{w} = \frac{k+l}{u+v}\quad;\quad \mathrm{supp}(\phi)\subset \gamma_w + \frac{2\pi\Z}{w} \subset \gamma_w + \frac{2\pi\Z}{u+v},$$
so $u+v \in G_\phi$ and $G_\phi$ is indeed a subgroup of $\R$.
\bigskip

If $G_\phi$ is discrete and writes as $p\Z$ with $p>0$, then $\phi$ is supported on a lattice $\gamma + \lambda \Z$ with $\lambda = \frac{2\pi}{p}$, and $|\phi(\E^{\I u})|=1$ if and only if $u \in \frac{2\pi\Z}{\lambda}$. Otherwise, $G_\phi$ cannot be a dense subgroup of $\R$, because then by continuity of $u \mapsto \phi(\E^{\I u})$, we would have $G_\phi=\R$, which implies that $\phi$ is a Dirac, and this case has been excluded. So, the only other possibility is $G_\phi=0$, which is the last statement of the proposition. \end{proof}
\bigskip

In the remaining of this section, we place ourselves in the setting of Definition \ref{def:modphi}, and we suppose that the $X_{n}$'s and the (non-constant) infinitely divisible distribution $\phi$ both take values in the lattice $\Z$, and furthermore, that $\phi$ has period $2\pi$ (in other words, the lattice $\Z$ is \emph{minimal} for $\phi$). In particular, for every $u \in (0,2\pi)$,
$|\exp(\eta(\I u))|<1,$
since by the previous discussion the period of the characteristic function of a $\Z$-valued infinitely divisible distribution is also the smallest $u > 0$ such that $|\phi(\E^{\I u x})|=1$. For more details on (discrete) infinitely-divisible distributions, we refer to the aforementioned textbook \cite{SVH04}, and also to \cite{Kat67} and \cite[Chapter XVII]{Fel71}.
\bigskip

\subsection{Deviations at the scale \texorpdfstring{$O(t_{n})$}{O(tn)}} 

\begin{lemma}\label{lem:fourier}
Let $X$ be a $\Z$-valued random variable whose generating function $\varphi_{X}(z)=\esper[\E^{zX}]$ converges absolutely in the strip $\mathcal{S}_{(c,d)}$, with $c < 0 <d$. For $k \in \Z$,
\begin{align*}
\forall h \in (c,d),\,\,\,\proba[X=k]&=\frac{1}{2\pi}\int_{-\pi}^{\pi} \E^{-k(h+\I u)}\,\varphi_{X}(h + \I u)\,du;\\
\forall h \in (0,d),\,\,\,\proba[X\geq k]&=\frac{1}{2\pi}\int_{-\pi}^{\pi} \frac{\E^{-k(h+\I u)}}{1-\E^{-(h+ \I u)}}\,\varphi_{X}(h + \I u)\,du.
\end{align*}
\end{lemma}

\begin{proof}
Since $$\varphi_{X}(h+\I u )=\sum_{k \in \Z} \proba[X=k]\,\E^{k(h+\I u )},$$ $\proba[X=k]\,\E^{kh}$ is the $k$-th Fourier coefficient of the $2\pi$-periodic and smooth function $u \mapsto \varphi_{X}(h+\I u)$; this leads to the first formula. Then, assuming also $h >0$,
$$ \proba[X \geq k]=\sum_{l=k}^{\infty} \proba[X=l] = \sum_{l=k}^{\infty} \frac{1}{2\pi}\int_{-\pi}^{\pi} \E^{-l(h+\I u)}\,\varphi_{X}(h + \I u)\,du, $$
and the sum of the moduli of the functions on the right-hand side is dominated by the integrable function $\frac{\E^{-kh}}{1-\E^{-h}}\,\varphi_{X}(h)$; so by Lebesgue's dominated convergence theorem, one can exchange the integral and the summation symbol, which yields the second equation.
\end{proof}
\bigskip

We now work under the assumptions of Definition \ref{def:modphi}, with a lattice infinitely divisible distribution $\phi$. Furthermore, we assume that the convergence is at speed $O((t_{n})^{-v})$, on a strip $\mathcal{S}_{(c,d)}$ containing $0$. Note that necessarily $\eta(0)=0$ and $\psi(0)=1$. A simple computation gives also the following approximation formulas:
\begin{align*}
    \esper(X_n)&=\varphi'_n(0)=t_n \eta'(0) + \psi'(0) \sim t_n \eta'(0) = t_n \eta'(0) +O(1); \\
    \Var(X_n)&=\varphi''_n(0)-\varphi'_n(0)^2 = t_n \eta''(0) +O(1).
\end{align*}

\begin{theorem}\label{thm:mainlattice}
Let $x$ be a real number in the interval $(\eta'(c),\eta'(d))$, and $h$ defined by the implicit equation $\eta'(h)=x$. We assume $t_{n}x \in \N$. \vspace{2mm}
\begin{enumerate}
\item The following expansion holds:
\begin{align*}
\proba[X_{n}=t_{n}x]&=\frac{\exp(-t_{n}F(x))}{\sqrt{2\pi t_{n}\eta''(h)}} \left(\psi(h)+\frac{a_{1}}{t_{n}}+\frac{a_{2}}{(t_{n})^{2}}+\cdots+\frac{a_{v-1}}{(t_{n})^{v-1}}+O\!\left(\frac{1}{(t_{n})^{v}}\right)\!\right)\\
&= \exp(-t_{n}F(x))\, \sqrt{\frac{F''(x)}{2\pi t_{n}}} \left(\psi(F'(x))+\frac{a_{1}}{t_{n}}+\cdots+\frac{a_{v-1}}{(t_{n})^{v-1}}+O\!\left(\frac{1}{(t_{n})^{v}}\right)\!\right),
\end{align*}
for some numbers $a_k$.\vspace{2mm}
\item Similarly, if $x$ is a real number in the range of $\eta'_{|(0,d)}$, then
$$\proba[X_{n}\geq t_{n}x]= \frac{\exp(-t_{n}F(x))}{\sqrt{2\pi t_{n}\eta''(h)}}\, \frac{1}{1-\E^{-h}}\left(\psi(h)+\frac{b_{1}}{t_{n}}+\cdots+\frac{b_{v-1}}{(t_{n})^{v-1}}+O\!\left(\frac{1}{(t_{n})^{v}}\right)\right),$$
for some numbers $b_k$.\vspace{2mm}
\end{enumerate}
Both $a_k$ and $b_k$ are rational fractions in the derivatives of $\eta$ and $\psi$ at $h$, that can be computed explicitly --- see Remark \ref{rem:computation_ak_bk}.
\end{theorem}\medskip 

\begin{proof}
With the notations of Definition \ref{def:modphi}, the first equation of Lemma \ref{lem:fourier} becomes
\begin{align}\proba[X_{n}=t_{n}x]&=\frac{1}{2\pi}\int_{-\pi}^{\pi} \E^{-t_{n}x\, (h+\I u)}\,\varphi_{n}(h + \I u)\,du\nonumber\\
&=\frac{1}{2\pi}\int_{-\pi}^{\pi} \E^{-t_{n}xh}\,\E^{t_{n}(\eta(h+\I u) - \I u x)}\,\psi_{n}(h + \I u)\,du \nonumber\\
&=\frac{ \E^{-t_{n}F(x)}}{2\pi}\int_{-\pi}^{\pi}\,\E^{t_{n}(\eta(h+\I u) - \eta(h)- \I u \eta'(h))}\,\psi_{n}(h + \I u)\,du.\label{eq:tobesplit}
\end{align}
The last equality uses the facts that $x h=F(x)+\eta(h)$ and $x=\eta'(h)$. We perform the Laplace method on \eqref{eq:tobesplit}, and to this purpose we split the integral in two parts. Fix $\delta >0$, and denote $q_{\delta}=\max_{u \in (-\pi,\pi)\setminus (-\delta,\delta)} |\exp(\eta(h+\I u)-\eta(h))|$. This is strictly smaller than $1$, since
$$\exp(\eta(h+\I u)-\eta(h))=\frac{\esper[\E^{(h+\I u)X}]}{\esper[\E^{hX}]} = \esper_{\mathbb{Q}}[\E^{\I u X}]$$
is the characteristic function of $X$ under the new probability $d\mathbb{Q}(\omega)=\frac{\E^{hX(\omega)}}{\esper[\E^{hX}]}d\mathbb{P}(\omega)$
(and $X$ has minimum lattice $\Z$). Note that Lemma \ref{lem:technicqdelta_lattice} hereafter is a more precise version of this inequality.
\bigskip

As a consequence, if $I_{(-\delta,\delta)}$ and $I_{(-\delta,\delta)^{\mathrm{c}}}$ denote the two parts of \eqref{eq:tobesplit} corresponding to $\int_{-\delta}^{\delta}$ and $\int_{-\pi}^{-\delta}+\int_{\delta}^{\pi}$, then
$$
|I_{(-\delta,\delta)^{\mathrm{c}}}|\leq \frac{ \E^{-t_{n}F(x)}}{2\pi} \int_{(-\delta,\delta)^{\mathrm{c}}} (q_{\delta})^{t_{n}} \,|\psi_{n}(h + \I u)|\,du \leq 2\,(\E^{-F(x)}\,q_{\delta})^{t_{n}} \max_{u \in (-\pi,\pi)} |\psi(h+\I u)|
$$
for $n$ big enough, since $\psi_{n}$ converges uniformly towards $\psi$ on the compact set $K=h+\I[-\pi,\pi]$. Since $q_{\delta}<1$, for any $\delta>0$ fixed, $I_{(-\delta,\delta)^{\mathrm{c}}}\,\E^{t_{n}F(x)}$ goes  to $0$ faster than any negative power of $t_{n}$, so $I_{(-\delta,\delta)^{\mathrm{c}}}$ is negligible in the asymptotics
(recall that $F(x)$ is non-negative by definition, as $\eta(0)=0$).\bigskip

As for the other part, we can first replace $\psi_{n}$ by $\psi$ up to a $(1+O((t_{n})^{-v}))$, since the integral is taken on a compact subset of $\mathcal{S}_{(c,d)}$.
We then set $u=\frac{w}{\sqrt{t_{n}\eta''(h)}}$:
\begin{equation}
I_{(-\delta,\delta)}=\frac{\E^{-t_{n}F(x)}\left(1+O\!\left(\frac{1}{(t_{n})^{v}}\right)\right)}{2\pi \sqrt{t_{n}\eta''(h)}}\int_{-\delta \sqrt{t_{n}\eta''(h)}}^{\delta  \sqrt{t_{n}\eta''(h)}} \psi\!\left(h+\frac{\I w}{ \sqrt{t_{n}\eta''(h)}}\right) \,\E^{t_{n}\Delta_n(w)-\frac{w^{2}}{2}}\,dw,
\label{eq:principal}
\end{equation}
where $\Delta_n(w)$ is the Taylor expansion
\begin{align*}
\eta\left(h+\I u\right)&-\eta(h)-\eta'(h) \left(\I u \right)-\frac{\eta''(h)}{2} \left(\I u \right)^{2}\\
&=\sum_{k=3}^{2v+1} \frac{\eta^{(k)}(h)}{k!}\left(\frac{\I w}{ \sqrt{t_{n}\eta''(h)}}\right)^{\!k}+O\!\left(\frac{1}{(t_{n})^{v+1}}\right)\\
&=\frac{1}{t_{n}} \left(-\frac{w^{2}}{\eta''(h)}\sum_{k=1}^{2v-1} \frac{\eta^{(k+2)}(h)}{(k+2)!}\left(\frac{\I w}{ \sqrt{t_{n}\eta''(h)}}\right)^{\!k}+O\!\left(\frac{1}{(t_{n})^{v}}\right)\right).
\end{align*}
We also replace $\psi$ by its Taylor expansion
$$\psi\!\left(h+\frac{\I w}{ \sqrt{t_{n}\eta''(h)}}\right)=\sum_{k=0}^{2v-1} \frac{\psi^{(k)}(h)}{k!}\left(\frac{\I w}{ \sqrt{t_{n}\eta''(h)}}\right)^{\!k}+O\!\left(\frac{1}{(t_{n})^{v}}\right).$$
Thus, if one defines $\alpha_k$ by the equation
\begin{align*}
f_{n}(w)&:=\left(\sum_{k=0}^{2v-1} \frac{\psi^{(k)}(h)}{k!}\left(\frac{\I w}{ \sqrt{t_{n}\eta''(h)}}\right)^{\!k}\right)\,\exp\!\left(-\frac{w^{2}}{\eta''(h)}\sum_{k=1}^{2v-1} \frac{\eta^{(k+2)}(h)}{(k+2)!}\left(\frac{\I w}{ \sqrt{t_{n}\eta''(h)}} \right)^{\!k}\right)\\
&=\sum_{k=0}^{2v-1} \frac{\alpha_{k}(w)}{(t_{n})^{k/2}}+O\!\left(\frac{1}{(t_{n})^{v}}\right),
\end{align*}
then one can replace $\psi(h+\I u) \,\E^{t_{n}\Delta_n(w)}$ by $f_{n}(w)$ in Equation \eqref{eq:principal}. Moreover, observe that each coefficient $\alpha_{k}(w)$ writes as
$$\alpha_{k}(w)=\alpha_{k,0}(h)\left(\frac{w}{\sqrt{\eta''(h)}}\right)^{k}+\alpha_{k,1}(h)\left(\frac{w}{\sqrt{\eta''(h)}}\right)^{k+2}+\cdots+\alpha_{k,r}(h)\left(\frac{w}{\sqrt{\eta''(h)}}\right)^{k+2r}$$
with the $\alpha_{k,r}(h)$'s polynomials in the derivatives of $\psi$ and $\eta$ at point $h$. So,
$$I_{(-\delta,\delta)}=\left(1+O\!\left(\frac{1}{(t_{n})^{v}}\right)\right)\frac{\E^{-t_{n}F(x)}}{\sqrt{2\pi t_{n}\eta''(h)}} \left(\sum_{k=0}^{2v-1} \int_{-\delta \sqrt{t_{n}\eta''(h)}}^{\delta  \sqrt{t_{n}\eta''(h)}} \frac{\alpha_{k}(w)}{(t_{n})^{k/2}}\,\frac{\E^{-\frac{w^{2}}{2}}}{\sqrt{2\pi}}\,dw\right).$$
For any power $w^{m}$,  
$$\left|\int_{-\infty}^{\infty} w^{m} \,\frac{\E^{-\frac{w^{2}}{2}}}{\sqrt{2\pi}}\,dw-\int_{-\delta\sqrt{t_{n}\eta''(h)}}^{\delta\sqrt{t_{n}\eta''(h)}} w^{m}\, \frac{\E^{-\frac{w^{2}}{2}}}{\sqrt{2\pi}}\,dw\right|$$ is smaller than any negative power of $t_{n}$ as $n$ goes to infinity (see Lemma \ref{lem:gaussintegral}, \eqref{item:gausstail} for the case $m=0$): indeed, by integration by parts, one can expand the difference as 
$\E^{-\delta^{2}\,t_{n}\eta''(h)/2}\,R_{m}(\sqrt{t_{n}}),$
where $R_{m}$ is a rational fraction that depends on $m,h,\delta$ and on the order of the expansion needed. Therefore, one can take the full integrals in the previous formula. On the other hand, the odd moments of the Gaussian distribution vanish. One concludes that
$$
\proba[X_{n}=t_{n}x]=\frac{\E^{-t_{n}F(x)}}{\sqrt{2\pi t_{n}\eta''(h)}} \left(\sum_{k=0}^{v-1} \frac{1}{(t_{n})^{k}} \left(\int_{\R} \alpha_{2k}(w)\,\frac{\E^{-\frac{w^{2}}{2}}}{\sqrt{2\pi}}\,dw\right)+O\!\left(\frac{1}{(t_{n})^{v}}\right)\right),$$
and each integral $\int_{\R} \alpha_{2k}(w)\,\frac{\E^{-\frac{w^{2}}{2}}}{\sqrt{2\pi}}\,dw$ is equal to 
$$
\frac{\alpha_{2k,0}(h) \,(2k-1)!!}{(\eta''(h))^{k}}+\cdots+\frac{\alpha_{2k,r}(h) \,(2k+2r-1)!!}{(\eta''(h))^{k+r}}
$$
where $(2m-1)!!$ is the $2m$-th moment of the Gaussian distribution (\emph{cf.} Lemma \ref{lem:gaussintegral}, \eqref{item:gaussmoment}). This ends the proof of the first part of our Theorem, the second formula coming from the identities $h=F'(x)$ and $\eta''(h)=\frac{1}{F''(x)}$. The second part is exactly the same, up to the factor
$$\frac{1}{1-\E^{-h-\I u}}=\frac{1}{1-\E^{-h}}\,\left(\frac{1-\E^{-h}}{1-\E^{-h-\frac{\I w}{\sqrt{t_{n}\eta''(h)}}}}\right)$$
in the integrals.
\end{proof}\medskip

\begin{remark}\label{rem:deviation}
For $x > \eta'(0)$, the first term of the expansion
$$\frac{\exp(-t_{n}F(x))}{\sqrt{2\pi t_{n}\eta''(h)}}$$ 
is the leading term in the asymptotics of $\proba[Y_{t_{n}}=t_{n}x]$, where $(Y_{t})_{t \in \R_{+}}$ is the L\'evy process associated to the analytic function $\eta(z)$. Thus, the residue $\psi$ measures the difference between the distribution of $X_{n}$ and the distribution of $Y_{t_{n}}$ in the interval $(t_{n}\eta'(0),t_{n}\eta'(d))$. 
\end{remark}

\begin{remark}
If the convergence is faster than any negative power of $t_{n}$, then one can simplify the statement of the theorem as follows: as formal power series in $t_{n}$,
$$ \sqrt{2\pi t_{n}\eta''(h)}\,\exp(t_{n}F(x))\,\proba[X_{n}=t_{n}x]=\int_{\R} f_{n}(w)\,\E^{-\frac{w^{2}}{2}}\,dw,$$
\emph{i.e.}, the expansions of both sides up to any given power $O\!\left(\frac{1}{(t_{n})^{v}}\right)$ agree.
\end{remark}

\begin{remark}\label{rem:computation_ak_bk}
As mentioned in the statement of the theorem, the proof also gives an algorithm to obtain formulas for $a_k$ and $b_k$. More precisely, denote  
\begin{align*}\Delta_{n}(w)&=t_{n}\left(\eta\!\left(h+\frac{\I w}{\sqrt{t_{n}\eta''(h)}}\right)-\eta(h)-\eta'(h)\,\frac{\I w}{\sqrt{t_{n}\eta''(h)}}+\frac{w^{2}}{2t_{n}}\right)\\
f_{n}(w)&=\psi \!\left(h+\frac{\I w}{\sqrt{t_{n}\eta''(h)}}\right)\,\exp(t_n \Delta_{n}(w))=\sum_{k=0}^{\infty} \frac{\alpha_{k}(w)}{(t_{n})^{k/2}},
\end{align*}
the last expansion holding in a neighborhood of zero. The coefficient $\alpha_{2k}(w)$ is an even polynomial in $w$ with valuation $2k$ and coefficients which are polynomials in the derivatives of $\psi$ and $\eta$ at $h$, and in $\frac{1}{\eta''(h)}$. Then, 
$$a_{k}=\int_{\R}\alpha_{2k}(w) \,\frac{\E^{-\frac{w^{2}}{2}}}{\sqrt{2\pi}}\,dw,$$
and in particular,
\begin{align*}
a_{0}&= \psi(h);\\ 
a_{1}&= -\frac{1}{2}\,\frac{\psi''(h)}{\eta''(h)}+\frac{1}{24}\,\frac{\psi(h)\,\eta^{(4)}(h)+4\,\psi'(h)\,\eta^{(3)}(h)}{(\eta''(h))^{2}}-\frac{15}{72}\,\frac{\psi(h)\,(\eta^{(3)}(h))^{2}}{(\eta''(h))^{3}}.
\end{align*}
the $b_{k}$'s are obtained by the same recipe as the $a_{k}$'s, but starting from the power series
$$g_{n}(w)=\frac{1-\exp(-h)}{1-\exp\!\left(-h- \frac{\I w}{\sqrt{t_{n}\eta''(h)}}\right)}\,f_{n}(w).$$
\end{remark}
\medskip


\begin{example}\label{ex:randomcycle}
Suppose that $(X_{n})_{n \in \N}$ is mod-Poisson convergent, that is to say that $\eta(z)=\E^{z}-1$. The expansion of Theorem \ref{thm:mainlattice} reads then as follows:
$$\proba[X_{n}=t_{n}x]=\frac{\E^{t_{n} (x-1-x\log x)}}{\sqrt{2\pi x t_{n} }}\left(\psi(h)+\frac{\psi'(h)-3\psi''(h)-\psi(h)}{6xt_{n}}+O\!\left(\frac{1}{(t_{n})^{2}}\right)\right) $$
with $h = \log x$. For instance, if $X_{n}$ is the number of cycles of a random permutation in $\mathfrak{S}(n)$, then the discussion of Example \ref{ex:cycle} shows that for $x > 0$ such that $x\log n \in \N$, 
$$\proba[X_{n}= x(\log n)] = \frac{n^{-(x\log x-x+1)}}{\sqrt{2\pi x\,\log n}}\,\frac{1}{\Gamma(x)}\big(1+O(1/\log n)\big).$$
Similarly, for $x > 1$ such that $x\log n \in \N$, one has 
$$\proba[X_{n}\geq x(\log n)] = \frac{n^{-(x\log x-x+1)}}{\sqrt{2\pi x\,\log n}}\,\frac{x}{x-1}\,\frac{1}{\Gamma(x)}\big(1+O(1/\log n)\big).$$
As the speed of convergence is very good in this case, precise expansions in $1/\log n$ to any order could  also be given.
\end{example}
\bigskip

\subsection{Central limit theorem at the scales \texorpdfstring{$o(t_{n})$}{o(tn)} and \texorpdfstring{$o((t_{n})^{2/3})$}{o(tn23)}}
The previous paragraph has described in the lattice case the fluctuations of $(X_n)_{n \in \N}$ in the regime $O(t_n)$, with a result akin to large deviations. In this section, we establish in the same setting an extended central limit theorem, for fluctuations of order up to $o(t_n)$. In particular, for fluctuations of order $o((t_n)^{2/3})$, we obtain the usual central limit theorem. Hence, we describe the panorama of fluctuations drawn on Figure \ref{fig:panoramalattice}.

\begin{center}
\begin{figure}[ht]
\begin{tikzpicture}
\draw (-0.5,1.5) node {order of fluctuations};
\draw (2.3,0) node {large deviations ($\eta'(0)<x$):};
\draw (1.75,-1) node {extended central limit};
\draw (3,-3) node {central limit theorem ($y \ll (t_n)^{1/6}$):};
\draw (2.7,-1.4) node {theorem ($(t_n)^{1/6}\lesssim y \ll (t_n)^{1/2}$):};
\draw (9.2,-0.05) node {$\proba[\frac{X_n}{t_n} \geq x] \simeq \frac{\exp(-t_n\,F(x))}{\sqrt{2\pi t_n\eta'(x)}}\,\frac{1}{1-\E^{-F'(x)}}\,\psi(F'(x));$};
\draw (9.2,-1.5) node {$\proba[\frac{X_n-t_n\eta'(0)}{\sqrt{t_n\,\eta''(0)}} \geq y] \simeq \frac{\exp(-t_n\,F(x))}{F'(x)\,\sqrt{2\pi t_n \eta'(x)}};$};
\draw (9.23,-3.8) node {$\proba[\frac{X_n-t_n\eta'(0)}{\sqrt{t_n\,\eta''(0)}} \geq y] \simeq \proba[\mathcal{N}_{\R}(0,1)\geq y].$};
\draw[->,thick] (-2,-5) -- (-2,1) ;
\draw[thick] (-2.1,0) -- (-1.9,0);
\draw (-1.25,0) node {$O(t_n)$} ;
\draw[->,thick] (-2.1,-2) -- (-1.9,-2) -- (-1.9,-0.1);
\draw (-0.85,-2) node {$O((t_n)^{2/3})$} ;
\draw[->,thick] (-2.1,-4) -- (-1.9,-4) -- (-1.9,-2.1) ;
\draw (-0.85,-4) node {$O((t_n)^{1/2})$} ;
\end{tikzpicture}
\caption{Panorama of the fluctuations of a sequence of random variables $(X_n)_{n\in\N}$ that converges modulo a lattice distribution (with $x=\eta'(0) + \sqrt{\eta''(0)/t_n}\,y$).}\label{fig:panoramalattice}
\end{figure}
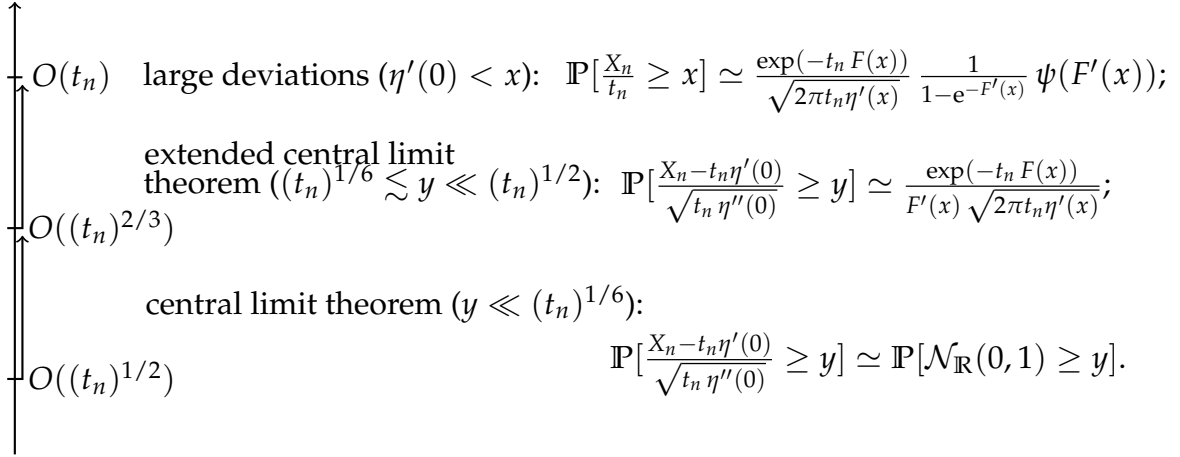
\end{center}

\begin{theorem}\label{thm:cltlattice}
Consider a sequence $(X_n)_{n \in \N}$ that converges mod-$\phi$, with a reference infinitely divisible law $\phi$ that is a lattice distribution. Assume $y=o((t_{n})^{1/6})$. Then,
$$\proba\!\left[X_{n}\geq t_{n}\eta'(0)+\sqrt{t_{n}\eta''(0)}\,y\right]=\proba[\mathcal{N}_{\R}(0,1)\geq y]\left(1+o(1)\right).$$
On the other hand, assuming $y \gg 1$ and $y=o((t_n)^{1/2})$, if $x=\eta'(0)+\sqrt{\eta''(0)/t_n}\,y$ and $h$ is the solution of $\eta'(h)=x$, then
\begin{align}
    \proba\!\left[X_{n}\geq t_{n}\eta'(0)+\sqrt{t_{n}\eta''(0)}\,y\right]
& = \frac{\E^{-t_n F(x)}}{h \sqrt{2\pi t_{n}\,\eta''(h)}}(1+o(1)) ; \nonumber\\
& = \frac{\E^{-t_n F(x)}}{y \sqrt{2\pi}} (1+o(1)).
\label{EqGaussianTailY}
\end{align}
\end{theorem}
\medskip

\begin{remark}
The case $y=O(1)$, which is the classical central limit theorem, follows immediately from the assumptions of Definition \ref{def:modphi}, since by a Taylor expansion around $0$ of $\eta$ the characteristic functions of the rescaled r.v. 
\[Y_{n}=\frac{X_{n}-t_{n}\eta'(0)}{\sqrt{t_{n}\eta''(0)}}\] 
converge pointwise to $\E^{-\frac{\zeta^{2}}{2}}$, the characteristic function of the standard Gaussian distribution.
In the first statement, the improvement here is the weaker assumption $y=o((t_{n})^{1/6})$.
\end{remark}\medskip

As we shall see, the ingredients of the proof are very similar to the ones in the previous paragraph. We start with a technical lemma of control of the module of the Fourier transform of the reference law $\phi$. 

\begin{lemma}\label{lem:technicqdelta_lattice}
Consider a non-constant infinitely divisible law $\phi$, of type lattice, and with convergent moment generating function $\int \E^{zx}\,\phi(dx) = \E^{\eta(z)}$ on a strip $\mathcal{S}_{(c,d)}$ with $c<0<d$. We assume without loss of generality that $\phi$ has minimal lattice $\Z$. Then, there exists a constant $D>0$ only depending on $\phi$, and an interval $(-\eps,\eps)\subset (c,d)$, such that for all $h \in (-\eps,\eps)$ and all $\delta$ small enough,
$$q_{\delta}= \max_{u \in [-\pi,\pi]\setminus (-\delta,\delta)} |\exp(\eta(h+\I u)-\eta(h))|$$ 
is smaller than $1-D\,\delta^{2}$.
\end{lemma}

\begin{proof}
Denote $X$ a random variable under the infinitely divisible distribution $\phi$. We claim that there exist two consecutive integers $n$ and $m=n-1$ with $\proba[X=n]\neq 0$ and $\proba[X=m]\neq 0$. Indeed, under our hypotheses, if $\Pi$ is the L\'evy measure of $\phi$, then 
$$\Z=\Z[\supp(\Pi)]=\N[\supp(\Pi)]-\N[\supp(\Pi)],$$ 
so there exist $a$ and $b$ in $\N[\supp(\Pi)]$ such that $b-a=1$. However, $\supp(\phi)=\gamma+\N[\supp(\Pi)]$ for some $\gamma \in \Z$, so $n=\gamma+b$ and $m=\gamma+a$ satisfy the claim.\bigskip

Now, we have seen that $\exp(\eta(h+\I u)-\eta(h))$ can be interpreted as the characteristic function of $X$ under the new probability measure $d\mathbb{Q}=\frac{\E^{hX}}{\esper[\E^{hX}]}\,d\proba$. So, for any $u$,
\begin{align*}
|\exp(\eta(h+\I u)-\eta(h))|^{2}&=\left|\esper_{\mathbb{Q}}[\E^{\I u X}]\right|^{2} = \sum_{n,m \in \Z} \mathbb{Q}[X=n]\,\mathbb{Q}[X=m]\,\E^{\I u (n-m)} \\
&= \sum_{k \in \Z}  \left(\sum_{n-m=k} \mathbb{Q}[X=n]\,\mathbb{Q}[X=m] \right)\cos ku.
\end{align*}
Fix two integers $n$ and $m=n-1$ such that $\proba[X=n]\neq 0$ and $\proba[X=m]\neq 0$. Then one also has $\mathbb{Q}[X=n]\neq 0$, $\mathbb{Q}[X=m]\neq 0$, and there exists
$D >0$ such that
$$\mathbb{Q}[X=n]\,\mathbb{Q}[X=m] \geq 15\,D>0$$
for $h$ small enough ($\mathbb{Q}$ tends to $\mathbb{P}$ for $h\to 0$).
As $\cos u \leq 1-\frac{u^{2}}{5}$ for all $u \in (-\pi,\pi)$, 
\begin{align*}
&|\exp(\eta(h+\I u)-\eta(h))|^{2} \leq 1+15\,D\,(\cos u-1) \leq 1-3\,D\,u^{2};\\ 
&q_{\delta} \leq \sqrt{1-3\,D\,\delta^{2}} \leq 1-D\,\delta^{2} \text{ for $\delta$ small enough}.
\end{align*}
This concludes the proof of the Lemma.\end{proof}
\medskip

\begin{proof}[Proof of Theorem \ref{thm:cltlattice}]
Notice that $\eta''(0)\neq 0$ since this is the variance of the law $\phi$, 
assumed to be non-trivial.
Set $x=\eta'(0)+s$, and assume $s =o(1)$. The analogue of Equation \eqref{eq:tobesplit} reads in our setting
\begin{equation}\proba[X_{n}\geq t_{n}x]=\frac{ \E^{-t_{n}F(x)}}{2\pi}\int_{-\pi}^{\pi}\,\frac{\E^{t_{n}(\eta(h+\I u) - \eta(h)- \I u \eta'(h))}}{1-\E^{-h-\I u}}\,\psi_{n}(h + \I u)\,du.\label{eq:tobesplit2}\end{equation}
Since $h'(x)=\frac{1}{\eta''(x)}$, one has $h=\frac{s}{\eta''(0)}+O(s^{2})$.
The same argument as in the proof of Theorem \ref{thm:mainlattice} shows that the integral over $(-\delta,\delta)^{\mathrm{c}}$ is bounded by $C\,\delta\,(q_{\delta})^{t_{n}},$ where $q_{\delta} <1$, and $C\,\delta$ (with $C$ a constant independent from $s$ and $\delta$) comes from the computation of 
$$\max_{u \in (-\delta,\delta)^{\mathrm{c}}}\left|\frac{\psi(h+\I u)}{1-\E^{-h-\I u}}\right|.$$
In the following we shall need to make $\delta$ go to zero sufficiently fast, but with $\delta\sqrt{t_{n}\eta''(0)}$ still going to infinity. Thus, set $\delta=(t_{n})^{-2/5}$, so that in particular $(t_{n})^{-1/2}\ll \delta \ll (t_{n})^{-1/3}$.
Notice that $I_{(-\delta,\delta)^\mathrm{c}}\,\,\E^{t_{n}F(x)}$ still goes to zero faster than any power of $t_{n}$; indeed, 
$$(q_{\delta})^{t_{n}}\leq \left(1-\frac{D}{(t_{n})^{4/5}}\right)^{t_{n}} \leq \E^{-D\,(t_{n})^{1/5}}$$
by Lemma \ref{lem:technicqdelta_lattice}. The other part of \eqref{eq:tobesplit2} is 
$$\frac{\E^{-t_{n}F(x)}}{2\pi \sqrt{t_{n}\eta''(h)}}\int_{-\delta \sqrt{t_{n}\eta''(h)}}^{\delta  \sqrt{t_{n}\eta''(h)}} \psi\!\left(h+\frac{\I w}{ \sqrt{t_{n}\eta''(h)}}\right) \,\E^{t_{n}\Delta_n(w)}\,\frac{\E^{-\frac{w^{2}}{2}}}{1-\E^{-h-\frac{\I w}{\sqrt{t_{n}\eta''(h)}}}}\,dw,$$
up to a factor $(1+o(1))$. Let us analyze each part of the integral:\vspace{2mm}
\begin{itemize}
\item The difference between $\psi\left(h+\frac{\I w}{ \sqrt{t_{n}\eta''(h)}}\right)$ and $\psi(0)$ is bounded by
$$\max_{z \in [-s,s]+\I[-\delta,\delta]}|\psi(z)-\psi(0)|=o(1)$$
by continuity of $\psi$, so one can replace the term with $\psi$ by the constant $\psi(0)=1$, up to factor $(1+o(1))$.\vspace{2mm}
\item The term $\Delta_n(w)$ has for Taylor expansion
$$\frac{\eta^{(3)}(h)}{6}\,\left(\frac{\I w}{\sqrt{t_{n}\eta''(h)}}\right)^{\!3}+O\left(\frac{1}{(t_{n})^{2}}\right),$$
so $t_{n}\,\Delta_n(w)$ is bounded by a $O(t_{n}\,\delta^{3})$, which is a $o(1)$ since $\delta\ll(t_{n})^{-1/3}$. So again one can replace $\E^{t_{n}\Delta_n(w)}$ by the constant $1$.\vspace{2mm}
\item The Taylor expansion of $\left(1-\E^{-h-\frac{\I w}{\sqrt{t_{n}\eta''(h)}}}\right)^{-1}$ is $\frac{1}{h+\frac{\I w}{\sqrt{t_{n}\eta''(h)}}}\,(1+o(1))$. Hence,
\begin{align*}\proba\left[X_{n}\geq t_{n}(\eta'(0)+s)\right]&= \frac{\E^{-t_{n}F(x)}}{2\pi} \left(\int_{\R} \frac{\E^{-\frac{w^{2}}{2}}}{\sqrt{t_{n}\eta''(h)} \,h+\I w} \,dw\right)\big(1+o(1)\big)\\
&=\E^{-t_{n}F(x)+\frac{h^{2}\,t_{n}\,\eta''(h)}{2}}\,\,\proba\!\left[\mathcal{N}_{\R}(0,1)\geq h\,\sqrt{t_{n}\eta''(h)}\right]\big(1+o(1)\big).
\end{align*}
Indeed, setting $\beta=h\sqrt{t_{n}\,\eta''(h)}$, this leads directly to the computation done in Lemma \ref{lem:gaussintegral}, \eqref{item:gausscomplex}.
\end{itemize}
Hence, we have shown so far that
\begin{equation}
\proba\left[X_{n}\geq t_{n}(\eta'(0)+s)\right]=\E^{-t_{n}F(\eta'(0)+s)}\,\E^{\frac{\beta^{2}}{2}}\,\proba[\mathcal{N}_{\R}(0,1)\geq \beta]\big(1+o(1)\big),\label{eq:forcomparisionwithradziwill}
\end{equation}
with $\beta=h\sqrt{t_{n}\,\eta''(h)}$.
\medskip

We now set $y=s\sqrt{t_{n}/\eta''(0)}=o(t_n^{1/2})$, and we consider the following regimes. If $y \gg 1$ (and \emph{a fortiori} if $y$ is of order bigger than $O(t_n)^{1/6}$), then  $s\gg (t_n)^{-1/2}$, so $h \gg (t_n)^{-1/2}$ and $\beta \gg 1$. We can then use Lemma \ref{lem:gaussintegral}, \eqref{item:gausstail} to replace in Equation \eqref{eq:forcomparisionwithradziwill} the function of $\beta$ by the tail-estimate of the Gaussian:
\begin{equation}
    \proba\left[X_n \geq t_n\eta'(0) + \sqrt{t_n\,\eta''(0)}y\right] 
= \frac{\E^{-t_n\,F(x)}}{h\,\sqrt{2\pi t_n\, \eta''(h)}}(1+o(1)).
\label{EqGaussianTail}
\end{equation}
Recall that $h=\frac{s}{\eta''(0)}\,(1+O(s))$, so that the denominator above can be approximated as follows:
\[h \sqrt{t_n\, \eta''(h)} =\frac{s}{\eta''(0)}\,(1+O(s))\, \sqrt{t_{n}\,(\eta''(0)+O(s))}= y\,(1+O(s))=y(1+o(1)).\]
This completes the proof of the second part of the theorem.
\medskip

Suppose on the opposite that $y=o((t_{n})^{1/6})$, or, equivalently, $s=o((t_{n})^{-1/3})$. Let us then see how everything is transformed.\vspace{2mm}
\begin{itemize}
\item By making a Taylor expansion around $\eta'(0)$ of the Legendre-Fenchel transform, we get (recall that $x=\eta'(0)$ implies $h=0$)
    \begin{equation}
        F(x)=F(\eta'(0))+F'(\eta'(0))\,s+\frac{F''(\eta'(0))}{2}\,s^{2}+O(s^{3})=\frac{y^{2}}{2t_{n}}+o((t_{n})^{-1}),
        \label{EqTaylorF}
    \end{equation}
so $\E^{-t_{n}F(\eta'(0)+s)}\simeq \E^{-\frac{y^{2}}{2}}$.  \vspace{2mm}
\item As above,
\[
\beta=h\sqrt{t_{n}\,\eta''(h)} = y\,(1+O(s))= y\,(1+o((t_{n})^{-1/3}))
\]
Consequently, $\beta^{2}=y^{2}(1+o((t_{n})^{-1/3}))=y^{2}+o(1)$, so $\E^{\frac{\beta^{2}}{2}}$ can be replaced safely by $\E^{\frac{y^{2}}{2}}$, which compensates the previous term.
\vspace{2mm}
\item Finally, fix $y$, and denote $F_{y}(\lambda)=\proba[\mathcal{N}_{\R}(0,1)\geq \lambda y]$. Then, for $|\lambda|$ say between $\frac{1}{2}$ and $2$,
$$|F_{y}'(\lambda)|=\left|\frac{y}{\sqrt{2\pi}}\,\E^{-\frac{\lambda^{2}y^{2}}{2}}\right|\leq \max_{y \in \R} \left|\frac{y}{\sqrt{2\pi}}\,\E^{-\frac{y^{2}}{8}}\right|=C<+\infty;$$
as a consequence,
\begin{align*}
\left|\proba[\mathcal{N}_{\R}(0,1)\geq \beta]-\proba[\mathcal{N}_{\R}(0,1)\geq y]\right|&=\left|F_{y}(1+o((t_{n})^{-1/3}))-F_{y}(1)\right|\\
&\leq \frac{C}{(t_{n})^{1/3}}=o(1).
\end{align*}
\end{itemize}
This ends the proof of Theorem \ref{thm:cltlattice}.
\end{proof}
\medskip

\begin{remark}
Equation \eqref{eq:forcomparisionwithradziwill} is the probabilistic counterpart of the number-theoretic results of \cite{Kub72,Rad09}, see in particular Theorems 2.1 and 2.2 in \cite{Rad09}. In Section \ref{subsec:arithmetic}, we shall explain how to recover the precise large deviation results of \cite{Rad09} for arithmetic functions whose Dirichlet series can be studied with the Selberg-Delange method.
\end{remark}\bigskip

The following corollary gives
a more explicit form of Theorem \ref{thm:cltlattice},
depending on the order of magnitude of $y$.
\begin{corollary}
    If $y=o((t_n)^{1/4})$, then one has
\begin{equation}
    \proba\!\big[X_n \ge t_n \eta'(0) + \sqrt{t_n \eta''(0)}\, y\big] 
    =\frac{ (1+o(1))}{y\sqrt{2\pi}} \,\E^{-\frac{y^2}{2}}\, \exp\left( \frac{\eta'''(0)}{6\,(\eta''(0))^{3/2}}\,\frac{y^3}{\sqrt{t_n}}
    \right) .
\label{EqCramerOrdre3}
\end{equation}
More generally, if $y=o((t_n)^{1/2-1/m})$, then one has
\begin{equation}
    \proba\!\big[X_n \ge t_n \eta'(0) + \sqrt{t_n \eta''(0)}\, y\big] 
    =\frac{(1+o(1))}{y\sqrt{2\pi}} \exp\left( - \sum_{i=2}^{m-1} \frac{F^{(i)}(\eta'(0))}{i!}\,\frac{(\eta''(0))^{i/2}\, y^i}{(t_n)^{(i-2)/2}}
    \right).
\label{EqCramerOrdreM}
\end{equation}
\end{corollary}
\begin{proof}
    As above in Equation \eqref{EqTaylorF}, we write $s=y\sqrt{\eta''(0)/t_{n}}$
    and $x=\eta'(0)+s$ and do a Taylor expansion of $F$ around $\eta'(0)$:
    \[F(x) = \sum_{i=0}^{m-1} \frac{F^{(i)}(\eta'(0))}{i!} \left( y\sqrt{\frac{\eta''(0)}{t_{n}}} \right)^i +O(s^{m}).\]
    Note that $F(\eta(0))=F'(\eta'(0))=0$.
    Because of the hypothesis $y=o((t_n)^{1/2-1/m})$, we have $t_n O(s^{m})=o(1)$.
    Therefore, plugging the equation above in Equation \eqref{EqGaussianTailY},
    we get \eqref{EqCramerOrdreM}.\medskip

    Observing that $F''(\eta'(0))=1/\eta''(0)$ and $F'''(\eta'(0))=\frac{-\eta'''(0)}{\eta''(0)^3}$,
    we get the first equation.
\end{proof}

To summarize, in the lattice case, mod-$\phi$ convergence implies a large deviation principle (Theorem \ref{thm:mainlattice}) and a precised central limit theorem (Theorem \ref{thm:cltlattice}), and these two results cover a whole interval of possible scalings for the fluctuations of the sequence $(X_n)_{n \in \N}$. As we shall see in the next Section \ref{sec:nonlattice}, the same holds for non-lattice reference distributions.
\bigskip
\bigskip

\section{Fluctuations in the non-lattice case}\label{sec:nonlattice}
In this section we prove the analogues of Theorems \ref{thm:mainlattice} and \ref{thm:cltlattice} when $\phi$ is not lattice-distributed; hence, by Proposition \ref{prop:latticeandnonlattice}, $|\E^{\eta(\I u)}|<1$ for any $u\neq 0$. In this setting, assuming $\phi$ absolutely continuous w.r.t.~the Lebesgue measure, there is a formula equivalent to the one given in Lemma \ref{lem:fourier}, namely,
\begin{equation}\proba[X\geq x]=\lim_{R \to \infty}\left(\frac{1}{2\pi}\int_{-R}^{R} \frac{\E^{-x(h+\I u)}}{h+\I u}\,\varphi_{X}(h+\I u)\,du\right)\label{eq:notuseful}
\end{equation}
if $\varphi_{X}(h)=\esper[\E^{hX}]<+\infty$ for $h>0$ (see \cite[Chapter XV, Section 3]{Fel71}). However, in order to manipulate this formula as in Section \ref{sec:lattice}, one would need strong additional assumptions of integrability on the characteristic functions of the random variables $X_{n}$. Thus, instead of Equation \eqref{eq:notuseful}, our main tool will be a Berry-Esseen estimate (see Proposition \ref{prop:berryesseen} hereafter), which we shall then combine with techniques of tilting of measures (Lemma \ref{lem:exponchange}) similar to those used in the classical theory of large deviations (see \cite[p. 32]{DZ98}).\bigskip

\subsection{Berry-Esseen estimates}
As explained above, we start by establishing some Berry-Esseen estimates in the setting of mod-$\phi$ convergence.

\begin{proposition}[Berry-Esseen expansion]\label{prop:berryesseen}
We place ourselves under the assumptions of Definition \ref{def:modphi}, with $\phi$ non-lattice infinitely divisible law, and the strip $\mathcal{S}_{(c,d)}$ that contains $0$. Denote 
$$g(y)=\frac{1}{\sqrt{2\pi}}\,\E^{-y^{2}/2}$$
the density of a standard Gaussian variable, and $F_{n}(x)=\proba[X_{n}\leq t_{n}\eta'(0)+\sqrt{t_{n}\eta''(0)}\,x]$. One has
$$F_{n}(x)=\int_{-\infty}^{x} \left(1+\frac{\psi'(0)}{\sqrt{t_n\eta''(0)}}\,y+\frac{\eta'''(0)}{6\sqrt{t_{n}(\eta''(0))^{3}}}\,(y^{3}-3y)\right)g(y)\,dy+o\!\left(\frac{1}{\sqrt{t_n}}\right)$$
with the $o(\cdot)$ uniform on $\R$.
\end{proposition}
\medskip

\begin{proof}
We use the same arguments as in the proof of \cite[Theorem XVI.4.1]{Fel71}, but adapted to the assumptions of Definition \ref{def:modphi}. Given an integrable function $f$, or more generally a distribution, its Fourier transform is $f^{*}(\zeta)=\int_{\R}\E^{\I \zeta x}\,f(x)\,dx.$
Consider a probability law $F(x)=\int_{-\infty}^{x}f(y)\,dy$ with vanishing expectation $(f^{*})'(0)=0$; and $G(x)=\int_{-\infty}^{x}g(y)\,dy$ a $m$-Lipschitz function with $g^{*}$ continuously differentiable and
$$(g^{*})'(0)=0 \qquad;\qquad \lim_{y \to -\infty}G(y)=0\qquad ;\qquad\lim_{y \to +\infty}G(y)=1.$$ 
Then Feller's Lemma \cite[Lemma XVI.3.2]{Fel71} states that, for any $x\in \R$ and any $T>0$,
$$|F(x)-G(x)|\leq \frac{1}{\pi} \int_{-T}^{T} \left|\frac{f^{*}(\zeta)-g^{*}(\zeta)}{\zeta}\right|\, d\zeta +\frac{24m}{\pi T}.$$
Notice that this is true even when $f$ is a distribution. Define the auxiliary variables
$$Y_{n}=\frac{X_{n}-t_{n}\eta'(0)}{\sqrt{t_{n}\eta''(0)}}$$
We shall apply Feller's Lemma to the functions
\begin{align*}
    F_{n}(x)&=\text{cumulative distribution function of } 
    Y_n;\\
G_{n}(x)&=\int_{-\infty}^{x}\left(1+\frac{\psi'(0)}{\sqrt{t_{n}\eta''(0)}}\,y+\frac{\eta'''(0)}{6\sqrt{t_{n}(\eta''(0))^{3}}}\,(y^{3}-3y)\right)g(y)\,dy.
\end{align*}
Note that each $G_n$ is clearly a Lipschitz function (with a uniform Lipschitz constant, \emph{i.e.} that does not depend on $n$). Besides, by Lemma \ref{lem:gaussintegral}, \eqref{item:gaussfourier}, the Fourier transform corresponding to the distribution function $G_n$ is, setting $z=\I\, \zeta$,
\begin{equation}\label{eq:gnstar}
    g_n^*(\zeta)= \E^{\frac{z^{2}}{2}}\,\left(1+\frac{\psi'(0)\,z}{\sqrt{t_{n}\eta''(0)}}+\frac{\eta'''(0)\,z^{3}}{6\sqrt{t_{n}(\eta''(0))^{3}}}\right).
\end{equation}
Consider now $f_{n}^{*}(\zeta)$:
if $z=\I\, \zeta$, then
\begin{align*}
    f_{n}^{*}(\zeta) &= \esper \left[\E^{z\left(\frac{X_{n}-t_{n}\eta'(0)}{\sqrt{t_{n}\eta''(0)}}\right)}\right]=\exp\left(-z\sqrt{\frac{t_{n}}{\eta''(0)}}\,\eta'(0)\right)\times \varphi_{n}\left(\frac{z}{\sqrt{t_{n}\eta''(0)}}\right)\\
&=\exp\left(t_{n}\left(\eta\left(\frac{z}{\sqrt{t_{n}\eta''(0)}}\right)-\eta'(0)\,\frac{z}{\sqrt{t_{n}\eta''(0)}}\right)\right)\times \psi_{n}\!\left(\frac{z}{\sqrt{t_{n}\eta''(0)}}\right)
\end{align*}
But
\[ \psi_{n}\!\left(\frac{z}{\sqrt{t_{n}\eta''(0)}}\right) = \left(1+\frac{\psi'_n(0)\,z}{\sqrt{t_{n}\eta''(0)}}+o\!\left(\frac{z}{\sqrt{t_{n}}}\right)\right) = \left(1+\frac{\psi'(0)\,z}{\sqrt{t_{n}\eta''(0)}}+o\!\left(\frac{z}{\sqrt{t_{n}}}\right)\right)\]
where the $o$ is uniform in $n$ because of the local uniform convergence of the analytic functions $\psi_n$ to $\psi$ (and hence, of $\psi'_n$ and $\psi''_n$ to $\psi'$ and $\psi$). Thus
\begin{align}
    f_{n}^{*}(\zeta)&=\exp\left(\frac{z^{2}}{2}+\frac{\eta'''(0)\,z^{3}}{6\sqrt{t_{n}(\eta''(0))^{3}}}+|z|^{2}\,o\!\left(\frac{z}{\sqrt{t_{n}}}\right)\right)\times \left(1+\frac{\psi'(0)\,z}{\sqrt{t_{n}\eta''(0)}}+o\!\left(\frac{z}{\sqrt{t_{n}}}\right)\right)\nonumber\\
&=\E^{\frac{z^{2}}{2}}\,\left(1+\frac{\psi'(0)\,z}{\sqrt{t_{n}\eta''(0)}}+\frac{\eta'''(0)\,z^{3}}{6\sqrt{t_{n}(\eta''(0))^{3}}}+(1+|z|^{2})\,o\!\left(\frac{z}{\sqrt{t_{n}}}\right)\right).\label{eq:asymptoticsfouriertransforms}
\end{align}
Beware that in the previous expansions, the $o(\cdot)$ is 
$$o\!\left(\frac{z}{\sqrt{t_{n}}}\right)=\frac{|z|}{\sqrt{t_{n}}}\,\eps\!\left(\frac{z}{\sqrt{t_{n}}}\right)\quad \text{with } \lim_{t \to 0}\eps(t)=0.$$ 
In particular, $z$ might still go to infinity in this situation. To make everything clear we will continue to use the notation $\eps(t)$ in the following. Fix $0<\delta<\Delta$ and take $T=\Delta\sqrt{t_{n}}$.  Comparing~\eqref{eq:gnstar} and \eqref{eq:asymptoticsfouriertransforms} and using Feller's lemma, we get:
\begin{align}
|F_{n}(x)-G_{n}(x)|&\leq\frac{1}{\pi}\int_{-\Delta\sqrt{t_{n}}}^{\Delta\sqrt{t_{n}}} \left|\frac{f_{n}^{*}(\zeta)-g_{n}^{*}(\zeta)}{\zeta}\right|d\zeta + \frac{24m}{\Delta\pi\sqrt{t_{n}}}\nonumber\\
&\leq \frac{1}{\pi\sqrt{t_{n}}}\int_{-\delta\sqrt{t_{n}}}^{\delta\sqrt{t_{n}}} \E^{-\frac{\zeta^{2}}{2}}\,(1+|\zeta|^{2})\,\, \eps\!\left(\frac{\zeta}{\sqrt{t_{n}}}\right)\,d\zeta+\frac{24m}{\Delta\pi\sqrt{t_{n}}}\nonumber\\
&\quad+\frac{1}{\pi\delta\sqrt{t_{n}}}\int_{[-\Delta\sqrt{t_{n}},\Delta\sqrt{t_{n}}]\setminus[-\delta\sqrt{t_{n}},\delta\sqrt{t_{n}}]}\left|f_{n}^{*}(\zeta)-g_{n}^{*}(\zeta)\right|d\zeta.
\label{ineq:FmoinsG}
\end{align}
In the right-hand side, the first part is of the form $\frac{\eps'(\delta)}{\sqrt{t_{n}}}$ when 
$\lim_{\delta \to 0} \eps'(\delta)=0$,
while the second part is smaller than $\frac{M}{\Delta\sqrt{t_{n}}}$ for some constant $M$.\bigskip

Let us show that the last integral goes to zero faster than any power of $t_{n}$. Indeed, for $|\zeta| \in [\delta\sqrt{t_{n}},\Delta\sqrt{t_{n}}]$,
$$|f_{n}^*(\zeta)|=\left|\varphi_{n}\left(\frac{\I \zeta}{\sqrt{t_{n}\eta''(0)}}\right)\right|\leq
 \left|\psi_n \!\left(\frac{\I \zeta}{\sqrt{t_{n}\eta''(0)}}\right) \right|
\times
 \left| \exp \!\left(t_n\, \eta\! \left(\frac{\I \zeta}{\sqrt{t_{n}\eta''(0)}}\right) \right) \right| 
 $$
The first part is bounded by a constant $K(\Delta)$ because of the uniform convergence of $\psi_n$ towards $\psi$ on the complex segment $[-\I\Delta/\sqrt{\eta''(0)},\I\Delta/\sqrt{\eta''(0)}]$. The second part can be bounded by
\[\left(\max_{\frac{\delta}{\sqrt{\eta''(0)}} \leq |u| \leq \frac{\Delta}{\sqrt{\eta''(0)}}} |\exp(\eta(\I u))|\right)^{\!t_{n}},\]
but the maximum is a constant $q_{\delta,\Delta}$ strictly smaller than 1, because $\eta$ is not lattice distributed. This implies that in the domain 
$[-\Delta\sqrt{t_{n}},\Delta\sqrt{t_{n}}]\setminus[-\delta\sqrt{t_{n}},\delta\sqrt{t_{n}}]$,
one has the bound
\[|f_{n}^*(\zeta)| \le K(\Delta) (q_{\delta,\Delta})^{t_n}.\]
The explicit expression~\eqref{eq:gnstar} shows that the same kind of bound holds for $|g_{n}^*(\zeta)|$. We shall use the notation $\widetilde{K}(\Delta)$ and $\widetilde{q}_{\delta,\Delta}$ for constants valid for both $|f_{n}^*(\zeta)|$ and $|g_{n}^*(\zeta)|$. Thus the third summand in the bound~\eqref{ineq:FmoinsG} is bounded by 
$$\frac{4 \Delta}{\pi \delta}\, \widetilde{K}(\Delta) \,(\widetilde{q}_{\delta,\frac{1}{\delta}})^{t_{n}}.$$
Fix $\eps>0$, then $\delta$ such that $\eps(\delta)<\eps$ and $M\delta<\eps$. Take $\Delta=\frac{1}{\delta}$; we get
$$|F_{n}(x)-G_{n}(x)| \leq \frac{2\eps}{\sqrt{t_{n}}}+\frac{4}{\pi\delta^2}\,\widetilde{K}(\delta^{-1})\,(\widetilde{q}_{\delta,\Delta})^{t_{n}} \leq \frac{3\eps}{\sqrt{t_{n}}}$$
for $t_{n}$ large enough. This completes the proof of the proposition.
\end{proof}
\medskip

\begin{remark}
Proposition \ref{prop:berryesseen} gives an approximation for the Kolmogorov distance between the law $\mu_n$ and the normal law. Indeed, assume to simplify that the reference law $\phi$ is the Gaussian law. Then, $\eta''(0)=1$ and $\eta'''(0)=0$, and one computes
\begin{align*}
d_\text{Kol}(\mu_n,\,\mathcal{N}_{\R}(0,1))&=\frac{1}{\sqrt{t_n}}\,\sup_{x \in \R} \left|\int_{-\infty}^x\psi'(0)\,y\,g(y)\,dy\right|+o\!\left(\frac{1}{\sqrt{t_n}}\right) \\
&= \frac{|\psi'(0)|}{\sqrt{2\pi t_n}}+o\!\left(\frac{1}{\sqrt{t_n}}\right).
\end{align*}
This makes explicit the bound given by Theorem 1 in \cite{Hwa98}.
If $\psi'(0) \ne 0$ ({\em e.g.}, as in Lemma \ref{lem:exponchange}), we get an equivalent of the Kolmogorov distance. However, if $\psi'(0)=0$, then the estimate $d_{\mathrm{Kol}}=o(1/\sqrt{t_n})$ may not be optimal. Indeed, in the case of a scaled sum of i.i.d.~random variables, $t_n=n^{1/3}$ and one obtains the bound 
$$d_{\mathrm{Kol}}\left(\frac{1}{\sqrt{n}}\sum_{i=1}^n Y_i,\,\mathcal{N}_{\R}(0,1)\right)=o\!\left(\frac{1}{n^{1/6}}\right),$$
which is not as good as the classical Berry-Esseen estimate $O(\frac{1}{n^{1/2}})$. There is a way to modify the arguments in order to get such optimal estimates, by controlling  the zone of mod-convergence. We refer to \cite{FMN14}, where such "optimal" computations of Kolmogorov distances is performed.
\end{remark}
\bigskip

\subsection{Deviations at scale \texorpdfstring{$O(t_{n})$}{O(tn)}} 
\begin{theorem}\label{thm:mainnonlattice}
Suppose $\phi$ non-lattice, and consider as before a sequence $(X_n)_{n \in \N}$ that converges mod-$\phi$ on a band $\mathcal{S}_{(c,d)}$ with $c < 0 <d$. If $x \in (\eta'(0),\eta'(d))$, then 
$$\proba[X_{n}\geq t_{n}x]=\frac{\exp(-t_{n}F(x))}{h\,\sqrt{2\pi t_{n}\eta''(h)}}\,\psi(h)\left(1+o(1)\right)$$
where as usual $h$ is defined by the implicit equation $\eta'(h)=x$.
\end{theorem}\medskip

\begin{remark}
By applying the result to $(-X_n)_{n \in \N}$, one gets similarly
$$\proba[X_n \leq t_nx] = \frac{\exp(-t_n F(x))}{|h|\sqrt{2\pi t_n \eta''(h)}}\, \psi(h)\left(1+o(1)\right)$$
for $x \in (\eta'(c),\eta'(0))$, with $h$ defined by the implicit equation $\eta'(h)=x$.
\label{rmq:neg_dev}
\end{remark}

\begin{remark}
Theorem \ref{thm:mainnonlattice} should be compared with \cite[Theorem 1]{Hwa96}, which studies another regime of deviations in the mod-$\phi$ setting, namely, when $h$ goes to zero (or equivalently, $x \to \eta'(0)$). We shall also look at this regime in our Theorem \ref{thm:cltnonlattice}.
\end{remark}

\begin{remark}
The main difference between Theorems \ref{thm:mainlattice} and \ref{thm:mainnonlattice} is the replacement of the factor $\psi(h)/(1-\E^{-h})$ by $\psi(h)/h$; the same happens with Bahadur-Rao's estimates when going from lattice distributions to non-lattice distributions. 
\end{remark}
\medskip 

\begin{lemma}\label{lem:exponchange}
Let $(X_n)_{n \in \N}$ be a sequence of random variables that converges mod-$\phi$ with parameters $(t_n)_{n \in \N}$ and limiting function $\psi$, on a strip $\mathcal{S}_{(c,d)}$ that does not necessarily contain $0$. For $h \in (c,d)$, we make the exponential change of measure
$$\mathbb{Q}[dy]=\frac{\E^{hy}}{\varphi_{X_{n}}(h)}\,\proba[X_{n} \in dy],$$
and denote $\widetilde{X}_{n}$ a random variable following this law. The sequence $(\widetilde{X}_n)_{n \in \N}$ converges mod-$\widetilde{\phi}$, where $\widetilde{\phi}$ is the infinitely divisible distribution with characteristic function $\E^{\eta(z+h)-\eta(h)}$. The parameters of this new mod-convergence are again $(t_n)_{n \in \N}$, and the limiting function is
$$\widetilde{\psi}(z)=\frac{\psi(z+h)}{\psi(h)}.$$
The new mod-$\widetilde{\phi}$ convergence occurs in the strip $\mathcal{S}_{(c-h,d-h)}$.
\end{lemma}

\begin{proof}
Obvious since $\varphi_{\widetilde{X}_n}(z)=\varphi_{X_n}(z+h)/\varphi_{X_n}(h)$.
\end{proof}

\begin{proof}[Proof of Theorem \ref{thm:mainnonlattice}]
Fix $h \in (c,d)$, and consider the sequence $(\widetilde{X}_n)_{n\in \N}$ of Lemma \ref{lem:exponchange}. All the assumptions of Proposition \ref{prop:berryesseen} are satisfied, so, the distribution function $F_{n}(u)$ of 
$$\frac{\widetilde{X}_{n}-t_{n}\eta'(h)}{\sqrt{t_{n}\eta''(h)}}$$ is 
$$G_{n}(u)=\int_{-\infty}^{u} \left(1+\frac{\psi'(h)}{\psi(h)\,\sqrt{t_{n}\eta''(h)}}\,y+\frac{\eta'''(h)}{\sqrt{t_{n}(\eta''(h))^{3}}}\,(y^{3}-3y)\right)g(y)\,dy$$
up to a uniform $o(1/\sqrt{t_{n}})$. Then,
\begin{align*}\proba[X_{n} \geq t_{n}\eta'(h)]&=\int_{y=t_{n}\eta'(h)}^{\infty}\varphi_{X_{n}}(h)\,\E^{-hy}\,\mathbb{Q}(dy)\\
&=\varphi_{X_{n}}(h)\int_{u=0}^{\infty} \E^{-h\left(t_{n}\eta'(h)+\sqrt{t_{n}\eta''(h)}\,u\right)} dF_{n}(u)\\
&=\psi_{n}(h)\,\E^{-t_{n}F(x)}\int_{u=0}^{\infty} \E^{-h\sqrt{t_{n}\eta''(h)}\,u}\, dF_{n}(u), \,\, \text{(as $F(x)=h\eta'(h)-\eta(h)$).}
\end{align*}
To compute the integral $I$, we choose the primitive $F_n(u)-F_n(0)$ of $dF_{n}(u)$ that vanishes at $u=0$, and we make an integration by parts. Notice that we now need $h>0$ (hence, $x>\eta'(0)$) in order to manipulate some of the terms below:
\begin{align*}
I&=h\sqrt{t_{n}\eta''(h)}\int_{u=0}^{\infty}\E^{-h\sqrt{t_{n}\eta''(h)}\,u}\,(F_{n}(u)-F_n(0))\,du\\
&= h\sqrt{t_{n}\eta''(h)}\int_{u=0}^{\infty}\E^{-h\sqrt{t_{n}\eta''(h)}\,u}\left(G_{n}(u)-G_{n}(0)+o\!\left(\frac{1}{\sqrt{t_{n}}}\right)\!\right)\,du\end{align*}
\begin{align*}
&\simeq h\sqrt{t_{n}\eta''(h)} \iint_{0\leq y \leq u} \!\!\!\!\!\!\!\!\!\E^{-h\sqrt{t_{n}\eta''(h)}\,u}\!\left(1+\frac{\psi'(h)\,y}{\psi(h)\,\sqrt{t_{n}\eta''(h)}}+\frac{\eta'''(h)\,(y^{3}-3y)}{\sqrt{t_{n}(\eta''(h))^{3}}}\right)g(y)\,dy\,du \\
&\simeq \int_{y=0}^{\infty} \E^{-h\sqrt{t_{n}\eta''(h)}\,y}\left(1+\frac{\psi'(h)}{\psi(h)\,\sqrt{t_{n}\eta''(h)}}\,y+\frac{\eta'''(h)}{\sqrt{t_{n}(\eta''(h))^{3}}}\,(y^{3}-3y)\right)g(y)\,dy\\
&\simeq \frac{\E^{\frac{h^2\, t_n\eta''(h)}{2}}}{\sqrt{2\pi}}\int_{y=0}^{\infty} \E^{-\frac{(y+h\sqrt{t_{n}\eta''(h)})^2}{2}}\left(1+\frac{\psi'(h)}{\psi(h)\,\sqrt{t_{n}\eta''(h)}}\,y+\frac{\eta'''(h)}{\sqrt{t_{n}(\eta''(h))^{3}}}\,(y^{3}-3y)\right)dy,
\end{align*}
where on the three last lines the symbol $\simeq$ means that the remainder is a $o((t_{n})^{-1/2})$. By Lemma \ref{lem:gaussintegral}, \eqref{item:gausstail}, the only contribution in the integral that is not a $o((t_{n})^{-1/2})$ is $$\frac{\E^{\frac{h^2\, t_n\eta''(h)}{2}}}{\sqrt{2\pi}}\int_{y=0}^{\infty} \E^{-\frac{(y+h\sqrt{t_{n}\eta''(h)})^2}{2}}\,dy=\frac{1}{h\sqrt{2\pi t_{n}\eta''(h)}}+o\!\left(\frac{1}{\sqrt{t_{n}}}\right).$$ This ends the proof since $\psi_{n}(h)\to\psi(h)$ locally uniformly.
\end{proof}\bigskip

\subsection{Central limit theorem at the scales \texorpdfstring{$o(t_n)$}{o(tn)} and \texorpdfstring{$o((t_n)^{2/3})$}{o(tn23)}}
As in the lattice case, one can also prove from the hypotheses of mod-convergence an extended central limit theorem:
 
\begin{theorem}\label{thm:cltnonlattice}
Consider a sequence $(X_n)_{n \in \N}$ that converges mod-$\phi$ with limiting distribution $\psi$ and parameters $t_n$, where $\phi$ is a non-lattice infinitely divisible law that is absolutely continuous w.r.t. Lebesgue measure.
Let $y=o((t_{n})^{1/6})$. Then,
$$\proba\!\left[X_{n}\geq t_{n}\eta'(0)+\sqrt{t_{n}\eta''(0)}\,y\right]=\proba[\mathcal{N}_{\R}(0,1)\geq y]\left(1+o(1)\right).$$
On the other hand, assume $y \gg 1$ and $y=o((t_n)^{1/2})$. If $x=\eta'(0)+\sqrt{\eta''(0)/t_n}\,y$ and $h$ is the solution of $\eta'(h)=x$, then
$$\proba\!\left[X_{n}\geq t_{n}\eta'(0)+\sqrt{t_{n}\eta''(0)}\,y\right]
= \frac{\E^{-t_n\,F(x)}}{h \sqrt{2\pi t_{n}\,\eta''(h)}}\,\left(1+o(1)\right)
= \frac{\E^{-t_n\,F(x)}}{y\sqrt{2\pi}}\,\left(1+o(1)\right).$$
\end{theorem}

As in the proof of Theorem \ref{thm:cltlattice}, we need to control the modulus of the Fourier transform of the reference law $\phi$. Thus, let us state the non-lattice analogue of Lemma \ref{lem:technicqdelta_lattice}:
\begin{lemma}\label{lem:technicqdelta_nonlattice}
Consider a non-constant infinitely divisible law $\phi$, of type non-lattice, with a convergent moment generating function in a strip $\mathcal{S}_{(c,d)}$ with $c<0<d$. We also assume that $\phi$ is absolutely continuous w.r.t.~the Lebesgue measure. Then, there exists a constant $D>0$ only depending on $\phi$, and an interval $(-\eps,\eps) \subset (c,d)$, such that for all $h \in (-\eps,\eps)$, and all $\delta$ small enough,
$$ q_\delta = \max_{u \in \R \setminus (-\delta,\delta)} |\exp(\eta(h+\I u)-\eta(h))| \leq 1-D\,\delta^2.$$
\end{lemma}

\begin{remark}
One can give a sufficient condition on the L\'evy-Khintchine representation of $\phi$ to ensure the absolute continuity with respect to the Lebesgue measure; \emph{cf.} \cite[Chapter 4, Theorem 4.23]{SVH04}. Hence, it is the case if $\sigma^2>0$, or if $\sigma^2=0$ and if the absolutely continuous part of the L\'evy measure $\Pi$ has infinite mass.
\end{remark}

\begin{remark} Let us explain why we need to add the assumption of absolute continuity with respect to Lebesgue measure, which is a strictly stronger hypothesis than being non-lattice. The hypotheses on the infinitely divisible law $\phi$ imply that it has finite variance, and therefore, that the L\'evy-Khintchine representation of the Fourier transform given by Equation \eqref{eq:levykhintchine} can be replaced by a Kolmogorov representation. This representation actually holds for the complex moment generating function (see \cite[Chapter 4, Theorem 7.7]{SVH04}):
$$\eta(z)=m z + \sigma^2\,\int_{\R}\frac{\E^{zx}-1-zx}{x^2}\,K(dx)$$
where $K$ is a probability measure on $\R$, and where the fraction in the integral is extended by continuity at $x=0$ by the value $-\frac{z^2}{2}$. As a consequence,
$$|\exp(\eta(h+\I u)-\eta(h))| = \exp\left(\sigma^2\int_{\R} \frac{\E^{hx}\,(\cos ux-1)}{x^2}\,K(dx)\right)\leq 1.$$
This expression can be expanded in series of $u$ as
$$1-\frac{\sigma^2u^2}{2} \int_{\R} \E^{hx} \,K(dx) + O_h(u^3).$$
Therefore, Lemma \ref{lem:technicqdelta_nonlattice} holds as soon as one can show that 
$$\sup_{h \in (-\eps,\eps)} \limsup_{|u| \to \infty} |\exp(\eta(h+\I u)-\eta(h))| <1,$$
because one has a bound of type $1-D\,u^2$ in a neighborhood of zero. Unfortunately, for general probability measures, the Riemann-Lebesgue lemma does not apply, and even for $h=0$, it is unclear whether for a general exponent $\eta$ the Cram\'er condition (C)
$$\limsup_{|u| \to \infty} |\exp(\eta(\I u))| <1 $$is satisfied (see \cite{Petrov95} for more discussion and references on  condition (C)).
We refer to \cite[Theorem 2]{Wol83}, where it is shown that \emph{decomposable} probability measures enjoy this property. This difficulty explains why one has to restrict oneself to absolutely continuous measures in the non-lattice case, in order to use the Riemann-Lebesgue lemma. In the following we provide an \emph{ad hoc} proof of Lemma \ref{lem:technicqdelta_nonlattice} in the absolutely continuous cases, that does not rely on the Kolmogorov representation.\footnote{V: Je ne comprends pas bien cette remarque. C'est une preuve alternative avec des hypothèses plus faibles ?
Si oui, pourquoi ne pas présenter celle-là, plutôt que l'autre\dots

PL: On pourrait effectivement donner des conditions plus faibles qu'absolument continu par rapport à la mesure de Lebesgue, mais ca deviendrait excessivement technique, et nettement plus dur. Dans l'article, j'ai l'impression qu'on essaie de se passer des résultats avancés sur les lois infiniment divisibles (on a juste besoin de la distinction lattice/non-lattice), c'est pour cela que je ne voulais pas plus entrer dans des details techniques.}
\end{remark}

\begin{proof}[Proof of Lemma \ref{lem:technicqdelta_nonlattice}]
We shall adapt the arguments of Lemma \ref{lem:technicqdelta_lattice} from the discrete to the continuous case. Though the density $f$ cannot be supported on a compact segment (\emph{cf.} Lemma \ref{lem:support} and the classification of the possible supports of an infinitely divisible law), one can work as if it were the case, thanks to the following calculation:
\begin{align*}
|\exp(\eta(h+\I u)-\eta(h))|&=\left|\frac{\phi(\E^{(h+\I u)x})}{\phi(\E^{hx})}\right|\\
&=\frac{|\phi_{<a}(\E^{(h+\I u)x})+\int_a^b \E^{(h+\I u) x}\,f(x)\,dx+\phi_{>b}(\E^{(h+\I u)x})|}{\phi_{<a}(\E^{hx})+\int_a^b \E^{hx}\,f(x)\,dx+\phi_{>b}(\E^{hx})}\\
&\leq \frac{\phi_{<a}(\E^{hx})+\phi_{>b}(\E^{hx})+|\int_a^b \E^{(h+\I u) x}\,f(x)\,dx|}{\phi_{<a}(\E^{hx})+\phi_{>b}(\E^{hx})+\int_a^b \E^{hx}\,f(x)\,dx}
\end{align*}
where $\phi_{<a}$ (respectively, $\phi_{>b}$) is the measure $\mathbbm{1}_{x<a}\,\phi(dx)$ (resp., $\mathbbm{1}_{x>b}\,\phi(dx)$). Therefore, it suffices to show:
$$\max_{u \in (-\delta,\delta)^\mathrm{c}} \frac{|\int_a^b \E^{(h+\I u)x}\,f(x)\,dx|}{\int_a^b \E^{hx}\,f(x)\,dx} \leq 1-D\,\delta^2$$
for $\delta$ and $h$ small enough. This reduction to a compact support will be convenient later in the computations.\bigskip

Set $g_h(x)=\frac{\E^{hx}\,f(x)}{\int_a^b \E^{hx}\,f(x)\,dx}$ and 
\begin{align*}
F(h,u)&=\left|\int_a^b g_h(x)\,\E^{\I u x}\,dx\right|^2=\iint_{[a,b]^2} g_h(x)g_h(y)\, \E^{\I u (x-y)}\,dx \,dy \\
&=\iint_{[a,b]^2} g_h(x)g_h(y)\, \cos (u (x-y))\,dx \,dy \\
&=\int_{t=-(b-a)}^{b-a}  \left(\int_{x=\max(a,t+a)}^{\min(b,t+b)}g_h(x)g_h(x-t)\,dx\right)\cos ut\,dt.
\end{align*}
The problem is to show that
$$\sup_{h \in (-\eps,\eps)} \sup_{u \in (-\delta,\delta)^\mathrm{c}} F(h,u) \leq 1-D\,\delta^2$$
for some constant $D$. With $h$ fixed, by the Riemann-Lebesgue lemma applied to the integrable function
$$m(t)=\int_{x=\max(a,t+a)}^{\min(b,t+b)}g_h(x)g_h(x-t)\,dx,$$ 
the limit as $|u|$ goes to infinity of $F(h,u)$ is $0$. On the other hand, if $u \neq 0$, then $F(h,u)<F(h,0)=1$. Indeed, suppose the opposite: then $\cos ut=1$ almost surely w.r.t. the measure $m(t)\,dt$. This means that this measure $m(t)\,dt$ is concentrated on the lattice $\frac{2\pi}{|u|}\,\Z$, which is impossible for a measure continuous with respect to  the Lebesgue measure. Combining these two observations, one sees that for any $\delta>0$,
$$\sup_{u \in (-\delta,\delta)^{\mathrm{c}}} F(h,u) \leq C_{(h,\delta)}<1$$
for some constant $C_{(h,\delta)}$. Since all the terms considered depend smoothly on $h$, for $h $ small enough, one can even take a uniform constant $C_\delta$:
\begin{equation}\forall \delta>0,\,\,\exists C_\delta<1 \text{ such that }\sup_{h \in (-\eps,\eps)} \sup_{u \in (-\delta,\delta)^{\mathrm{c}}} F(h,u) \leq C_\delta.\label{eq:boundatinfinity}
\end{equation}
On the other hand, notice that
$$
\frac{\partial F(h,u)}{\partial u} = -\iint_{[a,b]^2} g_h(x)g_h(y) \,(x-y)\,\sin(u(x-y))\,dx\,dy.$$
However, if $u(b-a)\leq \frac{\pi}{2}$, then $(x-y)\,\sin(u(x-y))\geq \frac{2 u}{\pi}(x-y)^2$ over the whole domain of integration, so,
$$\frac{\partial F(h,u)}{\partial u} \leq -\frac{2 B_h}{\pi}\, u$$
where $B_h=\iint_{[a,b]^2}g_h(x)g_h(y)\,(x-y)^2\,dx\,dy$. By integration, 
$$F(h,u) \leq 1-\frac{B_h}{\pi}\,u^2 \quad \text{for all } u \leq  \frac{\pi}{2(b-a)}.$$
Again, by continuity of the constant $B_h$ w.r.t. $h$, one can take a uniform constant :
\begin{equation}\exists B>0 \text{ such that for all } u \leq  \frac{\pi}{2(b-a)},\,\,\,\sup_{h \in (-\eps,\eps)} F(h,u) \leq 1-B\,u^2.\label{eq:boundaroundzero}\end{equation}
The two assertions \eqref{eq:boundatinfinity} (with $\delta=\frac{\pi}{2(b-a)}$) and \eqref{eq:boundaroundzero} enable one to conclude, with
\[D=\inf\left(B,\frac{1-C_\delta}{\delta^2}\right), \quad\text{where }\delta=\frac{\pi}{2(b-a)}. \qedhere\]
\end{proof}
\noindent We also refer to \cite[Theorem 6]{Ess45} for a general result on the Lebesgue measure of the set of points such that the characteristic function of a distribution is larger in absolute value than $1-\delta^{2}$. 
\bigskip

\begin{proof}[Proof of Theorem \ref{thm:cltnonlattice}]
The proof is now exactly the same as in the lattice case (Theorem \ref{thm:cltlattice}). Indeed, the conclusions of the technical Lemma \ref{lem:technicqdelta_nonlattice} hold, and on the other hand, the equivalents for $\proba[X_{n}\geq t_{n}x ]$ in the lattice and non-lattice cases (Theorems~\ref{thm:mainlattice} and \ref{thm:mainnonlattice}) differ only by the fact that $1-e^{-h}$ is replaced by $h$. But in the proof of Theorem \ref{thm:cltlattice}, the quantity $1-e^{-h}$ is approximated by $h$, so everything works the same way as in the non-lattice case.
\end{proof}
\bigskip

As in the non-lattice case, we have the following corollary (with the exact same statement and proof):
\begin{corollary}
    If $y=o((t_n)^{1/4})$, then one has
\begin{equation}
    \proba\big[X_n \ge t_n \eta'(0) + \sqrt{t_n \eta''(0)}\, y\big] 
    =\frac{(1+o(1))}{y\sqrt{2\pi}}\, \E^{-\frac{y^2}{2}}\, \exp\left( \frac{\eta'''(0)}{6\,(\eta''(0))^{3/2}}\,\frac{y^3}{\sqrt{t_n}}
    \right).
\label{EqCramerOrdre3NL}
\end{equation}
More generally, if $y=o((t_n)^{1/2-1/m})$, then one has
\begin{equation}
    \proba\big[X_n \ge t_n \eta'(0) + \sqrt{t_n \eta''(0)}\, y\big] 
    =\frac{(1+o(1))}{y\sqrt{2\pi}} \exp\left( -\sum_{i=2}^{m-1} \frac{F^{(i)}(\eta'(0))}{i!}\,\frac{(\eta''(0))^{i/2} \,y^i}{(t_n)^{(i-2)/2}}
    \right) .
\label{EqCramerOrdreMNL}
\end{equation}
\end{corollary}

Hence, one can again describe all the fluctuations of $X_n$ from order $O(\sqrt{t_n})$ up to order $O(t_n)$, see Figure \ref{fig:panoramanonlattice}.
\begin{center}
\begin{figure}[ht]
\begin{tikzpicture}
\draw (-1,1.5) node {order of fluctuations};
\draw (2.5,0) node {large deviations ($\eta'(0)<x$):};
\draw (1.9,-1) node {extended central limit};
\draw (3.25,-3) node {central limit theorem ($y \ll (t_n)^{1/6}$):};
\draw (2.9,-1.5) node {theorem ($(t_n)^{1/6}\lesssim y \ll (t_n)^{1/2}$):};
\draw (9.2,-0.05) node {$\proba[\frac{X_n}{t_n} \geq x] \simeq \frac{\exp(-t_n\,F(x))}{F'(x)\sqrt{2\pi t_n\eta'(x)}}\,\psi(F'(x));$};
\draw (9.4,-1.6) node {$\proba[\frac{X_n-t_n\eta'(0)}{\sqrt{t_n\,\eta''(0)}} \geq y] \simeq \frac{\exp(-t_n\,F(x))}{F'(x)\,\sqrt{2\pi t_n \eta'(x)}};$};
\draw (9.23,-3.8) node {$\proba[\frac{X_n-t_n\eta'(0)}{\sqrt{t_n\,\eta''(0)}} \geq y] \simeq \proba[\mathcal{N}_{\R}(0,1)\geq y].$};
\draw[->,thick] (-2,-5) -- (-2,1) ;
\draw[thick] (-2.1,0) -- (-1.9,0);
\draw (-1.25,0) node {$O(t_n)$} ;
\draw[->,thick] (-2.1,-2) -- (-1.9,-2) -- (-1.9,-0.1);
\draw (-0.85,-2) node {$O((t_n)^{2/3})$} ;
\draw[->,thick] (-2.1,-4) -- (-1.9,-4) -- (-1.9,-2.1) ;
\draw (-0.85,-4) node {$O((t_n)^{1/2})$} ;
\end{tikzpicture}
\caption{Panorama of the fluctuations of a sequence of random variables $(X_n)_{n\in\N}$ that converges modulo an absolutely continuous distribution (with $x=\eta'(0) + \sqrt{\eta''(0)/t_n}\,y$).}\label{fig:panoramanonlattice}
\end{figure}
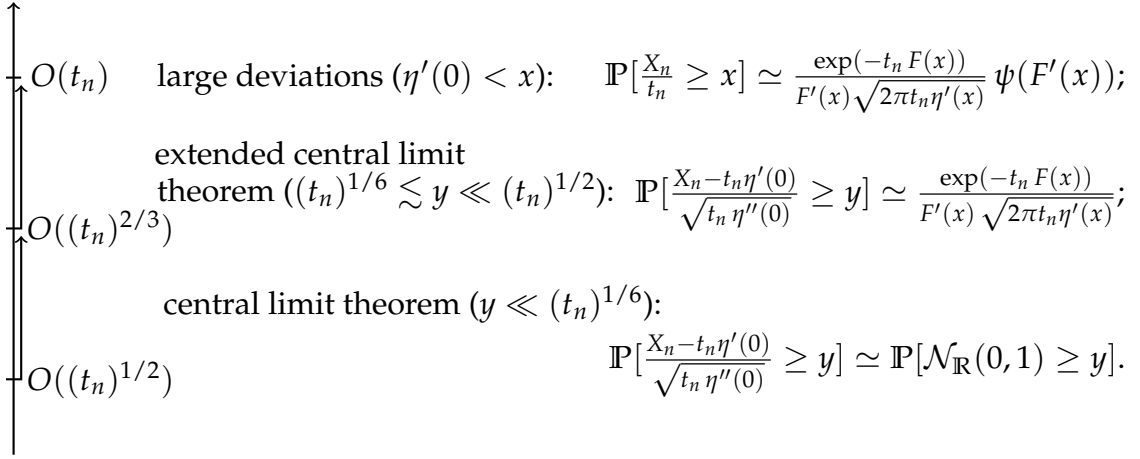
\end{center}
\vspace{-5mm}

To conclude this paragraph, let us mention an application that looks similar to the law of the iterated logarithm, and that also works in the lattice case. Consider a sequence $(X_n)_{n \in \N}$ converging mod-$\phi$ with parameters $t_n$ such that $t_n \gg (\log n)^3$. We also assume that the random variables $X_{n}$ are defined on the same probability space, and we look for sequences $\gamma_{n}$ such that almost surely,
$$\limsup_{n \to \infty} \left(\frac{X_{n}-t_n \eta'(0)}{\gamma_{n}} \right)\leq 1.$$
Unlike in the usual law of the iterated logarithm, we do not make any assumption of independence. Such assumptions are common in this setting, or at least some conditional independence (for instance, a law of the iterated logarithm can be stated for martingales); see the survey \cite{Bin86} or \cite[Chapter X]{Petrov75}. \bigskip

On the one hand we have a less precise result: we only obtain an upper bound, which is not tight in the case of sums of i.i.d. variable since we have a $\sqrt{\log(n)}$ factor instead the usual $\sqrt{\log(\log(n))}$ factor. On the other hand, our result does not depend at all on the way one realizes the random variables $X_{n}$. In other words, for every possible coupling of the variables $X_{n}$, the following holds:

\begin{proposition}
Let $(X_n)_{n \in \N}$ be a sequence that converges mod-$\phi$ with parameters $t_n$, where $\phi$ is either a non-constant lattice distribution, or a non-lattice distribution that is absolutely continuous with respect to the Lebesgue measure. We assume
$$\lim_{n \to \infty} \frac{t_n}{(\log n)^3} = +\infty.$$
Then,
$$\limsup_{n \to \infty} \frac{X_{n}-t_n \eta'(0)}{\sqrt{2\eta''(0)\,t_n\log n}} \leq 1 \quad\text{almost surely}.$$
\end{proposition}

\begin{proof}
Notice the term $\log n$ instead of $\log \log n$ for the usual law of iterated logarithm. One computes
$$\proba\big[X_{n}-t_n \eta'(0) \geq \sqrt{2 (1+\eps) \,\eta''(0)\, t_n \log n }\big].$$
Set $y=\sqrt{2 (1+\eps) \log n}$. Due to the hypotheses on $t_n$, one has $y=o\big( (t_n)^{1/6} \big)$ and one can apply Theorem \ref{thm:cltnonlattice}: using the classical equivalent 
\[\proba[\mathcal{N}_{\R}(0,1)\geq y] \sim \frac{\E^{-y^2/2}}{y\sqrt{2 \pi}},\]
we get
$$\proba\!\left[X_{n}-t_n \eta'(0)
\geq \sqrt{2 (1+\eps) \eta''(0) t_n \log n } \right]
\simeq \frac{\E^{-(1+\eps)\log n}}{\sqrt{4\pi (1+\eps) \log n}} \leq \frac{1}{n^{1+\eps}}.$$
for $n$ large enough. For any $\eps>0$, this is summable, so by the Borel-Cantelli lemma, one has almost surely 
$$X_{n}-t_n \eta'(0)<\sqrt{2(1+\eps)\,\eta''(0)\,t_n\log n}\quad\text{for $n$ large enough}.$$
Since this is true for every $\eps$, one has (almost surely): 
$$\limsup_{n \to \infty} \frac{X_{n}-t_n \eta'(0)}{\sqrt{2\eta''(0)\,t_n\log n}} \leq 1. $$
\end{proof}
\bigskip

\subsection{Normality zones for mod-\texorpdfstring{$\phi$}{phi} and mod-Gaussian sequences}
\label{subsec:normalityzone}
Let $(X_n)_{n\in \N}$ be a sequence of random variables that converges mod-$\phi$ (we do not assume $\phi$ non-lattice for the moment). Then we have seen that
$$Y_{n}=\frac{X_{n}-t_{n}\eta'(0)}{\sqrt{t_{n}\eta''(0)}}$$ 
satisfies a central limit theorem: for all fixed $y$,
\begin{equation}
    \label{eq:GaussianTail}
    \lim_{n \to \infty} \frac{\proba[Y_n \geq y]}{(2\pi)^{-1/2} \int_y^\infty \E^{-\frac{s^2}{2}} \,ds} =1.
\end{equation}
The question that we address here is the question of the {\em normality zone}: we want to identify the maximal scale $u_n$ such that Equation \eqref{eq:GaussianTail} holds for $y=o(u_n)$. The results of the previous Sections allows to identify this scale and to describe what happens for $y=O(u_n)$.
\bigskip

Suppose that $\phi$ is a lattice distribution, or a non-lattice distribution that is absolutely continuous with respect to the Lebesgue measure. From Theorem \ref{thm:cltlattice} or Theorem \ref{thm:cltnonlattice}, we get that if $y=o((t_n)^{1/6})$, then $\proba[Y_n \geq y]$ is given by the Gaussian distribution. Assume $\eta'''(0) \neq 0$. Then, the previous result is optimal, because
of Equations \eqref{EqCramerOrdre3} and \eqref{EqCramerOrdre3NL}: for $y=c (t_n)^{1/6}$ the second factor is different from $1$ and the approximation by the  Gaussian tail is no longer valid. Thus, if $\eta'''(0)\neq 0$, then the normality zone for the sequence $(Y_n)_{n \in \N}$ is $o((t_n)^{1/6})$. 
In particular, one has the same asymptotics and normality zone than in the case of the sum of $t_n$ i.i.d. variables of law $\phi$; see \cite{Cramer38} and \cite{Petrov54}.
\bigskip

The only case where this comparison does not give us the normality zone is 
the case of mod-Gaussian convergence, that we shall discuss now. 
\medskip

\begin{proposition}
    \label{prop:normalityzone}
    Assume that $(X_n)_{n \in \N}$ converges in the mod-Gaussian sense, with a non-trivial limiting function ({\em i.e.}, $\psi \not\equiv 1$). Then the normality zone for $Y_n=X_n/\sqrt{t_n}$ is $o(\sqrt{t_n})$.
\end{proposition}
\begin{proof}
\begin{itemize}
    \item Let $y=o((t_n)^{1/2})$. Set $x=h=y/\sqrt{t_n}$ as in Theorem \ref{thm:cltnonlattice} in the mod-Gaussian case. Then, the second part of Theorem \ref{thm:cltnonlattice} states that
        $$\proba[Y_n \geq y] =\proba[X_n \geq t_n x] = \frac{\E^{-\frac{ t_n x^2}{2}}}{h \sqrt{2\pi t_n}}\,(1+o(1)) = \frac{\E^{-\frac{y^2}{2}}}{y \sqrt{2\pi}}\,(1+o(1)).$$
        Thus, the normality zone is at least $o(\sqrt{t_n})$ in this case.\vspace{2mm}
    \item Set now $y=x \sqrt{t_n}$ for a fixed $x > 0$. Then, Theorem \ref{thm:mainnonlattice} states that
        \begin{equation}
            \proba[Y_n \geq y] = \proba[X_n \geq t_n\, x]= \frac{\E^{-\frac{y^2}{2} }}{y \sqrt{2\pi}}\, \psi(x)\, (1+o(1)).
            \label{eq:PosDev}
        \end{equation}
        Similarly, if $y=-x \sqrt{t_n}$ for a fixed $x > 0$, we get (see Remark \ref{rmq:neg_dev})
        \begin{equation}
            \proba[Y_n \leq y] = \proba[X_n \leq - t_n\, x] =\frac{\E^{-\frac{y^2}{2} }}{|y| \sqrt{2\pi}}\, \psi(-x)\, (1+o(1)).
            \label{eq:NegDev}
        \end{equation}
        In particular, if $\psi(x)$ is not identically equal to $1$, then the approximation \eqref{eq:GaussianTail} of $\proba[Y_n \geq y]$ by the Gaussian tail breaks and the normality zone is exactly $o(\sqrt{t_n})$.\qedhere
\end{itemize}
\end{proof}
As seen in the proof, from a simple application of Theorem \ref{thm:mainnonlattice},
the residue $\psi$ describes how to correct the Gaussian tail to find an equivalent for $\proba[Y_n \geq y]$. A standard and interesting case is the case where the limiting function is $\psi=\exp(L z^v)$, where $L \neq 0$ is a real number and $v$ a positive integer ($v \geq 3$).
This might seem restrictive, but we will see in the examples that $\psi$ is very often of this type
--- Examples \ref{ex:zeros_randomfunction}, \ref{ex:ising}, \ref{ex:sumiid} and Theorems \ref{thm:modgaussianrcv}, \ref{thm:moddevgraphs}, \ref{thm:modgaussianfromzeros} and \ref{thm:modgaussiangeneralsparsegraph}.
If $v$ is odd ($v=3$ is a common case), comparing Equations \eqref{eq:PosDev} and \eqref{eq:NegDev} shows the following phenomenon:
the negative and positive deviations of $Y_n$ at order $O(\sqrt{t_n})$ have different asymptotic behaviour,
one being larger than the other one depending on the sign of $L$. 
In other words, our results reveal a {\em breaking of symmetry} at the edge of the normality zone.

\begin{remark}
This breaking of symmetry also occurs in multi-dimensional setting. In particular, in two dimensions, the residue allows to compute the distribution of the angle
of a sum of i.i.d. random variables at the edge of the normality zone, see \cite{Multidim}.
\end{remark}
\bigskip
\bigskip

\subsection{Discussion and refinements}

\subsubsection{Bahadur-Rao theorem and Cramér-Petrov expansion}
We consider here the case of a sum $S_n=Y_1+\cdots+Y_n$ of i.i.d.~random variables
such that $Y=Y_1$ has an infinitely divisible distribution of Levy exponent $\eta$.
Then $S_n$ {\em converges mod-$Y$} with parameters $n$ and limiting function $\psi=1$;
see Example \ref{ex:sumiid}. In this case, Theorems \ref{thm:mainnonlattice} and \ref{thm:mainlattice}
correspond to Bahadur-Rao estimates   
$$\proba[S_n \geq n\,x] \simeq \begin{cases}\frac{\exp(-n\,F(x))}{(1-\E^{-h})\sqrt{2\pi n \,\eta''(h)}}
    &\!\!\!\text{in the lattice case (assume $\Z$ is the minimal lattice);} \\
\frac{\exp(-n\,F(x))}{h\sqrt{2\pi n \,\eta''(h)}}&\!\!\!\text{in the non-lattice case,}
 \end{cases}
$$
where $\eta(h)=\log \esper[\E^{hY}]$ and $F$ is the Legendre-Fenchel transform of $\eta$;
see Theorem 3.7.4 in \cite{DZ98}, and also the papers \cite{BR60,Ney83,Ilt95}.
\medskip

In the same setting,
 Theorems \ref{thm:cltnonlattice} and \ref{thm:cltlattice} correspond to 
 Cramér-Petrov expansion \cite{Petrov95} (in the non-lattice case,
 we assume in addition that the law of $Y$ is absolutely continuous with respect to Lebesgue measure).
 To the best of our knowledge, the link with the Legendre-Fenchel transform is new.\smallskip
\medskip

\subsubsection{On the infinite divisibility of $\phi$}
\label{subsect:OnInfDiv}
As above, consider the case of a sum $S_n=Y_1+\cdots+Y_n$ of i.i.d.~random variables,
but with the law of $Y_1$ not necessarily infinitely divisible.
In this case, $\esper[\E^{zS_n}]=\big(\esper[\E^{zY_1}] \big)^n$, 
but, if $\esper[\E^{zY_1}]$ vanishes for some complex value of $z$,
one cannot write this as $\exp(n \eta(z))$ as in Definition \ref{def:modphi}.

The proofs of our large deviation results --- Theorems \ref{thm:mainnonlattice} and \ref{thm:mainlattice} ---
can nevertheless be adapted to this setting.
For the extended central limit theorem --- Theorems \ref{thm:cltlattice} and \ref{thm:cltnonlattice} ---
we would need to assume the estimate given by
Lemma \ref{lem:technicqdelta_lattice} or Lemma \ref{lem:technicqdelta_nonlattice}.
This is satisfied in particular if:
\begin{itemize}
\item either $Y$ takes its values in $\Z$, and there are two consecutive integers $n,m=n-1$ such that $\proba[Y=n]\neq 0$ and $\proba[Y=m]\neq 0$;\vspace{2mm}
\item or, $Y$ has a component absolutely continuous w.r.t. Lebesgue measure.\vspace{2mm}
\end{itemize}
Since these are classical results and since our method are close to the usual ones,
we do not give details on how to adapt our proof to the non infinitely divisible setting.

\begin{remark}
For Bahadur-Rao theorem,
it should be noticed that the assumption that $\Z$ is the minimal lattice is necessary.
For instance, if one considers a sum $S_n$ of $n$ independent Bernoulli random variables with $\proba[B=1]=\proba[B=-1]=1/2$, then the estimate above is not true, because $S_n$ has always the same parity as $n$. This is related to the fact that $\esper[\E^{zB}]$ has modulus $1$ at $z=\I\pi$.
\end{remark}
\medskip

\subsubsection{Quasi powers}
\label{Subsect:QP}

Mod-$\phi$ convergent is reminiscent of the quasi-power theory
developed by Hwang \cite{Hwa96,Hwa98} --- see also \cite[Chapter IX]{FSe09}.
\begin{definition}\cite{Hwa96}
    A sequence $(X_n)$ of random variables satisfy the quasi-power hypothesis if
    \begin{equation}
        \esper[\E^{zX_n}] = \E^{\phi(n)u(z)+v(z)} \big( 1 +O(\ka_n^{-1}) \big),
        \label{EqHwangQP}
    \end{equation}
    where $\phi(n),\ka_n \to \infty$, $u(z)$ and $v(z)$ are analytic functions for $|s| < \rho$
    (with $u''(0) \ne 0$)
    and the $O$ symbol is uniform in the disk $D(0,\rho)$.    
\end{definition}
Clearly, any sequence converging mod-$\phi$ satisfies this hypothesis,
taking $\phi(n)=t_n$, $v(z)=\ln(\psi(z))$, $u(z)=\eta(z)$
(since $\psi(0)=1$, a determination of the $\ln$ always exists on a sufficiently small neighbourhood of the origin).

A major difference between mod-convergence and the quasi-power framework
is that we assume that $\eta(z)$ is the L\'evy exponent of an infinitely divisible distribution,
while Hwang does not have any hypothesis on $u(z)$ (except $u''(0) \ne 0$).
The fact that $\exp(\eta(z)) = \int_{\R} \E^{zx}\, \phi(dx)$ is important
to study deviations at scale $O(t_n)$,
since we used the inequality $|\exp(\eta(z))| \le 1$ for $z=i\xi$
in the proof of Theorem \ref{thm:mainlattice} and Proposition \ref{prop:berryesseen}.
\bigskip

At the scale $o(t_n)$, our results --- Theorems \ref{thm:cltnonlattice} and \ref{thm:cltlattice} ---
coincide with the ones of Hwang.
Note, however, that our hypotheses are slightly different.
We need $\eta(z)$ to be the Levy exponent of an infinitely divisible distribution,
while Hwang uses an hypothesis
on the speed of convergence in Equation \eqref{EqHwangQP}.
In most examples, $\eta$ is a Poisson or Gaussian Lévy exponent,
so that our hypothesis is automatically verified.
It can thus be considered as a slight improvement that 
we do not require any hypothesis on the speed of convergence
(but such an hypothesis allows us to refine our results at scale $O(t_n)$
in the lattice case, see Theorem \ref{thm:mainlattice}).

\begin{remark}[The disk or the strip?]
    In the quasi-power framework, we assume convergence of the renormalized 
    Laplace transform on a disk, while mod-$\phi$ convergence is defined
    as such convergence on the strip.
    It is thus natural to wonder which hypothesis is more natural. To this purpose, let us mention an old result of Lukacs and Szász \cite[Theorem 2]{LS52}: if $X$ is a random variable with an analytic moment generating function $\esper[\E^{zX}]$ defined on the open disk $\mathcal{D}_{(0,c)}$, then this function is automatically defined and analytic on the strip $\mathcal{S}_{(-c,c)}$. This implies that the left-hand side of Eq.~\eqref{eq:modphi} is automatically defined on a strip, as soon as it is defined on a disk. Of course it could converge on a disk and not on a strip, but we shall see throughout this paper that, in many examples, the convergence on the strip indeed occurs. Actually, in most of our examples, $c=-\infty$ and $d=+\infty$, and the distinction between disk and strip disappears as $\mathcal{D}_{(0,+\infty)}=\mathcal{S}_{(-\infty,+\infty)}=\C$.
\label{rem:diskorstrip}
\end{remark}\medskip

\section{An extended deviation result from bounds on cumulants}\label{sec:cumulantechnic}
In this section, we discuss a particular case of mod-Gaussian variables, that arises from bounds on cumulants. We will see that in this case the deviation result given in Theorem \ref{thm:mainnonlattice} is still valid at a scale larger than $t_n$; see Proposition \ref{prop:largedeviationscumulants}.

\subsection{Bounds on cumulants and mod-Gaussian convergence}
\label{subsec:modgaussfromcumulants}
Let $(S_{n})_{n \in \N}$ be a sequence of real-valued centered random variables that admit moments of all order, and such that  
\begin{equation}
    |\kappa^{(r)}(S_{n})| \leq (Cr)^{r}\,\alphan (\betan)^r
    \label{eq:superbound}
\end{equation}
for all $r \geq 2$ and for some sequences $(\alphan)_{n \to \infty} \to +\infty$ and $(\beta_n)_{n\in \N}$ arbitrary. Assume moreover that there exists an integer $v \geq 3$ such that \vspace{2mm}
\begin{enumerate}
\item $\kappa^{(r)}(S_{n})=0$ for all $3 \leq r < v$ and all $n \in \N$;\vspace{2mm}
\item we have the following approximations for second and third cumulants:
    \begin{align}
        \kappa^{(2)}(S_{n}) &= \sigma^{2}\,\alphan (\betan)^{2}\,\left(1+o\!\left((\alphan)^{-\frac{v-2}{v}}\right)\right); \nonumber \\
        \kappa^{(v)}(S_{n}) &=L\,\alphan (\betan)^{v}\,\big(1+o(1)\big).
        \label{eq:cv_sndthrd_cumulants}
    \end{align}
\end{enumerate}
Set $X_{n}=\frac{S_{n}}{(\alphan)^{\frac{1}{v}}\betan}$. The cumulant generating series of $X_{n}$ is
\begin{align*}
\log \varphi_{n}(z)&=\frac{\kappa^{(2)}(S_{n})}{2\,(\alphan)^{\frac{2}{v}}\,(\betan)^2}\,z^{2}+\frac{\kappa^{(v)}(S_{n})}{v!\,\alphan (\betan)^{v}}\,z^{v}+\sum_{r=v+1}^{\infty} \frac{\kappa^{(r)}(S_{n})}{r!\,(\alphan)^{\frac{r}{v}}\,(\betan)^r}\,z^{r}\\
&=\frac{\sigma^{2}}{2}\,(\alphan)^{\frac{v-2}{v}}z^{2}+\frac{L}{v!}\,z^{v}+ \sum_{r=v+1}^{\infty} \frac{\kappa^{(r)}(S_{n})}{r!\,(\alphan)^{\frac{r}{v}}\,(\betan)^r}\,z^{r}+o(1),
\end{align*}
where the $o(1)$ is locally uniform. The remaining series is locally uniformly bounded in absolute value by
$$\sum_{r=v+1}^{\infty} C^{r}\,\frac{r^{r}}{r!}\, \frac{1}{(\alphan)^{\frac{r-v}{v}}}\,R^{r} \leq \alphan \sum_{r=v+1}^{\infty} \left(\frac{\E\, CR}{(\alphan)^{\frac{1}{v}}}\right)^{r}=(\alphan)^{-\frac{1}{v}}\,\frac{(\E\, CR)^{v+1}}{1-\E\, CR\,(\alphan)^{-\frac{1}{v}}} \to 0.$$
Hence, 
$$\psi_{n}(z)=\exp\left(-(\alphan)^{\frac{v-2}{v}}\,\frac{\sigma^{2}z^{2}}{2}\right)\,\varphi_{n}(z) \to \exp\left(\frac{L}{v!}\,z^{v}\right)$$
locally uniformly on $\C$, so one has again mod-Gaussian convergence, with parameters $t_{n}=\sigma^{2}\,(\alphan)^{\frac{v-2}{v}}$ and limiting function $\psi(z)=\E^{\frac{L}{v!}\,z^{v}}$.

\begin{remark}
The case of i.i.d.~variables --- Example \ref{ex:sumiid} --- fits in this framework, with $\alphan=n$ and $\betan=1$.However, it includes many more examples than sums of i.i.d.~variables: in particular, in Section \ref{sec:depgraph}, we show that such bounds on cumulants typically occur in the framework of dependency graphs. Concrete examples are discussed in Sections \ref{sec:erdosrenyi} and \ref{sec:central}.
\end{remark}
\bigskip

\subsection{Precise deviations for random variables with control on cumulants}\label{subsec:deviationscumulant}
We use the same hypotheses as in the previous subsection, and without loss of generality, we suppose that $v=3$. Then, the sequence of random variables
$$X_n=\frac{S_{n}}{\betan\,(\alphan)^{1/3}}$$ 
conver\-ges mod-Gaussian with parameters $(\alphan)^{1/3}\,\sigma^{2}$ and limiting function $\psi(z)=\exp(Lz^{3}/6)$ (here, we may have $L=0$). So, one can apply the previous theorems to estimate the tail of the distribution of $S_n$.
In particular, $X_n/\sqrt{t_n} = S_n/(\betan \sigma \sqrt{\alphan})$ satisfies a central limit theorem
with a normality zone of size $o(\sqrt{t_n}) = o \big( \alphan^{1/6} \big)$ (as for the sum of $\alpha_n$ i.i.d. variables)
and one can describe the deviation probabilties at the edge of the normality zone 
--- see Proposition \ref{prop:normalityzone}.

We shall see now that, with stronger assumptions on the speed of convergence than Equation \eqref{eq:cv_sndthrd_cumulants},
we can extend these results to a larger scale.
More precisely, we will assume:
\begin{align}\kappa^{(2)}(S_{n})&= \sigma^{2}\,\alphan\,(\betan)^{2}\,(1+O((\alphan)^{-1/2})); \nonumber\\ 
\kappa^{(3)}(S_{n}) &= L \,\alphan\,(\betan)^{3}\,(1+O((\alphan)^{-1/4})).\label{eq:limitcumulant}
\end{align}
We then have the following result:
\begin{proposition}\label{prop:largedeviationscumulants}
Let $(S_{n})_{n \in \N}$ be a sequence of centered real-valued random variables.
Assume that the bound on cumulants \eqref{eq:superbound} and the asymptotics of second and third cumulants given by Equation \eqref{eq:limitcumulant} hold.
If $x_n$ is a positive sequence, bounded away from $0$ with $x_n=o\big(\alphan^{1/12} \big)$, then
\begin{align*}
    \proba\!\left[S_{n} \geq \betan \sigma^2 \alphan^{2/3} x_n \right]=
\proba[X_n \ge t_n x_n] =
\frac{\E^{-\frac{(x_n)^2(\alphan)^{1/3}\sigma^2}{2}}}{x_n(\alphan)^{1/6}\sigma \sqrt{2\pi}}\,\E^{\frac{L(x_n)^3}{6}}\,(1+o(1)).\\
    \proba\!\left[S_{n} \geq \betan \sigma^2 \alphan^{2/3} x_n \right]=
\proba[X_n \le - t_n x_n] =
\frac{\E^{-\frac{(x_n)^2(\alphan)^{1/3}\sigma^2}{2}}}{x_n(\alphan)^{1/6}\sigma \sqrt{2\pi}}\,\E^{\frac{-L(x_n)^3}{6}}\,(1+o(1)).
\end{align*}
\end{proposition}

\begin{remark}
The case where $x_n$ is a constant sequence equal to $x$ corresponds to Equations \eqref{eq:PosDev} and \eqref{eq:NegDev}, 
which gives an equivalent for the deviation probability at the edge of the normality zone.
Hence, the proposition asserts that, with appropriate assumptions on cumulants,
this result is valid at a larger scale.
Namely, we give an equivalent for the deviation probability of $X_n$ of order up to $o\big( t_n^{5/4}\big)$,
instead of the usual $O(t_n)$.
\end{remark}

\begin{proof}
Set $X_{n}=(\alphan)^{-1/3}\,S_{n}$; up to a renormalization of the random variables, one can suppose $\beta_n=1$, and also $\sigma^{2}=1$. Let $z_n$ be a sequence of complex numbers with $|z_{n}|=O((\alphan)^{1/12})$; we set $\eta_n = |z_n| \,(\alphan)^{-1/12}$. Then, following the computation of Section \ref{subsec:modgaussfromcumulants} with $v=3$, we get:
\begin{align*}
\log \varphi_{X_n}(z_n)&=\frac{\kappa^{(2)}(S_{n})}{2\,(\alphan)^{\frac{2}{3}}}\,(z_n)^{2}+\frac{\kappa^{(3)}(S_{n})}{6\,\alphan }\,(z_n)^{3}+\sum_{r=4}^{\infty} \frac{\kappa^{(r)}(S_{n})}{r!\,(\alphan)^{\frac{r}{3}}}\,(z_n)^{r}\\
&=\frac{1}{2}\,(\alphan)^{\frac{1}{3}}\,(z_n)^{2}+\frac{L}{6}\,(z_n)^{3}+ O((\eta_n)^2 + (\eta_n)^3+(\eta_n)^4).
\end{align*}
If $|z_{n}|=o((\alphan)^{1/12})$, then $\eta_n \to 0$, so the remainder above is $o(1)$ and we have:
\begin{equation}
\varphi_{X_{n}}(z_{n})=\exp\left((\alphan)^{\frac{1}{3}}\frac{(z_{n})^{2}}{2}+\frac{L(z_{n})^{3}}{6}\right)\big(1+o(1)\big).
\label{EqPhiXn}
\end{equation}
We make the change of probability measure 
$$\proba[Y_{n} \in dy]=\frac{\E^{x_ny}}{\varphi_{X_{n}}(x_n)}\,\proba[X_{n} \in dy]$$ 
with $x_n=o((\alphan)^{1/12})$; the generating function of $Y_{n}$ is $\varphi_{Y_{n}}(z)= \frac{\varphi_{X_{n}}(x_n+z)}{\varphi_{X_{n}}(x_n)}$. So, using the inequality
$$|(z+x_n)^r - (x_n)^r| \leq 2^r\,|z|\,\max(|z|,|x_n|)^{r-1},$$
we get, setting $\eta_n^z = |z|\,(\alphan)^{-1/12}$ and $\eta_n^{x,z} = \max(|x_n|,|z|)\,(\alphan)^{-1/12}$,
\begin{align}
\log \varphi_{Y_{n}}(z)&=(\alphan)^{1/3}\frac{(z+x_{n})^{2}-(x_n)^2}{2}+\frac{L((z+x_{n})^{3}-(x_n)^3)}{6}\nonumber\\
&\quad+O(\eta_{n}^z(\eta_n^{x,z} + (\eta_n^{x,z})^2+(\eta_n^{x,z})^3))\nonumber\\
&=\left((\alphan)^{1/3}\,x_{n} +\frac{L\,(x_{n})^{2}}{2}\right)z+\left(\frac{(\alphan)^{1/3}+L\,x_{n}}{2}\right)z^{2}+\frac{L}{6}\,z^{3}\nonumber\\
&\quad+O(\eta_{n}^z(\eta_n^{x,z} + (\eta_n^{x,z})^2+(\eta_n^{x,z})^3))\label{eq:push}
\end{align}
Thus, if 
$$Z_{n}=Y_{n}-(\alphan)^{1/3}\,x_{n} -\frac{L\,(x_{n})^{2}}{2},$$
then the sequence $(Z_{n})_{n \in \N}$ converges in the mod-Gaussian sense, with parameters $t_n=(\alphan)^{1/3}+L\,x_{n}$ and limiting function $\exp(\frac{L\,z^{3}}{6})$. Moreover, in Equation \eqref{eq:push}, the approximation is valid for any $z$ such that $|z|\leq \Delta(\alphan)^{1/12}$, for some constant $\Delta$ depending only on the constant $C$ in the bound \eqref{eq:superbound}.\bigskip

Besides, for $x_n=o((\alphan)^{1/12})$, one has
\begin{align*}
\proba&\!\left[S_{n} \geq x_{n}\,(\alphan)^{\frac{2}{3}}\right]=\proba\!\left[X_{n} \geq x_n\,(\alphan)^{\frac{1}{3}}\right]=\varphi_{X_{n}}(x_n)\int_{y=x_n\,(\alphan)^{1/3}}^{\infty}\E^{-x_ny}\,\,\proba[Y_{n} \in dy]\\
&=\varphi_{X_{n}}(x_n)\,\E^{-\left((\alphan)^{1/3}\,(x_{n})^{2}+\frac{L\,(x_{n})^{3}}{2}\right)}\!\!\int_{z=-\frac{L\,(x_{n})^{2}}{2}}^{\infty} \E^{-x_n z}\, \,\proba[Z_{n} \in dz]\\
&=\exp\left(-\frac{(\alphan)^{1/3}\,(x_n)^{2}}{2}-\frac{L(x_n)^3}{3}\right)\,R_{n}\,(1+o(1)),
\end{align*}
by replacing $\varphi_{X_n}(x_n)$ by its estimate \eqref{EqPhiXn}, which holds since $x_n=o((\alphan)^{1/12})$; $R_n$ is the integral of the second line.\bigskip

To estimate the integral $R_n$, we shall adapt the proof of Proposition \ref{prop:berryesseen} to the special case of a sequence $(Z_n)_{n \in \N}$ that converges in the mod-Gaussian sense, with parameters $t_n$, limit function $\exp(Kz^3)$, and with the approximation
\begin{equation}
\log \varphi_{Z_n}(z) = \frac{t_n\,z^2}{2} + K\,z^3 + O\!\left(\frac{z}{(t_n)^{1/4}}\right)
\label{EqLogPhi}
\end{equation}
that is valid for every $|z| \leq \Delta (t_n)^{1/4}$ with $\Delta>0$. Notice that the sequence $(Z_n)_{n \in \N}$ previously constructed satisfies these hypotheses with $t_n = (\alphan)^{1/3} + Lx_n \simeq (\alphan)^{1/3}$. If one applies Proposition \ref{prop:berryesseen} to the case of mod-Gaussian convergence with a limit $\exp(Kz^3)$, then $\eta'''(0)=\psi'(0)=0$, so the approximation of the law $dF_n(w)$ of $Z_n/\sqrt{t_n}$ is simply the Gaussian law $dG(w) = (2\pi)^{-1/2}\,\E^{-w^2/2}\,dw$, and the Kolmogorov distance between $F_n$ and $G$ is a $o((t_n)^{-1/2})$. However, by using the validity of the approximation \eqref{EqLogPhi} on a larger scale than $z=O(1)$, it is possible to obtain a better Berry-Esseen bound, namely, $O((t_n)^{-3/4})$.\bigskip

Recall that for any $T>0$ and any $w \in \R$, the distance between cumulative distribution functions is smaller than
$$|F_n(w)-G(w)|\leq \frac{1}{\pi} \int_{-T}^{T} \left|\frac{f_n^{*}(\zeta)-g^{*}(\zeta)}{\zeta}\right|\, d\zeta +\frac{24m}{\pi T}.$$
However, for any $\zeta=O((t_n)^{3/4})$, one has
\begin{align*}
\left|\frac{f_n^{*}(\zeta)-g^{*}(\zeta)}{\zeta}\right| &= \E^{-\frac{\zeta^2}{2}} \,\left|\frac{\exp\left(\frac{K\,\zeta^3}{(t_n)^{3/2}} + O\!\left(\frac{\zeta}{(t_n)^{3/4}}\right)\right)-1}{\zeta}\right| \\
\E^{\frac{\zeta^2}{2}}\left|\frac{f_n^{*}(\zeta)-g^{*}(\zeta)}{\zeta}\right| &\leq \left|\E^{\frac{K\,\zeta^3}{(t_n)^{3/2}}}\right|\,\left|\frac{\exp\left(O\!\left(\frac{\zeta}{(t_n)^{3/4}}\right)\right)-1}{\zeta}\right| + \left|\frac{\exp\left(\frac{K\,\zeta^3}{(t_n)^{3/2}}\right)-1}{\zeta}\right|  \\
&\leq \E^{\frac{K\,|\zeta|^3}{(t_n)^{3/2}}} \,O\!\left(\frac{1}{(t_n)^{3/4}} + \frac{\zeta^2}{(t_n)^{3/2}}\right).
\end{align*}
In these inequalities, the constant hidden in the big $O$ can be chosen uniform if $\frac{\zeta}{(t_{n})^{3/4}}$ stays in a bounded, sufficiently small interval $[-\Delta,\Delta]$. As a consequence, setting $T=\Delta\,(t_n)^{3/4}$, one obtains from Feller's lemma:
$$|F_n(w)-G(w)|\leq O\left(\,\int_{-\Delta\,(t_n)^{3/4}}^{\Delta\,(t_n)^{3/4}} \left(\frac{1}{(t_n)^{3/4}} + \frac{\zeta^2}{(t_n)^{3/2}}\right)\, \E^{-\frac{\zeta^2}{2} + \frac{K\,|\zeta|^3}{(t_n)^{3/2}}}\,d\zeta +\frac{1}{\Delta\,(t_n)^{3/4}}\right)$$
uniformly in $w$.
Since $|\zeta|$ stays smaller than $\Delta\,(t_n)^{3/4}$, in this integral,
$$-\frac{\zeta^2}{2} + \frac{K\,|\zeta|^3}{(t_n)^{3/2}} \leq -\frac{\zeta^2}{2}\left(1-\frac{K\,\Delta}{(t_n)^{3/4}}\right) \leq -\frac{\zeta^2}{4}$$
for $t_n$ large enough. As claimed before, it follows that
$$\sup_{w \in \R} |F_n(w)-G(w)| = O\left(\frac{1}{(t_n)^{3/4}}\right).$$
\bigskip

We can now compute the asymptotics of the integral $R_n$. Set
$$\eps_n = -\frac{L(x_n)^2}{2\sqrt{(\alphan)^{1/3}+Lx_n}} = -\frac{L(x_n)^2}{2\sqrt{t_n}};$$
since $x_n = o((\alphan)^{1/12})$, $\eps_n \to 0$. Now,
\begin{align*}
R_n&=\int_{w=\eps_n}^\infty \exp\left(-w\,x_n\sqrt{t_n} \right)\,dF_n(w)\\
&=x_n\sqrt{t_n} \int_{\eps_n}^{\infty}\exp\left(-w\,x_n\sqrt{t_n} \right)\,(F_n(w)-F_n(\eps_n))\,dw \\
&=x_n\sqrt{t_n} \int_{\eps_n}^{\infty}\exp\left(-w\,x_n\sqrt{t_n} \right)\left(G_n(w)-G_n(\eps_n) + O\!\left(\frac{1}{(t_n)^{3/4}}\right)\right)\,dw \\
&=\frac{1}{\sqrt{2\pi}}\,\int_{\eps_n}^{\infty} \exp\left(-w\,x_n\sqrt{t_n}-\frac{w^2}{2} \right)\,dw + O\!\left(\frac{\E^{\frac{L\,(x_n)^3}{2}}}{(t_n)^{3/4}}\right).
\end{align*}
The last Gaussian integral is given by Lemma \ref{lem:gaussintegral}, \eqref{item:gausstail}:

\begin{align*}
\int_{\eps_n}^{\infty} \exp\left(-w\,x_n\sqrt{t_n}-\frac{w^2}{2} \right)\,dw &= \E^{\frac{t_n(x_n)^2}{2}}\int_{0}^{\infty} \exp\left(-\frac{(y+\eps_n+x_n\sqrt{t_n})^2}{2} \right)\,dy\\
&= \frac{\E^{\frac{L(x_n)^3}{2}}}{\eps_n+x_n\sqrt{t_n}}\, (1+o(1)) = \frac{\E^{\frac{L(x_n)^3}{2}}}{x_n\sqrt{t_n}} \,(1+o(1))
\end{align*}
since $\eps_n \to 0$. Since $x_n=o((\alphan)^{1/12}) = o((t_n)^{1/4})$, $\frac{1}{x_n\sqrt{t_n}}$ becomes much larger than $O(\frac{1}{(t_n)^{3/4}})$ as $t_n$ goes to infinity, so finally:
$$R_n = \frac{1}{\sqrt{2\pi}}\,\frac{\E^{\frac{L(x_n)^3}{2}}}{x_n(\alphan)^{1/6}} \,(1+o(1))$$
as $t_n\simeq (\alphan)^{1/3}$. Gathering everything, we get
$$\proba\!\left[S_{n} \geq x_{n}\,(\alphan)^{\frac{2}{3}}\right] = \frac{\E^{-\frac{(x_n)^2(\alphan)^{1/3}}{2}}}{x_n(\alphan)^{1/6}\sqrt{2\pi}}\,\E^{\frac{L(x_n)^3}{6}}\,(1+o(1)),$$
and this ends the proof if $\beta_n=\sigma^2=1$ (set $T=x_n\,(\alphan)^{2/3}$ in the statement of the Proposition). In the general case, it suffices to replace $S_n$ by $\frac{S_{n}}{\sigma\beta_n}$, which changes $L$ into $\frac{L}{\sigma^3}$ in the previous computations.
\end{proof}
\medskip

\begin{remark}
The argument which allows one to get a better Berry-Esseen estimate than in Proposition \ref{prop:berryesseen} can be used in a very general setting of mod-stable convergence, in order to get optimal bounds on the Kolmogorov distance. This will be the main topic of the forthcoming paper \cite{FMN14}.
\end{remark}
\bigskip

\subsection{Link with the Cram\'er-Petrov expansion}
Proposition \ref{prop:largedeviationscumulants} hints at a possible expansion of the fluctuations up to any order $T=o((\alphan)^{1-\eps})$, and indeed, it is a particular case of the results given by Rudzkis, Saulis and Statulevi\v{c}ius in \cite{RSS78,SS91}, see in particular \cite[Lemma 2.3]{SS91}. Suppose that
$$|\kappa^{(r)}(S_{n})|\leq (Cr)^{r}\,\alphan\,(\betan)^{r}\quad;\quad \kappa^{(r)}(S_{n})=K(r)\,\alphan\,(\betan)^{r}\,(1+O((\alphan)^{-1}))$$
the second estimate holding for any $r \leq v$; we denote $\sigma^{2}=K(2)$. In this setting, one can push the expansion up to order $o((\alphan)^{1-1/v})$. Indeed, define recursively for a sequence of cumulants $(\kappa^{(r)})_{r \geq 2}$ the coefficients of the Cram\'er-Petrov series $\lambda^{(r)}=-b_{r-1}/r$, with
$$\sum_{r=1}^{j} \frac{\kappa^{(r+1)}}{r!}\left(\sum_{\substack{j_{1}+\cdots+j_{r}=j\\ j_{i} \geq 1}} b_{j_{1}}b_{j_{2}}\cdots b_{j_{r}}\right)=\mathbbm{1}_{j=1}.$$
For instance, $\lambda^{(2)}=-\frac{1}{2}$, $\lambda^{(3)}=\frac{\kappa^{(3)}}{6}$, $\lambda^{(4)}=\frac{\kappa^{(4)}-3(\kappa^{(3)})^{2}}{24}$, \emph{etc.} The appearance of these coefficients can be guessed by trying to push the previous technique to higher order; in particular, the simple form of $\lambda^{(3)}$ is related to the fact that the only term in $z^{3}$ in the expansion \eqref{eq:push} is $\frac{\kappa^{(3)}}{6}$. If for the cumulants $\kappa^{(r)}$'s one has estimates of order $(\alphan)^{1-r/2}(1+O((\alphan)^{-1}))$, then one has the same estimates for the $\lambda^{(r)}$'s, so there exists coefficients $L(r)$ such that
$$\lambda^{(r)}\left(\frac{S_{n}}{\sigma\,\betan\,(\alphan)^{\frac{1}{2}}}\right)=L(r)\,(\alphan)^{1-r/2}\,(1+O((\alphan)^{-1})).$$ 
Take then $T=x_n\,(\alphan)^{\frac{v-1}{v}}$ with $x_n=O(1)$; Lemma 2.3 of \cite{SS91} ensures that 
\begin{align*}\proba\left[\frac{S_{n}}{\sigma\betan}\geq T\right]&=\frac{\E^{-\frac{T^{2}}{2\alphan}}}{\sqrt{2\pi\frac{T^{2}}{\alphan}}}\,\exp\left(\sum_{r=3}^{v} \lambda^{(r)}\,\left(\frac{T}{\sigma (\alphan)^{1/2}}\right)^{r}\right) \big(1+o(1)\big)\\
&=\frac{\E^{-\frac{T^{2}}{2\alphan}}}{\sqrt{2\pi\frac{T^{2}}{\alphan}}}\,\exp\left(\sum_{r=3}^{v} \,\frac{L(r)\,T^{r}}{\sigma^{r}\,(\alphan)^{r-1}}\right) \big(1+o(1)\big).
\end{align*}
Thus, the method of cumulants of Rudzkis, Saulis and Statulevi\v{c}ius can be thought of as a particular case (and refinement in this setting) of the notion of mod-$\phi$ convergence. However, their works do not yield a bound
$$|\kappa^{(r)}(S_{n})|\leq (Cr)^{r}\,\alphan\,(\betan)^{r}$$
but for simple cases, such as sums of i.i.d.~random variables. 
In Section \ref{sec:depgraph}, we show that dependency graphs
are an adequate framework to provide such bounds.\bigskip
\bigskip

\section{A precise version of the Ellis-G\"artner theorem}\label{sec:precisellis}

In the classical theory of large deviations, asymptotic results are formulated not only for the probabilities of tails $\proba[X_{n}\geq t_{n}x]$, but more generally for probabilities
$$\proba[X_{n} \in t_{n}B]\quad \text{with $B$ arbitrary Borelian subset of $\R$}.$$
In particular, under some technical assumptions on the generating series (that look like, but are somehow weaker than mod-convergence),  Ellis-G\"artner theorem provides some asymptotic upper and lower bounds for $\log(\proba[X_{n} \in t_{n}B])$, these bounds relying on a limiting condition on $(t_n)^{-1} \log\varphi_n(\cdot)$. When the topology of $B$ is {\em nice} enough, these bounds coincide  (see \emph{e.g.} \cite[Theorem 2.3.6]{DZ98}). This generalizes Cram\'er's large deviations for sums of i.i.d.~random variables. \bigskip

Our Theorems~\ref{thm:mainlattice} and \ref{thm:mainnonlattice} give estimates for the probabilities $\proba[X_{n} \geq t_{n}x]$ themselves, instead of their logarithm). Therefore, it is  natural to establish in the framework of mod-convergence a precise version of Ellis-G\"artner theorem. In this section, we shall give some asymptotic upper and lower bounds for the probabilities $\proba[X_{n} \in t_{n}B]$ itself instead of their logarithms. Once again, the upper and lower bounds coincide for {\em nice} borelian sets $B$.\medskip

\begin{remark}
In \cite{Multidim}, we shall prove similar estimates of $\proba[\mathbf{X}_n \in t_n B]$ in the setting of sequences of random \emph{vectors} that converge in the multi-dimensional mod-Gaussian sense.
 \end{remark} \bigskip

\subsection{Technical preliminaries}
In this section, we make the following assumptions:\vspace{2mm}
\begin{enumerate}
\item The random variables $X_{n}$ satisfy the hypotheses of Definition \ref{def:modphi} with $c=+\infty$ (in particular, $\psi$ is entire on $\C$). \vspace{2mm}
\item The Legendre-Fenchel transform $F$ is essentially smooth, that is to say that it takes finite values on a non-empty closed interval $I_{F}$ and that $\lim F'(x)=\lim h = \pm \infty$ when $x$ goes to a bound of the interval $I_{F}$ (\emph{cf.} \cite[Definition 2.3.5]{DZ98}). \vspace{2mm}
\end{enumerate}
The latter point is verified if $\phi$ is a Gaussian or Poisson law, which are the most important examples.
\begin{lemma}
Let $C$ be a closed subset of $\R$. Either $\inf_{u \in C}F(u)=+\infty$, or $\inf_{u \in C}F(u)=m$ is attained and $\{x \in C \,\,|\,\,F(x)=\min_{u \in C}F(u)\}$ consists of one or two real numbers $a \leq b$, with $a < \eta'(0) < b$ if $a \neq b$.
\end{lemma}
\begin{center}
\begin{figure}[ht]
\includegraphics{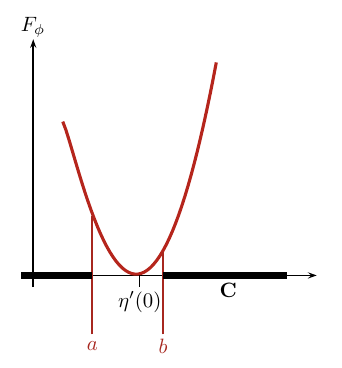}
\vspace{-5mm}
\caption{The infimum of $F$ on an admissible closed set $C$ is attained either at $a=\sup (C \cap (-\infty,\eta'(0)])$, or at $b=\inf( C \cap [\eta'(0),+\infty))$, or at both if $F(a)=F(b)$.}
\end{figure}
\end{center}
\begin{proof}
Recall that $F$ is strictly convex, since its second derivative is $1/\eta''(h)$, which is the inverse of the variance of a non-constant random variable. Also, $\eta'(0)$ is the point where $F$ attains its global minimum, and it is the expectation of the law $\phi$. If $C \cap I_{F}=\emptyset$, then $F_{|C}=+\infty$ and we are in the first situation. Otherwise, $F_{|C}$ is finite at some points, so there exists $M \in \R_{+}$ such that $C \cap \{x \in \R \,\,|\,\,F(x) \leq M\} \neq \emptyset$. However, the set $ \{x \in \R \,\,|\,\,F(x) \leq M\}$ is compact by the hypothesis of essential smoothness: it is closed as the reciprocal image of an interval $]-\infty,M]$ by a lower semi-continuous function, and bounded since $\lim_{x \to (I_{F})^{\mathrm{c}}} |F'(x)|=+\infty$. So, $C \cap \{x \in \R \,\,|\,\,F(x) \leq M\}$ is a non-empty compact set, and the lower semi-continuous $F$ attains its minimum on it, which is also $\min_{u\in C}F(u)$. Then, if $a \leq b$ are two points in $C$ such that $F(a)=F(b)=\min_{u \in C}F(u)$, then by strict convexity of $F$, $F(x)<F(a)$ for all $x \in (a,b)$, hence, $(a,b) \subset C^{\mathrm{c}}$. Also, $F(x)>F(a)$ if $a \neq b$ and $x \notin [a,b]$, so either $a=b$, or $\eta'(0) \in (a,b)$.
\end{proof}\bigskip

We take the usual notations $B^{\mathrm{o}}$ and $\overline{B}$ for the interior and the closure of a subset $B \subset \R$. Call \emph{admissible} a (Borelian) subset $B \subset \R$ such that there exists $b \in B$ with
$F(b) < +\infty$, and denote then 
$$F(B)=\inf_{u \in B}F(u)=\min_{u \in \overline{B}}F(u),$$
and $B_{\min}= \{a \in \overline{B}\,\,|\,\,F(a)=F(B)\}$; according to the previous discussion, $B_{\min}$ consists of one or two elements.

\subsection{A precise upper bound}
\begin{theorem}\label{thm:superellis}
Let $B$ be a Borelian subset of $\R$.\vspace{2mm}
\begin{enumerate}
\item\label{item:admissible} If $B$ is admissible, then
$$\limsup_{n \to \infty} \big(\sqrt{2\pi t_{n}}\,\exp(t_{n}F(B))\,\proba[X_{n} \in t_{n} B] \big) \leq \begin{cases}
&\sum_{a \in B_{\min}} \frac{\psi(h(a))}{(1-\E^{-|h(a)|})\,\sqrt{\eta''(h(a))}}\\
&\sum_{a \in B_{\min}} \frac{\psi(h(a))}{|h(a)|\,\sqrt{\eta''(h(a))}}
\end{cases}$$
the distinction of cases corresponding to $\phi$ lattice or non-lattice distributed. The sum on the right-hand side consists in one or two terms --- it is considered infinite if $a=\eta'(0) \in B_{\min}$. \vspace{2mm}
\item\label{item:nonadmissible} If $B$ is not admissible, then for any positive real number $M$,
$$\lim_{n \to \infty} \big( \exp(t_{n}M)\,\proba[X_{n} \in t_{n} B] \big)=0.$$
\end{enumerate}
\end{theorem}\medskip

\begin{proof}
For the second part, one knows that $\varphi_{n}(x)\,\exp(-t_{n}\eta(x))$ converges to $\psi(x)$ which does not vanish on the real line, so by taking the logarithms, 
$$\lim_{n \to \infty} \frac{\log \varphi_{n}(x)}{t_{n}}=\eta(x).$$
Then, Ellis-G\"artner theorem holds since $F$ is supposed essentially smooth. So, $$\limsup_{n \to \infty} \,\frac{\log \proba[X_{n} \in t_{n}B]}{t_{n}} \leq -F(B),$$ and if $B$ is not admissible, then the right-hand side is $-\infty$ and \eqref{item:nonadmissible} follows immediately.\bigskip

For the first part, suppose for instance $\phi$ non-lattice distributed. Take $C$ a closed admissible subset, and assume $\eta'(0) \notin C$ --- otherwise the upper bound in \eqref{item:admissible} is $+\infty$ and the inequality is trivially satisfied. Since $C^{\mathrm{c}}$ is an open set, there is an open interval $(a,b) \subset C^{\mathrm{c}}$ containing $\eta'(0)$, and which we can suppose maximal. Then $a$ and $b$ are in $C$ as soon as they are finite, and $C \subset (-\infty,a] \sqcup [b,+\infty)$. Moreover, by strict convexity of $F$, the minimal value $F(C)$ is necessarily attained at $a$ or $b$. Suppose for instance $F(a)=F(b)=F(C)$ --- the other situations are entirely similar. Then,
\begin{align*}\proba[X_{n} \in t_{n}C] &\leq \proba[X_{n} \leq t_{n}a]+\proba[X_{n}\geq t_{n}b] \\
&\lesssim \exp(-t_{n}F(C)) \left(\frac{\psi(h(a))}{-h(a)\sqrt{2\pi t_{n}\eta''(h(a))}}+\frac{\psi(h(b))}{h(b)\sqrt{2\pi t_{n}\eta''(h(b))}}\right) 
\end{align*}
by using Theorem \ref{thm:mainnonlattice} for $\proba[X_{n} \geq t_{n}b]$, and also for $\proba[X_{n}\leq t_{n}a]=\proba[-X_{n} \geq -t_{n}a]$ --- the random variables $-X_{n}$ satisfy the same hypotheses as the $X_{n}$'s with $\eta(x)$ replaced by $\eta(-x)$, $\psi(x)$ replaced by $\psi(-x)$, \emph{etc}. This proves the upper bound when $B$ is closed, and since $F(B)=F(\overline{B})$ by lower semi-continuity of $F$ and $B_{\min}=(\overline{B})_{\min}$, the result extends immediately to arbitrary admissible Borelian subsets.
\end{proof}\bigskip

\subsection{A precise lower bound}
One can then ask for an asymptotic lower bound on $\proba[X_{n} \in t_{n} B]$, and in view of the classical theory of large deviations, this lower bound should be related to open sets and to the exponent $F(B^{\mathrm{o}})$. Unfortunately, the result takes a less interesting form than Theorem \ref{thm:superellis}. If $B$ is a Borelian subset of $\R$, denote $B^{\delta}$ the union of the open intervals $(x,x+\kappa)$ of width $\kappa \geq \delta$ that are included into $B$. The interior $O=B^{\mathrm{o}}$ is a disjoint union of a countable collection of open intervals, and also the increasing union $\bigcup_{\delta >0}B^{\delta}$. 

\begin{center}
\begin{figure}[ht]\vspace{-1mm}
\includegraphics{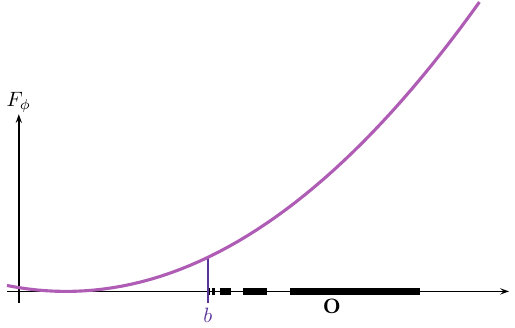}
\vspace{-2mm}
\caption{In some problematic situations, one is only able to prove a non-precise lower bound for large deviations.\label{fig:problem}}
\end{figure}
\end{center}

\noindent However, the topology of $B^{\mathrm{o}}$ may be quite intricate in comparison to the one of the $B^{\delta}$'s, as some points can be points of accumulation of open intervals included in $B$ and of width going to zero (see Figure \ref{fig:problem}). This phenomenon prevents us to state a precise lower bound when one of this point of accumulation is $a=\sup(B^{\mathrm{o}}\cap (-\infty,\eta'(0)])$ or $b=\inf (B^{\mathrm{o}} \cap [\eta'(0),+\infty))$. Nonetheless, the following is true:\medskip 

\begin{theorem}\label{thm:hyperellis}
For an admissible Borelian set $B$, 
$$\liminf_{\delta\to 0} \liminf_{n \to \infty} \big(\sqrt{2\pi t_{n}}\,\exp(t_{n}F(B^{\delta}))\,\proba[X_{n} \in t_{n} B] \big) \geq \begin{cases}
&\sum_{a \in (B^{\mathrm{o}})_{\min}} \frac{\psi(h(a))}{(1-\E^{-|h(a)|})\,\sqrt{\eta''(h(a))}}\\
&\sum_{a \in (B^{\mathrm{o}})_{\min}} \frac{\psi(h(a))}{|h(a)|\,\sqrt{\eta''(h(a))}},
\end{cases}$$
with the usual distinction of lattice/non-lattice cases. In particular, the right-hand side in Theorem \ref{thm:superellis} is the limit of $\sqrt{2\pi t_{n}}\,\exp(t_{n}F(B))\,\proba[X_{n} \in t_{n} B]$ as soon as $F(B^{\delta})=F(B)$ for some $\delta>0$.
\end{theorem}

\begin{proof}
Again we deal with the non-lattice case, and we suppose for instance that the set $(B^{\mathrm{o}})_{\min}$ consists of one point $b = \inf (B^{\mathrm{o}} \cap [\eta'(0),+\infty))$, the other situations being entirely similar. As $\delta$ goes to $0$, $B^{\delta}$ increases towards $B^{\mathrm{o}}=\bigcup_{\delta >0}B^{\delta}$, so the infimum $F(B^{\delta})$ decreases and the quantity $$L(\delta)=\liminf_{n \to \infty} \big(\sqrt{2\pi t_{n}}\,\exp(t_{n}F(B^{\delta}))\,\proba[X_{n} \in t_{n} B]\big)$$ is decreasing in $\delta$. 
Actually, if $b^{\delta} = \inf (B^{\delta} \cap [\eta'(0),+\infty))$, then for $\delta$ small enough $F(B^{\delta})=F(b^{\delta})$, so $\lim_{\delta \to 0}F(B^{\delta})=F(B^{\mathrm{o}})$ by continuity of $F$. On the other hand,  
$$R(\delta) = \frac{\psi(h(b^{\delta}))}{h(b^{\delta})\sqrt{\eta''(h(b^{\delta}))}}$$
tends to the same quantity with $b$ instead of $b^{\delta}$. Hence, it suffices to show that for $\delta$ small enough, $L(\delta)\geq R(\delta)$. However, by definition of $B^{\delta}$, the open interval $(b^{\delta},b^{\delta}+\delta)$ is included into $B$, so 
\begin{align*}\proba[X_{n} \in t_{n}B]&\geq \proba[X_{n} \in t_{n}B^{\delta}] \geq \proba[X_{n} > t_{n}b^{\delta}] - \proba[X_{n} \geq t_{n}(b^{\delta}+\delta)]\\
&\geq \left(\frac{\psi(h(b^{\delta}))\,\E^{-t_{n}F(b^{\delta})}}{h(b^{\delta})\sqrt{2\pi t_{n}\eta''(h(b^{\delta}))}}- \frac{\psi(h(b^{\delta}+\delta))\,\E^{-t_{n}F(b^{\delta}+\delta)}}{h(b^{\delta}+\delta)\sqrt{2\pi t_{n}\eta''(h(b^{\delta}+\delta))}}\right)\big(1+o(1)\big)\\
&\geq \frac{\psi(h(b^{\delta}))\,\E^{-t_{n}F(b^{\delta})}}{h(b^{\delta})\,\sqrt{2\pi t_{n}\eta''(h(b^{\delta}))}}\big(1+o(1)\big)
\end{align*}
since the second term on the second line is negligible in comparison to the first term --- $F(b^{\delta}+\delta)>F(b^{\delta})$. This ends the proof.
\end{proof}
\bigskip

\section{Examples with an explicit generating function}\label{sec:firstexamples}

The general results of Sections \ref{sec:lattice} and \ref{sec:precisellis} can be applied in many contexts, and the main difficulty is then to prove for each case that one has indeed the estimate on the Laplace transform given by Definition \ref{def:modphi}. Therefore, the development of techniques to obtain mod-$\phi$ estimates is an important part of the work. Such an estimate can sometimes be established from an explicit expression of the Laplace transform (hence of the characteristic function); we give several examples of this kind in Section \ref{subsec:explicit}. But there also exist numerous techniques to study sequences of random variables without explicit expression for the characteristic function: complex analysis methods in number theory (Section \ref{subsec:arithmetic}) and in combinatorics (Section \ref{subsec:cycles}), localization of zeros (Section \ref{sec:zeros}) and dependency graphs (Sections \ref{sec:central}, \ref{sec:erdosrenyi} and \ref{sec:depgraph}) to name a few. These methods are known to yield central limit theorems and we show how they can be adapted to prove mod-convergence. We illustrate each case with one or several example(s).
\footnote{V: Le paragraphe ci-dessus serait bien dans l'introduction peut-être\ldots}
\bigskip

In this section, we detail examples for which the mod-$\phi$ convergence has already been proved before (\emph{cf.} \cite{JKN11,DKN11}) or follows easily from formulas in the literature.\bigskip

\subsection{Mod-convergence from an explicit formula for the Laplace transform}
\label{subsec:explicit}
Examples where mod-convergence is proved using an explicit formula for the Laplace transform were already given in Section \ref{subsec:basicexamples}. We provide here two additional examples of mod-Gaussian convergence that can be obtained by this method.

\begin{example}
    \label{ex:zeros_randomfunction}
Let $f(z)=\sum_{n=0}^\infty a_n\,z^n$ be a random analytic function, where the coefficients $a_n$ are independent standard complex Gaussian variables. The random function $f$ has almost surely its radius of convergence equal to $1$, and its set of zeroes $\mathcal{Z}(f)=\{z \in \mathcal{D}_{(0,1)}\,\,|\,\,f(z)=0 \}$ is a determinantal point process on the unit disk, with kernel 
$$\mathcal{K}(w,z)=\frac{1}{\pi(1-w\overline{z})^2}=\frac{1}{\pi}\sum_{k=1}^\infty k(w\overline{z})^{k-1}.$$
We refer to \cite{Zero} for precisions on these results. It follows then from the general theory of determinantal point processes, and the radial invariance of the kernel, that the number $N_r$ of points of $\mathcal{Z}(f) \cap B_{(0,r)}$ can be represented in law as a sum of independent Bernoulli variables of parameters $\{r^{2k}\}_{k \geq 1}$:
$$N_r=\mathrm{card}\{z \in \mathcal{Z}(f)\,\,|\,\,|z|\leq r\} =_{\text{law}} \sum_{k=1}^\infty\mathcal{B}(r^{2k}).$$
This representation as a sum of independent variables allows one to estimate the moment generating function of $N_r$ under various renormalizations. Let us introduce the hyperbolic area $$h=\frac{4\pi\,r^2}{1-r^2}$$ of $\mathcal{D}_{(0,r)}$, and denote $N_r=N^h$; we are interested in the asymptotic behavior of $N^h$ as $h$ goes to infinity, or equivalently as $r$ goes to $1$. Since $\esper[\E^{zN_r}]=\prod_{k=1}^\infty (1+r^{2k}(\E^z-1))$, one has
$$
\log \left(\esper\!\left[\E^{\frac{zN^h}{h^{1/3}}}\right]\right)=\sum_{k=1}^\infty \,\log\left(1+r^{2k}\left(\E^\frac{z}{h^{1/3}}-1\right)\right) =\frac{h^{2/3}z}{4\pi}+ \frac{h^{1/3}z^2}{16\pi}+\frac{z^3}{144\pi}+o(1)
$$
with a remainder that is uniform when $z$ stays in a compact domain of $\C$. Therefore, $$\frac{N^h-\frac{h}{4\pi} }{h^{1/3}}\,\text{ converges mod-Gaussian with }\begin{cases} &\!\!\text{parameters } t_h=\frac{h^{1/3}}{8\pi},\\
 &\!\!\text{limiting function }\psi(z)=\exp\left(\frac{z^3}{144\pi}\right). 
 \end{cases}$$
Again, the limiting function is the exponential of a simple monomial. By Proposition \ref{prop:normalityzone}, 
$$X^h = \frac{N^h - \frac{h}{4\pi}}{\sqrt{\frac{h}{8\pi}}}$$
converges as $h \to \infty$ to a Gaussian law, with normality zone $o(h^{1/6})$. Moreover, by Theorem \ref{thm:mainnonlattice}, at the edge of this normality zone,
$$\proba\!\left[N^h - \frac{h}{4\pi} \geq \frac{h^{2/3}}{4\pi}\,x\right] = \frac{\E^{-\frac{h^{1/3}\,x^2}{4\pi}}}{h^{1/6}\,x}\,\exp\left(\frac{x^3}{18\pi}\right)\,(1+o(1))$$
for any $x>0$.
\end{example}
\medskip

\begin{example}\label{ex:ising}
Consider the Ising model on the discrete torus $\Z/n\Z$. Thus, we give to each spin configuration $\sigma : \Z/n\Z \to\{\pm1\}$ a probability proportional to the factor $\exp(-\beta \sum_{i \sim j} \mathbbm{1}_{\sigma(i)\neq \sigma(j)})$, the sum running over neighbors in the circular graph $\Z/n\Z$. The technique of the transfer matrix (see \cite[Chapter 2]{Baxter82}) ensures that if $M_n=\sum_{i=1}^n \sigma(i)$ is the total magnetization of the model, then
$$\esper[\E^{zM_n}]=\frac{\mathrm{tr}\, (T(z))^n}{\mathrm{tr}\, (T(0))^n},\qquad\text{where }T(z)=\begin{pmatrix}\E^{-z} & \E^{z-\beta}\\
\E^{-z-\beta}&\E^z\end{pmatrix}.$$
The two eigenvalues of $T(z)$ are $\cosh z \pm \sqrt{(\sinh z)^2 +\E^{-2\beta}}$, and their Taylor expansion shows that
$$\log \left(\esper\!\left[\E^{\frac{zM_n}{n^{1/4}}}\right]\right)=\frac{\E^\beta\,n^{1/2}\,z^2}{2}-\frac{(3\E^{3\beta}-\E^\beta)\,z^4}{24}+o(1).$$
So, one has mod-Gaussian convergence for $n^{-1/4}\,M_n$, and the estimates
$$\proba[M_n \geq n^{3/4}x]=\frac{\E^{-\frac{n^{1/2}\,\E^\beta\,x^2}{2}}}{x\,\sqrt{2\pi n^{1/2}}}\,\exp\left(-\frac{(3\E^{3\beta}-\E^\beta)\,x^4}{24}\right)\,(1+o(1)).$$
In particular, $\frac{M_n}{n^{1/2}\,\E^{\beta/2}}$ satisfies a central limit theorem with normality zone $o(n^{1/4})$.
\end{example}
\medskip

\begin{remark}
If instead of $\Z/n\Z$ we consider the graph $1 \leftrightarrow 2 \leftrightarrow \cdots \leftrightarrow n$ (\emph{i.e.}, one removes the link $n \leftrightarrow 1$), then one can realize the spins $\sigma(i)$ of the Ising model as the $n$ first states of a Markov chain with space of states $\{\pm 1\}$. The magnetization $M_n$ appears then as a linear functional of the empirical measure of this Markov chain. More generally, if $S_n = \sum_{i=1}^n f(X_i)$ is a linear functional of a Markov chain on a finite space, then under mild hypotheses this sum satisfies the Markov chain central limit theorem (\emph{cf.} for instance \cite{Cog72}). In \cite{FMN14}, we shall use arguments of the perturbation theory of operators in order to prove that one also has mod-Gaussian convergence for such linear functionals of Markov chains.
\end{remark}
\bigskip

\subsection{Additive arithmetic functions of random integers} 
\label{subsec:arithmetic} In this paragraph, we sometimes write $\log(\log n) = \log_2 n$, which is negligible in comparison to $\log n$ as $n$ goes to infinity. 

\subsubsection{Number of prime divisors, counted without multiplicities}
Denote $\mathbb{P}$ the set of prime numbers, $\omega(k)$ the number of distinct prime divisors of an integer $k$, and $\omega_{n}$ the random variable $\omega(k)$ with $k$ random integer uniformly chosen in $[n]:=\{1,2,\ldots,n\}$. The random variable $\omega_{n}$ satisfies the Erd\"os-Kac central limit theorem (\emph{cf.} \cite{EK40}):
$$\frac{\omega_{n}-\log\log n}{\sqrt{\log\log n}} \to \mathcal{N}_{\R}(0,1).$$ 
In this section, we show that $\omega_{n}$ converges mod-Poisson and present precise deviation results for it. Such a result was established in \cite[Section 4]{KN10}. But in the latter article, mod-$\phi$ convergence is defined via convergence of the renormalized Fourier transform, while here we work with Laplace transform. Therefore, we need to justify that the convergence also holds for Laplace transforms.
\bigskip

We start from the {\em Dirichlet series} of $y^{\omega(k)}$, which is:
\[\sum_{k \ge 1} \frac{y^{\omega(k)}}{k^s} = \prod_{p \in \mathbb{P}} \left( 1 + \frac{y}{p^s-1} \right),\]
well-defined and absolutely convergent if $\Re(s)>1$. The Selberg-Delange method allows to extract from this formula precise estimates for the generating function of $\omega_{n}$; see \cite{Ten95} and references therein.

\begin{proposition}\cite[Section II.6, Theorem 1]{Ten95}
For any $A >0$, we have, for any $y$ in $\C$ with $|y| \le A$
\[ \sum_{k \le n} y^{\omega(k)} = n\, (\log n)^{y-1} \, (\lambda_0(y) +O(1/\log n)), \]
where
\[\lambda_0(y)= \frac{1}{\Gamma(y)} \, \prod_{p \in \mathbb{P}} \left( 1 + \frac{y}{p-1} \right)\left( 1- \frac{1}{p} \right)^y \]
and the constant hidden in the $O$ symbol depends only on $A$.
\label{prop:selbergdelange_omega}
\end{proposition}

\begin{remark}
In fact, \cite[Section II.6, Theorem 1]{Ten95} gives a complete asymptotic expansion of $\sum_{k \le n} y^{\omega(k)}$ in terms of powers of $1/\log(n)$. For our purpose, the first term is enough: we will see that it implies mod-Poisson convergence with a speed of convergence as precise as wanted.
\end{remark}
\bigskip

Setting $y=\E^z$, this can be rewritten as an asymptotic formula for the Laplace transform of $\omega_{n}$.
\begin{align*}
\esper[\E^{z\,\omega_{n}}] &= \frac{1}{n} \sum_{k=1}^n \E^{z\,\omega(k)} 
= (\log n)^{\E^z-1} \lambda_0(\E^z) \,(1+ O(1/\log n))\\
&= \E^{(\E^z-1) \log \log n}\, \lambda_0(\E^z) \,(1+ O(1/\log n)).
\end{align*}
Recall that the constant in the $O$ symbol is uniform on sets $\{w, |w| \le A\}$, that is on bands $\{z,\, \Re(z) \le \log(A)\}$. In particular the convergence is uniform on compact sets. Therefore, one has the following result.

\begin{proposition}
    The sequence of random variables $(\omega_n)_{ n \ge 1}$ converges mod-Poisson with parameter $t_n=\log_2 n$ and limiting function $\psi(z)=\lambda_0(\E^z)$ on the whole complex plane. This takes place with speed of convergence $O(1/\log n)$, that is $O((t_n)^{-\nu})$ for all $\nu >0$.
\end{proposition}

Using Theorem \ref{thm:mainlattice}, we get immediately the following deviation result.
\begin{theorem}
    Let $x>0$. Assume $x \log_2 n \in \N$. Then
    \begin{equation}
        \proba[\omega_n=x \log_2 n] =
        \frac{ \lambda_0(x)\,\,(1 +O(1/\log_2 n))}{(\log n)^{x \log(x)-x+1} \sqrt{2 \pi x\, \log_2 n}}.
    \label{eq:LLL_omega}
\end{equation}
    Furthermore, if $x>1$, then
    \begin{equation}
        \proba[\omega_n \ge x \log_2 n] =
        \frac{ \lambda_0(x)\,\,(1 +O(1/\log_2 n ))}{(\log n)^{x \log(x)-x+1} \sqrt{2 \pi x\, \log_2 n }} \,\frac{1}{1-\frac{1}{x}}.
    \label{eq:deviations_omega}
\end{equation}
\end{theorem}

\begin{remark}
The first equation \eqref{eq:LLL_omega} is not new: it is due to Selberg \cite{Selberg54} and presented in a slightly different form than here in \cite[Section II.6, Theorem 4]{Ten95}. Note also that, as the speed of mod-Poisson convergence of $\omega_n$ is $O\big(\big(t_n\big)^{-\nu}\big)$ for all $\nu >0$, Theorem \ref{thm:mainlattice} gives asymptotic expansions of the above probabilities, up to an arbitrarily large power of $1/\log\log n$. Theorem 4 in \cite[Section II.6]{Ten95} also gives such estimates. The second statement \eqref{eq:deviations_omega} follows from \cite[Theorem 2.8]{Rad09}; it is a nice feature of the theory of mod-$\phi$ convergence to allow to recover quickly such deep arithmetic results (though we still need Selberg-Delange asymptotics).
\end{remark}
\bigskip

\subsubsection{Number of prime divisors, counted with multiplicities}
In this section, we give similar results for the number of prime divisors $\Omega_n$ of a random integer in $\{1,2,\ldots,n\}$, {\em counted with multiplicities}. An important difference is that, here, the mod-Poisson convergence occurs only on a band and not on the whole complex plane $\C$. In this case the Dirichlet series is given by:
\[\sum_{k \ge 1} \frac{y^{\Omega(k)}}{k^s} = \prod_{p \in \mathbb{P}} \left( 1 - \frac{y}{p^s} \right)^{\!-1},\]
which again is well-defined and absolutely convergent for $\Re (s)>1$. Note that, unlike in the case of $\omega(k)$, the right-hand side has some pole, the smallest in modulus being for $y=2^s$. Again, a precise estimate for the generating function follows from the work of Selberg and Delange; see \cite{Ten95} and references therein.

\begin{proposition}\cite[Section II.6, Theorem 2]{Ten95}
For any $\delta$ with $0<\delta<1$, we have, for any $y$ in $\C$ with $|y| \le 2-\delta$
\[ \sum_{k \le n} y^{\Omega(k)} = n\, (\log n)^{y-1} \, (\nu_0(y) +O(1/\log n)), \]
where
\[\nu_0(y)= \frac{1}{\Gamma(y)} \, \prod_{p \in \mathbb{P}} \left(1 - \frac{y}{p} \right)^{-1}\left( 1- \frac{1}{p} \right)^y \]
and the constant hidden in the $O$ symbol depends only on $\delta$.
\label{prop:selbergdelange_Omega}
\end{proposition}

\noindent Note the difference with Proposition \ref{prop:selbergdelange_omega}: the function $\nu_0$ has a simple pole for $y=2$ (while $\lambda_0$ is an entire function) and the estimate in Proposition \ref{prop:selbergdelange_Omega} holds (uniformly on compacts) {\em only for $|y| < 2$}.\bigskip

Setting again $y=\E^z$, this can be rewritten as an asymptotic formula for the Laplace transform of $\Omega_{n}$
 on the band $\mathcal{S}_{(-\infty,\log 2)}=\{z, \, \Re(z) < \log 2\}$.
\[
    \esper[\E^{z\,\Omega_{n}}] = \frac{1}{n} \sum_{k=1}^n \E^{z\,\Omega(k)} 
= \E^{(\E^z-1) \log \log n}\, \nu_0(\E^z) \,(1+ O(1/\log n)).
\]
Recall that the constant in the $O$ symbol is uniform on sets $\{y, |y| \le 2-\delta\}$, that is on bands $\{z,\, \Re(z) < \log(2-\delta)\}$. In particular the convergence is uniform on compact sets. Therefore, one has the following result.
\begin{proposition}
The sequence of random variables $(\Omega_n)_{ n \ge 1}$ converges mod-Poisson with parameter $t_n=\log_2 n$ and limiting function $\psi(z)=\nu_0(\E^z)$, on the band $\mathcal{S}_{(-\infty,\log 2)}$. This takes place with speed of convergence $O(1/\log(n))$, that is $O((t_n)^{-\nu})$, for all $\nu >0$.
\end{proposition}
\bigskip

As for $\omega_n$, this implies precise deviation results. However, as the convergence only takes place on a band,
the range of these results is limited (note the condition $x<2$ in the theorem below).
\begin{theorem}
    Fix $x$ with $0<x<2$. Assume $x \log_2 n \in \N$. Then
    \begin{equation}
        \proba[\Omega_n=x \log_2 n] =
        \frac{ \nu_0(x)\,\,(1 +O(1/\log_2 n))}{(\log n)^{x \log(x)-x+1} \sqrt{2 \pi x \,\log_2 n}}.
    \label{eq:LLL_Omega}
\end{equation}
    Furthermore, if $1<x<2$,
    \begin{equation}
        \proba[\Omega_n \ge x \log_2 n ] =
        \frac{ \nu_0(x)\,\,(1 +O(1/\log_2 n))}{(\log n )^{x \log(x)-x+1} \sqrt{2 \pi x \,\log_2 n}}\, \frac{1}{1-\frac{1}{x}}.
    \label{eq:deviations_Omega}
\end{equation}
\end{theorem}
\noindent Again, the fist equation was discovered by Selberg \cite{Selberg54} and can be found in a slightly different form in \cite[Section II.6, Theorem 5]{Ten95}.

\begin{remark}
An extension of Equation \eqref{eq:LLL_Omega} for $x>2$ is given in \cite[Section II.6, Theorem 6]{Ten95}. This involves the type of singularity of the limiting function $\nu_0(\E^z)$ at the edge of the convergence domain. Here, $\nu_0(y)$ has only a simple pole in $y=2$, and the residue appears in the deviation results. It would be interesting to see if this kind of idea can be used in the general framework of mod-$\phi$ convergence, but this is outside the scope of this already quite long paper.
\end{remark}
\bigskip

\subsubsection{Other arithmetic functions}
In this paragraph, we generalize the two examples $\omega_n$ and $\Omega_n$ above. The goal is to understand the phase transition between the convergence on a band and the convergence on the complex plane.
\medskip

Let $f: \N \to \Z$ be a function with the following properties:\vspace{2mm}
\begin{enumerate}[(i)]
\item\label{hyp:additivefunction1} $f$ is additive, that is $f(mn)=f(m)+f(n)$ if $m \wedge n =1$;\vspace{2mm}
\item\label{hyp:additivefunction2} for every prime $p$, one has $f(p)=1$.\vspace{2mm}
\end{enumerate}
We necessarily have $f(1)=0$. For $C>0$, we say that such a function {\em has a $C$-linear growth} if there exists $B$ such that
$$|f(p^k)| \le B+C\, k \quad\text{for any }p\in \mathbb{P} \text{ and any }k>0.$$
If $f$ has a $C$-linear growth for any $C>0$, we say that $f$ has a {\em sublinear growth}. In particular, since $\Omega(p^k)=k$ (for all $p$ and $k$), the function $\Omega$ has a $1$-linear growth, while $\omega$ has a sublinear growth (for all $p$ and $k$, $\omega(p^k)=1$). We are interested in the random variable $f_n$, which is equal to the value $f(k)$ of $f$ on a random integer $k$ chosen uniformly between $1$ and $n$. Note that the additivity condition ensures that the following Dirichlet series factorizes:
\[F(y,s):= \sum_{k=1}^\infty \frac{y^{f(k)}}{k^s} = \prod_{p\in \mathbb{P}} \left( 1+\frac{y^{f(p)}}{p^s}+\frac{y^{f(p^2)}}{p^{2s}}+\cdots \right). \]
Again, one can use the Selberg-Delange method to get precise estimates for the Laplace transform.
\begin{proposition}
\label{prop:selbergdelange_general}
Suppose that $f$ fulfills \eqref{hyp:additivefunction1} and \eqref{hyp:additivefunction2} and has a $C$-linear groth for some $C>0$. For any $\delta$ with $0<\delta<\frac{1}{2}$, we have, for any $y$ in $\C$ with $(2-2\delta)^{-1/C} \le |y| \le (2-2\delta)^{1/C}$,
\begin{equation}
 \sum_{k \le n} y^{f(k)} = n\, (\log n)^{y-1} \, (\tau_0(y) +O(1/\log n)),
\label{eq:asymptotics_generaladditivefunction}
\end{equation}
where the constant hidden in the $O$ symbol depends on $\delta$, $B$ and $C$ (but not on $y$), and
\[\tau_0(y)= \frac{1}{\Gamma(y)} \, \prod_{p \in \mathbb{P}}
\left( 1 + \frac{y^{f(p)}}{p}+\frac{y^{f(p^2)}}{p^2} + \cdots \right)\left( 1- \frac{1}{p} \right)^y. \]
\end{proposition}

\begin{proof}
The proof is an easy adaptation of the proof of \cite[Section II.6, Theorems 1 and 2]{Ten95}. We will only indicate the necessary modifications, assuming that the reader can consult Tenenbaum's book.\medskip

Fix $\delta\in(0,\frac{1}{2})$, and $y$ in $\C$ with $(2-2\delta)^{-1/C} \le |y| \le (2-2\delta)^{1/C}$. We set $Y = \max(|y|,|y|^{-1})$; by assumption, $1\leq Y\leq (2-2\delta)^{1/C}$, and on the other hand,
$$|y^{f(p^k)}| \leq Y^{B+Ck}
$$
as $f$ has a $C$-linear growth. We have, for any prime $p$ and $|\xi|<Y^{-C}$, the bound
\[\sum_{k \ge 0} |y^{f(p^k)}\, \xi^k| < 1+\sum_{k \ge 1} Y^{B+Ck} |\xi|^k \leq 1+ Y^B \frac{Y^C |\xi|}{1-Y^C |\xi|}.
\]
Therefore, for each prime $p$, the following formula defines a holomorphic function for $|\xi|<Y^{-C}\leq 1$:
\[h_{y,p}(\xi)= \left( 1+ y^{f(p)} \xi+ y^{f(p^2)} \xi^2 +\dots \right) \left( 1- \xi \right)^y. \]
Moreover, we have
\[|h_{y,p}(\xi)| < \left(1+ Y^B \frac{Y^C |\xi|}{1-Y^C |\xi|} \right)\,\E^{Y\,|\log(1-\xi)|}.\]
Consider now the coefficients of the power series expansion of $h_{y,p}$ around the origin,
{\em i.e.} the numbers $b_y(p^k)$ such that
\[h_{y,p}(\xi)= 1 + \sum_{k \ge 1} b_y(p^k) \xi^k.\]
Note that $b_y(p)=0$ for any $p \in \mathbb{P}$. With the previous hypotheses, 
$$\frac{1}{2-\delta}< Y^{-C}\leq |y|^{-C}.$$
By Cauchy's formula applied to the circle of radius $1/(2-\delta)$, one has, for any prime $p$, any complex number $y$ such that $(2-2\delta)^{1/C}\le |y| \le (2-2\delta)^{1/C}$, and any $k \ge 1$:
\[ |b_w(p^k)| \le M(\delta,B,C)\, (2-\delta)^k, \]
where 
\[M(\delta,B,C) = \sup_{\substack{|\xi| = 1/(2-\delta) \\  1 \leq Y \leq  (2-2\delta)^{1/C} }}
\left(1+ Y^B \frac{Y^C |\xi|}{1-Y^C |\xi|} \right)\,\E^{Y\,|\log(1-\xi)|} < +\infty. \]

Denote $\zeta$ is the Riemann zeta function. The same arguments as in the proof of \cite[Section II.6, Theorem 1]{Ten95} shows that
\[G(y,s):=F(y,s)\, \zeta(s)^{-y} = \prod_{p \in \mathbb{P}} \left (1 + \sum_{k \ge 2} \frac{b_y(p^k)}{p^{ks}} \right)\]
is convergent and uniformly bounded for $\Re(s)>3/4$ and $y $ such that $(2-2\delta)^{-1/C}\le |y| \le (2-2\delta)^{1/C}$.
Therefore, we can apply \cite[Section II.5, Theorem 3]{Ten95}, which gives the asymptotic formula \eqref{eq:asymptotics_generaladditivefunction}.
\end{proof}

Setting $y=\E^z$, Proposition \ref{prop:selbergdelange_general} can be rewritten as an asymptotic formula for the Laplace transform of $f_{n}$:
\[
    \esper[\E^{z\,f_{n}}] = \frac{1}{n} \sum_{k=1}^n \E^{z\,f(k)} 
= \E^{(\E^z-1) \log \log n)} \,(\tau_0(\E^z) + O(1/\log n)),
\]
where the constant hidden in $O$ depends on $\delta$, $B$ and $C$, but is uniform on $|\Re(z)| \le (\log(2-2\delta))/C$. This yields the following mod-convergence result:

\begin{proposition}
Assume $f$ is an arithmetic function that fulfills \eqref{hyp:additivefunction1} and \eqref{hyp:additivefunction2} and has a $C$-linear growth for some constant $C>0$. Then, the sequence of random variables $(f_n)_{ n \ge 1}$
converges mod-Poisson with parameter $t_n=\log_2 n$ and limiting function $\psi(z)=\tau_0(\E^z)$ on the band $$\mathcal{S}_{(-\frac{\log 2}{C},\frac{\log 2}{C})}.$$ 
This takes place with speed of convergence $O(1/\log n)$, that is $O((t_n)^{-\nu})$ for all $\nu >0$.
\medskip

\noindent As a consequence, if $f$ has a sublinear growth, then the convergence takes place on the whole complex plane.
\end{proposition}

\noindent It is then straightforward to write deviation results, analog to Equations \eqref{eq:deviations_omega} and \eqref{eq:deviations_Omega}, for a generic function $f$ as above.
\medskip

\begin{remark}
If the additive function $f$ is non-negative, as is the case for $\omega(n)$ and $\Omega(n)$, then it is easily seen that one can perform the whole proof of Proposition \ref{prop:selbergdelange_general} assuming only $|y| \leq (2-2\delta)^{1/C}$ (but not $|y| \geq (2-2\delta)^{-1/C}$). Therefore, one has in this case a mod-Poisson convergence on the band $$\mathcal{S}_{(-\infty,\frac{\log 2}{C})}\quad \text{instead of}\quad \mathcal{S}_{(-\frac{\log 2}{C},\frac{\log 2}{C})}.$$
\end{remark}
\bigskip

\subsection{Number of cycles in weighted probability measure}
\label{subsec:cycles}
We consider here the number of cycles of random permutations under the so-called weighted probability measure.
This example was already considered in \cite{NZ13} and we follow the presentation of this article.
\bigskip

Denote $X_{n}(\sigma)$ the number of disjoint cycles (including fixed points) of a permutation $\sigma$  in the symmetric group $\mathfrak{S}(n)$. We write $C_j(\sigma)$ for the number of cycles of length $j$ in the decomposition of $\sigma$ as a product of disjoint cycles; thus, $X_{n}(\sigma) =\sum_{j=1}^n C_j(\sigma)$ and $n = \sum_{j=1}^n j\,C_j(\sigma)$. Let $\Theta=(\theta_m)_{m\geq1}$ be given with $\theta_m\geq 0$. The generalized weighted measure is defined as the probability measure $\mathbb{P}_{\Theta} $ on the symmetric group $\mathfrak{S}(n)$:
$$\mathbb P_{\Theta}[\sigma]=\dfrac{1}{h_n\, n!}\prod_{m=1}^n (\theta_m)^{C_m(\sigma)}$$ with $h_n$ a normalization constant (or a partition function) and $h_0=1$. This model  is coming from statistical mechanics and the study of Bose quantum gases (see \cite{NZ13} for more references and details). It generalizes the classical cases of the uniform measure (corresponding to $\theta_m\equiv 1$) and the Ewens measure (corresponding to the case $\theta_m\equiv\theta>0$). It has been an open question to prove a central limit theorem for the total number of cycles $X_n$ under such measures (or more precisely under some specific regimes related to the asymptotic behavior of the $\theta_m$'s; such a central limit theorem was already known for Ewens measure). We refer to \cite{EU14} for a nice survey of this question. The difficulty comes from the fact the there is nothing such as the Feller coupling anymore, and we cannot apply the same method as in Example \ref{ex:cycle}. We now show how singularity analysis allows us to prove mod-Poisson convergence, and hence the central limit theorem, but also distributional approximations and precise large deviations.
\medskip

We consider the generating series 
$$g_\Theta(t)=\sum_{n=1}^\infty \dfrac{\theta_n}{n} \,t^n.$$ 
It is well known that 
\begin{equation}\label{eq:gf_hn}
\sum_{n=0}^\infty h_n \,t^n=\exp(g_\Theta(t)).
\end{equation}
Set $F(t,w)=\exp(w g_\Theta(t))$. Using elementary combinatorial arguments (which are detailed in \cite{NZ13}),
one can prove for each $z\in\mathbb{C}$ the following identity as formal power series:
\begin{equation}\label{eq:combtheta}
\sum_{n=0}^\infty h_n \,\mathbb E_{\Theta} [\exp(z X_n)]\,t^n=\exp(\E^z g_\Theta(t))=F(t,\E^z).
\end{equation}
Our goal is to obtain an asymptotic for $h_n$ and for the moment generating function of $X_n$. Note that in general (and unlike the case of Ewens measure), neither $h_n$ nor $\mathbb E_{\Theta} [\exp(z X_n)]$ have a closed expression.
Nevertheless, they correspond to the coefficient of $t^n$ in $F(t,1)$ (respectively $F(t,\E^z)$) and, thus, using Cauchy formula, they can be expressed as a contour integral. The idea of singularity analysis is to choose the contour in a clever way such that, asymptotically, the main part of this contour integral comes from the integral near the singularity. In this way, the integral depends asymptotically only of the type of singularity of $F(t,1)$ (resp. $F(t,\E^z)$).\bigskip

We note $r$ the radius of convergence of $g_\Theta(t)$.  We need suitable assumptions on the analyticity properties of $g$ together with assumptions on the nature of its singularity at the point $r$ on the circle of convergence. This motivates the next definition:

\begin{definition}
Let $0<r<R$ and $0<\phi<\pi/2$ be given. We then define
$$\Delta_0=\Delta_0(r,R,\phi)=\{z\in\mathbb C;\; |z|<R,z\neq r, |\arg (z-r)|>\phi\},$$
see Figure \ref{fig:domaindelta}. Assume we are further given $g(t)$, $\theta\geq0$ and $r>0$. We then say that $g(t)$ is in the class $\mathcal F(r,\theta)$ if\vspace{2mm}
\begin{enumerate}[(i)]
\item there exists $R>r$ and $0<\phi<\pi/2$ such that $g(t)$ is holomorphic in $\Delta_0(r,R,\phi)$;\vspace{2mm}
\item there exists a constant $K$ such that 
$$g(t)=\theta \log \left(\dfrac{1}{1-t/r}\right)+K+O(t-r)\;\; \text{ as } t\to r.$$
\end{enumerate}
\end{definition}

\begin{center}
\begin{figure}[ht]
\includegraphics{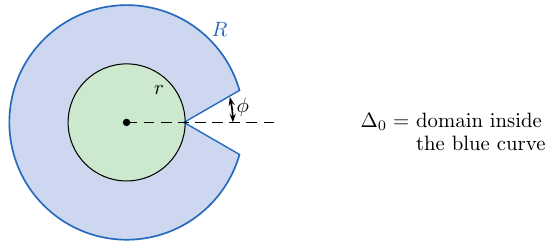}\caption{Domain $\Delta_0(r,R,\phi)$.\label{fig:domaindelta}}
\end{figure}
\end{center}

One readily notes that the generating series corresponding to the Ewens measure (\emph{i.e.} $\theta_m\equiv\theta$) is of class $\mathcal F(1,\theta)$ since in this case
$$g_\Theta(t)\equiv \theta\log\left(\dfrac{1}{1-t}\right).$$ Consequently our results will provide alternative proofs to this case as well.
The next theorem due to Hwang plays a key role in our example (we use the following notation: if $G(t)=\sum_{n=0}^\infty g_n t^n$,  we denote $[t^n][G]\equiv g_n$ the coefficient of $t^n$ in $G(t)$).
\begin{theorem}[Hwang, \cite{Hwa94}] \label{thm:hwang94}
Let $F(t,w)=\exp(w g(t))\,S(t,w)$ be given. Suppose\vspace{2mm}
\begin{enumerate}[(i)]
\item $g(t)$ is of class $\mathcal F(r,\theta)$,\vspace{2mm}
\item $S(t,w)$ is holomorphic in a domain containing $\{(t,w)\in\mathbb C^2;\;\; |t|\leq r, |w|\leq \hat r\}$, where $\hat r>0$ is some positive number.\vspace{2mm}
\end{enumerate}
Then 
$$[t^n][F(t,w)]=\dfrac{\E^{Kw}n^{w\theta-1}}{r^n}\left(\dfrac{S(r,w)}{\Gamma(\theta w)}+O\!\left(\dfrac{1}{n}\right)\right)$$
uniformly for $|w|\leq\hat r$ and with the same $K$ as in the definition above.
\end{theorem}
The idea of the proof consists in taking a suitable Hankel contour and to estimate the integral over each piece.  There exist several other versions of this theorem where one can replace $\log (1-t/r)$ by other functions and we refer the reader to the monograph \cite{FSe09}, chapter VI.5. As an application of Theorem \ref{thm:hwang94}, we obtain an asymptotic for $h_n$.

\begin{lemma}\label{lem:hn}
Let $\theta>0$ and assume that $g_\Theta(t)$ is of class $\mathcal F(r,\theta)$. We then have
$$h_n=\dfrac{\E^{K}n^{\theta-1}}{r^n\,\Gamma(\theta)}\left(1+O\!\left(\dfrac{1}{n}\right)\right).$$
\end{lemma}

\begin{proof}
We have already noted that $\sum_{n=1}^\infty h_n t^n=\exp(g_\Theta(t))$. We can apply Theorem \ref{thm:hwang94} with $g(t)=g_\Theta(t)$, $w=1$ and $S(t,w)=1$.
\end{proof}
\bigskip

Using identity \eqref{eq:combtheta} and Theorem \ref{thm:hwang94}, we can also obtain an asymptotic for the Laplace transform $\esper_{\Theta} [\exp(w X_n)]$:

\begin{theorem}[Nikeghbali-Zeindler, \cite{NZ13}]
If $g_\Theta(t)$ is of class $\mathcal F(r,\theta)$, then
$$ \esper_{\Theta} [\exp(z X_n)]=n^{\theta(\E^z-1)} \E^{K(\E^z-1)}\left(\dfrac{\Gamma(\theta)}{\Gamma(\theta \E^z)}+O\!\left(\dfrac{1}{n}\right)\right).$$
Consequently, the sequence $(X_n)_{n \in \N}$ converges in the mod-Poisson sense, with parameters $t_n=K+\theta \log n$ and limiting function $\psi(z)=\frac{\Gamma(\theta)}{\Gamma(\theta \E^z)}$. The convergence takes place of the whole complex plane with speed $O(1/n)$.
\label{thm:modpoissonNZ}
\end{theorem}

\begin{proof}
An application of Theorem \ref{thm:hwang94} yields
$$[t^n][\exp(\E^z g_\Theta(t))]=\dfrac{\E^{K\E^z}n^{\E^z\theta-1}}{r^n}\left(\dfrac{1}{\Gamma(\theta \E^z)}+O\!\left(\dfrac{1}{n}\right)\right),$$
with $O(\cdot)$ uniform for bounded $z\in\C$. Now a combination of identity (\ref{eq:combtheta}) and Lemma \ref{lem:hn} gives the desired result.
\end{proof}
\bigskip

The above theorem not only implies the central limit theorem, but also Poisson approximations and precise large deviations. We only state here the precise large deviation result which extends earlier work of Hwang in the case $\theta=1$ as a consequence of Theorem \ref{thm:mainlattice} and refer to \cite{NZ13} for the distributional approximations results.

\begin{proposition}[Nikeghbali-Zeindler, \cite{NZ13}]
Let $Y_n=X_n-1$ and let $x\in\mathbb R$ such that $t_nx\in\mathbb N$ with $t_n=K+\theta \log n$. We note $k=t_n x$. Then
$$\mathbb P[Y_n=x t_n]=\E^{-t_n}\dfrac{(t_n)^k}{k!}\left(\dfrac{\Gamma(\theta)}{x\,\Gamma(\theta x)}+O\!\left(\dfrac{1}{ t_n}\right)\right).$$
\end{proposition}
\noindent In fact, an application of Theorem \ref{thm:mainlattice} would immediately yield an arbitrary long expansion for $\mathbb P[Y_n=x t_n]$ and also for $\mathbb P[Y_n\geq x t_n]$, since the speed of convergence is fast enough.
\bigskip

\begin{remark}
    Equations \eqref{eq:gf_hn} and \eqref{eq:combtheta} fit naturally in the framework of 
    {\em labelled combinatorial classes}~\cite[Chapter II]{FSe09}:
    a permutation is a {\em set} of cycles and the weighted exponential generating series of cycles is $g_{\Theta}(t)$.
    Using this framework, one could give a more general statement with the same proof.
    Let $\mathcal{A}$ be a (weighted) labelled combinatorial class and
    consider the class $\mathcal{B}=\Set(\mathcal{A})$.
    We denote $X_n$ the number of components in a random element of size $n$ of $\mathcal{B}$
    (the probability of taking an element of $\mathcal{B}$ being proportional to its weight).
    Assume that the generating series of $\mathcal{A}$ is in the class $\mathcal{F}(r,\theta)$.
    Then $(X_n)_{n \in \N}$ converges in the mod-Poisson sense as in Theorem \ref{thm:modpoissonNZ}.
    We refer to \cite[Section 4.1]{Hwa96} for a similar discussion.\medskip

   However, the hypothesis that the generating series of $\mathcal{A}$ is in the class $\mathcal{F}(r,\theta)$
   is mainly natural in the case of cycles and permutations, which explains why we
   focused on this example here.
\end{remark}

\begin{remark}
   Singularity analysis is closely related to the Selberg-Delange method 
   used in previous Section.
   In both cases we do not have a closed expression for the Laplace transform of our variable,
   but for an appropriate sum of them --- Dirichlet series for arithmetic functions
   and exponential generating functions for cycles in weighted permutations.
    From Perron's or Cauchy's formula, 
    we get an integral formula for the Laplace transform.
    Then choosing an appropriate integral contour allows to find the asymptotics.
    It is quite striking that these closely related methods both
    yield mod-Poisson convergence. However, we are not aware of a way to bring both under one roof.
\end{remark}
\bigskip

\subsection{Rises in random permutations}
\label{subsec:rises}
As illustrated in the previous section,
the methods of singularity analysis and the \emph{transfer theorems} between the singularities of a generating series and the asymptotic behavior of its coefficients (see \emph{e.g.} \cite[Chapter VI]{FSe09})
can be use to prove mod-$\phi$ convergence
(this is not surprising as they can be used to establish the related concept of {\em quasi-powers},
see Section~\ref{Subsect:QP}).
We examine here another example where we do not have strictly speaking mod-$\phi$ convergence 
since the reference law is not infinitely divisible,
but where our methods can nevertheless be used.

If $\sigma \in \sym(n)$ is a permutation, a \emph{rise} of $\sigma$ is an integer $i \in \lle 1,n-1\rre$ such that $\sigma(i)<\sigma(i+1)$. The number of rises $R(\sigma)$ is thus a quantity between $0$ and $n-1$, the two extremal values corresponding to the permutations $n(n-1)\ldots 2 1$ and $12\ldots (n-1)n$. We denote $R_n=R(\sigma_n)$ the random variable which counts the number of rises of a random permutation $\sigma_n \in \sym(n)$ chosen under the uniform measure. The double generating series
$$R(z,t) = \sum_{n=0}^\infty \esper[\E^{z R_n}]\,t^n$$
is computed in \cite[III.7, p. 209]{FSe09}:
$$R(z,t)=\frac{\E^{z}-1}{\E^z-\E^{t(\E^z-1)}}.$$
If $\Re (z)\leq \log 2$, then $|\E^{z}|\leq 2$ and the smallest (in modulus) root of the denominator is 
$$t=\rho(z)=\frac{z}{\E^z-1}.$$
We then have $R(z,t) = \frac{1}{\rho(z)-t} + (\E^z-1)\,S(z,t)$ with $S(z,t)=\frac{1}{\E^z - \E^{t(\E^z-1)}}-\frac{1}{z-t(\E^z-1)}$ analytic in $t$ and on the band $z \in \mathcal{S}_{(-\infty,\log 2)}$. This implies:

\begin{theorem}[\cite{FSe09}, p. 658]\label{thm:flajoletrises}
For $z \in \mathcal{S}_{(-\infty,\log 2)}$, one has locally uniformly
$$\esper[\E^{zR_n}] = \tau(z)^{n+1} + O(|z|\,2^{-n}),$$
where $\tau(z) = \frac{\E^z-1}{z}$.
\end{theorem}
\noindent Notice that $\tau(z) = \frac{\E^z-1}{z} = \int_0^1 \E^{zu}\,du$ is the Laplace transform of a uniform random variable $U$ in the interval $[0,1]$. 
Informally, the previous should be thought as follows: $(R_n)_{n \in \N}$ converges mod-$U$ on $\mathcal{S}_{(-\infty,\log 2)}$ with parameters $t_n=n+1$ and limit $\psi(z)=1$.
However, the uniform law $U$ is not infinitely divisible, and $\tau(z)$ vanishes on any $z \in 2\I \pi\Z$,
therefore the error term $O(|z|\,2^{-n})$ cannot be rewritten as a multiplicative error term,
as we did before.
Still, every computation of Section \ref{sec:nonlattice} can be performed in this case. 
Let us stress out the necessary adjustments:
\begin{itemize}
    \item Berry-Esseen estimates. Set $F_n(x) = \proba[R_n - \frac{n+1}{2}\geq \sqrt{\frac{n+1}{12}}\,x]$ and $G_n(x)=\int_{-\infty}^x g(y)\,dy$. We have by Feller's lemma
    \begin{align*}
    |F_n(x)-G_n(x)| &\leq \frac{24m}{\Delta\pi\sqrt{t_n}} + \frac{1}{\pi\sqrt{t_n}} \int_{[\pm\delta \sqrt{t_n}]} \left|\frac{f_n^*(\xi)-\E^{-\frac{\xi^2}{2}}}{\xi}\right|\,d\xi  \\
    &\quad+ \frac{1}{\pi\delta\sqrt{t_n}} \int_{[\pm \Delta \sqrt{t_n}]\setminus [\pm \delta \sqrt{t_n}]} \left|f_n^*(\xi)-\E^{-\frac{\xi^2}{2}}\right|\,d\xi,
    \end{align*}
    where as in the proof of Proposition \ref{prop:berryesseen}, $f_n^*(\xi)$ is the Fourier transform of $\frac{R_n-\frac{t_n}{2}}{\sqrt{\frac{t_n}{12}}}$.
    From Theorem \ref{thm:flajoletrises} and a Taylor expansion of $\tau(\I \xi)$, one has 
    $$f_n^*(\xi) = \E^{-\frac{\xi^2}{2}\left(1+O\!\left(\frac{\xi^2}{t_n}\right)\right)} + O(|\xi|\,2^{-t_n}).$$
    In the first integral of Feller's bound, we are integrating 
    $$\left|\frac{f_n^*(\xi)-g^*(\xi)}{\xi}\right| = O\left(\delta^2\,\E^{-\frac{\xi^2}{2}}\right)+O(2^{-t_n}),$$
    so we get a $o(\frac{1}{\sqrt{t_n}})$ as $\delta$ goes to zero. Then, regarding the second integral, we are integrating a $O((q_\delta)^{t_n})$ for some $q_\delta\in [\frac{1}{2},1)$, so it does not contribute to the asymptotics. Thus, we get the uniform estimate $|F_n(x)-G_n(x)|=o(\frac{1}{\sqrt{t_n}})$ as in Proposition \ref{prop:berryesseen}.\vspace{2mm}
    \item Tilting method and large deviation estimates. When tilting by $h > 0$, the new random variable $R_n^h$ with law $\proba[R_n^h \in dy]=\frac{\E^{hy}}{\esper[\E^{hR_n}]}\,\proba[R_n \in dy]$ has its Fourier transform that has asymptotics
    $$\frac{\tau(h+\I\xi)^{n+1}+O(2^{-n})}{\tau(h)^{n+1}+O(2^{-n})} = \left(\frac{\tau(h+\I \xi)}{\tau(h)}\right)^{n+1} (1+o(1))$$ 
    the second formula holding if $h$ and $\xi$ are small enough (more precisely, we want $0<h<\log 2$ and then $\xi$ small enough with bounds depending on $h$).
    Notice that the new exponent $\frac{\tau(h+\I \xi)}{\tau(h)}$ does not vanish anymore, so we can now use the same argument as in Theorem \ref{thm:mainnonlattice}. Thus, if $F$ is the Legendre-Fenchel transform of the uniform law on $[0,1]$ (not really explicit), and if $h$ is the solution of $\frac{1+\eps}{2}=\eta'(h)$ with $\eta(h) = \log\left(\frac{\E^h-1}{h}\right)$, then for 
    $$\eps \in(2\eta'(0)-1,2\eta'(\log 2)-1) = \left(0,\,3-\frac{2}{\log 2}\right)\supseteq (0,0.114),$$
     we obtain the large deviation estimate:
    $$\proba\!\left[R_n \geq \frac{(1+\eps)(n+1)}{2}\right] = \frac{\E^{-t_n\,F(\frac{1+\eps}{2})}}{h\sqrt{2\pi t_n\,\eta''(h)}} (1+o(1))$$
    since we have mod-$U$ convergence of $(R_n)_{n \in \N}$ with parameters $t_n=\frac{n+1}{2}$ and residue $\psi(x)=1$. Similarly, for the negative deviations, the estimate of the Fourier transform of the tilted random variable $R^h_n$ can be used if $h>h_0$, where $h_0 \approx -1.594$ is the solution of the equation $\tau(h)=\frac{1}{2}$. Therefore, for
    $$\eps \in \left(0,\,\frac{2}{h_0}-1\right) \supseteq (0,0.254),$$
    we have
    $$\proba\!\left[R_n \leq \frac{(1-\eps)(n+1)}{2}\right] = \frac{\E^{-t_n\,F(\frac{1-\eps}{2})}}{|h|\sqrt{2\pi t_n\,\eta''(h)}} (1+o(1)),$$
    where $h$ is the solution of $\frac{1-\eps}{2}=\eta'(h)$.\vspace{2mm}
    \item Extended central limit theorem. The result of Lemma \ref{lem:technicqdelta_nonlattice} holds, because
\begin{align*}
\left|\frac{\tau(h+\I u)}{\tau(h)}\right|^2 &= \left|\frac{\E^{h+\I u}-1}{\E^{h}-1}\,\frac{h}{h+\I u}\right|^2 \\
&=\left(1+\frac{2\E^h\,(1-\cos u)}{(\E^h-1)^2}\right) \,\frac{h^2}{h^2+u^2} = \frac{h^2+2t(h)\,(1-\cos(u))}{h^2+u^2},
\end{align*}
where $t(h) = \frac{h^2\E^h}{(\E^h-1)^2} \leq 1$ for any $h$. From this inequality, one deduces that
$$q_\delta = \max_{u \in \R\setminus [-\delta,\delta]} \left|\frac{\tau(h+\I u)}{\tau(h)}\right| \leq 1-D\delta^2 $$
for some constant $D>0$ and any $h \in (-\eps,\eps)$, with $\eps$ small enough.\medskip

Since $\eta(x)=\log (\tau(x))$ has Taylor expansion $\frac{x}{2}+\frac{x^2}{24}-\frac{x^4}{2880}+o(x^5)$, we thus get from Equation \eqref{EqCramerOrdreMNL}:
$$\proba\!\left[R_n \geq \frac{n+1}{2} + \sqrt{\frac{n+1}{12}}\,y\right] =\frac{(1+o(1))}{y\sqrt{2\pi}}\,\exp\left(-\frac{y^2}{2} + \frac{y^4}{240t_n}\right)$$
for any positive $y$ with $y = o((t_n)^{\frac{1}{2}-\frac{1}{6}}) = o((t_n)^{\frac{5}{12}})$. 
\end{itemize}
\bigskip

\subsection{Characteristic polynomials of random matrices in a compact Lie group}\label{subsec:rmtpol}
Introduce the classical compact Lie groups of type $\mathrm{A}$, $\mathrm{C}$, $\mathrm{D}$:
\begin{align*}
\mathrm{U}(n)&=\{g \in \mathrm{GL}(n,\C)\,\,|\,\,gg^{\dagger}=g^{\dagger}g=I_{n}\}\qquad\qquad\quad\,\,\,\,\,\,\text{(unitary group)}\\
\mathrm{USp}(n)&=\{g \in \mathrm{GL}(n,\mathbb{H})\,\,|\,\,gg^{\star}=g^{\star}g=I_{n}\}\qquad\qquad\quad\,\,\,\,\,\text{(compact symplectic group)}\\
\mathrm{SO}(2n)&=\{g \in \mathrm{GL}(2n,\R)\,\,|\,\,gg^{t}=g^{t}g=I_{2n} \,\,;\,\,\det g =1\}\quad\text{(special orthogonal group)}
\end{align*}
where for compact symplectic groups $g^{\star}$ denotes the transpose conjugate of a quaternionic matrix, the conjugate of a quaternionic number $q=a+\I b+\mathrm{j} c+\mathrm{k} d$ being $q^{\star}=a-\I b-\mathrm{j} b- \mathrm{k}d$. In the following we shall consider quaternionic matrices as complex matrices of size $2n \times 2n$ by using the map
$$
a+\I b+\mathrm{j} c+\mathrm{k} d \mapsto \begin{pmatrix} a+\I b & c+\I d \\ -c + \I d & a-\I b\end{pmatrix}.
$$
The eigenvalues of a matrix $g \in G=\mathrm{SO}(2n) \text{ or }\mathrm{U}(n)\text{ or }\mathrm{USp(n)}$ are on the unit circle $\mathbb{S}^{1}$, and the value of the characteristic polynomial $\det(1-g)$ factorizes as
$$\det(1-g) = \prod_{i=1}^n (1-\E^{\I \theta_i}) $$
if $G=\mathrm{U}(n)$ and $\mathrm{Sp}(g) = (\E^{\I \theta_1},\ldots,\E^{\I \theta_n})$, and
$$\det(1-g) = \prod_{i=1}^{n} (1-\E^{\I\theta_i})(1-\E^{-\I\theta_i}) = \prod_{i=1}^n |1-\E^{\I\theta_i}|^2 $$
if $G=\mathrm{SO}(2n)$ or $\mathrm{USp}(n)$ and $\mathrm{Sp}(g) = (\E^{\I \theta_1},\E^{-\I \theta_1},\ldots,\E^{\I \theta_n},\E^{-\I\theta_n})$.
We define
\[Y_{n}^{\mathrm{A},\mathrm{C},\mathrm{D}}=
\begin{cases}
    \Re \big[ \log\det(1-g) \big] & \text{ in type $\mathrm{A}$;}\\
    \log\det(1-g) & \text{ in types $\mathrm{C}$ and $\mathrm{D}$.}
\end{cases}\]
An exact formula for the Laplace transform of $\log \det(1-g)$, and the corresponding asymptotics have been given in \cite{KS2,KS1,HKO01}, see also \cite[Sections 3 and 4]{KN12}. Hence, in type A,
$$\esper\!\left[\E^{z_1\,\Re(\log \det(1-g)) +z_2\,\mathrm{Im}(\log \det(1-g))}\right] = \prod_{j=1}^n \frac{\Gamma(j)\,\Gamma(j+z_1)}{\Gamma(j+\frac{z_1+\I z_2}{2}) \,\Gamma(j+\frac{z_1-\I z_2}{2} )}$$
for every $z_1$ with $\Re(z_1)>-1$, and every $z_2 \in \C$ (see \emph{e.g.} \cite[Formula (71)]{KS1}. In particular, if $\Re(z)>-1$, then
\begin{align*}
\esper\!\left[\E^{z Y_n^{\mathrm{A}}}\right] &= \prod_{j=1}^n \frac{\Gamma(j)\,\Gamma(j+z)}{\Gamma(j+\frac{z}{2})^2} \\
&= \frac{G(1+\frac{z}{2})^2}{G(1+z)}\,n^{\frac{z^2}{4}}\,(1+o(1)).
\end{align*}
Here $G$ denotes Barnes' $G$-function, which is the entire solution of the functional equation $G(z+1)=\Gamma(z)\,G(z)$ with $G(1)=1$. To go from the exact formula with ratios of Gamma functions, to the asymptotic formula with Barnes' functions, one has to interpret the formula as a Toeplitz determinant with one Fisher-Hartwig singularity, see the details in the aforementioned paper \cite{KN12} (the arguments therein apply \emph{mutatis mutandi} to the case of a complex random variable in the Laplace transform $z$, since the exact formula holds as long as $\Re(z)>-1$, and the theory of Toeplitz determinants works from the outset with complex functions). Therefore, $(Y_n^{\mathrm{A}})_{n \in \N}$ converges in the mod-Gaussian sense with parameters $t_n=\frac{\log n}{2}$, limiting function
$$\psi^{\mathrm{A}}(z)= \frac{G(1+\frac{z}{2})^2}{G(1+z)},$$
and on the band $\mathcal{S}_{(-1,+\infty)}$.\bigskip

Similarly, in type $\mathrm{C}$ and $\mathrm{D}$, one has the exacts formulas
\begin{align*}
\esper\!\left[\E^{zY_n^{\mathrm{C}}}\right] &= 2^{2n\,z}\,\prod_{j=1}^n \frac{\Gamma(j+n+1)\,\Gamma(z+j+\frac{1}{2})}{\Gamma(j+\frac{1}{2})\,\Gamma(z+j+n+1)}\quad\text{for }\Re(z)>-\frac{3}{2};\\
\esper\!\left[\E^{zY_n^{\mathrm{D}}}\right] &= 2^{2n\,z}\,\prod_{j=1}^n \frac{\Gamma(j+n-1)\,\Gamma(z+j-\frac{1}{2})}{\Gamma(j-\frac{1}{2})\,\Gamma(z+j+n-1)}\quad\text{for }\Re(z)>-\frac{1}{2}
\end{align*}
as well as the asymptotic formulas
\begin{align*}
\esper\!\left[\E^{zY_n^{\mathrm{C}}}\right] &= \left(\frac{\pi n}{2}\right)^{\frac{z}{2}}\,\left(\frac{n}{2}\right)^{\frac{z^2}{2}}\,\frac{G(\frac{3}{2})}{G(\frac{3}{2}+z)}\,(1+o(1));\\
\esper\!\left[\E^{zY_n^{\mathrm{D}}}\right] &= \left(\frac{8\pi}{n}\right)^{\frac{z}{2}}\,\left(\frac{n}{2}\right)^{\frac{z^2}{2}}\,\frac{G(\frac{1}{2})}{G(\frac{1}{2}+z)}\,(1+o(1))
\end{align*}
which hold in the same range for $z$, locally uniformly. Therefore, setting
\begin{align*}
X_n^{\mathrm{C}} &= Y_n^{\mathrm{C}} - \frac{1}{2}\,\log \frac{\pi n}{2} \\
X_n^{\mathrm{D}} &= Y_n^{\mathrm{D}} - \frac{1}{2}\,\log \frac{8\pi}{n}
\end{align*}
then $(X_n^{\mathrm{C}})_{n \in \N}$ and $(X_n^{\mathrm{D}})_{n \in \N}$ converges in the mod-Gaussian sense, with parameters $t_n = \log \frac{n}{2}$, limiting functions
$$\psi^{\mathrm{C}}(z)= \frac{G(\frac{3}{2})}{G(\frac{3}{2}+z)}\quad\text{and}\quad \psi^{\mathrm{D}}(z)= \frac{G(\frac{1}{2})}{G(\frac{1}{2}+z)}, $$
and respectively on $\mathcal{S}_{(-\frac{3}{2},+\infty)}$ and on $\mathcal{S}_{(-\frac{1}{2},+\infty)}$.
\bigskip

Now, our large deviation theorems apply and one obtains:

\begin{theorem}\label{thm:characteristicpolynomial}
Over $\mathrm{U}(n)$, one has:
\begin{align*}
\forall x>0,\,\,\,\proba_{n}\!\left[\,|\det(1-g)| \geq n^{\frac{x}{2}}\right]
&= \frac{G(1+\frac{x}{2})^{2}}{G(1+x)}\,\frac{1}{x\,n^{\frac{x^2}{4}}\, \sqrt{\pi \log n}}\, (1+o(1)); \\
\forall x \in (0,1),\,\,\,\proba_{n}\!\left[\,|\det(1-g)| \leq n^{-\frac{x}{2}}\right]
&= \frac{G(1-\frac{x}{2})^{2}}{G(1-x)}\,\frac{1}{x\,n^{\frac{x^2}{4}}\, \sqrt{\pi \log n}}\, (1+o(1)).
\end{align*}
Over $\mathrm{USp}(n)$, one has:
\begin{align*}
\forall x>0,\,\,\,\proba_{n}\!\left[\frac{\det(1-g)}{\sqrt{\frac{\pi}{2}}} \geq n^{\frac{1}{2}+x}\right] & = \frac{G(\frac{3}{2})}{G(\frac{3}{2}+x)}\frac{1}{x\,n^{\frac{x^{2}}{2}}\,2^{\frac{x^2+x}{2}}\,\sqrt{2\pi \log n}}\, (1+o(1));\\
\forall x\in \left(\!0\,,\frac{3}{2}\right)\!,\,\,\,\proba_{n}\!\left[\frac{\det(1-g)}{\sqrt{\frac{\pi}{2}}} \geq n^{\frac{1}{2}-x}\right] & = \frac{G(\frac{3}{2})}{G(\frac{3}{2}-x)}\,\frac{1}{x\,n^{\frac{x^{2}}{2}}\,2^{\frac{x^2-x}{2}}\,\sqrt{2\pi \log n}}\, (1+o(1)).
\end{align*}
Finally, over $\mathrm{SO}(2n)$, one has: 
\begin{align*}\forall x>0,\,\,\,\proba_{n}\!\left[\frac{\det(1-g)}{\sqrt{8\pi}} \geq n^{-\frac{1}{2}+x}\right] &= \frac{G(\frac{1}{2})}{G(\frac{1}{2}+x)}\,\frac{1}{x\,n^{\frac{x^{2}}{2}}\,2^{\frac{x^2-x}{2}}\,\sqrt{2\pi \log n}}\, (1+o(1));\\
\forall x\in \left(\!0\,,\frac{1}{2}\right)\!,\,\,\,\proba_{n}\!\left[\frac{\det(1-g)}{\sqrt{8\pi}} \leq n^{-\frac{1}{2}-x}\right] &= \frac{G(\frac{1}{2})}{G(\frac{1}{2}-x)}\,\frac{1}{x\,n^{\frac{x^{2}}{2}}\,2^{\frac{x^2+x}{2}}\,\sqrt{2\pi \log n}}\, (1+o(1)).
\end{align*}
\end{theorem}
\medskip

\begin{proof}
These are immediate computations by using Theorem \ref{thm:mainnonlattice}.
\end{proof}\bigskip

\begin{remark}
From Proposition \ref{prop:normalityzone}, one also gets normality zones for the random variables $X_n/\sqrt{t_n}$. On the other hand, the computations performed in the unitary case hint at a phenomenon of $2$-dimensional mod-Gaussian convergence for the complex numbers $\log \det(1-g)$, with $g \in \mathrm{U}(n)$. In \cite{Multidim}, we shall prove this rigorously, and compute various probabilistic consequences, for instance, estimates of large deviations for the random vectors $\log \det (1-g)$.
\end{remark}

\begin{remark}
The analogue of Theorem \ref{thm:characteristicpolynomial} in the setting of random matrices in the $\beta$-ensembles, or of general Wigner matrices, has been studied in the recent paper \cite{DE13}. They can be easily restated in the mod-Gaussian language, since their proofs rely on the computation of the asymptotics of the cumulants of the random variables $X_{n}=\log|\det M_{n}|$, with for instance $(M_{n})_{n \in \N}$ random matrices of the Gaussian unitary ensembles.
\end{remark}
\bigskip
\bigskip

\section{\for{toc}{Mod-Gaussian convergence from a factorization of the PGF} \except{toc}{Mod-Gaussian convergence from a factorization\texorpdfstring{\\}{} of the probability generating function}}
\label{sec:zeros} 
Let $(X_n)_{n \in \N}$ be a sequence of bounded random variables with nonnegative integer values such that $\sigma_n^2:=\Var(X_n)$ tends to infinity. Denote $P_n(t)= \esper[t^{X_n}]$ the {\em probability generating function} of $X_n$.  Each $P_n(t)$ is a polynomial in $t$. Then, it is known that a sufficient condition so that $X_n$ is asymptotically Gaussian is that $P_n(t)$ has negative real roots (see references below). In this section, we prove that if the third cumulant $L_n^3:=\kappa^{(3)}(X_n)$ also tends to infinity such that $L_n=o(\sigma_n)$, then a suitable renormalized version of $X_n$ converges in the mod-Gaussian sense. We then give an application for the number of blocks in a uniform set-partition of $[n]$.
\bigskip

\subsection{Background: central limit theorem from location of zeros}
The idea of proving central limit theorem by looking at the zeros of the probability generating function originates from an article of Harper \cite{HarperCLTBlocks}. Harper was interested in the number of blocks of a random set-partition, which will be our main example below. The argument was then generalized by Haigh \cite{HaighTCLZeros}. Hwang and Steyaert \cite[Lemma 4]{HwangSteyaertCLTZeros} refined Haigh's result by giving a bound on the speed of convergence towards the normal distribution. Finally, let us mention a recent work of Lebowitz, Pittel, Ruelle and Speer~\cite{Lebowitz2014central}, where the authors prove central limit theorems and local limit laws for random variables under various assumptions on the location of the zeros of the probability generating function. 
Lebowitz, Pittel, Ruelle and Speer apply their theoretical results to graph counting polynomials
and to the Ising model in $\Z^d$ (using Lee-Yang's theorem). 
We plan to address the problem of mod-Gaussian convergence for these models in the future.

The presentation of the results here is inspired from the one in \cite{HwangSteyaertCLTZeros}.

\subsection{Mod-Gaussian convergence from non-negative factorizations}
We will prove the following statement, which is stronger than what was announced.
\begin{theorem}
Let $(X_n)_{n \in \N}$ be a sequence of bounded random variables with non-negative integer values, with mean $\mu_n$, variance $\sigma_n^2$ and third cumulant $L_n^3$. Suppose that the probability generating function $P_n(t)$ of $X_n$ can be factorized as
\begin{equation}
    P_n(t) = \prod_{1 \le j \le k_n} P_{n,j}(t), 
    \label{eq:factorizationPGF}
\end{equation}
where \vspace{2mm}    
\begin{itemize}
    \item $(k_n)_{n \in \N}$ is a sequence of positive integers;\vspace{2mm}
    \item each $P_{n,j}$ is a polynomial with non-negative coefficients;\vspace{2mm}
    \item we assume $L_n=o(\sigma_n)$ and $M_n=o((L_n)^2/\sigma_n)$, where $M_n=\max_{1 \le j \le k_n} \deg(P_{n,j})$.\vspace{2mm}
\end{itemize}
Then, the sequence 
\[\widetilde{X_n}=\frac{X_n - \mu_n}{L_n}\]
converges in the mod-Gaussian sense with limiting function $\psi=\exp(z^3/6)$ and parameters $t_n=\sigma_n^2/L_n^2$.
The convergence takes place of the whole complex plane, with speed of convergence $O(M_n/|L_n|)$.
\label{thm:modgaussianfromzeros}
\end{theorem}

\noindent Notice that under the hypotheses of Theorem \ref{thm:modgaussianfromzeros}, $M_n = o(|L_n|)$: indeed,
$$M_n \ll \frac{(L_n)^2}{\sigma_n} = |L_n|\, \frac{|L_n|}{\sigma_n} \ll |L_n| $$
since $L_n  = o(\sigma_n)$. On the other hand, in the case where $P_n$ has only real negative roots $-r_{n,j}$, then 
\[P_n(t)= C \prod_{1 \le j \le n} (t+r_{n,j}).\] Therefore, we can then apply the above theorem with $k_n=n$ and $M_n=1$. Besides, our theorem also contains mod-Gaussian convergence of a sum of i.i.d. bounded variables with a non-zero third cumulant (see Example \ref{ex:sumiid}).

\begin{proof}
Since $P_n(1)= \prod_{1 \le j \le k_n} P_{n,j}(1)=1$, we have
    \[P_n(t) = \prod_{1 \le j \le k_n} \frac{P_{n,j}(t)}{P_{n,j}(1)}. \]
Therefore, we can assume without loss of generality that $P_{n,j}(1)=1$ (otherwise replace $P_{n,j}(t)$ by $P_{n,j}(t)/P_{n,j}(1)$). Since $P_{n,j}(t)$ has non-negative coefficients, $P_{n,j}(t)$ is the probability generating function of a variable $X_{n,j}$ (for $1 \le j \le k_n$), defined by $\proba[X_{n,j}=k]=[t^k] P_{n,j}(t)$.
\bigskip

The factorization \eqref{eq:factorizationPGF} implies that $X_n$ can be represented as the sum of independent copies of the variables $X_{n,j}$. In particular, 
\begin{align*}
\mu_n&=\sum_{1 \le j \le k_n} \mu_{n,j};\\
\sigma_n^2&=\sum_{1 \le j \le k_n} \sigma_{n,j}^2; \\
L_n^3&= \sum_{1 \le j \le k_n} L_{n,j}^3
\end{align*}
where $\mu_{n,j}$, $\sigma_{n,j}^2$ and $L_{n,j}^3$ are the three first cumulants of $X_{n,j}$. Notice that each $X_{n,j}$, and hence $|X_{n,j}-\mu_{n,j}|$ is bounded by $M_n$. We will use this repeatedly below. In particular one has: $\sigma_{n,j}^2 \le (M_n)^2$ and $|L_{n,j}^3| \le (M_n)^3$.
\bigskip

Call $\widetilde{\varphi_n}$ the Laplace transform of $\widetilde{X_n}$. We have
\[ \widetilde{\varphi_n}(z) = \esper\!\left[ \exp\left(\frac{z}{L_n}\, (X_n-\mu_n)\right) \right] = \prod_{j=1}^{k_n} \,\esper\!\left[ \exp \left(\frac{z}{L_n} (X_{n,j}-\mu_{n,j}) \right) \right]. \]
Fix $K>0$ and assume that $\frac{M_n |z| }{|L_n|} \le K$. From the Taylor expansion 
\[\E^w=1+w+\frac{w^2}{2}+\frac{w^3}{6} +O(w^4) \text{ (uniformly for $|w| \le K$),}\] 
we have that, uniformly for $\frac{M_n |z| }{|L_n|} \le K$,
 \begin{equation}
    \esper\!\left[\exp \left(\frac{z}{L_n} (X_{n,j}-\mu_{n,j}) \right) \right] = 1+\frac{\sigma_{n,j}^2\,z^2}{2\,(L_n)^2} + \frac{L_{n,j}^3\,z^3}{6\,(L_n)^3} +  O\!\left(\esper[(X_{n,j}-\mu_{n,j})^4] \left(\frac{|z|}{|L_n|}\right)^{\!4} \right).
    \label{eq:techniclaplace}
\end{equation}
Note that, since $M_n |z| /|L_n| \le K$, the two first terms are bounded. Besides, we have the following bounds:
    \begin{align*}
        \left| \left(\frac{\sigma_{n,j}^2\,z^2}{2\,(L_n)^2} \right)^2 \right| &\le \frac{\esper[(X_{n,j}-\mu_{n,j})^4]}{4} \left(\frac{|z|}{|L_n|}\right)^{\!4} ;\\
        \left| \frac{\sigma_{n,j}^2\,z^2}{2\,(L_n)^2}  \times \frac{L_{n,j}^3\,z^3}{6\,(L_n)^3} \right| & \le \frac{\esper[|X_{n,j}-\mu_{n,j}|^5]}{12} \left(\frac{|z|}{|L_n|}\right)^{\!5} \\
        &\le \frac{\esper[(X_{n,j}-\mu_{n,j})^4]}{12} \left(\frac{|z|}{|L_n|}\right)^{\!4}\, \frac{M_n |z|}{|L_n|};\\
        \left| \left(\frac{L_{n,j}^3\,z^3}{6\,(L_n)^3}\right)^2 \right| &\le \frac{\esper[(X_{n,j}-\mu_{n,j})^6]}{36} \left(\frac{|z|}{|L_n|}\right)^{\!6}\\
        & \le \frac{\esper[(X_{n,j}-\mu_{n,j})^4]}{36} \left(\frac{|z|}{|L_n|}\right)^{\!4} \,\left( \frac{M_n |z|}{|L_n|} \right)^2.
    \end{align*}
Taking the logarithm of Equation \eqref{eq:techniclaplace} and using $\log(1+t)=t+O(t^2)$, we get, thanks to the above bounds, that
$$\log \esper\!\left[ \exp \left(\frac{z}{L_n} (X_{n,j}-\mu_{n,j}) \right) \right] = \frac{\sigma_{n,j}^2\,z^2}{2\,(L_n)^2} + \frac{L_{n,j}^3\,z^3}{6\,(L_n)^3} +  O\!\left(\esper[(X_{n,j}-\mu_{n,j})^4] \left(\frac{|z|}{|L_n|}\right)^{\!4} \right).$$
Summing these identities, we obtain
\[\log \widetilde{\varphi_n}(z) = \left(\frac{\sigma_n}{L_n}\right)^2\,\frac{z^2}{2}+ \frac{z^3}{6} + O\!\left(\left(\frac{|z|}{|L_n|}\right)^{\!4} \,\sum_{j=1}^{k_n}\esper[(X_{n,j}-\mu_{n,j})^4] \right).\]
The error term can be bounded as follows:
$$\left(\frac{|z|}{|L_n|}\right)^{\!4} \,\sum_{j=1}^{k_n}\esper[(X_{n,j}-\mu_{n,j})^4] \leq \left(\frac{|z|}{|L_n|}\right)^{\!4} (M_n)^2\,\sum_{j=1}^{k_n}\esper[(X_{n,j}-\mu_{n,j})^2] \leq |z|^4\,\frac{(M_n)^2(\sigma_n)^2}{|L_n|^4}.$$
By the assumption made on $M_n$, this is a $o(|z|^4)$. Finally, we get
\[\log \widetilde{\varphi_n}(z) = \left(\frac{\sigma_n}{L_n}\right)^2\,\frac{z^2}{2} + \frac{z^3}{6} + o(|z|^4),\]
which can be rewritten as
\[ \widetilde{\varphi_n}(z)\, \E^{-\left(\frac{\sigma_n}{L_n}\right)^2\,\frac{z^2}{2}} = \exp\left(\frac{z^3}{6}\right)\, \big(1+o(|z|^4)\big), \]
this being uniform for $M_n |z| /|L_n| \le K$, thus on compacts of $\C$. This ends the proof of the theorem.
\end{proof}

\begin{remark}
Suppose for instance that $M_n=1$, \emph{i.e.}, the probability generating function of $X_n$ has non-negative real roots. Then, the conditions on the first cumulants of $X_n$ in order to apply Theorem \ref{thm:modgaussianfromzeros} are
$$\sqrt{\sigma_n} \ll L_n \ll \sigma_n,$$
which are usually easy to check.
\end{remark}
\bigskip

\subsection{Two examples: uniform permutations and uniform set-partitions}\label{subsec:twoexamples}
The first example that fills in this framework is the number of disjoint cycles $X_n$ of a uniform random permutation in $\mathfrak{S}(n)$. As mod-convergence of $X_n$ has already been discussed in this article (Example \ref{ex:cycle}), we will skip details. Using Feller's coupling, it is easily seen that
\[P_n(t)= \prod_{i=1}^{n} \frac{t+i-1}{1+i-1}.\]
Moreover, a straight-forward computation yields $\mu_n=H_n+O(1)$, $\sigma^2_n=H_n+O(1)$ and $L^3_n=H_n+O(1)$, where $H_n$ is the $n$-th harmonic number as in Example \ref{ex:cycle}. Theorem \ref{thm:modgaussianfromzeros} implies that $(X_n-H_n)/(H_n)^{1/3}$ converges in the mod-Gaussian sense, as established at the end of Example \ref{ex:cycle}. Note however that Theorem \ref{thm:modgaussianfromzeros} does not give the stronger mod-Poisson convergence of $(X_n)_{n \in \N}$ without renormalization.
\bigskip

The second and more interesting example is the number of blocks in a random uniform set-partition of $[n]$.
By definition, a \emph{set-partition} of $[n]$ is a set of disjoint non-empty subsets of $[n]$, whose union is $[n]$.
These subsets are called \emph{blocks} or \emph{parts} of the set-partition. For instance, $\{\{2,4\},\{1\},\{3\}\}$ is a set-partition of $[4]$ with $3$ blocks. We denote $\qym(n)$ the set of all set-partitions of $[n]$. For each integer $n \ge 0$, we then consider a random uniform set-partition in $\qym(n)$, and denote $X_n$ its number of blocks.\bigskip

It was proved by Harper \cite[Lemma 1]{HarperCLTBlocks} that the probability generating function of $X_n$ has only real non-negative roots. Moreover, the asymptotic behaviour of $\mu_n=\esper[X_n]$ and $\sigma_n^2=\Var(X_n)$ are known  --- see {\it e.g.} \cite[Theorem 2.1]{DiaconisEtAlCLTSetPartitions} ---
\[ \mu_n = \frac{n}{\log n}\, (1+o(1)), \quad \sigma_n^2=\frac{n}{(\log n)^2}\,(1+o(1)).\]
We will prove in next subsection that
\[ L_n^3= \frac{2n}{(\log n)^3}\,(1+o(1)). \]
Since $M_n=1$ and
$$\sqrt{\sigma_n} = O\!\left(\frac{n^{1/4}}{(\log n)^{1/2}}\right) \ll L_n = O\!\left(\frac{n^{1/3}}{\log n}\right) \ll \sigma_n =  O\!\left(\frac{n^{1/2}}{\log n}\right),$$ 
we can apply Theorem \ref{thm:modgaussianfromzeros}: the variables $(X_n-\mu_n)/L_n$ converge in the mod-Gaussian sense with parameter $t_n=\sigma_n^2/L_n^2=(\frac{n}{4})^{1/3}\,(1+o(1))$ and limiting function $\psi(z)=\exp(z^3/6)$.
As a corollary, applying Theorem \ref{thm:mainnonlattice} and Proposition \ref{prop:normalityzone} yields the following precise deviation result.
\begin{proposition}
Let $X_n$ be the number of blocks in a uniform set-partition of $[n]$. Define $\mu_n$, $\sigma_n$, $L_n$ and $t_n$ as above. Then the random variable $\frac{X_n - \mu_n}{\sigma_n}$ converges towards a Gaussian law, with a normality zone of size $o(n^{1/6})$. Moreover, at the edge of this normality zone,
the deviation probabilities are given by: for any fixed $x>0$, 
\begin{align*}
\proba\!\left[\frac{X_n-\mu_n}{L_n} \ge t_n x\right] &= \frac{\exp(-t_n \frac{x^2}{2})}{x\sqrt{2\pi t_n}} \,\exp\left(\frac{x^3}{6}\right)\, (1+o(1)); \\
\proba\!\left[\frac{X_n-\mu_n}{L_n} \le -t_n x\right] &= \frac{\exp(-t_n \frac{x^2}{2})}{x\sqrt{2\pi t_n}} \,\exp\left(-\frac{x^3}{6}\right)\, (1+o(1)). 
\end{align*}
\end{proposition}
\bigskip

\subsection{Third cumulant of the number of blocks in uniform set-partitions}
Fix $n \ge 0$. Let $B_n$ be the number of set-partitions of $[n]$, known as the $n$-th Bell number. Dobinski's formula states that $B_n=\E^{-1} \sum_{k =0}^\infty k^n/k!$, which allows us to consider a random variable $M$ with the following distribution:
\[\proba[M=k] = \frac{1}{\E B_n} \frac{k^n}{k!}.\]
An easy observation, useful below, is that $\esper[M^r]=B_{n+r}/B_n$. As above, we denote $X_n$ the number of blocks in a uniform set-partition of size $[n]$. We also consider a Poisson variable $P$ of parameter $1$, independent from $X_n$.
Stam \cite{StamRandomPartition} proved the following relation (with a very nice probabilistic explanation).

\begin{lemma}
    \cite[Theorem 2]{StamRandomPartition}
    We have the following equality of random variables in law:
    \[M \stackrel{\text{law}}{=} X_n + P.\] 
\end{lemma}

\noindent Therefore, $\kappa^{(3)}(M) = \kappa^{(3)}(X_n) + \kappa^{(3)}(P)$. But $\kappa^{(3)}(P)=1$ is a constant (independent of $n$), whose value will not be relevant for the asymptotic of $\kappa^{(3)}(X_n)$. Let us consider $\kappa^{(3)}(M)$. It is given by:
\[\kappa^{(3)}(M) = \esper[M^3] -3 \,\esper[M^2] \esper[M] +2 (\esper[M])^3=\frac{B_{n+3}}{B_n} -3 \frac{B_{n+2}\, B_{n+1}}{(B_n)^2}
+2 \frac{(B_{n+1})^3}{(B_n)^3}.\]
In order to find the asymptotic of the above formula, we use the following estimate for Bell numbers. This is a variant of the Moser-Wyman formula \cite{MoserWymanFormula}, that can be found in an unpublished note of Mohr \cite{NoteMohr}.

\begin{lemma}
Let $\alpha_n$ be the positive real number defined by the equation $\alpha_n\, \E^{\alpha_n}=n$; in particular $\alpha_n=(\log n)\,(1-o(1))$. Then, one has, for any integer $r$,
    \begin{multline*}
        B_{n+r}=\frac{(n+r)!}{(\alpha_n)^{n+r}}\, \frac{\exp(\E^{\alpha_n}-1)}{\sqrt{2\pi \beta_n}} \\
        \times \left( 1+\frac{P_0+rP_1+r^2P_2}{\E^{\alpha_n}}+\frac{Q_0+rQ_1+r^2Q_2+r^3Q_3+r^4Q_4}{\E^{2\alpha_n}} +O(\E^{-3\alpha_n}) \right),
    \end{multline*}
where $\beta_n=((\alpha_n)^2+\alpha_n)\,\E^{\alpha_n}$, and the $P_i$'s and $Q_i$'s are explicit rational functions of $\alpha_n$, with $P_i=O((\alpha_n)^{-i})$ and $Q_i=O((\alpha_n)^{-i})$.
\end{lemma}

This lemma yields an estimate for quotients of Bell numbers (we substitute $\E^{\alpha_n}$ by $n/\alpha_n$):
\begin{multline*}
    (\alpha_n)^r \, \frac{B_{n+r}}{B_n}=(n+1) \cdots (n+r) \left( 
1+\alpha_n \, \frac{rP_1+r^2P_2}{n} \right.\\
\left.+(\alpha_n)^2 \, \frac{-P_0 (rP_1+r^2P_2)+rQ_1+r^2Q_2+r^3Q_3+r^4Q_4}{n^2} +O(\alpha_n^3/n^3)\right)
\end{multline*}
Consider the expression 
\[k^3:=\alpha_n^3 \left(\frac{B_{n+3}}{B_n} -3 \frac{B_{n+2}\, B_{n+1}}{B_n^2} +2 \frac{(B_{n+1})^3}{(B_n)^3} \right).\]
With the help of a computer algebra program, we find
\[k^3 = n \left( (-6 \, P_2\, (P_1+P_2) + 6 \, Q_3 + 36 \, Q_4 ) (\alpha_n)^2 + 12 P_2\, \alpha_n + 2\right) +O((\alpha_n)^3).\]
From the estimate $P_i=O((\alpha_n)^{-i})$ and $Q_i=O((\alpha_n)^{-i})$, we see that the dominant term in $k^3$ comes 
from the constant $2$ in the above equation. Namely,
\[k^3= 2 n + O\!\left(\frac{n}{\alpha_n}\right), \quad\text{ that is }\kappa^{(3)}(X_n)= \frac{2n}{(\alpha_n)^3} + O\!\left(\frac{n}{(\alpha_n)^4}\right),\]
as claimed in the previous subsection.
\bigskip
\bigskip

\section{Dependency graphs and mod-Gaussian convergence}\label{sec:depgraph}

Dependency graphs are a classical tool in the literature to prove convergence in distribution towards a Gaussian law of the sum of {\em partly} dependent random variables. They are used in various domains, such as random graphs~\cite[pages 147-152]{JLR00}, random polytopes~\cite{BV07}, patterns in random permutations~\cite{Bona10}. As dependency graphs give a natural framework to deal in a uniform way with different kinds of objects, a natural question is the following:  when we have a dependency graph with good properties, can we obtain more precise or other results than convergence in distribution? Here is a brief presentation of the literature around this question.\vspace{2mm}
\begin{itemize}
    \item In \cite{BR89}, P. Baldi and Y. Rinott give precise estimates for the total variation distance between the relevant sequence of random variables and the Gaussian distribution.\vspace{2mm}
  \item In \cite{Jan04}, S. Janson has established some large deviation result involving the fractional chromatic number of the dependency graph.\vspace{2mm}
  \item More recently, H. D\"oring and P. Eichelsbacher have shown how dependency graphs can be used to obtain some moderate deviation principles
      \cite[Section 2]{DE12}.\vspace{2mm}
\end{itemize}
Here, we shall see a link between dependency graphs and mod-Gaussian convergence. This gives us a large collection of examples, for which the material of this article gives automatically some {\em precise} moderate deviation results. Our deviation result has a larger domain of validity than the one of D{\"o}ring and Eichelsbacher --- see below.\bigskip

In this section, we establish a general result involving dependency graphs (Theorem~\ref{thm:dependencygraphsrefined}). In the next two sections, we focus on examples and derive the mod-Gaussian convergence of the following renormalized statistics:\vspace{2mm}
\begin{itemize}
  \item subgraph count statistics in Erd\"os-R\'enyi random graphs (Section \ref{sec:erdosrenyi});\vspace{2mm}
  \item random character values from central measures on partitions (Section \ref{sec:central}).
\end{itemize}
\medskip

\subsection{The theory of dependency graphs}
\label{subsec:dependencygraphs}
Let us consider a variable $X$, which writes as a sum
$$X=\sum_{\alpha \in V} \,Y_\alpha$$
of random variables $Y_\alpha$ indexed by a set $V$.

\begin{definition}
A graph $G$ with vertex set $V$ is called a {\em dependency graph} for the family of random variables $\{Y_\alpha,\,\, \alpha \in V\}$ if the following property is satisfied:\vspace{2mm}
\begin{quote}
If $V_1$ and $V_2$ are disjoint subsets of $V$ such that there are no edges in $G$ with one extremity in $V_1$ and one in $V_2$, then the sets of random variables $\{Y_\alpha\}_{\alpha \in V_1}$ and $\{Y_\alpha\}_{\alpha \in V_2}$ are independent ({\em i.e.}, the $\sigma$-algebras generated by these sets are independent).
\end{quote}
\end{definition}\medskip

\begin{example}
Let $(Y_1,\ldots,Y_7)$ be a family with dependency graph
\begin{center}
\begin{tikzpicture}[scale=2]
\draw (1,0.2) -- (0.4,0.8) -- (0,0) -- (1,0.2) -- (1.7,-0.1) -- (2,0.6);
\draw (3,0) -- (3.7,0.7);
\fill[color=white!50!black] (0,0) circle [radius=0.5mm];
\fill[color=white!50!black] (1,0.2) circle [radius=0.5mm];
\fill[color=white!50!black] (0.4,0.8) circle [radius=0.5mm];
\fill[color=white!50!black] (1.7,-0.1) circle [radius=0.5mm];
\fill[color=white!50!black] (2,0.6) circle [radius=0.5mm];
\fill[color=white!50!black] (3,0) circle [radius=0.5mm];
\fill[color=white!50!black] (3.7,0.7) circle [radius=0.5mm];
\draw (0,-0.17) node {$1$};
\draw (0.25,0.85) node {$2$};
\draw (1,0.03) node {$3$};
\draw (1.7,-0.27) node {$4$};
\draw (2.15,0.65) node {$5$};
\draw (2.8,0) node {$6$};
\draw (3.85,0.75) node {$7$};
\end{tikzpicture}
\end{center}
Then, $(Y_1,Y_2,Y_3,Y_4,Y_5)$ and $(Y_6,Y_7)$ are independent (obvious), but the vectors $(Y_1,Y_2)$ and $(Y_4,Y_5)$ are also independent: although they are in the same connected component of the graph $G$, they are not directly connected by an edge $e \in E$.
\end{example}
\medskip

\begin{remark}
Note that a family of random variables may admit several dependency graphs. In particular, the complete graph with vertex set $V$ is always a dependency graph. We are interested in dependency graphs with as few edges as possible. Note that a family of random variables does not always have a unique minimal dependency graph (minimal for edge-set inclusion), as shown by the following example.
\end{remark}
\medskip

\begin{example}
Consider three independent Bernoulli random variables $X_1,X_2,X_3$, and $Y_1=\mathbbm{1}_{(X_2=X_3)}$, $Y_2=\mathbbm{1}_{(X_1=X_3)}$ and $Y_3=\mathbbm{1}_{(X_1=X_2)}$. Then, the following graphs are minimal dependency graphs for $(Y_1,Y_2,Y_3)$:
\begin{center}
\begin{tikzpicture}[scale=2]
\draw (0,0) -- (0.5,0.7) -- (1,0);
\draw (2.5,0.7) -- (2,0) -- (3,0);
\draw (4.5,0.7) -- (5,0) -- (4,0);
\fill[color=white!50!black] (0,0) circle [radius=0.5mm];
\fill[color=white!50!black] (1,0) circle [radius=0.5mm];
\fill[color=white!50!black] (0.5,0.7) circle [radius=0.5mm];
\fill[color=white!50!black] (2,0) circle [radius=0.5mm];
\fill[color=white!50!black] (3,0) circle [radius=0.5mm];
\fill[color=white!50!black] (2.5,0.7) circle [radius=0.5mm];
\fill[color=white!50!black] (4,0) circle [radius=0.5mm];
\fill[color=white!50!black] (5,0) circle [radius=0.5mm];
\fill[color=white!50!black] (4.5,0.7) circle [radius=0.5mm];
\draw (-0.15,-0.1) node {$1$};
\draw (1.85,-0.1) node {$1$};
\draw (3.85,-0.1) node {$1$};
\draw (1.15,-0.1) node {$3$};
\draw (3.15,-0.1) node {$3$};
\draw (5.15,-0.1) node {$3$};
\draw (0.35,0.8) node {$2$};
\draw (2.35,0.8) node {$2$};
\draw (4.35,0.8) node {$2$};
\end{tikzpicture}
\end{center}
\end{example}

\begin{example}
Fix a finite graph $G=(V,E)$. Take a family of independent non-constant random variables $(Y_e)_{e \in E}$ 
indexed by the edge-set $E$ of $G$. For a vertex $v \in V$, define $X_v=\sum_e Y_e$ where the sum runs over incident edges to $v$. Then $G$ is a dependency graph for the family $(X_v)_{v \in V}$.  Moreover, it is minimal for edge-set inclusion and unique with this property. 
\end{example}
\bigskip

The following bound on cumulants of sums of random variables has been established by S. Janson~\cite[Lemma 4]{Jan88}.

\begin{theorem}\label{thm:dependencygraphs}
For any integer $r \ge 1$, there exists a constant $C_r$ with the following property. Let $\{Y_\alpha\}_{\alpha \in V}$ be a family of random variables with dependency graph $G$. We denote $N=|V|$ the number of vertices of $G$ and $D$ the maximal degree of $G$. 
Assume that the variables $Y_\alpha$ are uniformly bounded by a constant $A$. 
Then, if $X=\sum_{\alpha} Y_\alpha$, one has:
$$| \kappa^{(r)}(X)| \leq C_r \, N \, (D+1)^{r-1} \, A^r.$$
\end{theorem}

In most applications for counting substructures in random objects, the $Y_\alpha$ are indicator variables, so that the {\em uniformly bounded} assumption is not restrictive. This theorem is often used to prove some central limit theorem. In \cite{DE12}, D\"oring and  Eichelsbacher have analysed Janson's original proof and have established that the theorem holds with $C_r=(2\E)^r (r!)^3$. Then they have used this new bound to obtain some moderate deviation results. Here, we will give a new proof of Janson's result, with a smaller value of the constant $C_r$. Namely, we will prove:

\begin{theorem}\label{thm:dependencygraphsrefined}
Theorem \ref{thm:dependencygraphs} holds with $C_r=2^{r-1} \,r^{r-2}$.    
\end{theorem}

\noindent We shall see at the end of this section, and in the next Sections that this stronger version can be used to establish mod-Gaussian convergence and, thus, precise moderate deviation results. In fact, we prove a slightly more general statement.

\begin{theorem}
With the same assumptions as above, one has:
    \[ | \kappa^{(r)}(X)| \leq 2^{r-1} \,r^{r-2} \left( \sum_\alpha \esper[|Y_\alpha|] \right) (D+1)^{r-1} A^{r-1}.\]
    \label{thm:dependencygraphs_sumY}
\end{theorem}

\noindent Note that $\sum_\alpha \esper[|Y_\alpha|] \le N \, A$. On the other hand, if $Y_\alpha$ are Bernoulli variables indexed by $\alpha \in \{1,\dots,N\}$ with parameters $1/\alpha$, then $\sum_\alpha \esper[|Y_\alpha|] \sim \log N$. Thus the second bound is exponentially smaller.\medskip

The next few subsections are devoted to the proof of Theorem~\ref{thm:dependencygraphs_sumY}.
\medskip

\begin{remark}
The hypothesis of bounded random variables can sometimes be lifted by mean of truncation methods. Indeed, if $(Y_{\alpha})_{\alpha \in V}$ is a family of unbounded random variables with dependency graph $G$, then any truncated family $(Y_{\alpha}\,\mathbbm{1}_{|Y_{\alpha}|\leq L_{\alpha}})_{\alpha \in V}$ with fixed levels of truncation $L_{\alpha}$ has the same dependency graph $G$. Thus, in many situations, one can prove the mod-Gaussian convergence of (an adequate renormalization of) the truncated sum $S_{\text{truncated}}=\sum_{\alpha \in V} Y_{\alpha}\,\mathbbm{1}_{|Y_{\alpha}| \leq L_{\alpha}}$, and then to use \emph{ad-hoc} arguments in order to control the remainder $S_{\text{remainder}}=\sum_{\alpha \in V} Y_{\alpha}\,\mathbbm{1}_{|Y_{\alpha}| > L_{\alpha}}$, such as moments inequalities (Bienaymé-Chebyshev). We shall develop these arguments in details in the forthcoming paper \cite{FMN14}.
\end{remark}
\bigskip

\subsection{Joint cumulants}
There exists a multivariate version of cumulants, called joint cumulants, that we shall use to prove Theorem~\ref{thm:dependencygraphsrefined}. We present in this paragraph its definition and basic properties. Most of this material can be found in Leonov's and Shiryaev's paper~\cite{LS59} (see also \cite[Proposition 6.16]{JLR00}).
\medskip

\subsubsection{Preliminaries: set-partitions}
We denote by $[n]$ the set $\{1,\dots,n\}$. Recall from Section \ref{subsec:twoexamples} that a {\em set partition} of $[n]$ is a (non-ordered) family of non-empty disjoint subsets of $S$ (called parts of the partition), whose union is $[n]$. For instance,
$$\{\{1,3,8\},\{4,6,7\},\{2,5\}\}$$
is a set partition of $[8]$. Denote $\qym(n)$ the set of set partitions of $[n]$. Then $\qym(n)$ may be endowed with a natural partial order: the {\em refinement} order. We say that $\pi$ is {\em finer} than $\pi'$ or $\pi'$ {\em coarser} than $\pi$ (and denote $\pi \leq \pi'$) if every part of $\pi$ is included in a part of $\pi'$.\medskip

Lastly, denote $\mu$ the M\"obius function of the poset $\qym(n)$. In this paper, we only use evaluations of $\mu$ at pairs $(\pi,\{[n]\})$ (the second argument is the partition of $[n]$ in only one part, which is the maximum element of $\qym(n)$), so we shall use abusively the notation $\mu(\pi)$ for $\mu(\pi,\{[n]\})$.  In this case, the value of the M\"obius function is given by:
\begin{equation}\label{eq:valuemobius}
    \mu(\pi) = \mu(\pi, \{[n]\})=(-1)^{\#(\pi)-1} (\# (\pi)-1)!\,.
\end{equation}
\medskip

\subsubsection{Definition and properties of joint cumulants}
If $X_1, \ldots, X_r$ are random variables with finite moments on the same probability space (denote $\esper$ the expectation on this space), we define their joint cumulant by
\begin{equation}
    \kappa (X_1,\dots,X_r) = [t_1 \cdots t_r] \,\,\log 
    \bigg( \esper\!\left[ \E^{t_1 X_1 + \cdots + t_r X_r} \right] \bigg).
    \label{eq:defcumulant}
\end{equation}
As usual, $[t_1 \dots t_r] F$ stands for the coefficient of $t_1 \cdots t_r$  in the series expansion of $F$ in positive powers of $t_1, \dots, t_r$. Note that joint cumulants are multilinear functions. In the case where all the $X_i$'s are equal, we recover the $r$-th cumulant $\kappa^{(r)}(X)$ of a single variable. Using set-partitions, joint cumulants can be expressed in terms of joint moments, and {\em vice-versa}:

\begin{align}
    \esper [X_1 \cdots X_r] &= \sum_{\pi \in \qym(r)} \prod_{C \in \pi}
    \kappa(X_i\,;\, i \in C);
    \label{eq:cumulant2moment} \\
    \kappa (X_1,\dots,X_r) &= \sum_{\pi \in \qym(r)} \mu(\pi)
    \prod_{C \in \pi} \esper\!\left[ \prod_{i \in C} X_i \right].
    \label{eq:moment2cumulant}
\end{align}
In these equations, $C \in \pi$ shall be understood as ``$C$ is a part of the set partition $\pi$''. Recall that $\mu(\pi)$ has an explicit expression given by Equation~\eqref{eq:valuemobius}. For example the joint cumulants of one or two variables are simply the mean of a single random variable and the covariance of a couple of random variables:
$$\kappa(X_1)=\esper[X_1]\qquad;\qquad \kappa(X_1,X_2)=\esper[X_1 X_2] - \esper[X_1]\, \esper[X_2].$$
For three variables, one has
\begin{align*}
\kappa(X_1,X_2,X_3) &= \esper[X_1 X_2 X_3] - \esper[X_1 X_2]\, \esper[X_3] - \esper[X_1 X_3]\, \esper[X_2] \\
&\quad- \esper[X_2 X_3]\, \esper[X_1] +2\, \esper[X_1] \,\esper[X_2] \,\esper[X_3].
\end{align*}

\begin{remark}
    The most important property of cumulants is their relation with independence: if the variables $X_1,\dots,X_r$ can be split in two non-empty sets of variables which are independent with each other, then $\kappa (X_1,\dots,X_r)$ vanishes \cite[Proposition 6.16 (v)]{JLR00}. We will not need this property here. In fact, we will prove a more precise version of it, see Equation~\eqref{eq:boundkappa}.
\end{remark}\medskip

\subsubsection{Statement with joint cumulants}
Theorems~\ref{thm:dependencygraphs} and \ref{thm:dependencygraphsrefined} have some analog with joint cumulants. Let $\{Y_\alpha\}_{\alpha \in V}$ be a family of random variables with dependency graph $G$. As in Theorem \ref{thm:dependencygraphs}, we assume that the variables $Y_\alpha$ are uniformly bounded by a constant $A$: {\em i.e.}, for all $\alpha \in V$, 
$$ \|Y_\alpha\|_\infty \leq A.$$
Consider $r$ subsets $V_1, V_2,\ldots, V_r$ of $V$, non necessarily distinct and set $X_i=\sum_{\alpha \in V_i} Y_{\alpha}$ (for $i \in [r]$). We denote $D_i$ the maximal number of vertices in $V_i$ adjacent to a given vertex (not necessarily in $V_i$). Then one has the following result.

\begin{theorem} \label{thm:boundjointcumulant}
With the notation above,
$$ |\kappa(X_1,\dots,X_r)| \le 2^{r-1}\, r^{r-2} \,|V_1|\,(D_2 +1) \cdots (D_r +1)\, A^r .$$
\end{theorem}

\noindent The proof of this theorem is very similar to the one of Theorem \ref{thm:dependencygraphsrefined}.
However, to simplify notation, we only prove the latter here.\bigskip

\subsection{Useful combinatorial lemmas}
We start our proof of Theorem \ref{thm:dependencygraphsrefined} by stating a few lemmas on graphs and spanning trees.

\subsubsection{A functional on graphs}
In this section, we consider graphs $H$ with multiple edges and loops. We use the standard notation $V(H)$ and $E(H)$ for their vertex and edge sets. For a graph $H$ and a set partition $\pi$ of $V(H)$, we denote $\pi \perp H$ when the following holds: for any edge $\{i,j\} \in E(H)$, the elements $i$ and $j$ lie in different parts of $\pi$ (in this case we use the notation $i \nsim_\pi j$). We introduce the following functional on graphs $H$:
$$\SF_{H}=(-1)^{|V(H)|-1}\sum_{\pi \perp H}\mu(\pi).$$

\begin{lemma}
For any graph $H$, one has
$$\SF_{H}= \sum_{\substack{E \subset E(H) \\ (V(H),E)\text{ connected}}} (-1)^{|E|-|V(H)|+1}.$$
\end{lemma}

\begin{proof}
To simplify notation, suppose $V(H)=[r]$. We denote $\mathbbm{1}_{(P)}$ the characteristic function of the property $(P)$. By inclusion-exclusion,
\begin{align*}
(-1)^{|V(H)|-1}\,\SF_{H}&= \sum_{\pi \in \qym(r)} \!\left( \prod_{(i,j) \in E(H)}
\mathbbm{1}_{i \nsim_{\pi} j} \right) \mu(\pi)=\sum_{\pi \in \qym(r)}\!
\left( \prod_{(i,j) \in E(H)} (1- \mathbbm{1}_{i \sim_{\pi} j})\right) \mu(\pi)\\
&= \sum_{E \subset E(H)} \sum_{\pi \in \qym(r)} (-1)^{|E|} \left( \prod_{(i,j) \in E}  \mathbbm{1}_{i \sim_{\pi} j} \right)  \mu(\pi) \\
&= \sum_{E \subset E(H)} (-1)^{|E|} \left[ \sum_{\substack{\pi \text{ such that } \\ \forall (i,j) \in E,\,\, i \sim_\pi j}} \mu(\pi) \right].
\end{align*}
But the quantity in the bracket is $0$ unless the only partition in the sum is the maximal partition $\big\{[r]\big\}$, in which case it is $1$. This corresponds to the case where the edges in $E$ form a connected subgraph of $H$.
\end{proof}
\medskip

\begin{corollary}
The functional $\SF_H$ fulfills the deletion-contraction induction, \emph{i.e.}, if $e$ is an edge of $H$ which is not a loop, then
$$\SF_H = \SF_{H / e} + \SF_{H \setminus e},$$
where $H \setminus e$ (respectively $H / e$) are the graphs obtained from $H$ by deleting (resp. contracting) the edge $e$.
\end{corollary}

\begin{proof}
The first term corresponds to sets of edges containing $e$, and the second to those that do not contain $e$.
\end{proof}
\bigskip

This induction (over-)determines $\SF_H$ together with the initial conditions:
$$\begin{cases}
&\SF_{\includegraphics{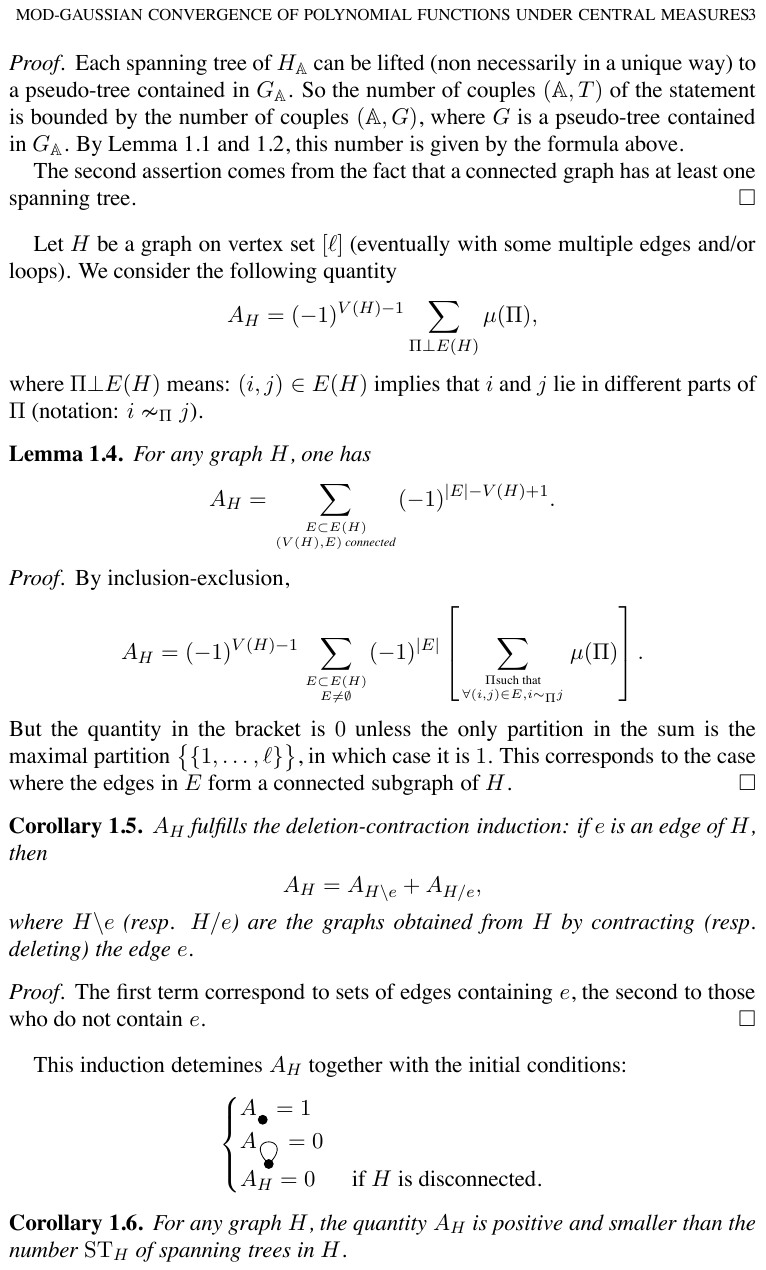}}=1, \\
&\SF_{\! \!\! \GraphOneLoop } \! =\SF_{\!  \!\! \!\GraphTwoLoops } \!=\cdots=0,\phantom{}\vspace{2mm}\\
&\SF_{H}=0\quad\text{ if $H$ is disconnected}.
\end{cases}$$

\begin{corollary}\label{cor:ineq_SFST}
For any graph $H$, the quantity $\SF_H$ is nonnegative and less or equal than the number $\ST_H$ of spanning trees in $H$.
\end{corollary}

\begin{proof}
The quantity $\ST_H$ fulfills the same induction as $\SF_H$ with initial conditions:
$$\hspace{2cm}\begin{cases}
&\ST_{\includegraphics{node.pdf}}=1, \\
&\ST_{\! \! \GraphOneLoop } \!=\ST_{ \! \! \GraphTwoLoops } \!=\cdots=1, \phantom{}\vspace{2mm}\\
&\ST_{H}=0\quad\text{ if $H$ is disconnected}.
\end{cases}\vspace*{-7mm}$$
\end{proof}
\medskip

\begin{remark}
If $H$ is connected, both $\SF_H$ and $\ST_{H}$ are actually specializations of the bivariate \emph{Tutte polynomial} $T_{H}(x,y)$ of $H$ (\emph{cf.} \cite[Chapter X]{Bol98}):
$$\SF_{H}=T_{H}(1,0)\qquad;\qquad\ST_{H}=T_{H}(1,1).$$
This explains the deletion-contraction relation. As the bivariate Tutte polynomials has non-negative coefficients, it also explains the inequality $0 \le \SF_H \le \ST_H$.
\end{remark}
\bigskip

\subsubsection{Induced graphs containing spanning trees}
Fix a graph $G$ (typically the dependency graphs of our family of variables). For a list $(v_1,\dots,v_r)$ of $r$ vertices of $G$, we define the induced graph $G[v_1,\dots,v_r]$ as follows:\vspace{2mm}
\begin{itemize}
    \item its vertex set is $[r]$;\vspace{2mm}
    \item there is an edge between $i$ and $j$ if and only if 
        $v_i=v_j$ or $v_i$ and $v_j$ are linked in $G$.\vspace{2mm}
\end{itemize}
We will be interested in spanning trees of induced graphs. As the vertex set is $[r]$, these spanning trees may be seen as {\em Cayley trees}. Recall that a Cayley tree of size $r$ is by definition a tree with vertex set $[r]$ (Cayley trees are neither rooted, nor embedded in the plane, they are only specified by an adjacency matrix). These objects are enumerated by the well-known Cayley formula established by C. Borchardt in \cite{Borchardt60}: there are exactly $r^{r-2}$ Cayley trees of size $r$.

\begin{lemma}
    Fix a Cayley tree $T$ of size $r$ and a graph $G$ with $N$ vertices and maximal degree $D$.
    Fix a vertex $v_1$ of $G$.
    The number of lists $(v_1,\dots,v_r)$ of $r$ vertices of $G$
    such that $T$ is contained in the induced subgraph $G[v_1,\dots,v_r]$ is bounded from above by
    \[ (D+1)^{r-1}. \]
\end{lemma}

\begin{proof}
    Lists $(v_1,\dots,v_r)$ as in the lemma are constructed as follows. First consider a neighbor $j$ of $1$ in $T$. As we require $G[v_1,\dots,v_r]$ to contain $T$, the vertices $1$ and $j$ must also be neighbors in $G[v_1,\dots,v_r]$, {\it i.e.} $v_j=v_1$ or $v_j$ is a neighbor of $v_1$ in $T$. Thus, once $v_1$ is fixed, there are at most $D+1$ possible values for $v_j$. The same is true for all neighbors of $1$ and then for all neighbors of neighbors of $1$ and so on.
\end{proof}
\medskip

We have the following immediate consequence.
\begin{corollary}
    \label{cor:countingcouplessequences_ST}
    Let $G$ be a graph with $n$ vertices and maximal degree $D$ and $r\ge 1$.
    Fix a vertex $v_1$ of $G$.
    The number of couples 
    $\big( (v_1,\dots,v_r), T  \big) $
    where each $v_i$ is a vertex of $V$ and $T$ a spanning tree of
    the induced subgraph $G[v_1,\dots,v_r]$ is bounded above by
    \[r^{r-2} \, (D+1)^{r-1}.\]
\end{corollary}\bigskip

\subsubsection{Spanning trees and set partitions of vertices}
Recall that $\ST_H$ denotes the number of spanning trees of a graph $H$. Consider now a graph $H$ with vertex set $[r]$ and a set partition $\pi=(\pi_1,\dots,\pi_t)$ of $[r]$. For each $i$, we denote $\ST^{\pi_i}(H)=\ST_{H[\pi_i]}$ the number of spanning trees of the graph induced by $H$ on the vertex set $\pi_i$. We also use the multiplicative notation
$$\ST^\pi(H) = \prod_{j=1}^t \ST^{\pi_j}(H).$$
We can also consider the contraction $H/\pi$ of $H$ with respect to $\pi$. By definition, it is the multigraph ({\em i.e.} graph with multiple edges, but no loops) defined as follows. Its vertex set is the index set $[t]$ of the parts of $\pi$ and, for $i \neq j$, there are as many edges between $i$ and $j$ as  edges between a vertex of $\pi_i$ and a vertex of $\pi_j$ in $H$. Denote $\ST_{\pi}(H)=\ST_{H/\pi}$ the number of spanning trees of this contracted graph (multiple edges are here important). This should not be confused with $\ST^\pi(H)$: in the latter, $\pi$ is placed as an exponent because the quantity is multiplicative with respect to the part of $\pi$.
\bigskip

Note that the union of a spanning tree $\overline{T}$ of $H/\pi$ and  of spanning trees $T_i$ of $H[\pi_i]$ (one for each $1 \le i \le t$) gives a spanning tree $T$ of $H$. Conversely, take a spanning tree $T$ on $H$ and a bicoloration of its edges. Edges of color $1$ can be seen as a subgraph of $H$ with the same vertex set $[r]$. This graph is of course acyclic. Its connected components define a partition $\pi=\{\pi_1,\dots,\pi_t\}$ of $[r]$ and edges of color $1$ correspond to a collection of spanning trees $T_i$ of $H[\pi_i]$ (for $1 \le i \le t$). Besides, edges of color $2$ define a spanning tree $\overline{T}$ on $H/\pi$.\bigskip

Therefore, we have described a bijection between spanning trees $T_0$ of $H$ with a bicoloration of their edges and triples $(\pi,\overline{T},(T_i)_{1\le i \le t})$ where:\vspace{2mm}
\begin{itemize}
    \item $\pi$ is a set partition of the vertex set $[r]$ of $H$ (we denote $t$ its number of parts);\vspace{2mm}
    \item $\overline{T}$ is a spanning tree of the contracted graph $H/\pi$;\vspace{2mm}
    \item for each $1 \le i \le t$,  $T_i$ is a spanning tree of the induced graph $H[\pi_i]$.\vspace{2mm}
\end{itemize}
Before giving a detailed example, let us state the numerical corollary of this bijection:
\begin{equation}\label{eq:coloredpseudotree}
    2^{r-1} \,\ST_H = \sum_{\pi} \ST_\pi(H) \,\ST^\pi(H),
\end{equation}
where the sum runs over all set partitions $\pi$ of $[r]$.

\begin{center}
\begin{figure}[ht]
    $$T=\begin{array}{c}
        \begin{tikzpicture}
    \tikzstyle{vertex1}=[circle,draw,inner sep=0.5pt,minimum size=1mm]
    \tikzstyle{vertex2}=[rectangle,draw,inner sep=1.5pt,minimum size=1mm]
    \tikzstyle{rededge}=[red, dotted,thick,line width=1.2pt]
    \tikzstyle{blueedge}=[blue!90!white,thick,line width=1.2pt]
    \tikzstyle{greenedge}=[ForestGreen!80!white, dashed,thick,line width=1.2pt]
    \node (v1) at (0:1) [vertex1] {\footnotesize $1$};
    \node (v2) at (60:1) [vertex1] {\footnotesize $2$};
    \node (v3) at (120:1) [vertex1] {\footnotesize $3$};
    \node (v4) at (180:1) [vertex1] {\footnotesize $4$};
    \node (v5) at (240:1) [vertex1] {\footnotesize $5$};
    \node (v6) at (300:1) [vertex1] {\footnotesize $6$};
    \draw [blueedge] (v1) to (v2);
    \draw [blueedge] (v2) to (v3);
    \draw [blueedge] (v4) to (v5);
    \draw [blueedge] (v4) to (v6);
    \draw [greenedge] (v3) to (v5);
    \draw [rededge] (v3) to (v1);
    \draw [rededge] (v1) to (v4);
    \draw [rededge] (v2) to (v6);
    \draw [rededge] (v1) to (v6);
\end{tikzpicture} \end{array}\,\,\leftrightarrow\,\,
        \begin{cases}
        \,\,\,\,\pi=\{\pi_1,\pi_2\} \text{ with }\pi_1=\{1,2,3\},\ \pi_2=\{4,5,6\};\\
        \\
        \,\,\,\,T_1=\begin{array}{c}
        \begin{tikzpicture}
    \tikzstyle{vertex1}=[circle,draw,inner sep=0.5pt,minimum size=1mm]
    \tikzstyle{vertex2}=[rectangle,draw,inner sep=1.5pt,minimum size=1mm]
    \tikzstyle{rededge}=[red, dotted,thick,line width=1.2pt]
    \tikzstyle{blueedge}=[blue!90!white,thick,line width=1.2pt]
    \tikzstyle{greenedge}=[ForestGreen!80!white, dashed,thick,line width=1.2pt]
    \node (v1) at (0:1) [vertex1] {\footnotesize $1$};
    \node (v2) at (60:1) [vertex1] {\footnotesize $2$};
    \node (v3) at (120:1) [vertex1] {\footnotesize $3$};
    \draw [blueedge] (v1) to (v2);
    \draw [blueedge] (v2) to (v3);
    \draw [rededge] (v3) to (v1);
        \end{tikzpicture} 
        \end{array},\quad 
        T_2=\begin{array}{c}
        \begin{tikzpicture}
    \tikzstyle{vertex1}=[circle,draw,inner sep=0.5pt,minimum size=1mm]
    \tikzstyle{vertex2}=[rectangle,draw,inner sep=1.5pt,minimum size=1mm]
    \tikzstyle{rededge}=[red, dotted,thick,line width=1.2pt]
    \tikzstyle{blueedge}=[blue!90!white,thick,line width=1.2pt]
    \tikzstyle{greenedge}=[ForestGreen!80!white, dashed,thick,line width=1.2pt]
    \node (v4) at (180:1) [vertex1] {\footnotesize $4$};
    \node (v5) at (240:1) [vertex1] {\footnotesize $5$};
    \node (v6) at (300:1) [vertex1] {\footnotesize $6$};
    \draw [blueedge] (v4) to (v5);
    \draw [blueedge] (v4) to (v6);
        \end{tikzpicture} 
        \end{array};\\
        \\
        \,\,\,\,\overline{T}=\begin{array}{c}
        \begin{tikzpicture}
    \tikzstyle{vertex1}=[circle,draw,inner sep=0.5pt,minimum size=1mm]
    \tikzstyle{vertex2}=[rectangle,draw,inner sep=1.5pt,minimum size=1mm]
    \tikzstyle{rededge}=[red, dotted,thick,line width=1.2pt]
    \tikzstyle{blueedge}=[blue!90!white,thick,line width=1.2pt]
    \tikzstyle{greenedge}=[ForestGreen!80!white, dashed,thick,line width=1.2pt]
    \node (v123) at (60:1) [vertex1] {\footnotesize $\pi_1$};
    \node (v456) at (240:1) [vertex1] {\footnotesize $\pi_2$};
    \draw [rededge] (v123) to [bend right=-60] (v456);
    \draw [rededge] (v123) to [bend right=20] (v456);
    \draw [rededge] (v123) to [bend right=-20] (v456);
    \draw [greenedge] (v123) to [bend right=60] (v456);
\end{tikzpicture}
        \end{array}.
        \end{cases}
$$
    \caption{Bijection explaining Identity \eqref{eq:coloredpseudotree}.}
    \label{fig:partitionspanningtrees}
\end{figure}
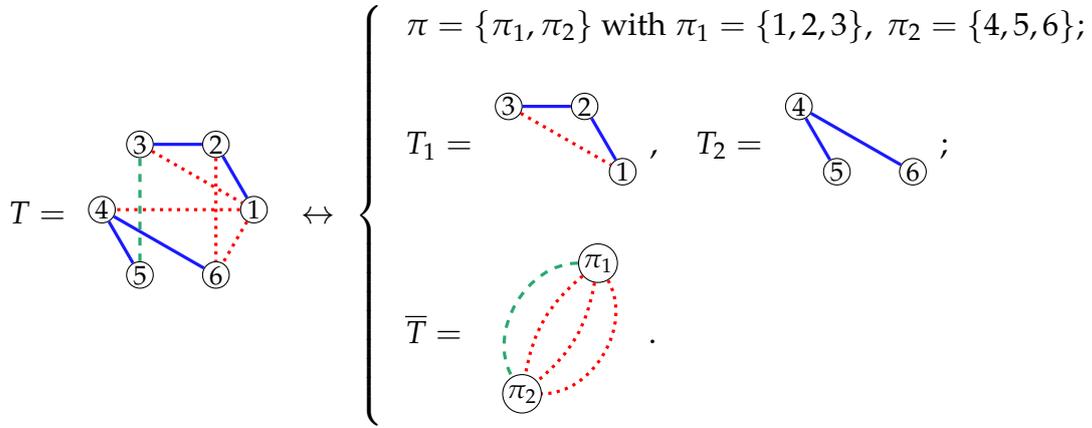
\end{center}
Our bijection is illustrated on Figure \ref{fig:partitionspanningtrees}, with the following conventions.\vspace{1.5mm}
\begin{itemize}
    \item On the left, blue plain edges are edges of color $1$ in the tree $T_0$; on the right, these blue plain edges are the edges of the spanning trees $T_i$.\vspace{1.5mm}
    \item On the left, green dashed edges are edges of color $2$ in the tree $T_0$; on the right, these green dashed edges are edges of the spanning tree $\overline{T}$.\vspace{1.5mm}
    \item Red dotted edges belong to the graphs $H$, $H/\pi$, $H[\pi_i]$ but not to their spanning tree $T_0,\overline{T},T_i$.\vspace{1.5mm}
\end{itemize}
Note that in this example the graph $H/\pi$ has two vertices linked by four edges. These four edges correspond to the edges $(1,4)$, $(1,6)$, $(2,6)$ and $(3,5)$ of $H$. In the example, the spanning tree $\overline{T}$ is the edge $\{(3,5)\}$. If we had chosen another edge, the tree $T$ on the left-hand side would have been different. Hence, in the Equality \eqref{eq:coloredpseudotree}, the multiple edges of $H/\pi$ must be taken into account: in our example, $\ST_\pi(H)=4$.\bigskip

\subsection{Proof of the bound on cumulants}
Recall that we want to find a bound for $\kappa^{(r)}(X)$. As $X$ writes $X=\sum_{\alpha \in V} Y_{\alpha}$, we may use joint cumulants and expand by multilinearity:
\begin{equation}\label{eq:expcumulant}
    \kappa^{(r)}(X) =
\sum_{\alpha_1,\dots,\alpha_r}
\kappa(Y_{\alpha_1},\dots,Y_{\alpha_r}).
\end{equation}
The sum runs over lists of $r$ elements in $V$, that is vertices of the dependency graph $G$. The proof consists in bounding each summand $\kappa(Y_{\alpha_1},\dots,Y_{\alpha_r})$, with a bound depending on the induced subgraph $G[\alpha_1,\dots,\alpha_r]$.

\subsubsection{Bringing terms together in joint cumulants.}
Recall the moment-cumulant formula \eqref{eq:moment2cumulant} which states $\kappa(Y_{\alpha_1},\dots,Y_{\alpha_r}) = \sum_\pi \mu(\pi) M_\pi$, where 
$$M_\pi = \prod_{B \in \pi} \,\esper\! \left[ \prod_{i \in B} Y_{\alpha_i} \right]. $$
We warn the reader that the notation $M_\pi$ is a little bit abusive as this quantity depends also on the list $(\alpha_1,\dots,\alpha_r)$. By hypothesis, $G$ is a dependency graph for the family $\{Y_{\alpha}\}_{\alpha \in V}$. Hence if some block $B$ of a partition $\pi$ can be split into two sub-blocks $B_1$ and $B_2$ such that the sets of vertices $\{\alpha_i\}_{i \in B_1}$ and $\{\alpha_i\}_{i \in B_2}$ are disjoint and do not share an edge, then
\begin{equation}\esper\! \left[ \prod_{i \in B} Y_{\alpha_i} \right] = 
\esper\! \left[ \prod_{i \in B_1} Y_{\alpha_i} \right] \times
\esper\! \left[ \prod_{i \in B_2} Y_{\alpha_i} \right]. \label{eq:alsoinnoncommutative}\end{equation}
Therefore, $M_\pi=M_{\phi_H(\pi)}$, where $H=G[\alpha_1,\dots,\alpha_r]$ and $\phi_H(\pi)$ is the refinement of $\pi$ obtained as follows: for each part $\pi_i$ of $\pi$, consider the induced graph $H[\pi_i]$ and replace $\pi_i$ by the collection of vertex sets of the connected components of $H[\pi_i]$. This construction is illustrated on Figure \ref{fig:phiH}.

\begin{figure}[ht]
    \hfill \begin{minipage}{.3\linewidth}
        \[H=\begin{array}{c}\begin{tikzpicture}
    \tikzstyle{vertex1}=[circle,blue!90!white,draw,inner sep=0.5pt,minimum size=1mm]
    \tikzstyle{vertex2}=[rectangle,red,draw,inner sep=1.5pt,minimum size=1mm]
    \node (v1) at (0:1) [vertex1] {\footnotesize $1$};
    \node (v2) at (60:1) [vertex1] {\footnotesize $2$};
    \node (v3) at (120:1) [vertex1] {\footnotesize $3$};
    \node (v4) at (180:1) [vertex1] {\footnotesize $4$};
    \node (v5) at (240:1) [vertex2] {\footnotesize $5$};
    \node (v6) at (300:1) [vertex2] {\footnotesize $6$};
    \draw [line width=1pt] (v1) to (v5);
    \draw [line width=1pt] (v1) to (v6);
    \draw [line width=1pt] (v2) to (v1);
    \draw [line width=1pt] (v3) to (v4);
    \draw [line width=1pt] (v3) to (v6);
    \draw [line width=1pt] (v5) to (v4);
\end{tikzpicture} \end{array}\]
    \end{minipage}\hfill \hfill
    \begin{minipage}{.6\linewidth}
For example, consider the graph $H$ here opposite and the partition $\pi=\{\pi_1,\pi_2\}$ with $\pi_1=\{1,2,3,4\}$ and $\pi_2=\{5,6\}$. Then $H[\pi_1]$ (respectively, $H[\pi_2]$) has two connected components with vertex sets $\{1,2\}$ and $\{3,4\}$ (resp. $\{5\}$ and $\{6\}$). Thus
\[\phi_H(\pi)=\{\{1,2\},\{3,4\},\{5\},\{6\}\}. \]
    \end{minipage} \hfill
    \vspace{3mm}
    \caption{Illustration of the definition of $\phi_H$.}
    \label{fig:phiH}
\end{figure}

We can thus write
$$\kappa(G)= \sum_{\pi'}\, M_{\pi'} \! \left( \sum_{\pi \in \phi_H^{-1}(\pi')} \mu(\pi) \right).$$
Fix $\pi'=(\pi'_1,\dots,\pi'_t)$ and let us have a closer look to the expression in the parentheses that we will call $\pa_{\pi'}$. To compute it, it is convenient to consider the contraction $H/\pi'$ of the graph $H$ with respect to the partition $\pi'$.
\begin{lemma} 
    Let $\pi'$ be a set partition of $[r]$. If one of the induced graph $H[\pi'_i]$ is disconnected, then $\pa_{\pi'}=0$. Otherwise, $\pa_{\pi'}=(-1)^{\ell(\pi')-1} \SF_{H/\pi'}$.
\end{lemma}
\begin{proof}
The first part is immediate, as $\phi_H^{-1}(\pi')=\emptyset$ in this case.\bigskip

\noindent If all induced graphs are connected, let us try to describe $\phi_H^{-1}(\pi')$. All set partitions $\pi$ of this set are coarser than $\pi'$, so can be seen as set partitions of the index set $[r]$ of the parts of $\pi'$. This identification does not change their M\"obius functions, which depends only on the number of parts. Then, it is easy to see that $\pi$ lies in $\phi_H^{-1}(\pi')$ if and only if $\pi$ is coarser than $\pi'$ and  two elements in the same part of $\pi$ never share an edge in $H/\pi'$ (here, $\pi$ is seen as a set partition of $[r]$). In other words, $\pi$ lies in $\phi_H^{-1}(\pi')$ if and only if $\pi \perp (H/\pi')$. This implies the Lemma.
\end{proof}
\bigskip

\noindent Consequently,
\begin{equation}\label{eq:jointcumulantgroup}
\kappa(Y_{\alpha_1},\dots, Y_{\alpha_r}) =
\sum_{\pi'} (-1)^{\ell(\pi')-1} M_{\pi'}\, \SF_{H/\pi'}\left( \prod_{i=1}^t \mathbbm{1}_{H[\pi'_i]\text{ connected}}\right),
\end{equation}
where the sum runs over all set partitions $\pi'$ of $[r]$.\medskip

\subsubsection{Bounding all the relevant quantities}
Using $|Y_{\alpha_i}| \le A$ for all $i$, we get the inequality: 
$$|M_\pi| \le A^{r-1}\, \esper[|Y_{\alpha_1}|].$$
Finally, to bound each summand $\kappa(Y_{\alpha_1},\dots,Y_{\alpha_r})$, we shall use the following bounds:
\begin{align*}
    |\SF_{H/\pi'}| &\le \ST_{H/\pi'}  \qquad \text{by Corollary \ref{cor:ineq_SFST};}\\
    \mathbbm{1}_{H[\pi'_i]\text{ connected}} &\le \ST_{H[\pi'_i]}\,.
\end{align*}
Thus, Equation \eqref{eq:jointcumulantgroup} gives
\begin{equation}\label{eq:boundkappa}
    |\kappa(Y_{\alpha_1},\dots, Y_{\alpha_r})| \le 
    A^{r-1}\, \esper [|Y_{\alpha_1}|]
    \,\sum_{\pi'} \,\ST_{H/\pi'} \left(\prod_{i=1}^t \ST_{H[\pi'_i]}\right) =
    A^{r-1}\, \esper[|Y_{\alpha_1}|]
    \, 2^{r-1}\, \ST_H,
\end{equation}
the last equality corresponding to Equation \eqref{eq:coloredpseudotree}.
Recall that $H=G[\alpha_1,\dots,\alpha_r]$.
Summing over $\alpha_1,\dots,\alpha_r$ and using Equation \eqref{eq:expcumulant}, we get 
$$
|\kappa^{(r)}(X)| \le (2A)^{r-1} \sum_{\alpha_1} \esper[|Y_{\alpha_1}|] \left[ \sum_{\alpha_2,\dots,\alpha_r} \ST_{G[\alpha_1,\dots,\alpha_r]} \right].
$$

\noindent From Corollary \ref{cor:countingcouplessequences_ST}, we know that the expression in the bracket is bounded by the quantity \hbox{$r^{r-2} (D+1)^{r-1}$} (for any fixed $\alpha_1$). This completes the proof of Theorem \ref{thm:dependencygraphs_sumY}.
\bigskip

\subsection{Sums of random variables with a sparse dependency graph}
An immediate application of Theorem \ref{thm:dependencygraphsrefined} is the following general result on sums of weakly dependent random variables, to be compared with \cite[\S2.3]{Pen02}. Let $X_n=\sum_{i=1}^{N_n} Y_i$ be a sum of random variables, where the $Y_i$'s have a dependency graph of degree $D_n-1$, and satisfy $\|Y_i\|_\infty \leq A$ for some $A \geq 0$ (independent of $i$). We also assume that $X_n$ is not deterministic, so that its variance is non-zero.

\begin{theorem}\label{thm:modgaussiangeneralsparsegraph}
We assume that the dependency graph of $(Y_i)_{1\leq i \leq N_n}$ is sparse, in the sense that $\lim_{n \to \infty} \frac{D_n}{N_n} = 0 $.\vspace{2mm}
\begin{enumerate}
	\item There exists a positive constant $C$ such that, for all $r \geq 2$, 
$ \left|\kappa^{(r)}\!\left(\frac{X_n}{D_n}\right)\right| \leq (Cr)^r\,\frac{N_n}{D_n}.$\vspace{2mm}
	\item Consider the bounded sequences 
	$$\sigma^2_n=\frac{D_n}{N_n}\,\,\kappa^{(2)}\!\left(\frac{X_n}{D_n}\right)\qquad;\qquad L_n=\sigma_n^3\,\frac{D_n}{N_n}\,\,\kappa^{(3)}\!\left(\frac{X_n}{D_n}\right).$$
	If they have limits $\sigma^2>0$ and $L$, then $$\frac{X_n-\esper[X_n]}{D_n\,\sigma_n}\,\left(\frac{D_n}{N_n}\right)^{1/3}$$ converges in the mod-Gaussian sense with parameters $t_n=(N_n/D_n)^{1/3}$ and limiting function $\psi(z)=\exp(Lz^3/6)$.\vspace{2mm}
\item If furthermore,
    \[ \sigma^2_n = \sigma^2 \left(1 + o\!\left( (D_n/N_n)^{1/3}\right) \right), \]
    then the variable
    \[\frac{X_n-\esper[X_n]}{D_n\,\sigma}\,\left(\frac{D_n}{N_n}\right)^{1/3}\]
    converges in the mod-Gaussian sense with parameters $(N_n/D_n)^{1/3}$ and limiting function $\psi(z)=\exp(Lz^3/6)$
    (the difference with the previous item is that $\sigma_n$ has been replaced by $\sigma$).
\end{enumerate}
\end{theorem}

Of course, if $(2)$ holds, then the results of this paper imply that 
\begin{itemize}
\item $X_n/(\sigma_n \sqrt{N_n D_n})$ satisfies a central limit
theorem, with a normality zone of size $o\!\left(\left(\frac{D_n}{N_n}\right)^{1/6}\right)$;
\item furthermore, for $x>0$,
\begin{equation}
    \label{EqDevProbAtTheEdgeInDepGraphs}
    \proba[X_n-\esper[X_n] \geq (N_n)^{2/3}\,(D_n)^{1/3}\,\sigma_n\,x]=\frac{\E^{-\left(\frac{N_n}{D_n}\right)^{1/3}\frac{x^2}{2}}}{x\sqrt{2\pi (N_n/D_n)^{1/3} }}\,\exp\left(\frac{Lx^3}{6}\right)\,(1+o(1))
\end{equation}
and a similar result holds for negative deviations.
\end{itemize}
The same holds replacing $\sigma_n$ by $\sigma$ if one has the assumption on the speed of convergence of $\sigma_n$ towards $\sigma$, given in the third item of Theorem \ref{thm:modgaussiangeneralsparsegraph}.
If, moreover, the cumulants of $X_n$ satisfy \eqref{eq:limitcumulant},
then \eqref{EqDevProbAtTheEdgeInDepGraphs} still holds when we replace $x$ by a positive sequence $x_n$
tending to infinity with $x_n=o(N_n/D_n)^{1/12}$ --- see Proposition \ref{prop:largedeviationscumulants}.

\begin{center}
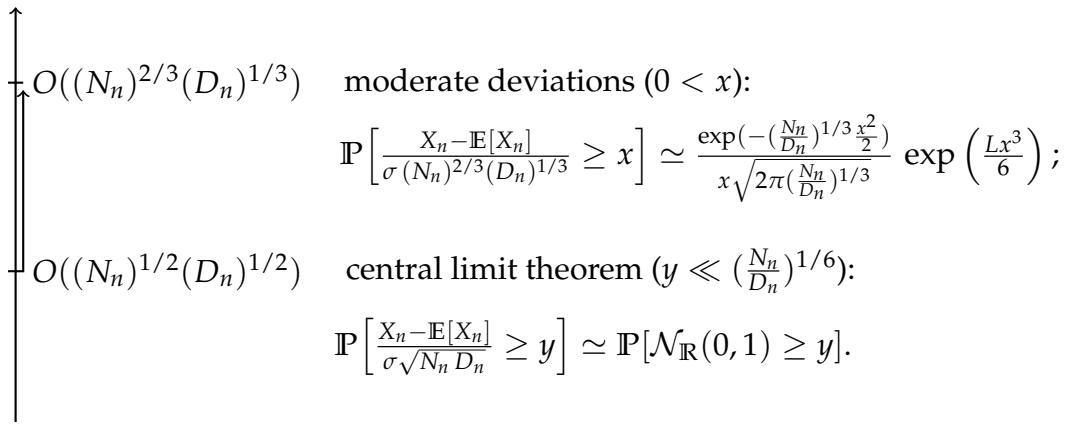
\begin{figure}[ht]
\begin{tikzpicture}
\draw (-1,1.5) node {order of fluctuations};
\draw (5,-0) node {moderate deviations ($0<x$):};
\draw (7,-1) node {$\proba\!\left[\frac{X_n-\esper[X_n]}{\sigma \,(N_n)^{2/3}(D_n)^{1/3}} \geq x\right] \simeq \frac{\exp(-(\frac{N_n}{D_n})^{1/3}\frac{x^2}{2})}{x\sqrt{2\pi(\frac{N_n}{D_n})^{1/3}}}\,\exp\left(\frac{Lx^3}{6}\right);$};
\draw (5.7,-2.5) node {central limit theorem ($y \ll (\frac{N_n}{D_n})^{1/6}$):};
\draw (5.6,-3.5) node {$\proba\!\left[\frac{X_n-\esper[X_n]}{\sigma \sqrt{N_n\,D_n}} \geq y\right] \simeq \proba[\mathcal{N}_{\R}(0,1)\geq y].$};
\draw[->,thick] (-2,-4.5) -- (-2,1) ;
\draw[thick] (-2.1,0) -- (-1.9,0);
\draw (0,0) node {$O((N_n)^{2/3}(D_n)^{1/3})$} ;
\draw[->,thick] (-2.1,-2.5) -- (-1.9,-2.5) -- (-1.9,-0.1) ;
\draw (0,-2.5) node {$O((N_n)^{1/2}(D_n)^{1/2})$} ;
\end{tikzpicture}
\caption{Panorama of the fluctuations of a sum of $N_n$ random variables that have a sparse dependency graph, of degree $D_n-1$ with $\lim_{n \to \infty}\frac{D_n}{N_n}=0$.}
\label{FigPanoramaSparseDepGraphs}
\end{figure}
\end{center}
Notice that in the case of independent random variables, $D_n=1$
and we recover the mod-Gaussian convergence of Example \ref{ex:sumiid}.
\bigskip

In the next Sections, we shall study some particular cases of Theorem \ref{thm:modgaussiangeneralsparsegraph}. 
We will see that most computations are related to the asymptotic analysis of the first cumulants of $X_n$.

\bigskip
\bigskip

\section{Subgraph count statistics in Erd\"os-R\'enyi random graphs}\label{sec:erdosrenyi}
In this section, we consider Erd\"os-R\'enyi model $\Gamma(n,p_n)$ of random graphs. A random graph $\Gamma$ with this distribution is described as follows. Its vertex set is $[n]$ and for each pair $\{i,j\} \subset [n]$  with $i \neq j$, there is an edge between $i$ and $j$ with probability $p_n$. Moreover, all these events are independent. We are then interested in the following random variables, called {\em subgraph count statistics}. If $\gamma$ is a fixed graph of size $k$, then $\SubGraph_\gamma$ is the number of copies of $\gamma$ {\em contained} in the graph $\Gamma(n,p_n)$ (a more formal definition is given in the next paragraph). This is a classical parameter in random graph theory; see, {\em e.g.}, the book of S. Janson, T. {\L}uczak and A. Ruci\'nski \cite{JLR00}. \bigskip

The first result on this parameter was obtained by P. Erd\"os and A. R\'enyi, \emph{cf.} \cite{ER60}. They proved that, if $\gamma$ belongs to some particular family of graphs (called {\em balanced}),
one has a threshold: namely,
$$\lim_{n \to \infty} \proba[\SubGraph_\gamma >0]=\begin{cases}
    0 & \text{if }p_n=o\left(n^{-1/m(\gamma)}\right); \\
1 & \text{if }n^{-1/m(\gamma)}=o(p_n), \end{cases}
$$
where $m(\gamma)=|E(\gamma)|/|V(\gamma)|$.
This result was then generalized to all graphs by B. Bollob\'as~\cite{Bol01}, but the parameter $m(\gamma)$ is in general more complicated than the quotient above. Consider the case $n^{-1/m(\gamma)}=o(p_n)$, when the graph $\Gamma(n,p_n)$ contains with high probability a copy of $\gamma$. It was then proved by A. Ruci\'nski (see \cite{Ruc88}) that, under the additional assumption $n^2 (1-p_n) \to \infty$, the fluctuations of $\SubGraph_\gamma$ are Gaussian. This result can be obtained using dependency graphs; see {\em e.g.}  \cite[pages 147-152]{JLR00}.\bigskip

Here, we consider the case where $p_n=p$ is a constant sequence ($0<p<1$). The possibility of relaxing this hypothesis is discussed in Subsection \ref{subsec:nonconstantproba}. Denote $\alphan=n^2$ and $\betan=n^{k-2}$, where $k$ is the number of vertices of $\gamma$. It is easy to check that 
$$\esper[\SubGraph_\gamma]=c \,\alphan \,\betan \quad;\quad \Var(\SubGraph_\gamma)=\sigma^2 \,\alphan \,(\betan)^2$$ 
for some positive constants $c$ and $\sigma$ --- see, {\em e.g.}, \cite[Lemma 3.5]{JLR00}. Hence, Ruci\'nski's central limit theorem asserts that,  if $T \sim x \sqrt{\alphan}$ for some fixed real $x$, then
$$ \lim_{n\to \infty} \proba \left[ \frac{\SubGraph_\gamma - \esper[\SubGraph_\gamma]}
\betan \geq T \right] = \frac{1}{ \sqrt{2\pi}}
\int_{-\infty}^{x/\sigma} e^{-u^2/2} du.$$
Using Theorem \ref{thm:dependencygraphsrefined}, we shall extend this result to a framework where $x$ tends to infinity, but not to quickly: $(\alphan)^{1/2} \ll T \ll (\alphan)^{3/4}$.

\begin{theorem} \label{thm:moddevgraphs}
    Let $0<p<1$ and $\gamma$ be a graph with $k$ vertices. We consider $\SubGraph_\gamma$ the number
    of copies of $\gamma$ contained in Erd\"os-R\'enyi random graph $\Gamma(n,p)$. Let $\alphan$ and $\betan$ be defined as above.\vspace{2mm}
\begin{enumerate}
\item The renormalized variable 
    $(\SubGraph_\gamma - \esper[\SubGraph_\gamma])/(\alphan^{1/3} \betan)$
    converges mod-Gaussian with parameters $t_{n}=\sigma^{2}\,\alphan^{1/3}$
    and limiting function $\psi(z)=\exp\big(\frac{L}{6} z^{3}\big)$,
    where $\sigma$ and $L$ are computed in Section \ref{subsec:sigma_L}.
The convergence takes place on the whole complex plane with speed of convergence $o(\alphan^{-1/3})$.
\vspace{2mm}

\item Therefore, $(\SubGraph_\gamma - \esper[\SubGraph_\gamma])/(\alphan^{1/2} \betan)$ converges in distribution
towards a Gaussian law, with normality zone $o(\alphan^{1/6})$.
Moreover, the deviation probabilities at the edge of the normality zone and at a slightly larger scale 
are given as follows:
is $x_n$ is a positive sequence bounded away from $0$ with $x_n=o\big(\alphan^{1/12})$, then
\begin{align*}
        \proba \left[ \frac{\SubGraph_\gamma - \esper[\SubGraph_\gamma]}    
        {\alphan^{1/2} \, \betan} \geq \sigma^2 \alphan^{1/6} x_n \right] &=
\frac{\E^{-\frac{(x_n)^2(\alphan)^{1/3}\sigma^2}{2}}}{x_n(\alphan)^{1/6}\sigma \sqrt{2\pi}}\,\E^{\frac{L(x_n)^3}{6}}\,(1+o(1)).\\
        \proba \left[ \frac{\SubGraph_\gamma - \esper[\SubGraph_\gamma]}    
        {\alphan^{1/2} \, \betan} \leq - \sigma^2 \alphan^{1/6} x_n \right] &=
\frac{\E^{-\frac{(x_n)^2(\alphan)^{1/3}\sigma^2}{2}}}{x_n(\alphan)^{1/6}\sigma \sqrt{2\pi}}\,\E^{\frac{-L(x_n)^3}{6}}\,(1+o(1)).
\end{align*}
    \end{enumerate}
\end{theorem}
\smallskip

A similar result has been obtained by H. D\"{o}ring and P. Eichelsbacher in \cite[Theorem 2.3]{DE12}. However, \vspace{2mm}
\begin{itemize}
    \item their result is less precise as they only obtain the equivalence of the logarithms of the relevant quantities (in particular, when we look at the logarithm, the second factor of the right-hand side is negligible);\vspace{2mm}
    \item and they cover a smaller zone of deviation.\vspace{2mm}
\end{itemize}
Unfortunately, we cannot get deviation results when $x_n \sim t \alphan^{1/3}$ for some real number $t$;
this would amount to evaluate $\proba [\SubGraph_\gamma > (1+\varepsilon)\,\esper[\SubGraph_\gamma]]$. For large deviations equivalents of
$$ \log \, \proba [\SubGraph_\gamma >  (1+\varepsilon)\,\esper[\SubGraph_\gamma]], $$
there is a quite large literature, see \cite[Theorem 4.1]{CV11} and \cite{Cha12} for recent results in this field. As we consider deviations of a different scale, our result is neither implied by, nor implies these results. Note, however, that their large deviation results are equivalents of the logarithm of the probability, while our statement is an equivalent for the probability itself.
\bigskip

\subsection{A bound on cumulants}
\subsubsection{Subgraph count statistics}
In the following we denote $\arr(n,k)$ the set of \emph{arrangements} in $[n]$ of length $k$, \emph{i.e.}, lists of $k$ distinct elements in $[n]$. 
The cardinality of $\arr(n,k)$ is  the falling factorial 
$n\fall{k}=n(n-1) \cdots (n-k+1)$.
Let $A=(a_1,\dots,a_k)$ be an arrangement in $[n]$ of length $k$, and $\gamma$ be a fixed graph with vertex set $[k]$. Recall that $\Gamma=\Gamma(n,p_{n})$ is a random Erd\"os-R\'enyi graph on $[n]$. We denote $\delta_\gamma(A)$ the following random variable:
\begin{equation}\delta_\gamma(A) = \begin{cases}
     1 & \text{ if }\gamma \subseteq \Gamma[a_1,\dots,a_k];\\
     0 & \text{ else.}
\end{cases}\label{eq:deltadef}
\end{equation}
Here $\Gamma[a_1,\dots,a_k]$ denotes the graph induced by $\Gamma$ on vertex set $\{a_1,\dots,a_k\}$.
As our data is an {\em ordered} list $(a_1,\dots,a_k)$, this graph can canonically be seen as a graph on vertex set $[k]$, that is the same vertex set as $\gamma$. Then the inclusion should be understood as inclusion of edge sets.\medskip

\noindent For any graph $\gamma$ with $k$ vertices
and any integer $n \ge 1$, we then define the random variable $\SubGraph_{\gamma}$ by
$$\SubGraph_{\gamma} = \sum_{A \in \arr(n,k)} \delta_\gamma(A).$$

\begin{remark}
It would also be natural to replace in Definition \eqref{eq:deltadef} the inclusion by an equality $\gamma = \Gamma[a_1,\dots,a_k]$. This would lead to other random variables $Y_\gamma^{(n)}$, called {\em induced} subgraph count statistics. Their asymptotic behavior is harder to study (in particular, fluctuations are not always Gaussian; see \cite[Theorem 6.52]{JLR00}). Notice that if $\gamma$ is a complete graph, then both definitions coincide.
\end{remark}\medskip

\subsubsection{A dependency graph of the subgraph count statistics}
Fix some graph $\gamma$ with vertex set $[k]$. By definition, the variable we are interested in writes as a sum
$$\SubGraph_{\gamma} = \sum_{A \in \arr(n,k)} \delta_\gamma(A).$$
We shall describe a dependency graph for the variables $\{\delta_\gamma(A)\}_{A \in \arr(n,k)}$.\bigskip

For each pair $e=\{v,v'\} \subset [n]$, denote $I_e$ the indicator function of the event: {\em $e$ is in the graph $\Gamma(n,p)$}. By definition of the model $\Gamma(n,p)$, the random variables $I_e$ are independent Bernouilli variables of parameter $p$. Then, for an arrangement $A$, denote $E(A)$ the set of pairs $\{v,v'\}$ where $v$ and $v'$ appear in the arrangement $A$. One has
\[\delta_\gamma(A) = \prod_{e \in E'(A)} I_e, \]
where $E'(A)$ is a subset of $E(A)$ determined by the graph $\gamma$. In particular, if, for two arrangements $A$ and $A'$, one has $|E(A) \cap E(A')| =\emptyset$ (equivalently, $|A \cap A'| \le 1$), then the variables $\delta_\gamma(A)$ and $\delta_\gamma(A')$ are defined using different variables $I_e$ (and, hence, are independent). This implies that the following graph denoted $\B$ is a dependency graph for the family of variables $\{\delta_\gamma(A)\}_{A \in \arr(n,k)}$:\vspace{2mm}
\begin{itemize}
    \item its vertex set is $\arr(n,k)$;\vspace{2mm}
    \item there is an edge between $A$ and $A'$ if  $|A \cap A'| \ge 2$.\vspace{2mm}
\end{itemize}
Considering this dependency graph is quite classical --- see, {\em e.g.}, \cite[Example 1.6]{JLR00}.\bigskip

All variables in this graph are Bernoulli variables and, hence, bounded by $1$. Besides the graph $\B$ has $N=n\fall{ k}$ vertices, and is regular of degree $D$ smaller than 
$$\binom{k}{2}^2  \,2\,  (n-2)(n-3)\cdots(n-k+1) < k^4 \,n^{k-2}.$$ 
Indeed, a neighbour $A'$ of a fixed arrangement $A \in \arr(n,k)$ is given as follows:\vspace{2mm}
\begin{itemize}
    \item choose a pair $\{a_i,a_j\}$ in $A$ that will appear in $A'$;\vspace{2mm}
    \item choose indices $i'$ and $j'$ such that $a'_{i'}=a_i$ and $a'_{j'}=a_j$
        (these indices are different but their order matters);\vspace{2mm}
    \item choose the other values in the arrangement $A'$.\vspace{2mm}
\end{itemize}
So, we may apply Theorem \ref{thm:dependencygraphsrefined} and we get:
\begin{proposition} \label{prop:boundcumulantsubgraph}
Fix a graph $\gamma$ of vertex set $[k]$. For any $r\le 1$, one has
$$ \big| \kappa^{(r)}(\SubGraph_\gamma) \big| \le 2^{r-1} r^{r-2}\,  n^k\, (k^4 \,n^{k-2})^{r-1}. $$
\end{proposition}
\bigskip

\subsection{Polynomiality of cumulants}
\subsubsection{Dealing with several arrangements}
Consider a list $(A^1,\dots,A^r)$ of arrangements. We associate to this data two graphs  (unless said explicitly, we always consider loopless simple graphs, and $V(G)$ and $E(G)$ denote respectively the edge and vertex sets of a graph $G$):\vspace{2mm}
\begin{itemize}
 \item the graph $G_\AA$ has vertex set
   $$V_{\Bbbk}=V(G_{\AA})= \{ (t,i) \,\,|\,\, 1 \le i \le r, 1 \le t \le k_i \}$$
and an edge between $(t,i)$ and $(s,j)$ if and only if $a^i_t=a^j_s$. It is always a disjoint union of cliques. If the $A^i$'s are arrangements, then the graph $G_\AA$ is endowed with a natural proper $r$-coloring, $(t,i)$ being of color $i$.\vspace{2mm}
\item the graph $H^m_\AA$  has vertex set $[r]$ and an edge between $i$ and $j$ if $|A^i \cap A^j| \geq m$. \vspace{2mm}
\end{itemize}
Notice that $H^1_{\AA}$ is the \emph{contraction} of the graph $G_{\AA}$ by the map $\varphi : (t,i) \mapsto i$ from the vertex set of $G_\AA$ to the vertex of $H^1_\AA$. Indeed,
$$(i,j) \in E(H^1_\AA) \ \Leftrightarrow \ \exists v_i \in \varphi^{-1}(i), v_j  \in \varphi^{-1}(j) \text{ such that } (v_i,v_j) \in E(G_\AA).$$
An example of a graph $G_\AA$ and of its 1- and 2-contractions $H^1_\AA$ and $H^2_\AA$ is drawn on Figure \ref{fig:graphsGandH}. For $m \ge 2$, the definition of $H^m_\AA$ is less common. We call it the $m$-contraction of $G_\AA$. The $2$-contraction is interesting for us. It corresponds exactly to the graph induced by the dependency graph $\B$ on the list of arrangement $\AA$, considered in the proof of Theorem~\ref{thm:dependencygraphsrefined}.

\begin{figure}[ht]
\includegraphics{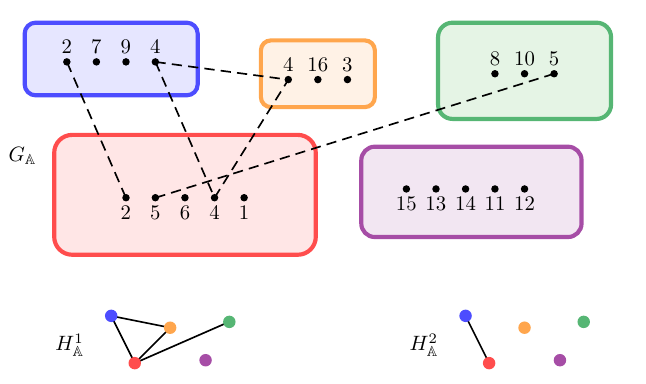}
\caption{The graphs $G_\AA$, $H^1_\AA$ and $H^{2}_{\AA}$ corresponding to the family of arrangements $(15,13,14,11,12)$, $(2,5,6,4,1)$, $(2,7,9,4)$, $(4,16,3)$,  and $(8,10,5)$.}
\label{fig:graphsGandH}
\end{figure}

\begin{remark}
Graphs associated to families of arrangements are a practical way to encode some information and should not be confused with the random graphs or their induced subgraphs. Therefore we used greek letters for the latter and latin letter for graphs $G_\AA$ and their contractions. The dependency graph will always be called $\B$ to avoid confusions.
\end{remark}
\bigskip

\subsubsection{Exploiting symmetries}
The dependency graph of our model has much more structure than a general dependency graph. In particular, all variables $\delta_\gamma(A)$ are identically distributed. More generally, the joint distribution of
\[ \left( \delta_\gamma(A^1),\dots,\delta_\gamma(A^r) \right) \]
depends only on $G_\AA$. Here, we state a few consequences of this invariance property that will be useful
in the next Section.

\begin{lemma}\label{lem:indjointmoments}
Fix a graph $\gamma$ of size $k$, the quantity
$$\esper\!\left[ \delta_{\gamma}(A^1) \cdots \delta_{\gamma}(A^r) \right]$$
depends only on the graph $G_\AA$ associated to the family of arrangements $(A^1,\dots,A^r)$. The same is true for the joint cumulant $\kappa\left( \delta_{\gamma}(A^1), \dots, \delta_{\gamma}(A^r) \right).$
\end{lemma}

\begin{proof}
The first statement follows immediately  from the invariance of the model $\Gamma(n,p)$ by relabelling of the vertices. The second is a corollary, using the moment-cumulant relation \eqref{eq:moment2cumulant}.
\end{proof}
\medskip 

\begin{corollary}
Fix some graph $\gamma$.
Then the joint cumulant $\kappa(\SubGraph_{\gamma},\dots,\SubGraph_{\gamma})$ is a polynomial in $n$.
\label{cor:jointcumulantpoly}
\end{corollary}

\begin{proof}
Using Lemma \ref{lem:indjointmoments}, we can rewrite the expansion \eqref{eq:expcumulant} as
\begin{equation}
    \kappa(\SubGraph_{\gamma},\dots,\SubGraph_{\gamma}) = \sum_G \kappa(G) N_G,
    \label{eq:expansioncumulantgraphs}
\end{equation}
where:\vspace{2mm}
\begin{itemize}
    \item the sum runs over graphs $G$ of vertex set $V_\Bbbk$ that correspond to some arrangements (that is $G$ is a disjoint union of cliques and, for any $s$, $t$ and $i$, there is no edge between $(s,i)$ and $(t,i)$);\vspace{2mm}
    \item $\kappa(G)$ is the common value of $\kappa\left( \delta_{\gamma}(A^1), \dots, \delta_{\gamma}(A^r) \right)$, where $(A^1,\dots,A^r)$ is any list of arrangements with associated graph $G$;\vspace{2mm}
    \item $N_G$ is the number of lists of arrangements with associated graph $G$.\vspace{2mm}
\end{itemize}
But it is clear that the sum index is finite and that neither the summation index nor the quantity $\kappa(G)$ depend on $n$. Besides, the number $N_G$ is simply the falling factorial $n(n-1)\dots(n-c_G+1)$, where $c_G$ is the number of connected components of $G$. The corollary follows from these observations.
\end{proof}
\bigskip

\subsection{Moderate deviations for subgraph count statistics}
\subsubsection{End of the proof of Theorem \ref{thm:moddevgraphs}}
We would like to apply Proposition \ref{prop:largedeviationscumulants} to the sequence $S_n=\SubGraph_\gamma-\esper[\SubGraph_\gamma]$ with $\alphan=n^2$ and $\betan=n^{k-2}$. Let us check that $S_n$ indeed fulfills the hypothesis.\vspace{2mm}
\begin{enumerate}
\item The uniform bound $|\kappa^{(r)}(S_n)| \leq (C r)^r \alphan (\betan)^r$, where $C$ does not depend on $n$,  corresponds to Proposition~\ref{prop:boundcumulantsubgraph}; we may even choose $C=2k^4$.\vspace{2mm}
\item We also have to check the speed of convergence:
\begin{equation}\label{eq:speedconvK2K3}
    \kappa^{(2)}(S_{n})= \sigma^{2}\,\alphan\,(\betan)^2\,(1+O(\alphan^{-1/2}))
\quad;\quad
\kappa^{(3)}(S_{n}) = L \,\alphan\, (\betan)^3\,(1+O(\alphan^{-1/4})).
\end{equation}
But these estimates follow directly from the bound above for $r=2,3$ and the fact that $\kappa^{(r)}(S_{n})$ is always a polynomial in $n$ --- see Corollary \ref{cor:jointcumulantpoly}.\vspace{2mm}
\end{enumerate}
Finally, the mod-Gaussian convergence follows from the observations in Section \ref{sec:cumulantechnic}.
The normality zone result follows from Proposition \ref{prop:normalityzone} and we can apply Proposition \ref{prop:largedeviationscumulants} to get the moderate deviation statement.
 This ends the proof of Theorem \ref{thm:moddevgraphs}. \qed
\medskip

\begin{remark}
    Using Theorem \ref{thm:boundjointcumulant}, we could obtain a bound for joint cumulants of subgraph count statistics. Hence, it would be possible to derive mod-Gaussian convergence and moderate deviation results for {\em linear combinations} of subgraph count statistics. However, since we do not have a specific motivation for that and since the statement for a single subgraph count statistics is already quite technical, we have chosen not to present such a result.
\end{remark}\bigskip

\subsubsection{Computing $\sigma^2$ and $L$}
\label{subsec:sigma_L}
The proof above does not give an explicit value for $\sigma^2$ and $L$. Yet, these values can be obtained by analyzing the graphs $G$ that contribute to the highest degree term of $\kappa^{(2)}$ and $\kappa^{(3)}$.

\begin{lemma} 
Let $\gamma$ be a graph with $k$ vertices and $h$ edges. Then the positive number $\sigma$ appearing in Theorem \ref{thm:moddevgraphs} is given by $$\sigma^2 = 2 h^2 p^{2h-1} (1-p). $$
\end{lemma}

\begin{proof}
By definition, $\sigma^2$ is the coefficient of $n^{2k-2}$ in $\kappa^{(2)}(\SubGraph_\gamma)$. As seen in Equation \eqref{eq:expansioncumulantgraphs}, the quantity $\kappa^{(2)}(\SubGraph_\gamma)$ can be written as 
$$\sum_{G} \kappa(G) N_G,$$
where the sum runs over some graphs $G$ with vertex set $V \sqcup V$. However, we have seen that $\kappa(G)=0$ unless the $2$-contraction $H^2$ of $G$ is connected --- see Inequality \eqref{eq:boundkappa} --- and on the other hand, $N_G$ is a polynomial in $n$, whose degree is the number $c_G$ of connected component of $G$.
\bigskip

\noindent As we are interested in the coefficient of $n^{2k-2}$, we should consider only graphs $G$ with at least $2k-2$ connected components and a connected $2$-contraction. These graphs are represented on Figure~\ref{fig:computationsigma}.
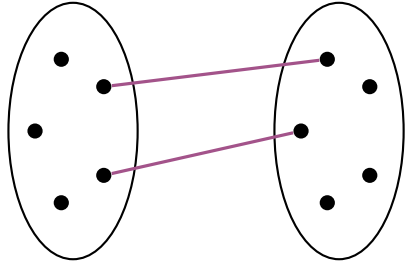
\begin{figure}[ht]
    $$\begin{array}{c} 
        \begin{tikzpicture}
    \tikzstyle{vertex}=[circle,fill=black,inner sep=0pt,minimum size=2mm]
    \begin{scope}[xscale=.5]
        \node (v1) at (36:1) [vertex] {};
        \node (v2) at (108:1) [vertex] {};
        \node (v3) at (180:1) [vertex] {};
        \node (v4) at (-108:1) [vertex] {};
        \node (v5) at (-36:1) [vertex] {};
        \draw [line width=0.8pt] (0,0) circle (1.7);
    \end{scope}
    \begin{scope}[xshift= 3.5cm,xscale=.5]
        \node (v1b) at (36:1) [vertex] {};
        \node (v2b) at (108:1) [vertex] {};
        \node (v3b) at (180:1) [vertex] {};
        \node (v4b) at (-108:1) [vertex] {};
        \node (v5b) at (-36:1) [vertex] {};
        \draw [line width=0.8pt] (0,0) circle (1.7);
    \end{scope}
    \draw [line width=1.2pt,DarkOrchid] (v1) to (v2b);
    \draw [line width=1.2pt,DarkOrchid] (v5) to (v3b);
\end{tikzpicture}\end{array}$$
\caption{Graphs involved in the computation of the main term in $\kappa^{(2)}(\SubGraph_{\gamma})$.}
\label{fig:computationsigma}
\end{figure}

\noindent Namely, we have to choose a pair of vertices on each side and connect each of these vertices to one vertex of the other pair (there are $2$ ways to make this connection, if both pairs are fixed). A quick computation shows that, for such a graph $G$,
\[\kappa(G)=\begin{cases}
    p^{2h-1} (1-p) &\text{ if both pairs correspond to an edge of }\gamma;\\
    0 & \text{ else.}
\end{cases}\]
Finally there are $2 h^2$ graphs with a non-zero contribution to the coefficient of $n^{2k-2}$ in $\kappa^{(2)}(\SubGraph_\gamma)$.
For each of these graphs, $N_G=n(n-1)\cdots(n-2k+3)=n^{2k-2}(1+o(1))$. Therefore, the coefficient of $n^{2k-2}$ in $\kappa^{(2)}(\SubGraph_\gamma)$ is $2 h^2 p^{2h-1} (1-p)$, as claimed.
\end{proof}\bigskip

The number $L$ can be computed by the same method.
\begin{lemma}
Let $\gamma$ be a graph with $k$ vertices and $h$ edges. Then the number $L$ appearing in Theorem \ref{thm:moddevgraphs} is given by
$$L =12 h^3 (h-1) p^{3h-2} (1-p)^2 + 4 h^3 p^{3h-2} (1-p) (1-2p).$$
\end{lemma}

\begin{proof}
Here, we have to consider graphs $G$ on vertex set $V \sqcup V \sqcup V$ with at least $3k-4$ connected components and with a connected $2$-contraction. These graphs are of two kinds, see Figure~\ref{fig:computationL}.
 \begin{figure}[ht]
 $$\begin{array}{c|c}
     \begin{tikzpicture}[scale=.8]
    \tikzstyle{vertex}=[circle,fill=black,inner sep=0pt,minimum size=2mm]
    \begin{scope}[xscale=.5]
        \node (v1) at (36:1) [vertex] {};
        \node (v2) at (108:1) [vertex] {};
        \node (v3) at (180:1) [vertex] {};
        \node (v4) at (-108:1) [vertex] {};
        \node (v5) at (-36:1) [vertex] {};
        \draw [line width=0.8pt] (0,0) circle (1.7);
    \end{scope}
    \begin{scope}[xshift= 3cm,xscale=.5]
        \node (v1b) at (36:1) [vertex] {};
        \node (v2b) at (108:1) [vertex] {};
        \node (v3b) at (180:1) [vertex] {};
        \node (v4b) at (-108:1) [vertex] {};
        \node (v5b) at (-36:1) [vertex] {};
        \draw [line width=0.8pt] (0,0) circle (1.7);
    \end{scope}
    \begin{scope}[xshift= 6cm,xscale=.5]
        \node (v1t) at (36:1) [vertex] {};
        \node (v2t) at (108:1) [vertex] {};
        \node (v3t) at (180:1) [vertex] {};
        \node (v4t) at (-108:1) [vertex] {};
        \node (v5t) at (-36:1) [vertex] {};
        \draw [line width=0.8pt] (0,0) circle (1.7);
    \end{scope}
    \draw [line width=1pt,DarkOrchid] (v1) to (v2b);
    \draw [line width=1pt,DarkOrchid] (v5) to (v3b);
    \draw [line width=1pt,DarkOrchid] (v1b) to (v3t);
    \draw [line width=1pt,DarkOrchid] (v4b) to (v5t);
\end{tikzpicture}
\qquad & \qquad
     \begin{tikzpicture}[scale=.8]
    \tikzstyle{vertex}=[circle,fill=black,inner sep=0pt,minimum size=2mm]
    \begin{scope}[xscale=.5]
        \node (v1) at (36:1) [vertex] {};
        \node (v2) at (108:1) [vertex] {};
        \node (v3) at (180:1) [vertex] {};
        \node (v4) at (-108:1) [vertex] {};
        \node (v5) at (-36:1) [vertex] {};
        \draw [line width=0.8pt] (0,0) circle (1.7);
    \end{scope}
    \begin{scope}[xshift= 3cm,xscale=.5]
        \node (v1b) at (36:1) [vertex] {};
        \node (v2b) at (108:1) [vertex] {};
        \node (v3b) at (180:1) [vertex] {};
        \node (v4b) at (-108:1) [vertex] {};
        \node (v5b) at (-36:1) [vertex] {};
        \draw [line width=0.8pt] (0,0) circle (1.7);
    \end{scope}
    \begin{scope}[xshift= 6cm,xscale=.5]
        \node (v1t) at (36:1) [vertex] {};
        \node (v2t) at (108:1) [vertex] {};
        \node (v3t) at (180:1) [vertex] {};
        \node (v4t) at (-108:1) [vertex] {};
        \node (v5t) at (-36:1) [vertex] {};
        \draw [line width=0.8pt] (0,0) circle (1.7);
    \end{scope}
    \draw [line width=1pt,DarkOrchid] (v1) to (v2b);
    \draw [line width=1pt,DarkOrchid] (v5) to (v3b);
    \draw [line width=1pt,DarkOrchid] (v2b) to (v3t);
    \draw [line width=1pt,DarkOrchid] (v3b) to (v4t);
    \draw [line width=1pt,DarkOrchid] (v1) to (v3t);
    \draw [line width=1pt,DarkOrchid] (v5) to (v4t);
\end{tikzpicture}
 \end{array}$$
 \caption{Graphs involved in the computation of the main term in $\kappa^{(3)}(\SubGraph_{\gamma})$.\label{fig:computationL}}
\end{figure}
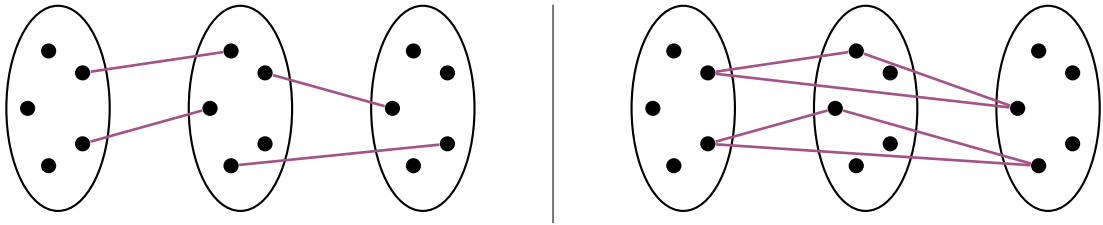\bigskip

In the first case (left-hand side picture),  an edge on the left can possibly have an extremity in common with an edge on the right. In this case, one has to add an edge to complete the triangle (indeed, all graphs $G$ are disjoint unions of cliques).  But this cannot happen for both edges on the left simultaneously, otherwise the graph belong to the second family.\bigskip
 
The following is now easy to check. There are $12 h^3 (h-1)$ graphs of the first kind with a non-zero cumulant $\kappa(G)$ --- 3 choices for which copy of $V$ plays the central role, $h^3 (h-1)$ for pairs of vertices and $4$ ways to link the chosen vertices --- and, for these graphs, the corresponding cumulant is always $\kappa(G) = p^{3h-2} (1-p)^2$. Similarly, there are $4 h^3$ graphs of the second kind with a non-zero cumulant $\kappa(G)$. For these graphs, $\kappa(G)=p^{3h-2}(1-p)(1-2p)$.
In both cases, $N_G=n^{3k-4}(1+o(1))$. This completes the proof.
\end{proof}
\bigskip

\begin{example}
Denote $T_{n}$ the number of triangles in a random Erd\"os-R\'enyi graph $\Gamma(n,p)$ (each triangle being counted $6$ times). According to the previous Lemmas, the parameters $\sigma^{2}$ and $L$ are respectively
$$\sigma^{2}=18\,p^{5}\,(1-p)\quad \text{and}\quad L=108\,p^{7}\,(1-p)(7 -8 p).$$
Moreover, $\esper[T_{n}]=n\fall{3}\,p^{3}=n^{3}\,p^{3}-3\,n^{2}\,p^{3}+O(n)$. So,
$$
\proba \!\left[ T_{n} \geq n^{3}p^{3}+n^{2}(v-3p^{3}) \right] \simeq \sqrt{\frac{ 9p^{5}(1-p)}{\pi\, v^{2}}}\, \exp\!\left(-\frac{v^{2}}{36\,p^{5}(1-p) }+\frac{(7-8p)\,v^{3}}{324 \,n \,p^8(1-p)^2 } \right) 
$$
for $1 \ll v\ll n^{1/2}$.
\end{example}
\medskip

\subsubsection{Case of a non-constant sequence $p_n$}
\label{subsec:nonconstantproba}
Proposition \ref{prop:boundcumulantsubgraph} still holds when $p_n$ is a non-constant sequence (the particularly interesting case is $p_n \to 0$). Applying Theorem \ref{thm:dependencygraphs_sumY} instead of Theorem \ref{thm:dependencygraphsrefined}, one gets a slightly sharper bound 
$$\left| \kappa^{(r)} (\SubGraph_\gamma) \right| \le C^{r} \,r^{r-2} \,n^{r(k-2)+2}\, (p_n)^h.$$
But, unlike in the case $p_n=p$ constant, this bound is not always optimal for a fixed $r$ (even up to a multiplicative constant) . Indeed, \cite[Lemma 6.17]{JLR00} gives stronger bounds than ours (see also Example 6.19 in \cite{JLR00}). Finding a uniform bound for cumulants, whose dependence in $r$ is of order $(Cr)^r$  (so that we have mod-Gaussian convergence), and which is optimal for fixed $r$ is an open problem.\bigskip

Yet, we can still give a lower bound on the normality zone. Let $\alphan$ and $\betan$ be defined as follows:
\begin{align*}
    \alphan &= n^2 \,(p_n)^{4h-3} \,(1-p_n)^3;\\
    \betan &=n^{k-2}\, (p_n)^{1-h}\, (1-p_{n})^{-1}.
\end{align*}
With these choices, one has
$$\left| \kappa^{(r)} (\SubGraph_\gamma) \right| \le (Cr)^r\, \alphan (\betan)^r.$$
Notice that $\kappa^{(2)}(\SubGraph_\gamma)$ and $\kappa^{(3)}(\SubGraph_\gamma)$ are polynomials in $n$ and $p_n$, of degree $2h$ and $3h$ in $p_n$. Thanks to this observation, if $p_n \to 0$, then
\begin{align*}
\kappa^{(2)}(\SubGraph_\gamma) &= 2h^2\,n^{2k-2}\,(p_n)^{2h-1}(1-p_n) + O(n^{2k-3}) \\
&= 2h^2\,\alphan\,(\betan)^2 \left(1+O\!\left(\frac{1}{n\,(p_n)^{2h-1}}\right)\right) ;\\
\kappa^{(3)}(\SubGraph_\gamma) &= h^3\,n^{3k-4}\,(p_n)^{3h-2}(1-p_n)\,(12 (h-1) (1-p_n) + 4 (1-2p_n)) + O(n^{3k-5})\\
&=\alphan\,(\betan)^3\,\left(O((p_n)^{2h-2})+O\!\left(\frac{1}{n\,(p_n)^{h}}\right)\right).
\end{align*}
From the discussion of Section \ref{subsec:modgaussfromcumulants}, we conclude that we have mod-Gaussian convergence 
if $(\alphan)^{1/3}=o(n\,(p_n)^{2h-1})$ and $n\,(p_n)^h \to +\infty$. The first condition is equivalent to $n\,(p_n)^{2h} \to +\infty$, so, if
$$ n^{-\frac{1}{2h}} \ll p_n \ll 1 ,$$
then Conditions \eqref{eq:cv_sndthrd_cumulants} are met, and one has a mod-Gaussian convergence of 
$$\frac{\SubGraph_\gamma-\esper[\SubGraph_\gamma]}{n^{k-\frac{4}{3}}\,(p_n)^{\frac{h}{3}}}$$
with parameters $t_n = 2h^2\,n^{\frac{2}{3}}\,(p_n)^{\frac{4h}{3}-1}\,(1-p_n)$ and limit $\psi(z)=1$. Therefore, in this setting, 
$$\frac{\SubGraph_\gamma-\esper[\SubGraph_\gamma]}{h\,n^{k-1}\,\sqrt{2(p_n)^{2h-1}(1-p_n)}}$$
has normality zone at least $O(n^{\frac{1}{3}}\,(p_n)^{\frac{2h}{3}-\frac{1}{2}})$.
Note that we only have a lower bound on the normality zone,
since the limit in the mod-Gaussian convergence is trivial.

\begin{remark}
Unfortunately the convergence speed hypotheses \eqref{eq:speedconvK2K3} are not satisfied, so one cannot apply Proposition \ref{prop:largedeviationscumulants}.
\end{remark}
\bigskip
\bigskip

\section{Random character values from central measures on partitions}\label{sec:central}
Our Theorem~\ref{thm:dependencygraphsrefined} can also be used to study certain models of random integer partitions. Recall that if $G$ is a finite group and if $\tau$ is a function $G \to \C$ with $\tau(e_{G})=1$ and $\tau(gh)=\tau(hg)$ (such a function is called a \emph{trace} on $G$), then $\tau$ can be expanded uniquely as a linear combination of normalized irreducible characters:
$$\tau=\sum_{\lambda \in \widehat{G}} \,\proba_{\tau}[\lambda]\,\widehat{\chi}^{\lambda},$$
where $\widehat{G}$ is the (finite) set of isomorphism classes of irreducible representations of $G$ and $\widehat{\chi}^{\lambda}$ the normalized ({\em i.e.} divided by the dimension of the space)
character of the irreducible representation associated to $\lambda$.
\begin{definition}
The map $\lambda \mapsto \proba_{\tau}[\lambda]$ is called the spectral measure of the trace $\tau$. It takes non-negative values if and only if, for every family $(g_1,\ldots,g_n)$ of elements of $G$, the matrix $(\tau(g_i(g_j)^{-1}))_{1 \leq i,j \leq n}$ is Hermitian non-negative definite. Then, the spectral measure is a probability measure on $\widehat{G}$.
\end{definition}
\bigskip

When $G=\sym(n)$ is the symmetric group of order $n$, the irreducible representations are indexed by \emph{integer partitions} of size $n$, that is non-increasing sequences of positive integers $\lambda=(\lambda_{1},\ldots,\lambda_{\ell})$ with $\sum_{i=1}^{\ell}\lambda_{i}=n$. In the following we denote $\pym(n)=\widehat{\sym}(n)$ the set of integer partitions of size $n$, and $\ell(\lambda)=\ell$ the length of a partition. We shall also consider the infinite symmetric group $\sym(\infty)=\bigcup_{n \ge 1} \sym(n)$, which is the group of permutations of the set of natural numbers that have a finite support.

\begin{definition}
A \emph{central measure} on partitions is a family $(\proba_{\tau,n})_{n \in \N}$ of spectral measures on the sets $\pym(n)$ that come from the same trace of the infinite symmetric group $\sym(\infty)$. In other words $(\proba_{\tau,n})_{n \in \N}$ is a central measure if there exists a trace $\tau : \sym(\infty) \to \C$ such that
$$\tau_{|\sym(n)}=\sum_{\lambda \in \pym(n)} \proba_{\tau,n}[\lambda]\,\widehat{\chi}^{\lambda}.$$
\end{definition}

\begin{example}
The regular trace $\tau(\sigma)=\mathbbm{1}_{\sigma=\mathrm{id}}$ corresponds to the Plancherel measures of the symmetric groups, given by the formula $\proba_{n}[\lambda]=(\dim V^\lambda)^{2}/n!$, where $V^{\lambda}$ is the $\sym(n)$-irreducible module of label $\lambda$. They have been extensively studied in connection with Ulam's problem of the longest increasing subsequence and with random matrix theory, see \emph{e.g.} \cite{BDJ99,BOO00,Oko00,IO02}.
\end{example}
\bigskip

A central measure $(\proba_{\tau,n})_{n \in \N}$ is non-negative if and only if $(\tau(\rho_{i}\rho_{j}^{-1}))_{1\leq i,j \leq n}$ is Hermitian non-negative definite for any finite family of permutations $\rho_{1},\ldots,\rho_{n}$. The set of non-negative central measures, \emph{i.e.}, coherent systems of probability measures on partitions has been identified in \cite{Tho64} and later studied in \cite{KV81}. Call extremal a non-negative trace on $\sym(\infty)$  that is not a positive linear combination of non-negative traces. Then, extremal central measures are indexed by the infinite-dimensional \emph{Thoma simplex}
$$\Omega=\bigg\{\omega=(\alpha,\beta)=\big((\alpha_{1}\geq \alpha_{2}\geq \cdots \geq 0),(\beta_{1}\geq \beta_{2}\geq \cdots \geq 0)\big)\,\,\bigg|\,\,\sum_{i=1}^{\infty}\alpha_{i}+\beta_{i}\leq 1\bigg\}.$$
The trace on the infinite symmetric group corresponding to a parameter $\omega$ is given by

\begin{equation}\label{eq:def_tau_omega}
\tau_{\omega}(\sigma)=\prod_{c \in C(\sigma)} p_{|c|}(\omega) \quad\text{with }p_{1}(\omega)=1,\,\,\,p_{k\geq 2}(\omega)=\sum_{i=1}^{\infty}(\alpha_{i})^{k}+(-1)^{k-1}(\beta_{i})^{k},
\end{equation}
$C(\sigma)$ denoting the set of cycles of $\sigma$. Kerov and Vershik have shown that if $\omega \in \Omega$ and $\rho \in \sym(\infty)$ are fixed, then the random character value $\widehat{\chi}^{\lambda}(\rho)$ with $\lambda$ chosen according to the central measure $\proba_{\omega,n}$ converges in probability towards the trace $\tau_{\omega}(\rho)$. Informally, central measures on partitions and extremal traces of $\sym(\infty)$ are {\em concentrated}. More recently, it was shown by F\'eray and M\'eliot that this concentration is Gaussian, see \cite{FM12,Mel12}. The aim of this section is to use the techniques of Sections \ref{sec:depgraph}-\ref{sec:erdosrenyi} in order to prove the following:

\begin{theorem}\label{thm:modgaussianrcv}
Fix a parameter $\omega \in \Omega$, and denote $\rho$ a $k$-cycle in $\sym(\infty)$ (for $k \ge 2$). Assume $p_{2k-1}(\omega)-(p_{k}(\omega))^{2}\neq 0$. Denote then $\RCV$ the random character value $\widehat{\chi}^{\lambda}(\rho)$, where:\vspace{2mm}
\begin{itemize}
\item $\lambda \in \pym(n)$ is picked randomly according to the central measure $\proba_{\omega,n}$, \vspace{2mm}
\item and $\widehat{\chi}^{\lambda}$ denotes the normalized irreducible character indexed by $\lambda$ of $\sym(n)$, that is 
$$\widehat{\chi}^{\lambda}(\rho)=\frac{\mathrm{tr}\, \Pi^{\lambda}(\rho)}{\dim V^{\lambda}}\qquad \text{ with }(V^{\lambda},\Pi^{\lambda}) \text{ irreducible representation of }\sym(n).$$
\end{itemize}
The rescaled random variable $n^{2/3}\,\big[\RCVk-p_{k}(\omega)\big]$ converges in the mod-Gaussian sense with parameters $t_{n}=n^{1/3}\,\sigma^{2}$ and limiting function $\psi(z)=\exp(L\,\frac{z^{3}}{6})$, where
\begin{align*}
\sigma^{2}&=k^{2}\,\big(p_{2k-1}(\omega)-p_{k}(\omega)^{2}\big);\\
L&= k^{3}\,\big((3k-2)\,p_{3k-2}(\omega)-(6k-3)\,p_{2k-1}(\omega)\,p_{k}(\omega)+(3k-1)\,p_{k}(\omega)^{3}\big).
\end{align*}
\end{theorem}

As in Section \ref{sec:erdosrenyi}, the mod-Gaussian convergence is proved by bounds on cumulants of type \eqref{eq:superbound}.
Furthermore, the hypothesis \eqref{eq:limitcumulant} on the third and second cumulant are ensured
by the polynomiality of joint cumulants --- Lemma \ref{lem:cumulantcharacterpoly}.
Therefore, we can apply Propositions \ref{prop:largedeviationscumulants} and \ref{prop:normalityzone}.
\begin{corollary}\label{cor:moderatedeviationcharacter}
Let $\SubGraph_{k}=\RCVk$ be the random character value on a $k$-cycle as defined above.
Then $n^{1/2}\, (\SubGraph_{k}-p_{k}(\omega))$ satisfies a central limit theorem with normality zone 
$o(n^{1/6})$.
Moreover, at the edge of this normality zone and at a slightly larger scale, the deviation probabilities are given by:
for any sequence $x$ of positive numbers bounded away from $0$ with $x_n=o\big( n^{1/12} \big)$,
one has
\begin{align*}
    \proba\left[n^{1/2} (\SubGraph_{k}-p_{k}(\omega))\geq n^{1/6} \sigma^2 x_n\right]&=
\frac{\E^{-\frac{(x_n)^2\, n^{1/3}\sigma^2}{2}}}{x_n\, n^{1/6}\sigma \sqrt{2\pi}}\,\E^{\frac{L(x_n)^3}{6}}\,(1+o(1)).\\
    \proba\left[n^{1/2} (\SubGraph_{k}-p_{k}(\omega))\leq - n^{1/6} \sigma^2 x_n\right]&=
\frac{\E^{-\frac{(x_n)^2\, n^{1/3}\sigma^2}{2}}}{x_n\, n^{1/6}\sigma \sqrt{2\pi}}\,\E^{\frac{-L(x_n)^3}{6}}\,(1+o(1)).
\end{align*}
where $\sigma^{2}$ and $L$ as in Theorem \ref{thm:modgaussianrcv}.
\end{corollary}\bigskip

\begin{remark}\label{rem:limitvarianceispositive}
The condition $\sigma^{2}>0$ is satisfied as soon as the sequence $\alpha \sqcup \beta$ associated to the Thoma parameter $\omega$ contains two different non-zero coordinates. Indeed, for any summable non-increasing non-negative sequence $\gamma=(\gamma_{1},\gamma_{2},\ldots)$ with $\gamma_{i}>\gamma_{i+1}$ for some $i$, one has
$$\left(\sum_{i=1}^{\infty} (\gamma_{i})^{k-\eps} \right)\left(\sum_{i=1}^{\infty} (\gamma_{i})^{k+\eps} \right) \geq \left(\sum_{i=1}^{\infty} (\gamma_{i})^{k} \right)^{2}$$
with equality if and only if $\eps=0$. Indeed, the derivative of the function of $\eps$ on the left-hand side is
$$\sum_{i<j} \left(\gamma_{i}\gamma_{j}\right)^{k}\,\log\!\left(\frac{\gamma_{i}}{\gamma_{j}}\right)\,\left(\left(\frac{\gamma_{i}}{\gamma_{j}}\right)^{\!\eps}-\left(\frac{\gamma_{j}}{\gamma_{i}}\right)^{\!\eps}\,\right).$$
Applying the result to $\gamma=\alpha \sqcup \beta$ and $\eps=k-1$, we obtain on the left-hand side
$$\left(\sum_{i=1}^{\infty} \gamma_{i} \right)\left(\sum_{i=1}^{\infty} (\gamma_{i})^{2k-1} \right) = \left(\sum_{i=1}^{\infty} \alpha_{i}+\beta_{i} \right)\,p_{2k-1}(\omega)\leq p_{2k-1}(\omega)$$
and on the right-hand side
$$\left(\sum_{i=1}^{\infty} (\gamma_{i})^{k}\right)^{2} \geq \left(\sum_{i=1}^{\infty} (\alpha_{i})^{k} + (-1)^{k-1}(\beta_{i})^{k}\right)^{2}=(p_{k}(\omega))^{2}.$$
This proof shows that the condition $\sigma^{2}>0$ is also satisfied if $0<\sum_{i=1}^{\infty} \alpha_{i}+\beta_{i}<1$.
\end{remark}

This Section is organized as follows. In Section \ref{subsec:prelim_chmu}, we present the necessary material. In Section \ref{subsec:boundcumulantchmu}, we then prove bounds on these cumulants similar to those of Proposition \ref{prop:boundcumulantsubgraph}, and we compute the limits of the second and third cumulants. This will allow us to use in Section \ref{subsec:asymptoticrcv} the framework of Section \ref{subsec:deviationscumulant} in order to prove the results stated above. We shall also detail some consequences of these results for the shapes of the random partitions $\lambda \sim \proba_{n,\omega}$ viewed as Young diagrams, in the spirit of \cite{FM12,Mel12}. \bigskip

\subsection{Preliminaries}\label{subsec:prelim_chmu}
\subsubsection{Non-commutative probability theory}
The originality of this section is that, although the problem is formulated in a classical probability space,
it is natural to work in the setting of non-commutative probability theory.
\begin{definition}
    A {\em non-commutative probability space} is a complex unital $\star$-algebra $\mathscr{A}$ 
    with some linear functional $\varphi:\mathscr{A} \to \C$ such that $\varphi(1)=1$ and,
    for any $a \in \mathscr{A}$, $\varphi(a^\star a) \ge 0$.
\end{definition}
This generalizes the notion of probability space: elements of $\mathscr{A}$ should be thought as random variables, while $\varphi$ should be thought as the expectation. The difference is that, unlike random variables in usual probability theory, elements of $\mathscr{A}$ are not assumed to commute. Usually one also assumes that $\varphi$ is tracial, \emph{i.e.}, $\varphi(ab)=\varphi(ba)$ for any $a,b\in\mathscr{A}$.\bigskip

There are five natural analogues of the notion of independence in non-commutative probability theory, see~\cite{Muraki03}. In our context, the relevant one is the following one, sometimes called {\em tensor independence} \cite[Definition 1.1]{HS10}.
\begin{definition}
    Two subalgebras $\mathscr{A}_1$ and $\mathscr{A}_2$ of $\mathscr{A}$ are {\em tensor independent} if and only if, for any sequence $a_1,\dots,a_r$ with, for each $i$, $a_i \in \mathscr{A}_1$ or $a_i \in \mathscr{A}_2$, one has
    \[\varphi(a_1 \dots a_r) = 
    \varphi \left(\prod_{\substack{1\le i \le r \\ a_i \in \mathscr{A}_1}}^{\rightarrow} a_i\right)
    \varphi \left(\prod_{\substack{1\le i \le r \\ a_i \in \mathscr{A}_2}}^{\rightarrow} a_i\right)
    .\]
    The arrow on the product sign means that the $a_i$ in the product appear in the same order as in $a_1,\dots,a_r$.
\end{definition}

With this definition of independence, the notion of dependency graph presented in Section \ref{subsec:dependencygraphs} is immediately extended to the non-commutative framework. One can also define joint cumulants as follows: if $a_1,\dots,a_r$ are elements in $\mathscr{A}$, we set
$$\kappa(a_1,\dots,a_r) = \sum_{\pi} \,\mu(\pi) \,\prod_{B \in \pi} \varphi \!\left( 
\prod^{\rightarrow}_{i \in B} a_i \right).$$
Note that, in the proof of Theorem~\ref{thm:dependencygraphsrefined}, independence is only used in equation~\eqref{eq:alsoinnoncommutative}. By definition of tensor independence, equation~\eqref{eq:alsoinnoncommutative} also holds in the non-commutative setting and hence, so does Theorem~\ref{thm:dependencygraphsrefined}.
\bigskip

\subsubsection{Two probability spaces}
From now on, we fix an element $\omega$ in the Thomas simplex $\Omega$. Denote $\C\sym(n)$  the group algebra of $\sym(n)$. The function $\tau_\omega$, defined by Equation \eqref{eq:def_tau_omega}, can be linearly extended to $\C\sym(n)$. Then $(\C\sym(n),\tau_{\omega})$ is a non-commutative probability space.
\bigskip

Note that we are now working with two probability spaces: the non-commutative probability space $(\C\sym(n),\tau_{\omega})$ and the usual probability space that we want to study (the set of Young diagrams of size $n$ with the probability measure $\proba_{\omega,n}$). They are related as follows. Consider an element $y$ in $\C\sym(n)$ and define the random variable (for the usual probability space) $X_y$ by $X_y(\lambda)=\widehat{\chi}(y)$.
Then one has
\[\esper_{\proba_{\omega,n}}(X_y) = \sum_{\lambda \vdash n} \proba_{\omega,n}(\lambda)
\widehat{\chi}(y) = \tau_\omega(y).\]
In other words, the {\em expectations} of $y$ (in the non-commutative probability space) and of $X_y$ (in the usual probability space) coincide (recall that the trace of a non-commutative probability space, here $\tau_\omega$, is considered as an expectation).
\bigskip

Besides, if we restrict to the {\em center} of $(\C\sym(n),\tau_{\omega})$, then the map $y \mapsto X_y$ is an {\em algebra morphism}, called {\em discrete Fourier transform}. As a consequence, if $y_1,\dots,y_r$ lie in the center of $(\C\sym(n),\tau_{\omega})$, their joint moments (and joint cumulants) are the same as those of $X_{y_1},\dots,X_{y_r}$.

\subsubsection{Renormalized conjugacy classes and polynomiality of cumulants}
Given a partition $\mu=(\mu_{1},\ldots,\mu_{\ell})$ of size $|\mu|=\sum_{i=1}^{\ell}\mu_{i}=k$, we denote 
$$\varSigma_{\mu,n}=\sum_{A} \rho_\mu(A) \text{ where } \rho_\mu(A)=(a_1,\ldots,a_{\mu_{1}})(a_{\mu_1+1},\ldots,a_{\mu_1+\mu_{2}})\cdots.$$
In the equation above, the formal sum is taken over arrangements in $\arr(n,k)$ and is considered as an element of the group algebra $\C\sym(n)$. It clearly lies in its center. Besides, $\varSigma_{\mu,n}$ is the sum of $n\fall{k}$ elements of cycle-type $\mu$, hence, if $\rho$ is a fixed permutation of type $\mu$, one has:
\[X_{\varSigma_{\mu,n}} = n\fall{k} X_\rho.\]
Note that considering these elements $\varSigma_{\mu,n}$ and their normalized characters is a classical trick in the study of central measures on Young diagrams; see \cite{IO02,Sni06b,FM12}.
\bigskip

Fix some permutations $\rho_1,\dots,\rho_r$ of respective size $k_1,\dots,k_r$. Denote $\mu^{1},\ldots,\mu^{r}$ their cycle-types.
\begin{align}
n\fall{k_{1}}\cdots n\fall{k_{r}}\,\,\kappa\!\left(\SubGraph_{\rho_{\mu^{1}}},\ldots,\SubGraph_{\rho_{\mu^{r}}}\right)&=\kappa(\varSigma_{\mu^{1}},\ldots,\varSigma_{\mu^{r}})\nonumber\\
&=\sum_{\substack{A^{1}\in \arr(n,k_{1}) \vspace{-1mm} \\ \vdots  \vspace{1mm}\\ A^{r} \in \arr(n,k_{r})}}
\kappa\left(\rho_{\mu^{1}}(A^{1}),\ldots,\rho_{\mu^{r}}(A^{r})\right).\label{eq:listequality}
\end{align}
As in the framework of subgraph count statistics,
the invariance of $\tau_\omega$ by conjugacy of its argument
implies that the joint cumulant $\kappa\left(\rho_{\mu^{1}}(A^{1}),\ldots,\rho_{\mu^{r}}(A^{r})\right)$
depends only on the graph $G_{\mathbb{A}}$ associated to the family of arrangement
$\mathbb{A}=(A^{1},\dots,A^{r})$.
Copying the proof of Corollary~\ref{cor:jointcumulantpoly},
we get:
\begin{lemma}\label{lem:cumulantcharacterpoly}
Fix some integer partitions  $\mu^1,\dots,\mu^r$. Then the rescaled joint cumulant 
$$\kappa(\varSigma_{\mu^{1}},\ldots,\varSigma_{\mu^{r}})=n\fall{k_{1}}\cdots n\fall{k_{r}}\,\,\kappa\!\left(\SubGraph_{\rho_{\mu^{1}}},\ldots,\SubGraph_{\rho_{\mu^{r}}}\right)
$$ 
is a polynomial in $n$.
\end{lemma}\bigskip

\subsection{Bounds and limits of the cumulants}\label{subsec:boundcumulantchmu} 
\subsubsection{The dependency graph for random character values} If $k=|\mu|$, we are interested in the non-commutative random variables
$$\varSigma_{\mu,n}=\sum_{A \in \arr(n,k)} \rho_{\mu}(A) \in \C\sym(n).$$
Again, to control the cumulants, we shall exhibit a dependency graph for the families of random variables $\{\rho_{\mu}(A)\}_{A \in \arr(n,k)}$.

Due to the multiplicative form of Equation \eqref{eq:def_tau_omega}, if $I$ and $J$ have disjoint subsets of $[n]$, the subalgebras $\C\sym(I)$ and $\C\sym(J)$ are tensor independent (here, $\sym(I)$ denotes the group of permutations of $I$, canonically included in $\sym(n)$). Therefore, one can associate to $\{\rho_{\mu}(A)\}_{A \in \arr(n,k)}$
the dependency graph $\B$ defined by: \vspace{2mm}
\begin{itemize}
    \item its vertex set is $\arr(n,k)$;\vspace{2mm}
    \item there is an edge between $A$ and $A'$ if  $|A \cap A'| \ge 1$.\vspace{2mm}
\end{itemize}
The graph $\B$ is obviously regular with degree strictly smaller than $k^{2}\,n\fall{k-1}$. On the other hand, all joint moments of the family $(\rho_{\mu}(A))_{A \in \arr(n,k)}$ are normalized characters of single permutations and hence bounded by $1$ in absolute value. So one can once again apply Theorem \ref{thm:dependencygraphsrefined} and we get:

\begin{proposition}\label{prop:boundcumulantcharacter} 
Fix a partition $\mu$ of size $k$. For any $r\le 1$, one has
\begin{align*} 
\big| \kappa^{(r)}(\varSigma_\mu) \big| &\le 2^{r-1} r^{r-2}\,  n\fall k\, (k^{2} n\fall{k-1})^{r-1};\\
\big| \kappa^{(r)}(\SubGraph_{\rho_\mu}) \big| &\le r^{r-2}\,  \left(\frac{2k^{2}}{n} \right)^{r-1}.
\end{align*}
\end{proposition}

\begin{remark}
Using Theorem~\ref{thm:boundjointcumulant}, one can also obtain a bound for joint cumulants of the $\SubGraph_{\rho_\mu}$, namely,
$$\left| \kappa\!\left (\SubGraph_{\rho_{\mu^{1}}},\dots,\SubGraph_{\rho_{\mu^{r}}}\right) \right| \le k_1 \cdots k_r\, \big(r \cdot \max_{1 \le i \le r} k_i \big)^{r-2} \, \left(\frac{2}{n}\right)^{r-1}$$
for integer partitions $\mu^1,\dots,\mu^r$ of sizes $k_1,\ldots,k_r$. In the following, we shall focus on the case of simple random variables $\SubGraph_{\rho_{\mu}}$, though most results also hold in the multi-dimensional setting. Actually, in order to compute the asymptotics of the first cumulants of $X_{\rho_{\mu}}$, it will be a bit clearer to manipulate joint cumulants of variables $\varSigma_{\mu^{1}},\ldots,\varSigma_{\mu^{r}}$ with arbitrary integer partitions.
\end{remark}
\bigskip

\subsubsection{Limits of the second and third cumulants}\label{subsubsec:limitchar_K2K3}
Because of Lemma \ref{lem:cumulantcharacterpoly} and Proposition \ref{prop:boundcumulantcharacter}, for any fixed integer partitions, $$\kappa\left(\SubGraph_{\rho_{\mu^{1}}},\dots,\SubGraph_{\rho_{\mu^{r}}}\right)\,n^{1-r}\simeq \kappa\left(\varSigma_{\mu^{1}},\ldots,\varSigma_{\mu^{r}}\right)\,n^{k_{1}+\cdots+k_{r}-(r-1)}$$ converges to a constant. Let us compute this limit when $r=2$ or $3$; we use the same reasoning as in Section \ref{subsec:sigma_L}.\bigskip

As $\kappa$ is invariant by simultaneous conjugacy of its arguments, the summand in Equation \eqref{eq:listequality} depends only on the graph $G=G_{\AA}$ associated to the collection $\AA=(A^1,\dots,A^r)$, and we shall denote it $\kappa(G)$.  We fix partitions $\mu^{1},\ldots,\mu^{r}$ of respective sizes $k_{1},\ldots,k_{r}$, and write $$\kappa(\varSigma_{\mu^{1}},\ldots,\varSigma_{\mu^{r}})=\sum_{G}\, \kappa(G)\,N_{G}.$$ 
Here, as in Section~\ref{sec:erdosrenyi}, $N_G$ denotes the number of list $\AA$ of arrangements with associated graph $G$.\medskip

When $r=2$, we have to look for graphs $G$ on vertex set $V_{\Bbbk}=[k_{1}]\sqcup [k_{2}]$ with $1$-contraction connected and at least $k_{1}+k_{2}-1$ connected components, because these are the ones that will give a contribution for the coefficient of $n^{k_{1}+k_{2}-1}$. For $i \in [\ell(\mu^{1})]$ and $j \in [\ell(\mu^{2})]$, denote $$(\mu^{1} \Join \mu^{2} )(i,j)=(\mu^{1} \setminus \mu_{i}^{1})\sqcup(\mu^{2} \setminus \mu_{j}^{2}) \sqcup \{\mu_{i}^{1}+\mu_{j}^{2}-1\}.$$ This is the cycle type of a permutation $\rho_{\mu^{1}}(A^{1})\,\rho_{\mu^{2}}(A^{2})$, where $G_{\AA}$ is the graph with one edge joining an element of $A^{1}$ in the cycle of length $\mu^{1}_{i}$ with an element of $A^{2}$ in the cycle of length $\mu^{2}_{j}$. These graphs are the only ones involved in our computation, and they yield 
$$\kappa(G)=p_{(\mu^{1} \Join \mu^{2} )(i,j)}(\omega)-p_{\mu^{1}\sqcup\mu^{2}}(\omega),$$
where for a partition $\mu$ we denote $p_{\mu}(\omega)$ the product $\prod_{i=1}^{\ell(\mu)}p_{\mu_{i}}(\omega)$. So,

\begin{proposition}
For any partitions $\mu$ and $\nu$, the limit of $n\,\kappa(\SubGraph_{\rho_\mu},\SubGraph_{\rho_\nu})$ is 
$$\sum_{i=1}^{\ell(\mu)}\sum_{j=1}^{\ell(\nu)} \,\mu_{i}\,\nu_{j}\,\big(p_{(\mu\Join \nu)(i,j)}(\omega)-p_{\mu\sqcup\nu}(\omega)\big).$$
\end{proposition}

\noindent In particular, for cycles $\mu=(k)$ and $\nu=(l)$,
$$\lim_{n \to \infty} n\,\,\kappa\!\left(\SubGraph_{k},\SubGraph_{l}\right)=kl\,\big(p_{k+l-1}(\omega)-p_{k,l}(\omega)\big).$$
On the other hand, if $\mu=\nu$, then the limit of $n\,\kappa^{(2)}(X^{(n)}_{\rho_{\mu}})$ is
$$(p_{\mu}(\omega))^{2}\sum_{1 \leq i,j\leq \ell(\mu)} \mu_{i}\,\mu_{j}\,\left(\frac{p_{(\mu_{i} \mu_{j}-1)}(\omega)}{p_{\mu_{i}}(\omega)\,p_{\mu_{j}}(\omega)}-1\right).$$
\bigskip

When $r=3$, we look for graphs $G$ on vertex set $V_{\Bbbk}=[k_{1}]\sqcup [k_{2}]\sqcup [k_{3}]$ with $1$-contraction connected and at least $k_{1}+k_{2}+k_{3}-2$ connected components. They are of three kinds: \vspace{2mm}
\begin{enumerate}
\item One cycle in $\rho_{\mu^{2}}(A^{2})$ is connected to two cycles in $\rho_{\mu^{1}}(A^{1})$ and $\rho_{\mu^{3}}(A^{3})$, but not by the same point in this cycle of $\rho_{\mu^{2}}(A^{2})$. This gives for $\rho_{\mu^{1}}(A^{1})\,\rho_{\mu^{2}}(A^{2})\,\rho_{\mu^{3}}(A^{3})$ a permutation of cycle type
$$( \mu^{1} \Join \mu^{2} \Join \mu^{3}) (i,j,k)=(\mu^{1}\setminus \mu^{1}_{i})\sqcup (\mu^{2}\setminus \mu^{2}_{j})\sqcup (\mu^{3}\setminus \mu^{3}_{k})\sqcup\{\mu^{1}_{i}+\mu^{2}_{j}+\mu^{3}_{k}-2\};$$
and the corresponding cumulant is
\begin{align*}\kappa(G)&=p_{\mu^{1}\sqcup \mu^{2}\sqcup \mu^{3}}(\omega)+p_{(\mu^{1}\Join \mu^{2}\Join \mu^{3})(i,j,k)}(\omega) \\
&\,\,\,\,\,\,- p_{((\mu^{1}\Join \mu^{2})(i,j)) \sqcup \mu^{3}}(\omega)-p_{((\mu^{2}\Join \mu^{3})(j,k)) \sqcup \mu^{1}}(\omega).\end{align*}
In this description, one can permute cyclically the indices $1,2,3$, and this gives $3$ different graphs.\vspace{2mm}
\item One cycle in $\rho_{\mu^{2}}(A^{2})$ is connected to two cycles in $\rho_{\mu^{1}}(A^{1})$ and $\rho_{\mu^{3}}(A^{3})$, and by the same point in this cycle of $\rho_{\mu^{2}}(A^{2})$. In other words, there is an identity  $a^{1}_{s}=a^{2}_{t}=a^{3}_{u}$. This gives again for $\rho_{\mu^{1}}(A^{1})\,\rho_{\mu^{2}}(A^{2})\,\rho_{\mu^{3}}(A^{3})$ a permutation of cycle type $( \mu^{1} \Join \mu^{2} \Join \mu^{3}) (i,j,k)$, but the corresponding cumulant takes now the form
\begin{align*}\kappa(G)&=2\,p_{\mu^{1}\sqcup \mu^{2}\sqcup \mu^{3}}(\omega)+p_{(\mu^{1}\Join \mu^{2}\Join \mu^{3})(i,j,k)}(\omega) - p_{((\mu^{1}\Join \mu^{2})(i,j)) \sqcup \mu^{3}}(\omega)\\
&\,\,\,\,\,\,-p_{((\mu^{2}\Join \mu^{3})(j,k)) \sqcup \mu^{1}}(\omega)-p_{((\mu^{1}\Join \mu^{3})(i,k)) \sqcup \mu^{2}}(\omega).\end{align*}
Here, there is no need to permute cyclically the indices in the enumeration for $N_{G}$.\vspace{2mm}
\item Two distinct cycles in $\rho_{\mu^{2}}(A^{2})$ are connected to a cycle of $\rho_{\mu^{1}}(A^{1})$ and to a cycle of $\rho_{\mu^{3}}(A^{3})$, which gives a permutation of cycle type
\begin{align*}
&( \mu^{1} \Join \mu^{2} \Join \mu^{3}) (i,j;k,l)\\
&=(\mu^{1}\setminus \mu^{1}_{i})\sqcup (\mu^{2}\setminus \{\mu^{2}_{j},\mu^{2}_{k}\})\sqcup (\mu^{3}\setminus \mu^{3}_{l})\sqcup\{\mu^{1}_{i}+\mu^{2}_{j}-1,\mu^{2}_{k}+\mu^{3}_{l}-1\}.
\end{align*}
The cumulant corresponding to this last  case is
\begin{align*}\kappa(G)&=p_{\mu^{1}\sqcup \mu^{2}\sqcup \mu^{3}}(\omega)+p_{(\mu^{1}\Join \mu^{2}\Join \mu^{3})(i,j;k,l)}(\omega)\\
&\,\,\,\,\,\,-p_{((\mu^{1}\Join \mu^{2})(i,j)) \sqcup \mu^{3}}(\omega)-p_{((\mu^{2}\Join \mu^{3})(k,l)) \sqcup \mu^{1}}(\omega),\end{align*}
and again one can permute cyclically the indices $1,2,3$ to get $3$ different graphs.\vspace{2mm}
\end{enumerate}
Consequently:

\begin{proposition}
For any partitions $\mu$, $\nu$ and $\delta$, the limit of $n^{2}\,\kappa(\SubGraph_{\rho_\mu},\SubGraph_{\rho_\nu},\SubGraph_{\rho_\delta})$ is 
\begin{align*}\sum_{\Z/3\Z}&\left(\sum_{i=1}^{\ell(\mu)}\sum_{j=1}^{\ell(\nu)}\sum_{k =1}^{\ell(\delta)} \mu_{i}\,\nu_{j}\,(\nu_{j}-1)\,\delta_{k}\left(\substack{p_{\mu\sqcup \nu\sqcup \delta}(\omega)+p_{(\mu\Join \nu\Join \delta)(i,j,k)}(\omega)\\-p_{((\mu\Join \nu)(i,j)) \sqcup \delta}(\omega)-p_{((\nu\Join \delta)(j,k)) \sqcup \mu}(\omega)}\right)\right.\\
&\left.+\sum_{i=1}^{\ell(\mu)}\sum_{(j\neq k) =1}^{\ell(\nu)} \sum_{l=1}^{\ell(\delta)}\mu_{i}\,\nu_{j}\,\nu_{k}\,\delta_{l}\left(\substack{ p_{\mu\sqcup \nu\sqcup \delta}(\omega)+p_{(\mu\Join \nu\Join \delta)(i,j;k,l)}(\omega)\\-p_{((\mu\Join \nu)(i,j)) \sqcup \delta}(\omega)-p_{((\nu\Join \delta)(k,l)) \sqcup \mu}(\omega)}\right) \right)\\
&\!\!\!\!\!\!\!\!\!\!\!\!+ \sum_{i=1}^{\ell(\mu)}\sum_{j=1}^{\ell(\nu)}\sum_{k =1}^{\ell(\delta)} \mu_{i}\,\nu_{j}\,\delta_{k}\left(\substack{2\,p_{\mu\sqcup \nu\sqcup \delta}(\omega)+p_{(\mu\Join \nu\Join \delta)(i,j,k)}(\omega)-p_{((\mu\Join \nu)(i,j)) \sqcup \delta}(\omega)\\
-p_{((\nu\Join \delta)(j,k)) \sqcup \mu}(\omega)-p_{((\mu\Join \delta)(i,k))\sqcup \nu }(\omega)}\right),\end{align*}
where $\sum_{\Z/3\Z}$ means that one permutes cyclically the partitions $\mu,\nu,\delta$.
\end{proposition}
\noindent In particular, for cycles $\mu=(k)$, $\nu=(l)$ and $\delta=(m)$, $\lim_{n \to \infty} n^{2}\,\kappa(\SubGraph_{k},\SubGraph_{l},\SubGraph_{m})$ is equal to
\begin{align*}
klm&\big((k+l+m-1)\,p_{k,l,m}(\omega)+(k+l+m-2)\,p_{k+l+m-2}(\omega)\\
&-(k+l-1)\,p_{k+l-1,m}(\omega)-(l+m-1)\,p_{l+m-1,k}(\omega)-(l+m-1)\,p_{k+m-1,l}(\omega)\big).
\end{align*}
One recovers for $k=l=m$ the values of $\sigma^{2}$ and $L$ announced in Theorem \ref{thm:modgaussianrcv}. On the other hand, if $\mu=\nu=\delta$, then the limit of $\frac{n^{2}\,\kappa^{(3)}(X_{\rho_{\mu}}^{(n)})}{(p_{\mu}(\omega))^{3}}$ is 
\begin{align*}&\sum_{1 \leq  i,j,k \leq \ell(\mu)} 3\,\mu_{i}\,\mu_{j}^{2}\,\mu_{k}\left(1+\frac{p_{\mu_{i}+\mu_{j}+\mu_{k}-2}(\omega)}{p_{\mu_{i}}(\omega)\,p_{\mu_{j}}(\omega)\,p_{\mu_{k}}(\omega)}-\frac{p_{\mu_{i}+\mu_{j}-1}(\omega)}{p_{\mu_{i}}(\omega)\,p_{\mu_{j}}(\omega)}-\frac{p_{\mu_{j}+\mu_{k}-1}(\omega)}{p_{\mu_{j}}(\omega)\,p_{\mu_{k}}(\omega)}\right)\\
&\quad+\sum_{1\leq i,j\neq k,l \leq \ell(\mu)}3\,\mu_{i}\,\mu_{j}\,\mu_{k}\,\mu_{l}\left( 1- \frac{p_{\mu_{i}+\mu_{j}-1}(\omega)}{p_{\mu_{i}}(\omega)\,p_{\mu_{j}}(\omega)}\right)\left(1 - \frac{p_{\mu_{k}+\mu_{l}-1}(\omega)}{p_{\mu_{k}}(\omega)\,p_{\mu_{l}}(\omega)}\right) \\
&\quad+\sum_{1 \leq  i,j,k \leq \ell(\mu)} \mu_{i}\,\mu_{j}\,\mu_{k}\left(3\,\frac{p_{\mu_{i}+\mu_{j}-1}(\omega)}{p_{\mu_{i}}(\omega)\,p_{\mu_{j}}(\omega)}-2\,\frac{p_{\mu_{i}+\mu_{j}+\mu_{k}-2}(\omega)}{p_{\mu_{i}}(\omega)\,p_{\mu_{j}}(\omega)\,p_{\mu_{k}}(\omega)}-1\right).\end{align*}
\bigskip

\subsection{Asymptotics of the random character values and partitions}\label{subsec:asymptoticrcv}
Fix an integer $k \ge 2$, and consider the random variable
$$V_{n,k}=n^{2/3}\,(\SubGraph_{\rho}-p_{k}(\omega)),$$ 
where $\rho$ is a $k$-cycle. If $p_{2k-1}(\omega)-(p_{k}(\omega))^{2}>0$, then the previous results show that  
$$\esper[\E^{zV_{n,k}}]=\exp\left(\frac{n^{1/3}}{2}\,\sigma^2\,z^2 + \frac{1}{6} \,L \,z^3\right)
(1+o(1))\!,$$ 
where $\sigma$ and $L$ are the quantities given in the statement of Theorem~\ref{thm:modgaussianrcv}.
This ends the proof of Theorem \ref{thm:modgaussianrcv} and Corollary \ref{cor:moderatedeviationcharacter}. Note that the speed of convergence of the cumulants is each time a $O((\alpha_{n})^{-1})$ because of the polynomial behavior established in Lemma \ref{lem:cumulantcharacterpoly}; therefore, one can indeed apply Proposition \ref{prop:largedeviationscumulants} with $\alpha_{n}=n$ and $\beta_{n}=n^{-1}$. The theorem can be extended to other permutations $\rho$ of $\sym(\infty)$
than cycles: if $\rho$ has cycle-type $\mu$, then define
\[V_{n,\mu} = n^{2/3}\,(\SubGraph_{\rho}-p_{\mu}(\omega)).\]
Then the  generating series of $V_{n,\mu}$ is asymptotically given by:
$$\esper[\E^{zV_{n,\mu}}]=\exp\left(\frac{n^{1/3}}{2}\,\sigma^2(\mu)\,z^2 + \frac{1}{6} \,L(\mu) \,z^3\right)\!,$$ 
where $\sigma^2(\mu)=\lim_{n \to \infty}n\,\kappa^{(2)}(\SubGraph_{\rho_{\mu}})$ and $L(\mu)=\lim_{n \to \infty}n^{2}\,\kappa^{(3)}(\SubGraph_{\rho_{\mu}})$ are the limiting quantities given in Section \ref{subsubsec:limitchar_K2K3}.
Hence, \emph{provided that $\sigma^{2}(\mu)>0$}, one has mod-Gaussian convergence of $V_{n,\mu}$, and the limiting variance $\sigma^{2}(\mu)$ is non-zero under the same conditions as those given in Remark \ref{rem:limitvarianceispositive}. Under these conditions, one can also easily establish mod-Gaussian convergence for every vector of 
renormalized random character values $(V_{n,\mu_1},\dots, V_{n,\mu_\ell})$
\bigskip

\begin{remark} There is one case which is not covered by our theorem,
but is of particular interest: the case $\omega=((0,0,\ldots),(0,0,\ldots))$.
This parameter of the Thoma simplex corresponds to the Plancherel measures of the symmetric groups, and in this case, since $p_{2}(\omega)=p_{3}(\omega)=\cdots=0$, the parameters of the mod-Gaussian convergence are all equal to $0$. Indeed, the random character values under Plancherel measures do not have fluctuations of order $n^{-1/2}$. For instance, Kerov's central limit theorem (\emph{cf.} \cite{Hora98,IO02}) ensures that the random character values
$$\frac{n^{k/2}\,\widehat{\chi}^{\lambda}(c_{k})}{\sqrt{k}}$$
on cycles $c_{k}$ of lengths $k \geq 2$ converges in law towards independent Gaussian variables; so the fluctuations are of order $n^{-k/2}$ instead of $n^{-1/2}$. One still expects a mod-Gaussian convergence for adequate renormalizations of the random character values; however, the combinatorics underlying the asymptotics of Plancherel measures are much more complex than those of general central measures --- see \cite{Sni06a} --- and we have not been able to prove mod-Gaussian convergence here.\bigskip
\end{remark}

From the estimates on the laws of the random character values, one can prove many estimates for the parts $\lambda_{1},\lambda_{2},\ldots$ of the random partitions taken under central measures. The arguments of algebraic combinatorics involved in these deductions are detailed in \cite{FM12,Mel12}, so here we shall only state results. Given a partition $\lambda=(\lambda_{1},\ldots,\lambda_{\ell})$ of size $n$, the Frobenius coordinates of $\lambda$ are the two sequences
$$\left(\lambda_{1}-\frac{1}{2},\lambda_{2}-\frac{3}{2},\ldots,\lambda_{d}-d+\frac{1}{2}\right),\left(\lambda_{1}'-\frac{1}{2},\lambda_{2}'-\frac{3}{2},\ldots,\lambda_{d}'-d+\frac{1}{2}\right)$$
where $\lambda_{1}'$, $\lambda_{2}'$, \emph{etc.} are the sizes of the columns of the Young diagram of $\lambda$, and $d$ is the size of the diagonal of the Young diagram. Denote $(a_{1},\ldots,a_{d}),(b_{1},\ldots,b_{d})$ these coordinates, and
$$X_{\lambda}=\sum_{i=1}^{d} \frac{a_{i}}{n}\,\delta_{\left(\frac{a_{i}}{n}\right)}+\sum_{i=1}^{d}\frac{b_{i}}{n}\,\delta_{\left(-\frac{b_{i}}{n}\right)}.$$
This is a (random) discrete probability measure on $[-1,1]$ whose moments
$$p_{k}(\lambda)=n^{k}\,\int_{-1}^{1} x^{k-1}\, X_{\lambda}(dx)$$
are also the moments of the Frobenius coordinates, so $X_{\lambda}$ encodes the geometry of the Young diagram $\lambda$. We shall also need
$$X_{\omega}=\sum_{i=1}^{d} \alpha_{i}\,\delta_{\left(\alpha_{i}\right)}+\sum_{i=1}^{d}\beta_{i}\,\delta_{\left(-\beta_{i}\right)}+\gamma\,\delta_{\left(0\right)},$$
which will appear in a moment as the limit of the random measures $X_{\lambda}$. Here, $\gamma=1-\sum_{i=1}^{\infty}\alpha_{i}-\sum_{i=1}^{\infty}\beta_{i}$. Notice that $\esper_{n,\omega}[\SubGraph_{k}]=\tau_{\omega}(c_{k})=X_{\omega}(x^{k-1})$.\bigskip

It is shown in \cite{IO02} that for any partition $\lambda$ of size $n$ and for any $k$,
$$p_{k}(\lambda)=\varSigma_{k}(\lambda)+\text{remainder},$$
where the remainder is a linear combination of symbols $\varSigma_{\mu}$ with $|\mu|<k$. It follows that the cumulants of the $p_{k}$'s satisfy the same estimates as the cumulants of the $\varSigma_{k}$'s. Therefore, the rescaled random variable 
$\nabla(x^{k-1})=n^{2/3}\,(X_{\lambda}(x^{k-1})-X_{\omega}(x^{k-1}))$
converges in the mod-Gaussian sense with parameters $n^{1/3}\,\sigma^{2}$ and limiting function $\psi(z)=\exp(L\,\frac{z^{3}}{6})$, where $\sigma^{2}$ and $L$ are given by the same formula as in the case of the random character value $\SubGraph_{k}$, that is to say
\begin{align*}
\sigma^{2}&=k^{2}\,\big(p_{2k-1}(\omega)-p_{k}(\omega)^{2}\big);\\
L&= k^{3}\,\big((3k-2)\,p_{3k-2}(\omega)-(6k-3)\,p_{2k-1}(\omega)\,p_{k}(\omega)+(3k-1)\,p_{k}(\omega)^{3}\big).
\end{align*}
Actually, one has mod-Gaussian convergence for any finite vector of random variables $\nabla(x^{k-1})$, $k \geq 1$.
\bigskip

We don't know how to obtain from there moderate deviations for the parts $\lambda_{1},\lambda_{2},\ldots$ of the partition; but one has at least a central limit theorem when one has strict inequalities $\alpha_{1} > \alpha_{2} > \cdots > \alpha_{i} > \cdots$ and $\beta_{1} > \beta_{2} > \cdots > \beta_{i} > \cdots$ (see \cite{Mel12}, and also \cite{Buf12}). Indeed, for any smooth test function $\psi_{i}$ equal to $1$  around $\alpha_{i}$ and to $0$ outside a neighborhood of this point, one has
\begin{equation}
X_{\lambda}(\psi_{i})-X_{\omega}(\psi_{i}) = \frac{a_{i}}{n}-\alpha_{i} \label{eq:almostalwaystrue}
\end{equation}
with probability going to $1$, and then the quantities in the left-hand side renomalized by $\sqrt{n}$ converge jointly towards a Gaussian vector with covariance
\begin{equation}\kappa(i,j)=\delta_{ij}\,\alpha_{i}-\alpha_{i}\alpha_{j}.\label{eq:last}\end{equation}
So, the fluctuations $$\sqrt{n}\left(\frac{\lambda_{i}}{n}-\alpha_{i}\right)$$ of the rows of the random partitions taken under central measures $\proba_{n,\omega}$ converge jointly towards a Gaussian vector with covariances given by Equation \eqref{eq:last}, and one can include in this result the fluctuations
$$\sqrt{n}\left(\frac{\lambda_{j}'}{n}-\beta_{j}\right)$$ of the columns of the random partitions, with a similar formula for their covariances. \bigskip

The reason why it becomes difficult to get by the same technique the moderate deviations of the rows and columns is that in Equation \eqref{eq:almostalwaystrue}, one throws away an event of probability going to zero (because of the law of large numbers satisfied by the rows and the columns, see \emph{e.g.} \cite{KV81}). However, one cannot \emph{a priori} neglect this event in comparison to rare events such as $\{ a_{i}-n\alpha_{i} \geq n^{2/3}\,x\};$ indeed, these rare events are themselves of probability going exponentially fast to zero. Also, there is the problem of approximation of smooth test functions by polynomials, which one has to control precisely when doing these computations. One still conjectures these moderate deviations to hold, and $n^{2/3}\left(\frac{a_{i}}{n}-\alpha_{i}\right)$ to converge in the mod-Gaussian sense with parameters $n^{1/3}(\alpha_{i}-\alpha_{i}^{2})$ and limiting function $$\psi(z)=\exp\left(\frac{\alpha_{i}-3\alpha_{i}^{2}+2\alpha_{i}^{3}}{6}\,z^{3}\right)$$ --- this is what one obtains if we ignore the previous caveats, and still suppose the $\alpha_{i}$ and $\beta_{j}$ all distinct. As explained in \cite{Mel12} (see also \cite{KV86}), this would give moderate deviations for the lengths of the longest increasing subsequences in a random permutation obtained by generalized riffle shuffle.\bigskip
\bigskip
\bigskip

\newpage

\bibliographystyle{alpha}
\bibliography{deviant}

\end{document}